\documentclass[10pt]{article}
\usepackage[english]{babel}
\usepackage{natbib}
\usepackage{url}
\usepackage[utf8]{inputenc}
\usepackage{amsfonts}
\usepackage{amsmath}
\usepackage{mathtools}
\usepackage{empheq}
\usepackage{graphicx}
\usepackage{subcaption}
\usepackage{parskip}
\usepackage{fancyhdr}
\usepackage{bbold}
\usepackage{amssymb}
\usepackage{vmargin}
\usepackage{array}
\usepackage{amsthm}
\usepackage[ruled,vlined]{algorithm2e}
\usepackage{hyperref}
\usepackage{caption}
\usepackage{float}
\usepackage{fourier-orns}
\usepackage{listings}
\usepackage{epsfig}
\usepackage[scr=boondoxo,scrscaled=1.05]{mathalfa}
\usepackage[shortlabels]{enumitem}
\mathtoolsset{showonlyrefs=true}

\newtheorem{theorem}{Theorem}[section]
\newtheorem{corollary}{Corollary}[section]
\newtheorem{lemma}[theorem]{Lemma}
\newtheorem{remark}{Remark}[section]
\theoremstyle{definition}
\newtheorem{definition}{Definition}[section]

\newtheorem{prop}{Proposition}[section]
\usepackage{listings}

\newcommand{\R}{{\mathbb{R}}}

\newcommand{\bpi}{\text{\boldmath{$\Pi$}}}
\newcommand{\blambda}{\text{\boldmath{$\lambda$}}}
\newcommand{\obpi}{\overline{\bpi}}
\newcommand{\bydef}{\stackrel{\mbox{\tiny\textnormal{\raisebox{0ex}[0ex][0ex]{def}}}}{=}}
\usepackage[dvipsnames]{xcolor}

\makeatletter
\newcommand\subsubsubsection{\@startsection{paragraph}{4}{\z@}{-2.5ex\@plus -1ex \@minus -.25ex}{1.25ex \@plus .25ex}{\normalfont\normalsize\bfseries}}
\newcommand\subsubsubsubsection{\@startsection{subparagraph}{5}{\z@}{-2.5ex\@plus -1ex \@minus -.25ex}{1.25ex \@plus .25ex}{\normalfont\normalsize\bfseries}}
\makeatother
\title{Rigorous Validation of Cusp Bifurcations of Stationary Periodic Patterns in Partial Differential Equations}
\author{
Dominic Blanco
\footnote{McGill University, Department of Mathematics and Statistics, 805 Sherbrooke Street West, Montreal, QC, H3A 0B9, Canada. {\tt dominic.blanco@mail.mcgill.ca}}
\and 
Jean-Philippe Lessard \footnote{McGill University, Department of Mathematics and Statistics, 805 Sherbrooke Street West, Montreal, QC, H3A 0B9, Canada. {\tt jp.lessard@mcgill.ca}}}
\begin{document}

\maketitle

\begin{abstract}
In this paper, we present a computer-assisted framework for the rigorous validation of cusp bifurcations of spatially symmetric stationary periodic patterns in parabolic semilinear partial differential equations. Our approach extends to an infinite-dimensional setting the cusp map formulation previously developed in finite dimensions. We formulate the cusp conditions as a zero-finding problem on a Hilbert space of Fourier coefficients, whose non-degenerate solutions correspond to cusp bifurcation points. A key technical ingredient is a careful treatment of the adjoint multiplication operator in the resulting sequence space, which is more involved than in the finite-dimensional case. Starting from a numerically computed approximation, we develop a constructive Newton-Kantorovich argument to prove the existence and local uniqueness of a nearby zero of the cusp map. The non-degeneracy of this solution directly yields the non-vanishing of the cubic normal form coefficient $c$. To complete the verification, we rigorously enclose the spectrum of the linearized operator, confirming that exactly one eigenvalue has zero real part via Gershgorin-type estimates adapted to the symmetry structure of the problem. As a further consequence, the number $k$ of eigenvalues with strictly positive real part at the cusp and the sign of $c$ (both rigorously certified by the framework) together determine the stability structure of the three coexisting solutions inside the cusp region: bistability arises when $k=0$ and $c<0$, monostability when $k=0$ and $c>0$, and no stable solution exists when $k \geq 1$. We apply the method to the Swift-Hohenberg equation and the Gray--Scott system in both one and two spatial dimensions, obtaining rigorous proofs of cusp bifurcations in all five distinct settings; in all five the framework certifies $k=0$, yielding bistability in one case and monostability in the other four.  
\end{abstract}

\begin{center}
\end{center}

\section{Introduction} \label{sec:introduction}

A \emph{cusp bifurcation} is a codimension-two bifurcation that serves as a fundamental organizing center in the global bifurcation landscape of parameter-dependent dynamical systems. It marks the point in two-dimensional parameter space at which two branches of fold bifurcation curves meet tangentially and terminate. The local solution structure near a cusp is universal: in a neighborhood of the cusp point, the parameter plane is divided by two fold curves into a region with three solutions, a region with one solution, and the curves themselves where two solutions collide. In many situations of practical importance, this solution-counting picture gives rise to the fundamental notion of \emph{bistability}: when the two outer equilibria are stable (sinks) and the intermediate one is an unstable saddle with a one-dimensional unstable manifold, the region of three coexisting solutions corresponds to a region of coexistence of two stable states, separated by an unstable one. The system then exhibits a \emph{hysteresis loop}: as parameters are varied across the fold curves, the state undergoes abrupt, irreversible transitions between the two stable branches. This phenomenon has been identified as an organizing principle across a wide range of applications, including ecological pattern formation and desertification transitions \cite{SITEUR201481}, wave pinning and localization phenomena in reaction--diffusion systems \cite{Champneys_Bistability}, and the organization of complex spot and stripe patterns in models of chemical autocatalysis \cite{pearson_gs,gs_original,doleman_pulse}. 

In this paper, we study cusp bifurcations of stationary periodic patterns arising in systems of semi-linear parabolic partial differential equations (PDEs) of the form
\begin{equation}\label{eq : gen PDE}
    \partial_t \mathbf{u} = f(\boldsymbol{\lambda},\mathbf{u}) \bydef
    l_{\boldsymbol{\lambda}}\mathbf{u} + g(\boldsymbol{\lambda},\mathbf{u}), \quad
    \mathbf{u} = (u_1(x,t),\dots,u_p(x,t)), \quad x \in \mathbb{R}^m,
\end{equation}
where $\boldsymbol{\lambda} = (\lambda_1,\lambda_2) \in \mathbb{R}^2$ is a two-dimensional parameter, $l_{\boldsymbol{\lambda}}$ is a linear differential operator depending on $\boldsymbol{\lambda}$, and $g(\boldsymbol{\lambda},\mathbf{u})$ is a nonlinear operator that may also involve spatial derivatives of $\mathbf{u}$. Stationary periodic patterns are time-independent solutions satisfying $\partial_t \mathbf{u}=0$. Equivalently, writing $2r$ for the order of the leading differential operator in $l_{\boldsymbol{\lambda}}$, they are solutions $\mathbf{u} \in H^{2r}(\mathscr{D})^\mathscr{p}$ of
\[
f(\boldsymbol{\lambda},\mathbf{u}) = 0,
\]
posed on a periodic domain $\mathscr{D}$ (e.g. see \eqref{def:mathscr_D} for the formal definition). We choose the domain in such a way that we can encode the spatial symmetry of the pattern (e.g. rolls, hexagons, squares, and more complex dihedral structures in two spatial dimensions) and the periodic setting provides a natural functional framework for spatially extended periodic states. 
Such patterns arise ubiquitously in models of pattern formation, including the Swift-Hohenberg equation \cite{radial2009,hexagon2021lloyd,hexagon2008,stability1997Mielke,squareSakaguchi_1997}, the Gray--Scott reaction--diffusion system \cite{gs_cadiot_blanco,castelli_gs,doleman_pulse,gs_original,mcgough_riley_gs,pearson_gs}, and phase-field crystal models \cite{gabriel_pfc}, among many others.

From a dynamical systems perspective, the bifurcation structure of these periodic states in the two-dimensional parameter space is highly intricate. In particular, cusp points play a central role in organizing regions of multistability and mediating transitions between distinct pattern types \cite{ladder2009Beck,jason_review_paper,jason_spot_paper,jason_ring_paper}. Despite their importance, establishing rigorously that a cusp bifurcation occurs at a specific point in the infinite-dimensional phase space of a PDE (i.e.\ verifying all required non-degeneracy conditions with mathematical certainty) remains largely open. The goal of this work is to rigorously prove the existence and local uniqueness of cusp bifurcation points $(\boldsymbol{\lambda},\mathbf{u})$ of~\eqref{eq : gen PDE}.

Let us recall the conditions that characterize a cusp bifurcation by first briefly discussing a fold (saddle-node) bifurcation. At an equilibrium $(\boldsymbol{\lambda},\mathbf{u})$ of~\eqref{eq : gen PDE}, a fold bifurcation occurs when the Fréchet derivative $D_u f(\boldsymbol{\lambda},\mathbf{u})$ has a simple zero eigenvalue and all remaining eigenvalues have nonzero real part. We denote by $L^2(\mathscr{D})$ the usual $L^2$ space on a bounded domain $\mathscr{D}$. We also let $\mathcal{G}$ be a finite discrete symmetry group. We then denote by $L^2_{\mathcal{G}}(\mathscr{D})$ the restriction of $L^2(\mathscr{D})$ to $\mathcal{G}$-symmetric functions. That is,
\begin{align}
    L^2_{\mathcal{G}}(\mathscr{D}) \bydef \{u \in L^2(\mathscr{D}) ~:~ u(x) = u(\mathscr{g} \cdot x), ~ \text{for all} ~ x \in \mathscr{D}, \mathscr{g} \in \mathcal{G}\},
\end{align}
which is a closed subspace of $L^2(\mathscr{D})$.
We denote by $(\cdot,\cdot)_{L^2}$ the inner product on $L^2_{\mathcal{G}}(\mathscr{D})^{\mathscr{p}}$. That is, given $\mathbf{f} \bydef (f_1,\dots,f_{\mathscr{p}})$ and $\mathbf{g} \bydef (g_1,\dots,g_{\mathscr{p}}) \in (L^2_{\mathcal{G}}(\mathscr{D}))^{\mathscr{p}}$, we define
\begin{align}\label{def : L2 inner product}
    (\mathbf{f},\mathbf{g})_{L^2} \bydef \frac{1}{|\mathscr{D}|} \sum_{j = 1}^{\mathscr{p}} \int_{\mathscr{D}} f_j(x) \mathrm{conj}(g_j(x)) dx,
\end{align}
where $\mathrm{conj}$ denotes complex conjugation and $|\mathscr{D}|$ is the Lebesgue measure of $\mathscr{D}$. The normalisation by $|\mathscr{D}|$ is a pure convention, and in fact none of the quantities below depends on it: replacing $(\cdot,\cdot)_{L^2}$ by $\varrho\,(\cdot,\cdot)_{L^2}$ for some $\varrho>0$ leaves the adjoint $\cdot^\dagger$ unchanged and, because of the constraint $(\mathbf{w},\mathbf{v})_{L^2}=1$, replaces $\mathbf{w}$ by $\varrho^{-1}\mathbf{w}$; the quantities $b$, $c$, $\sigma_i$, $\tau_i$ and $\Delta$ defined below are therefore unchanged. (They do depend on the normalisation chosen for $\mathbf{v}$, which in this paper is fixed by the first row of the cusp map \eqref{def : cusp map}; the conditions $b=0$, $c\neq0$, $\Delta\neq0$ and the sign of $c$ are unaffected.) The advantage of the factor $1/|\mathscr{D}|$ is that it makes the passage to Fourier coefficients an exact identity rather than an identity up to a constant; see Remark~\ref{rmk : L2 vs ell2} in Section~\ref{sec : cusp map}, where $(\cdot,\cdot)_{L^2}$ is expressed in terms of the sequence-space inner product $(\cdot,\cdot)_2$ actually used in the computations.
Throughout, all the function spaces above are understood as \emph{real} Hilbert spaces (the solutions and null vectors we compute are real-valued functions); complex conjugation is retained in \eqref{def : L2 inner product} only because we shall represent real functions by their complex Fourier coefficients. In particular $(\cdot,\cdot)_{L^2}$ is symmetric and bilinear on the real subspace, which is what is used when differentiating the maps below.
We also denote by $\cdot^\dagger$ the adjoint operator defined with respect to this inner product (we reserve $\cdot^*$ for the adjoint with respect to the sequence-space inner product $(\cdot,\cdot)_2$ of Section~\ref{sec : adjoint mult}; the two differ, see Remark~\ref{rmk : L2 vs ell2}). Denoting by $\mathbf{v}$ and $\mathbf{w}$ the (suitably normalized) right and left null vectors ($D_u f(\blambda,\mathbf{u})\mathbf{v} = 0$, $D_u f(\blambda,\mathbf{u})^\dagger\mathbf{w} = 0$, and $(\mathbf{v},\mathbf{w})_{L^2} = 1$), the fold is characterized by the nonvanishing of the quadratic coefficient
\begin{equation} \label{eq:normal_form_coeff_b}
  b \bydef \tfrac{1}{2} \big( \mathbf{w}, D_{uu}f(\boldsymbol{\lambda},\mathbf{u})(\mathbf{v},\mathbf{v}) \big)_{L^2}
\end{equation}
together with a transversality condition with respect to the unfolding parameter.

A cusp bifurcation occurs at a point where $b=0$ while the cubic coefficient below is nonzero, so that two fold curves meet. In this case, the dynamics on the one-dimensional center manifold is governed by the cubic coefficient
\begin{equation} \label{eq:normal_form_coeff_c}
  c \bydef \tfrac{1}{6} \bigl( \mathbf{w},D_{uuu}f(\boldsymbol{\lambda},\mathbf{u})(\mathbf{v},\mathbf{v},\mathbf{v})
    + 3\,D_{uu}f(\boldsymbol{\lambda},\mathbf{u})(\mathbf{v},\mathbf{h})\bigr)_{L^2},
\end{equation}
where $\mathbf{h}$ solves
\[
D_u f(\boldsymbol{\lambda},\mathbf{u})\mathbf{h}
=
-D_{uu}f(\boldsymbol{\lambda},\mathbf{u})(\mathbf{v},\mathbf{v})
\]
in the complement $\{\boldsymbol{\psi} : (\mathbf{w},\boldsymbol{\psi})_{L^2} = 0\}$ of $\operatorname{span}(\mathbf{v})$. This normalization makes $\mathbf{h}$ unique, and it is the one enforced by the third row of the cusp map \eqref{def : cusp map}. Moreover, the value of $c$ is independent of the choice of solution: replacing $\mathbf{h}$ by $\mathbf{h}+t\mathbf{v}$ changes $c$ by $t b=0$. We refer to $c\neq0$ as the nondegeneracy condition. Under this condition (e.g. see \cite{kuznetsov_2,kuznetsov_book}), the center-manifold dynamics takes the canonical cusp normal form $\dot{y}=cy^3+O(y^4)$.

Since a cusp is a codimension-two phenomenon, one further condition is required: the two-parameter family must unfold the degeneracy transversally. To state it, denote by $D_{uu}f(\blambda,\mathbf{u})^\dagger(\mathbf{w},\mathbf{v})$ the vector characterized by $\bigl(D_{uu}f(\blambda,\mathbf{u})^\dagger(\mathbf{w},\mathbf{v}),\mathbf{y}\bigr)_{L^2} = \bigl(\mathbf{w},D_{uu}f(\blambda,\mathbf{u})(\mathbf{v},\mathbf{y})\bigr)_{L^2}$ for all $\mathbf{y}$, and let $\mathbf{z}$ be any solution of
\begin{equation}\label{eq : z vector}
    D_u f(\blambda,\mathbf{u})^\dagger\mathbf{z} = D_{uu}f(\blambda,\mathbf{u})^\dagger(\mathbf{w},\mathbf{v}),
\end{equation}
which is solvable precisely because $b = 0$: by the Fredholm alternative the solvability condition is $\bigl(\mathbf{v},D_{uu}f(\blambda,\mathbf{u})^\dagger(\mathbf{w},\mathbf{v})\bigr)_{L^2} = 2b = 0$. Setting, for $i=1,2$,
\begin{equation}\label{eq : transversality}
   \sigma_i \bydef \bigl(\mathbf{w},\partial_{\lambda_i}f(\blambda,\mathbf{u})\bigr)_{L^2},
   \qquad
   \tau_i \bydef \bigl(\mathbf{w},\partial_{\lambda_i}\!\left[D_u f\right]\!(\blambda,\mathbf{u})\,\mathbf{v}\bigr)_{L^2} - \bigl(\mathbf{z},\partial_{\lambda_i}f(\blambda,\mathbf{u})\bigr)_{L^2},
\end{equation}
the \emph{transversality condition} reads
\begin{equation}\label{eq : Delta}
   \Delta \bydef \sigma_1\tau_2 - \sigma_2\tau_1 \neq 0.
\end{equation}
Note that $\Delta$ does not depend on the choice of $\mathbf{z}$: two solutions of \eqref{eq : z vector} differ by an element of $\ker D_u f(\blambda,\mathbf{u})^\dagger = \mathrm{span}(\mathbf{w})$, which changes $\tau_i$ into $\tau_i - \beta\sigma_i$ for some $\beta \in \R$, and $\Delta$ is invariant under $\tau \mapsto \tau - \beta\sigma$. In terms of the bifurcation function $\phi$ introduced below, $\sigma_i = \partial_{\lambda_i}\phi$ and $\tau_i = \partial_{\lambda_i}\partial_\xi\phi$ (up to such a shift) at the cusp, so \eqref{eq : Delta} states that $\det \frac{\partial(\phi,\partial_\xi\phi)}{\partial(\lambda_1,\lambda_2)} \neq 0$: the parameters unfold the degeneracy in a non-degenerate way \cite{kuznetsov_book}. As shown in Theorem~\ref{thm : cusp map nondegen}, \emph{both} $c \neq 0$ and $\Delta \neq 0$ are automatic consequences of the computer-assisted argument developed in this paper, so neither requires a separate verification. We summarize these conditions in the following lemma, which is the central criterion verified constructively throughout this paper.

\begin{lemma}[\bf Cusp bifurcation] \label{lemma : Cusp bifurcation}
Let $\blambda \in \R^2$ and $\mathbf{u},\mathbf{v},\mathbf{w},\mathbf{h} \in H^{2r}_{\mathcal{G}}(\mathscr{D})^\mathscr{p}$ satisfy $f(\blambda,\mathbf{u}) = 0$, $(\mathbf{w}, \mathbf{v})_{L^2} = 1$, $D_u f(\blambda,\mathbf{u})\mathbf{v} = 0$, $D_u f(\blambda,\mathbf{u})^\dagger\mathbf{w} = 0$, and $D_u f(\blambda,\mathbf{u})\mathbf{h} + D_{uu}f(\blambda,\mathbf{u})(\mathbf{v},\mathbf{v}) = 0$. Suppose that $D_u f(\blambda,\mathbf{u})$ has exactly one zero eigenvalue, which is algebraically simple, and that its remaining spectrum is contained in $\{\mu \in \mathbb{C} : |\mathrm{Re}(\mu)| \geq \varepsilon\}$ for some $\varepsilon > 0$ (such a spectral gap is what the Gershgorin enclosures of Sections~\ref{sec : gershgorin swift hohenberg} and~\ref{sec : gershgorin gray scott} produce, and it is what is used in the proof of Lemma~\ref{lem : bistability}); assume that $b,c$ given in \eqref{eq:normal_form_coeff_b} and \eqref{eq:normal_form_coeff_c} satisfy $b=0$ and $c \ne 0$; and assume that the transversality condition \eqref{eq : Delta} holds. Then a cusp bifurcation occurs at $(\blambda,\mathbf{u})$: there are neighbourhoods $\mathcal{U}$ of $\mathbf{u}$ and $\mathcal{V}$ of $\blambda$ and a smooth change of coordinates and parameters taking $\{(\tilde{\blambda},\tilde{\mathbf{u}}) \in \mathcal{V}\times\mathcal{U} : f(\tilde{\blambda},\tilde{\mathbf{u}}) = 0\}$ onto the canonical cusp $\{\beta_1 + \beta_2 y + \mathrm{sign}(c)\,y^3 = 0\}$.
\end{lemma}
\begin{remark}
The algebraic simplicity of the zero eigenvalue is in fact automatic: a generalized eigenvector would be a solution $\mathbf{y}$ of $D_u f(\blambda,\mathbf{u})\mathbf{y} = \mathbf{v}$, whose existence would force $(\mathbf{w},\mathbf{v})_{L^2} = (D_u f(\blambda,\mathbf{u})^\dagger\mathbf{w},\mathbf{y})_{L^2} = 0$, contradicting $(\mathbf{w},\mathbf{v})_{L^2} = 1$. We keep it in the statement for emphasis. Note also that $H^{2r}_{\mathcal{G}}(\mathscr{D})$ (rather than $L^2$) is the correct domain: the equations $f(\blambda,\mathbf{u}) = 0$ and $D_u f(\blambda,\mathbf{u})\mathbf{v} = 0$ only make sense there.
\end{remark}

The proof of Lemma~\ref{lemma : Cusp bifurcation} (which is presented in Appendix~\ref{sec:proof_lemma_cusp}) shows that the Lyapunov--Schmidt reduction at the cusp point $(\blambda,\mathbf{u})$ yields a scalar bifurcation function $\phi(\xi,\tilde{\blambda})$, where $\xi \in \mathbb{R}$ is the coordinate along the kernel direction $\mathbf{v}$ and $\tilde{\blambda}$ is a parameter value near $\blambda$. Locally near $(\blambda,\mathbf{u})$, stationary solutions of~\eqref{eq : gen PDE} correspond bijectively to roots of $\phi(\cdot,\tilde{\blambda}) = 0$ near $\xi = 0$, with $\phi(\xi,\blambda) = c\xi^3 + O(\xi^4)$. The \emph{fold set} is
\[
  \mathcal{F} \bydef \bigl\{\tilde{\blambda} \text{ near } \blambda \;:\; \phi(\cdot,\tilde{\blambda}) \text{ has a multiple root near } \xi = 0\bigr\},
\]
a closed set containing $\blambda$ itself, which near the cusp consists of two smooth curves meeting tangentially at $\blambda$ (the \emph{fold curves}). The \emph{cusp region} is the connected component of $\{\tilde{\blambda}\text{ near }\blambda\}\setminus\mathcal{F}$ in which $\phi(\cdot,\tilde{\blambda})$ has \emph{exactly} three real roots near $\xi = 0$ (necessarily simple, since $\tilde{\blambda}\notin\mathcal{F}$). These roots $\xi_1 < \xi_2 < \xi_3$ give three nearby stationary solutions
\begin{equation}\label{eq : three solutions}
  \tilde{\mathbf{u}}_j = \mathbf{u} + \xi_j\mathbf{v} + \Psi(\xi_j,\tilde{\blambda}), \quad j = 1,2,3,
\end{equation}
where $\Psi(\xi_j,\tilde{\blambda})$ lies in $\{\boldsymbol{\psi} : (\mathbf{w},\boldsymbol{\psi})_{L^2} = 0\}$, a complement of $\mathrm{span}(\mathbf{v})$ (we write $\Psi$, and not $\mathbf{w}$, to avoid a clash with the left null vector $\mathbf{w}$; $\Psi$ is the map produced by the Lyapunov--Schmidt reduction in Appendix~\ref{sec:proof_lemma_cusp}). For $\tilde{\blambda}$ near $\blambda$ but outside the closure of the cusp region, $\phi(\cdot,\tilde{\blambda})$ has a unique real root near $\xi = 0$, while on the fold curves themselves (away from the cusp point $\tilde{\blambda}=\blambda$, where the root is triple) it has exactly two. We emphasise that this is a purely local picture near $(\blambda,\mathbf{u})$: other bifurcations and additional solutions may well exist elsewhere in the two-dimensional parameter space.

The following lemma (whose proof can be found in Appendix~\ref{sec:proof_lemma_stability}) gives the complete stability picture for the three solutions~\eqref{eq : three solutions}. Its conclusion depends on two quantities: the number $k \geq 0$ of eigenvalues of $D_u f(\blambda,\mathbf{u})$ with strictly positive real part (counted with algebraic multiplicity), and the sign of $c$. Both will be certified rigorously by our computer-assisted framework, as described below.

\begin{lemma}[\bf Stability structure near a cusp]\label{lem : bistability}
Under the hypotheses of Lemma~\ref{lemma : Cusp bifurcation}, let $k \geq 0$ be the (necessarily finite) number of eigenvalues of $D_u f(\blambda,\mathbf{u})$ with strictly positive real part (counted with algebraic multiplicity; the spectral gap $\varepsilon>0$ is part of the hypotheses of Lemma~\ref{lemma : Cusp bifurcation}, while $k<\infty$ follows from sectoriality of $D_uf(\blambda,\mathbf{u})$, whose spectrum is discrete and contained in a left half-plane sector, so that $\sigma(D_uf(\blambda,\mathbf{u}))\cap\{\mathrm{Re}\,\mu\geq-\varepsilon\}$ is bounded), and let $\tilde{\blambda}$ lie in the cusp region, sufficiently close to $\blambda$. The three solutions $\tilde{\mathbf{u}}_1, \tilde{\mathbf{u}}_2, \tilde{\mathbf{u}}_3$ defined in~\eqref{eq : three solutions} have the following unstable eigenvalue counts:
\begin{itemize}[leftmargin=1.5em]
\item if $c < 0$, then $D_u f(\tilde{\blambda},\tilde{\mathbf{u}}_1)$ and $D_u f(\tilde{\blambda},\tilde{\mathbf{u}}_3)$ each have exactly $k$ eigenvalues with positive real part, while $D_u f(\tilde{\blambda},\tilde{\mathbf{u}}_2)$ has exactly $k+1$;
\item if $c > 0$, then $D_u f(\tilde{\blambda},\tilde{\mathbf{u}}_1)$ and $D_u f(\tilde{\blambda},\tilde{\mathbf{u}}_3)$ each have exactly $k+1$ eigenvalues with positive real part, while $D_u f(\tilde{\blambda},\tilde{\mathbf{u}}_2)$ has exactly $k$.
\end{itemize}
In particular, when $k = 0$,
\begin{itemize}[leftmargin=1.5em]
\item if $c < 0$, then the system exhibits \emph{bistability} ($\tilde{\mathbf{u}}_1$ and $\tilde{\mathbf{u}}_3$ are asymptotically stable, $\tilde{\mathbf{u}}_2$ is unstable) (e.g. see Figure~\ref{fig : cusp});
\item if $c > 0$, then only $\tilde{\mathbf{u}}_2$ is asymptotically stable and the other two are unstable (\emph{monostability}).
\end{itemize}
When $k \geq 1$, none of the three solutions is asymptotically stable.
\end{lemma}

The statements about asymptotic stability are understood with respect to the semiflow generated by \eqref{eq : gen PDE}; they follow from the eigenvalue counts by the principle of linearized stability, since $l_{\blambda}$ generates an analytic semigroup and the Gershgorin enclosures of Sections~\ref{sec : gershgorin swift hohenberg} and \ref{sec : gershgorin gray scott} place the remaining spectrum at a positive distance from the imaginary axis \cite{henry_geometric,lunardi_analytic}.

\begin{figure}[ht]
\centering \includegraphics[width=0.5\textwidth]{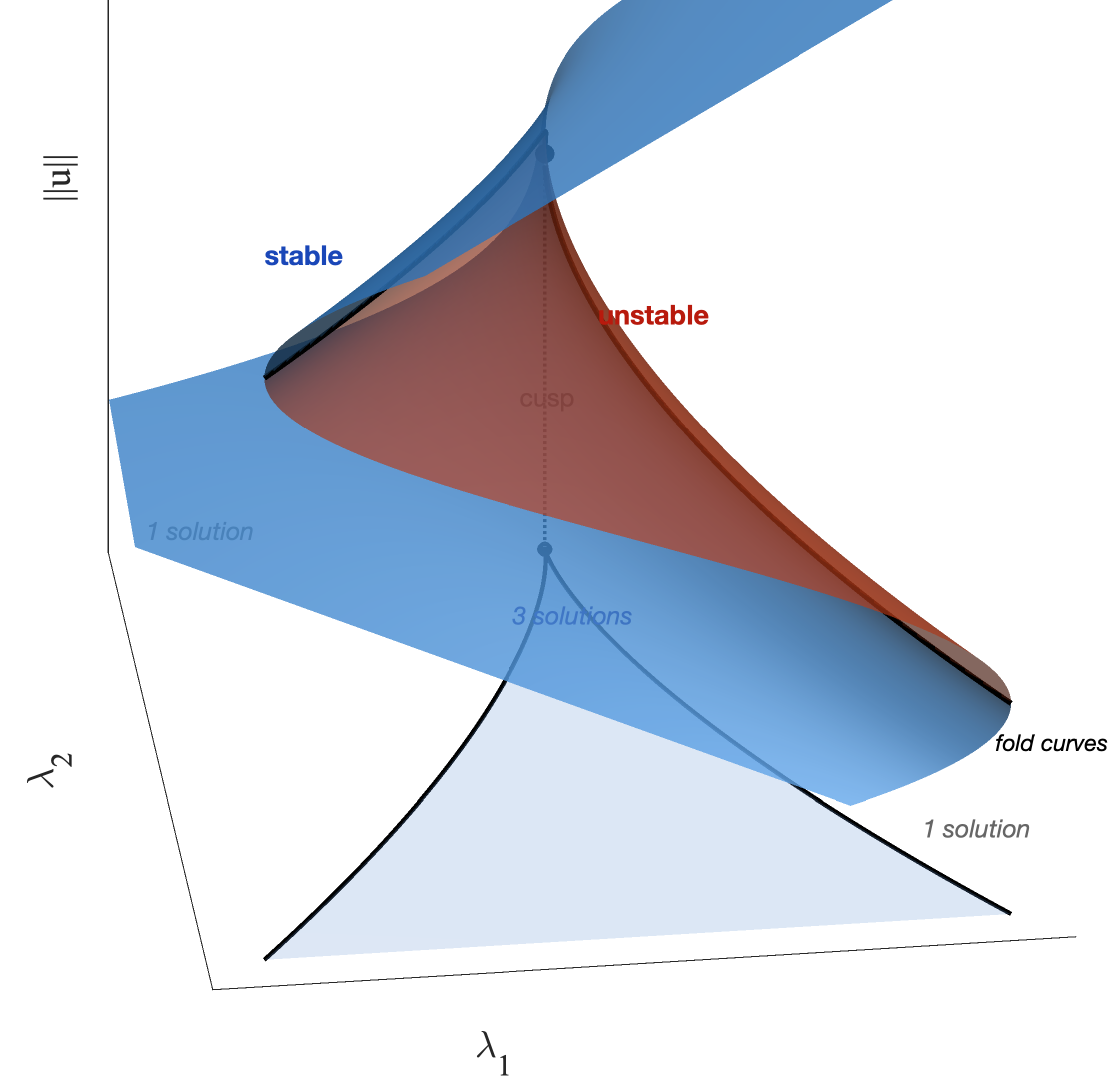} \caption{Equilibrium manifold for the cusp normal form $\dot{u}=f(\boldsymbol{\lambda},u) \bydef -u^3 + \lambda_2 u + \lambda_1$ (equilibria being the solutions of $f(\boldsymbol{\lambda},u)=0$) (case $k=0$, $c<0$). The surface shows the set of stationary solutions $(\lambda_1,\lambda_2, u)$ satisfying $f(\boldsymbol{\lambda},u)=0$, folded over the two-dimensional parameter plane $(\lambda_1,\lambda_2)$. \emph{Blue}: asymptotically stable solutions ($k=0$ eigenvalues with positive real part); \emph{red}: unstable solution (one eigenvalue with positive real part). The two black curves on the surface are the fold curves, whose projection onto the parameter plane forms the cusp. Inside the shaded region the system exhibits bistability ($\mathbf{u}_1$ and $\mathbf{u}_3$ stable, $\mathbf{u}_2$ unstable); outside it only one solution exists. } \label{fig : cusp}
\end{figure}

The use of computer-assisted proofs in bifurcation theory is now well established, with applications spanning saddle-node bifurcations \cite{continue3,blanco_cadiot_fassler_saddle,d13ba7de9191432e9c6e4912a64690bb,kepley_saddle,jp_saddle_node1,jp_saddle_node,tom_bifurcation,PiotrZgliczynsk2015JournalofComputationalDynamics}, symmetry-breaking bifurcations \cite{d13ba7de9191432e9c6e4912a64690bb,jp_saddle_node1,cyclic_sym,tom_bifurcation,PiotrZgliczynsk2015JournalofComputationalDynamics}, period-doubling bifurcations \cite{period_doubling}, cocoon bifurcations \cite{Kokubu_2007}, and Hopf bifurcations \cite{kevin_hopf,CHURCH2022133072,elena_otherhopf,elena_dissertation,wright_conjecture,jbjpelena_hopf,VANDENBERG2022106133}. These developments demonstrate the effectiveness of rigorous numerics for validating bifurcation phenomena in both finite- and infinite-dimensional dynamical systems. In the specific context of cusp bifurcations, the work of \cite{cusp_jp} developed a framework for verifying the hypotheses of Lemma~\ref{lemma : Cusp bifurcation} in finite-dimensional settings. The present paper extends this framework to stationary periodic solutions of PDEs. For broader introductions to computer-assisted proofs, we refer the reader to \cite{MR3444942,plum_numerical_verif,MR3990999,maxime_general}.

The strategy is as follows. All algebraic conditions of Lemma~\ref{lemma : Cusp bifurcation}, except the non-degeneracy condition $c \neq 0$ and the spectral hypothesis on $D_u f(\blambda,\mathbf{u})$, are encoded as zeros of an augmented system (the \emph{cusp map} $F(\mathbf{x})=0$, defined in Section~\ref{sec : cusp map}). Non-degenerate zeros of this system which satisfy the spectral hypothesis are cusp bifurcation points. Starting from a numerically computed approximation $\bar{\mathbf{x}}$, we apply a Newton--Kantorovich argument (Theorem~\ref{th : radii polynomial theorem periodic}) to establish the existence and local uniqueness of a true solution $\tilde{\mathbf{x}}$ nearby. A key feature of this approach is that non-degeneracy of $\tilde{\mathbf{x}}$ follows directly from the NK theorem, which in turn implies both $c \neq 0$ and the transversality condition \eqref{eq : Delta} by Theorem~\ref{thm : cusp map nondegen}. In addition, the NK construction provides a rigorous enclosure of the solution, which allows the sign of $c$ to be certified using interval arithmetic. The remaining spectral condition in Lemma~\ref{lemma : Cusp bifurcation}, namely that $D_u f$ has exactly one center eigenvalue, is verified independently via a Gershgorin circle argument (Sections~\ref{sec : gershgorin swift hohenberg} and \ref{sec : gershgorin gray scott}). This argument also determines the number $k$ of eigenvalues with positive real part. Finally, with $k$ and the sign of $c$ established, Lemma~\ref{lem : bistability} yields the complete stability characterization of the three nearby solutions.

To make the strategy above concrete (constructing the cusp map, verifying the contraction condition for the NK argument, and computing the Gershgorin enclosures) one needs a representation of the problem in which all quantities can be computed and bounded explicitly. For stationary periodic patterns, the natural setting is a \emph{sequence space of Fourier coefficients}. Since $\mathbf{u}$ is periodic on $\mathscr{D}$, it admits a Fourier series expansion, and the equation $f(\boldsymbol{\lambda},\mathbf{u})=0$ can be reformulated as an infinite-dimensional algebraic system on the sequence space $\ell^2(\mathbb{Z}^m)$. By Plancherel's theorem and Parseval's identity, $L^2(\mathscr{D})$ and $\ell^2(\mathbb{Z}^m)$ are naturally isometrically isomorphic (with the normalisation \eqref{def : L2 inner product}). Under $\mathcal{G}$-symmetry the picture is slightly more subtle: a $\mathcal{G}$-symmetric function is stored through the reduced set of coefficients indexed by $\mathcal{Z}_{\mathrm{red}}(\mathcal{G}) \subsetneq \mathbb{Z}^m$ introduced in Section~\ref{sec : periodic spaces}, and under this identification the $L^2_{\mathcal{G}}(\mathscr{D})$ inner product becomes the \emph{orbit-weighted} sequence inner product $(\mathcal{O}_{\mathcal{G}}\cdot,\cdot)_2$, where $\mathcal{O}_{\mathcal{G}}$ is the diagonal operator of orbit sizes defined in \eqref{def : D operator}; see Remark~\ref{rmk : L2 vs ell2}. The computations, on the other hand, are carried out in the polynomially weighted space $\ell^2_{\mathcal{G},\omega}$, which is \emph{not} isometric to $L^2_{\mathcal{G}}(\mathscr{D})$ but is continuously embedded in it. Keeping track of the operator $\mathcal{O}_{\mathcal{G}}$ is precisely what makes it possible to verify the $L^2$-conditions of Lemma~\ref{lemma : Cusp bifurcation} by solving a sequence-space equation; this is made explicit in Remark~\ref{rmk : L2 vs ell2}. The Fourier representation is particularly advantageous computationally: differential operators become diagonal multiplications, polynomial nonlinearities become discrete convolutions, and the Newton--Kantorovich estimates reduce to explicit algebraic bounds on sequences.

A crucial structural feature emerges in the tail: for sufficiently high Fourier modes, the diagonal operator $l_{\boldsymbol{\lambda}}$ dominates the convolution terms, so that $D_u f$ becomes asymptotically diagonally dominant. More precisely, the tail of $D_u f$ is represented by an infinite-dimensional matrix whose diagonal dominance can be quantified explicitly. This structure permits the construction of an approximate inverse $A$ consisting of a finite-dimensional block, computed and inverted numerically, together with an explicitly controlled diagonal tail. Variants of this strategy have proved highly effective for a broad class of PDEs \cite{dominic_sh_periodic,dominic_thomas,gabriel_pfc,breden2022computer,doi:10.1137/17M111938X,MR3633778,period_kuramoto,lessard_2,lessard_1,Sander_equilibrium,JB_symmetries_1,JB_symmetries_2,lindsay_suspensionbrdige}. In situations where diagonal dominance fails in the tail, alternative constructions can be employed \cite{MR3392647,MR4379799,cyranka2018construction}.

Extending the finite-dimensional cusp map framework of \cite{cusp_jp} to the PDE setting introduces two key challenges that do not arise in finite dimensions. 

The first concerns the treatment of the adjoint operator. In $L^2_{\mathcal{G}}(\mathscr{D})$ the left null vector is defined by $D_u f(\boldsymbol{\lambda},\mathbf{u})^\dagger\mathbf{w}=0$; in the sequence space it is the $(\cdot,\cdot)_2$-adjoint $D_uf(\boldsymbol{\lambda},\mathbf{u})^*$ that is computed, and the corresponding null vector is $\mathcal{O}_{\mathcal{G}}\mathbf{w}$ (Remark~\ref{rmk : L2 vs ell2}). In finite dimensions, the adjoint is simply given by matrix transposition. In the present sequence-space setting, however, the adjoint of the linearized operator involves the adjoint of a multiplication operator by a Fourier series. Characterizing this operator requires a careful analysis, particularly when symmetry constraints are imposed on the Fourier coefficients. In Section~\ref{sec : adjoint mult}, we develop the necessary theory and derive explicit formulas for the adjoint multiplication operator on $\ell^2$, which are subsequently used in the Newton--Kantorovich estimates. 

The second challenge concerns spatial symmetry. Solutions of~\eqref{eq : gen PDE} with periodic boundary conditions are generally not isolated, since every spatial translate of a solution is itself a solution. To obtain a well-posed zero-finding problem, we restrict our attention to solutions possessing a prescribed {\em finite discrete symmetry group} $\mathcal{G}$. This symmetry reduction significantly decreases the computational cost, as the Fourier coefficients of a $\mathcal{G}$-symmetric function are determined by a {\em reduced set} of independent modes $\mathcal{Z}_{\mathrm{red}}(\mathcal{G}) \subsetneq \mathbb{Z}^m$, corresponding to orbit representatives under the action of $\mathcal{G}$. The resulting framework, developed in \cite{JB_symmetries_1,JB_symmetries_2} and adapted here to the cusp setting, applies uniformly to dihedral, rectangular, hexagonal, and other crystallographic symmetries. The associated Banach space $\ell^2_{\mathcal{G},\omega}$ is described in Section~\ref{sec : periodic spaces}.

Beyond its computational advantages, symmetry reduction can also remove translational degeneracies and thereby isolate solutions. In all applications considered in this work, the imposed symmetry is sufficient to achieve this isolation. More generally, when symmetry alone does not eliminate the degeneracy, one may instead introduce unfolding parameters and analyze the resulting augmented problem. Although this situation lies beyond the scope of the present paper, we refer the interested reader to \cite{symmetry_blanco_cadiot,jay_unfold1,jay_unfold4,jay_unfold3,jay_unfold2} for further discussion.

The same $\mathcal{G}$-symmetry structure also shapes the spectral analysis needed to complete the verification of Lemma~\ref{lemma : Cusp bifurcation}: certifying that $D_u f$ has exactly one eigenvalue on the imaginary axis. For this we employ a \emph{Gershgorin circle argument} adapted to the infinite-dimensional, symmetry-structured setting. Originally introduced for finite matrices in \cite{gershgorin_original}, the Gershgorin circle theorem provides explicit regions containing the spectrum of an operator. Extensions to infinite diagonally dominant matrices were developed in \cite{FARID19917,Farid_Lancaster_1989,kato2013perturbation}, and related techniques have been used in computer-assisted proofs in \cite{continue3,gershgorin_bredencadiotzurek,cadiot2025stabilityanalysislocalizedsolutions,cusp_jp,marco_thesis,gershgorin_maxime2,JB_symmetries_2}. The present setting introduces an additional difficulty: after imposing $\mathcal{G}$-symmetry, the linearized operator no longer possesses the standard Fourier--convolution structure, and the resulting symmetry constraints couple Fourier modes according to the orbit structure of $\mathcal{G}$. To overcome this issue, we decompose the spectral analysis into three distinct components, each requiring a different treatment (see Sections~\ref{sec : gershgorin swift hohenberg} and \ref{sec : gershgorin gray scott}).

We apply this framework to two paradigmatic models. The first is the \emph{Swift-Hohenberg equation} \cite{radial2009,hexagon2008,stability1997Mielke}, a scalar PDE ($\mathscr{p}=1$). The second is the \emph{Gray--Scott reaction--diffusion system} \cite{doleman_pulse,gs_original,pearson_gs} ($\mathscr{p}=2$), which illustrates the applicability of the method to systems of PDEs. In both cases, we consider one- and two-dimensional spatial domains, thereby obtaining rigorous proofs of cusp bifurcations in five distinct settings. One of our main results is the following.

\begin{theorem}[$D_2$ cusp bifurcation in 2D Gray-Scott]\label{th : GS 2D cusp intro}
Fix $\delta_1 = 0.08583846158364704$, $\delta_2 = 1$ and $d = 1.5$. There exist $\tilde{\blambda} \approx (1,4.183827930189122)$ and a $D_2$-symmetric (rectangular) stationary solution $\tilde{\mathbf{u}}$ of the 2D Gray-Scott system, unique in an explicit ball of radius $3\times 10^{-12}$, such that $(\tilde{\blambda},\tilde{\mathbf{u}})$ is a non-degenerate cusp bifurcation point with $k = 0$ and $c > 0.1796$. In particular the three coexisting solutions exhibit monostability. The precise statement is Theorem~\ref{th : GS 2D cusp}.
\end{theorem}

\begin{figure}[H]
\centering
 \begin{minipage}{.5\linewidth}
  \centering\epsfig{figure=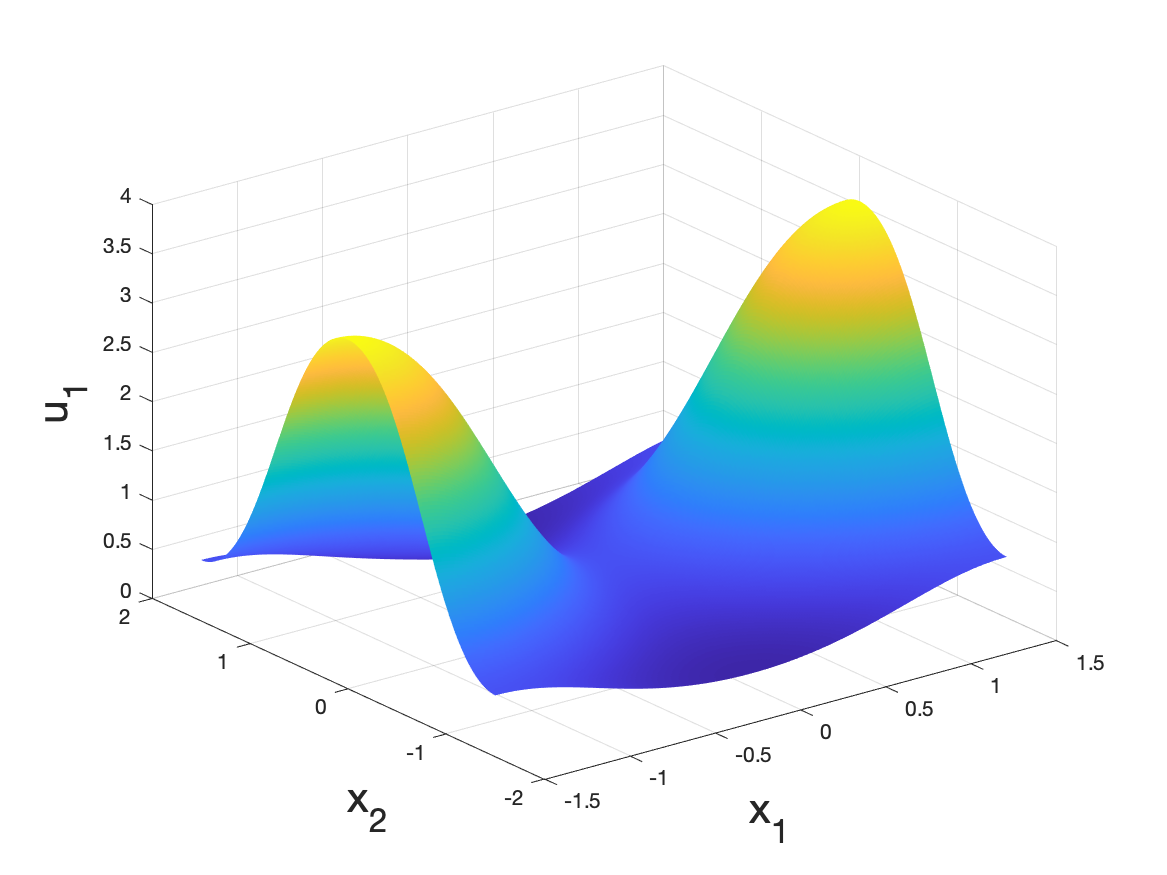,width=\linewidth}
  \end{minipage}%
 \begin{minipage}{.5\linewidth}
  \centering\epsfig{figure=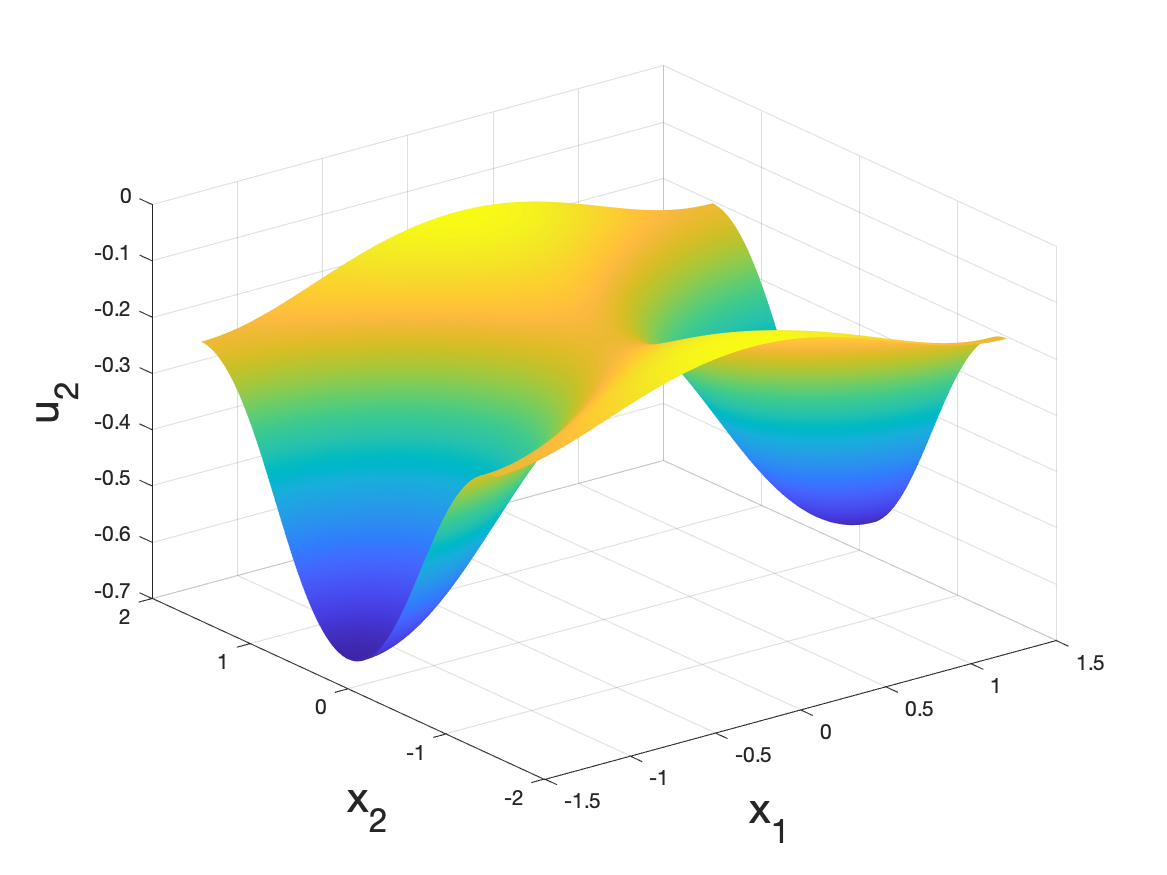,width=\linewidth}
 \end{minipage} 
 \caption{Plot of an approximation of $\tilde{\mathbf{u}}$ used in Theorem \ref{th : GS 2D cusp intro}. We denote this approximation $\overline{\mathbf{u}} = (\overline{u}_1,\overline{u}_2)$. We have $\overline{u}_1$ (L) and $\overline{u}_2$ (R) on $(-1.5,1.5)^2$.}\label{fig : th intro}
 \end{figure}%

Analogous theorems are proved for the 1D Swift-Hohenberg equation with $\mathbb{Z}_2$ symmetry (Theorem~\ref{th : SH 1D cusp}), the 2D Swift-Hohenberg equation with square $D_4$ and hexagonal $D_6$ symmetry (Theorems~\ref{th : SH 2D cusp D4} and \ref{th : SH 2D cusp D6}), and the 1D Gray-Scott system with $\mathbb{Z}_2$ symmetry (Theorem~\ref{th : GS 1D cusp}), for a total of five settings. In all five the framework certifies $k = 0$; the 1D Swift-Hohenberg cusp has $c < 0$ and hence exhibits bistability, while the other four have $c > 0$ and exhibit monostability. The computer code implementing all the proofs is publicly available at \cite{CuspProofs.jl}.

The remainder of this paper is organized as follows. Section~\ref{sec : set-up} introduces the necessary notation and background results: the symmetry-reduced Fourier sequence space $\ell^2_{\mathcal{G},\omega}$ and its convolution estimates (Section~\ref{sec : periodic spaces}), the adjoint multiplication operator in $\ell^2_{\mathcal{G},\omega}$ (Section~\ref{sec : adjoint mult}), the cusp map (Section~\ref{sec : cusp map}), and the Newton--Kantorovich theorem used throughout (Section~\ref{sec : radii poly}). Section~\ref{sec : swift hohenberg} presents the application to the Swift-Hohenberg equation: the numerical approximation and approximate inverse are constructed in Section~\ref{sec : numerical aspects}, the Newton--Kantorovich bounds are computed in Section~\ref{sec : bounds sh}, the Gershgorin spectral enclosure is developed in Section~\ref{sec : gershgorin swift hohenberg}, and the computer-assisted proofs are given in Section~\ref{sec : caps sh}. Section~\ref{sec : gray scott} presents the application to the Gray--Scott system, following the same structure; Section~\ref{sec : gershgorin gray scott} extends the Gershgorin enclosure of Section~\ref{sec : gershgorin swift hohenberg} to the block setting required when the linearization is not self-adjoint, so that the spectrum need not be real. We conclude in Section~\ref{sec : conclusion} with a discussion of open questions.

\section{Functional Setting and the Cusp Map}\label{sec : set-up}
In this section, we collect the notations and results needed for our computer-assisted approach. Section~\ref{sec : periodic spaces} introduces the symmetry-reduced Fourier sequence spaces $\ell^2_{\mathcal{G},\omega}$. Section~\ref{sec:conv_estimates} establishes convolution estimates on these spaces; since the PDEs under consideration have polynomial nonlinearities, products of sequences appear throughout and controlling their norms in $\ell^2_{\mathcal{G},\omega}$ is essential for all subsequent nonlinear analysis. Section~\ref{sec : adjoint mult} characterizes the adjoint multiplication operator in $\ell^2_{\mathcal{G},\omega}$, an ingredient needed to formulate the cusp map in Section~\ref{sec : cusp map}, whose non-degenerate zeros correspond to cusp bifurcation points. Finally, Section~\ref{sec : radii poly} states the Newton--Kantorovich theorem used to validate these zeros.

\subsection{Symmetry-Reduced Fourier Sequence Spaces}\label{sec : periodic spaces}
We seek solutions with a prescribed $\mathcal{G}$-symmetry. In this manuscript, we focus on the case where $\mathcal{G}$ is a finite discrete group which corresponds to a certain symmetry (i.e. squares, hexagons, etc). Since  $\mathcal{G}$ is a finite discrete symmetry group, it has a unitary representation. That is, for each $\mathscr{g} \in \mathcal{G}$, there exists an $\mathcal{A}_{\mathscr{g}} \in M_{m \times m}(\mathbb{R})$ and $ b_{\mathscr{g}} \in [0,1]^m$ such that
\[
\mathscr{g} \cdot x = \mathcal{A}_{\mathscr{g}} x + b_{\mathscr{g}},
\]
for every  $x \in \mathbb{R}^m$. We will base our theory around these types of groups. We now fix an invertible matrix $\mathcal{L} : \mathbb{R}^m \to \mathbb{R}^m$ such that $\mathcal{L}^{-1} \mathcal{A}_{\mathscr{g}}^T \mathcal{L} \in M_{m \times m}(\mathbb{Z})$ for every $\mathscr{g} \in \mathcal{G}$. This is exactly what is required for the map $\beta_{\mathscr{g}}(n) = \mathcal{L}^{-1}\mathcal{A}_{\mathscr{g}}^T\mathcal{L}n$ of Lemma~\ref{lem : G-fourier series} to preserve the Fourier index lattice $\mathbb{Z}^m$. (Quantified over all of $\mathcal{G}$, the requirement is unchanged if $\mathcal{A}_{\mathscr{g}}^T$ is replaced by $\mathcal{A}_{\mathscr{g}}$, since the representation is orthogonal, so that $\mathcal{A}_{\mathscr{g}}^T = \mathcal{A}_{\mathscr{g}}^{-1} = \mathcal{A}_{\mathscr{g}^{-1}}$, and $\mathscr{g}\mapsto\mathscr{g}^{-1}$ permutes $\mathcal{G}$; for a fixed $\mathscr{g}$, however, the two matrices are inverse to one another and are not the same. The condition stated with $\mathcal{L}^{-T}$ in place of $\mathcal{L}^{-1}$ is \emph{not} equivalent, and in fact fails here: for the hexagonal lattice $\mathcal{L} = \left[\begin{smallmatrix}1 & -1/2\\ 0 & \sqrt{3}/2\end{smallmatrix}\right]$ of Theorem~\ref{th : SH 2D cusp D6} and $\mathcal{A}$ the rotation by $60^\circ$ one computes $\mathcal{L}^{-T}\mathcal{A}\mathcal{L} = \left[\begin{smallmatrix}1/2 & -1\\ 1+\frac{1}{2\sqrt3} & -\frac{1}{\sqrt3}\end{smallmatrix}\right] \notin M_{2}(\mathbb{Z})$, whereas $\mathcal{L}^{-1}\mathcal{A}^T\mathcal{L} = \left[\begin{smallmatrix}0&1\\-1&1\end{smallmatrix}\right] \in M_2(\mathbb{Z})$.) This is so that the action of $\mathcal{G}$ preserves the Fourier index lattice $\mathbb{Z}^m$. This condition is satisfied for all symmetry groups considered in this work when $\mathcal{L}$ is chosen as the corresponding lattice generator. Now, let $ d > 0$ and introduce the shorthand
\[
    \tilde{n} \bydef \left(\frac{n_1 \pi}{d},\frac{n_2 \pi}{d},\dots,\frac{n_m \pi}{d}\right),
\]
used throughout. We also let the domain, $\mathscr{D}$ be defined as
\begin{equation} \label{def:mathscr_D}
    \mathscr{D} \bydef \{\mathcal{L}^{-T}x : ~ x \in (-d,d)^m\}.
\end{equation}
Henceforth we identify $\mathbf{u}$ with its Fourier coefficient sequence, that is, $\mathbf{u} = (\mathbf{u}_n)_{n \in \mathbb{Z}^m}$ where $\mathbf{u}_n \bydef ((u_1)_n,(u_2)_n,\dots,(u_p)_n)$. We now state the following key lemma.
\begin{lemma}\label{lem : G-fourier series}
Let $u \in L^2(\mathscr{D})$ be a function with a Fourier series representation of the form 
\begin{equation}\label{usual_fourier_series}
    u(x) = \sum_{n \in \mathbb{Z}^m} u_n e^{i (\mathcal{L} \tilde{n}) \cdot x}.
\end{equation}
Define $\beta_{\mathscr{g}} : \mathbb{Z}^m \to \mathbb{Z}^m$ and $\alpha_{\mathscr{g}} : \mathbb{Z}^m \to \mathbb{R}$ as
\begin{align}
    \beta_{\mathscr{g}}(n) \bydef \mathcal{L}^{-1} \mathcal{A}_{\mathscr{g}}^T \mathcal{L}n
    \qquad \text{and} \qquad \alpha_{\mathscr{g}}(n) \bydef \exp\left(\frac{\pi i}{d} \beta_{\mathscr{g}}(n) \cdot b_{\mathscr{g}}\right).
\end{align}
Then, it follows that $u(x) = u(\mathscr{g} \cdot x)$ for all $\mathscr{g} \in \mathcal{G}$ and $x \in \mathbb{R}^m$ if and only if $ \alpha_{\mathscr{g}}(n) u_{\beta_{\mathscr{g}}(n)} = u_n$. In this case, we say that $u$ has a $\mathcal{G}$-Fourier series representation.
\end{lemma}
\begin{proof}
The proof can be found in \cite{JB_symmetries_2}. 
\end{proof} 
Lemma~\ref{lem : G-fourier series} translates symmetry of a function into an algebraic condition on its Fourier coefficients, and applies to any finite discrete symmetry group in any dimension. Building on this, the authors of \cite{JB_symmetries_1,JB_symmetries_2} developed an algorithm for identifying a minimal collection of independent Fourier coefficients sufficient to represent any $\mathcal{G}$-symmetric function. This relies on choosing one representative from each orbit (cf.~\cite{gallianalgebra}), where 
\begin{equation} \label{def : orbit}
    \mathrm{orb}_{\mathcal{G}}(n) = \{\mathscr{g} \cdot n, \ \mathrm{for \ all} \ \mathscr{g} \in \mathcal{G}\}.
\end{equation}
A set containing exactly one representative from each orbit is called a \emph{fundamental domain}, denoted $\mathcal{Z}_{\mathrm{dom}}(\mathcal{G}) \subset \mathbb{Z}^m$. Its utility is immediate: for any $\mathcal{G}$-symmetric function $u$ with coefficients $(u_n)_{n \in \mathbb{Z}^m}$, the relation $u_n = \alpha_{\mathscr{g}}(n)u_{\beta_{\mathscr{g}}(n)}$ for all $\mathscr{g} \in \mathcal{G}$ means that the coefficients indexed by $\mathcal{Z}_{\mathrm{dom}}(\mathcal{G})$ already determine the full sequence. Additionally, the same framework identifies the \emph{trivial set} $\mathcal{Z}_{\mathrm{triv}}(\mathcal{G})$, consisting of all indices $n$ whose coefficient is forced to be zero by the symmetry constraints imposed by $\mathcal{G}$. Setting
\begin{align}
    \mathcal{Z}_{\mathrm{sym}}(\mathcal{G}) \bydef \mathbb{Z}^m \setminus \mathcal{Z}_{\mathrm{triv}}(\mathcal{G}),
\end{align}
the reduced set of Fourier coefficients is then defined as
 \begin{align}
     \mathcal{Z}_{\mathrm{red}}(\mathcal{G}) \bydef \mathcal{Z}_{\mathrm{dom}}(\mathcal{G}) \cap \mathcal{Z}_{\mathrm{sym}}(\mathcal{G}).\label{def : reduced set}
 \end{align}
We now have the following corollary.
\begin{corollary}\label{corr : reduced_set}
Let $u$ be a $\mathcal{G}$-Fourier series where $\mathcal{G}$ is a finite discrete symmetry group. Suppose that $u_j = u_n$ for all $j \in \mathrm{orb}_{\mathcal{G}}(n)$.
Then, the $\mathcal{G}$-Fourier series of $u$ can be written with indices in $\mathcal{Z}_{\mathrm{red}}(\mathcal{G})$. More specifically,
\begin{equation} \label{def : H_fourier series}
    u(x) = \sum_{n \in \mathcal{Z}_{\mathrm{red}}(\mathcal{G})} u_n \sum_{k \in \mathrm{orb}_{\mathcal{G}}(n)} e^{i \mathcal{L}\tilde{k} \cdot x}.
\end{equation}
\end{corollary}
\begin{proof}
The proof can be found in \cite{JB_symmetries_2}.
\end{proof}
\begin{remark}
The formula \eqref{def : H_fourier series} uses $u_j = u_n$ for all $j \in \mathrm{orb}_{\mathcal{G}}(n)$, which follows from $\mathcal{G}$-symmetry only when $\alpha_{\mathscr{g}}(n) = 1$ for all $\mathscr{g} \in \mathcal{G}$ and $n \in \mathbb{Z}^m$. This holds when all group elements act without a translational component (i.e., $b_{\mathscr{g}} = 0$ in $\mathscr{g} \cdot x = \mathcal{A}_{\mathscr{g}}x + b_{\mathscr{g}}$), as is satisfied by all symmetry groups considered in this work.
\end{remark}

Thus, $\mathcal{Z}_{\mathrm{red}}(\mathcal{G})$ provides a minimal set of Fourier modes from which any $\mathcal{G}$-symmetric function can be reconstructed. Restricting to this reduced index set offers two advantages. First, since $\mathcal{Z}_{\mathrm{red}}(\mathcal{G})$ is typically a strict subset of $\mathbb{Z}^m$, only the coefficients indexed by $\mathcal{Z}_{\mathrm{red}}(\mathcal{G})$ need to be stored and manipulated computationally. The remaining coefficients are uniquely determined by the symmetry relations, resulting in reduced memory requirements and computational cost. Second, this reduction eliminates the translational degeneracy inherent in Fourier-based formulations. Indeed, any spatial translate of a solution is again a solution, so zeros of $F(\mathbf{x}) = 0$, where $F$ denotes the cusp map introduced in Section~\ref{sec : cusp map}, are generally non-isolated. By imposing a sufficiently rich symmetry group $\mathcal{G}$, these translational degrees of freedom are removed. The symmetry groups considered in this work are chosen precisely to ensure this property. Consequently, we solve for
\[
\mathbf{u} = (\mathbf{u}_n)_{n \in \mathcal{Z}_{\mathrm{red}}(\mathcal{G})}.
\]

We now introduce the sequence space that serves as our functional setting. Let $\omega = (\omega_n)_{n \in \mathbb{Z}^m}$ be a sequence of weights. Let $\ell^p_{\omega}$ denote the Banach space
\begin{equation}\label{eq:ell_2_nu}
\ell^p_{\omega}
\bydef
\left\{
u=(u_n)_{n\in\mathbb{Z}^m} : \|u\|_{p,\omega}
\bydef
\left( \sum_{n\in\mathbb{Z}^m}
\,|u_n|^p\omega_n \right)^{\frac{1}{p}}
<\infty
\right\}.
\end{equation} 
We denote by  $\ell^p_{\mathcal{G},\omega}$ the restriction of $\ell^p_{\omega}$ to  $\mathcal{G}$-symmetric sequences. In order to provide a more formal definition of $\ell^p_{\mathcal{G},\omega}$, it will be useful to recover the full Fourier sequence from the $\mathcal{G}$-Fourier coefficients. We characterize this with the injective mapping $\iota : \ell^p_{\mathcal{G},\omega} \to \ell^p_{\omega}$ defined by
\begin{align}
    (\iota(u))_n \bydef \begin{cases} u_{\mathscr{n}}, & \mathscr{n} \in \mathcal{Z}_{\mathrm{red}}(\mathcal{G}) \\ 0, & \mathscr{n} \in \mathcal{Z}_{\mathrm{triv}}(\mathcal{G})\end{cases} ~ \text{where} ~ \mathscr{n} \in \mathrm{orb}_{\mathcal{G}}(n) \cap \mathcal{Z}_{\mathrm{dom}}(\mathcal{G}).
\end{align}
To clarify, for each $n$, there is a representative from its orbit chosen in $\mathcal{Z}_{\mathrm{dom}}(\mathcal{G})$. Suppose this index is denoted $\mathscr{n}$. Then the value of $(\iota(u))_n$ is defined to be $u_{\mathscr{n}}$, unless the symmetry constraints force the coefficient to vanish (i.e.\ $\mathscr{n} \in \mathcal{Z}_{\mathrm{triv}}(\mathcal{G})$), in which case it is defined to be $0$. This means $\iota$ is a map which formalizes the unfolding of a sequence with $\mathcal{G}$-Fourier coefficients to its equivalent with full Fourier coefficients.
We now define 
\begin{equation}\label{eq:ell_2_Gnu}
\ell^p_{\mathcal{G},\omega}
\bydef
\left\{
u=(u_n)_{n\in\mathcal{Z}_{\mathrm{red}}(\mathcal{G})} : \|u\|_{p,\omega}
\bydef
\left(\sum_{n\in\mathbb{Z}^m}
\,|\iota(u)_n|^p\omega_n\right)^{\frac{1}{p}}
<\infty
\right\}.
\end{equation}
By this definition, given $u \in \ell^p_{\mathcal{G},\omega}$ with its corresponding unfolding $\iota(u) \in \ell^p_{\omega}$, observe that
\begin{align}\label{eq : norm equality iota}
     \|u\|_{p,\omega} = \|\iota(u)\|_{p,\omega}
\end{align}
showing that the norms on $\ell^p_{\mathcal{G},\omega}$ and $\ell^p_{\omega}$ are equal. Equivalently, grouping the indices of $\mathbb{Z}^m$ into $\mathcal{G}$-orbits,
\begin{equation}\label{eq : orbit weights}
    \|u\|_{p,\omega}^p = \sum_{n \in \mathcal{Z}_{\mathrm{red}}(\mathcal{G})} W_n\,|u_n|^p, \qquad
    W_n \bydef \sum_{j \in \mathrm{orb}_{\mathcal{G}}(n)} \omega_j \;>\; 0 .
\end{equation}
We stress that $W_n = |\mathrm{orb}_{\mathcal{G}}(n)|\,\omega_n$ \emph{only} when $\omega$ is constant on $\mathcal{G}$-orbits. This is the case for $\mathbb{Z}_2$, $D_2$ and $D_4$ with $\mathcal{L}$ orthogonal, but it \emph{fails} for the hexagonal group $D_6$ of Theorem~\ref{th : SH 2D cusp D6}: there the action on Fourier indices is $n = (n_1,n_2) \mapsto (n_2, n_2 - n_1)$, which does not preserve $\omega_n = (1+|n_1|)^s(1+|n_2|)^s$ (e.g.\ $(1,1) \mapsto (1,0)$). All statements below are therefore formulated with the orbit weights $W_n$ of \eqref{eq : orbit weights}, which requires no invariance assumption on $\omega$.
\par We now introduce some additional notation. Let $\mathbb{Z}_1$ be the group only consisting of the identity element. We define the shorthand notations
\begin{align}
    &\ell^p_{\mathbb{Z}_1,\omega} \bydef \ell^p_{\omega}, 
    \qquad \ell^p_{\mathcal{G},1} \bydef \ell^p_{\mathcal{G}} \qquad \text{and} \qquad
    \ell^p_{\mathbb{Z}_1,1} \bydef \ell^p.
    \end{align} 
    
In particular, in the absence of symmetry we denote the space by $\ell^p_{\omega}$. When symmetry is present but no weighting is used (that is weights equal one), we write $\ell^p_{\mathcal{G}}$. Finally, in the absence of both symmetry and weighting, we recover the standard space $\ell^p$. Accordingly, we write $\|u\|_p \bydef \|u\|_{p,1}$ for the norm on $\ell^p$, that is
\begin{align}
    \|u\|_p \bydef \left(\sum_{n \in \mathbb{Z}^m} |u_n|^p\right)^{1/p}.
\end{align}
    
\subsection{Convolution Estimates} \label{sec:conv_estimates}

The PDEs considered in this work possess polynomial nonlinearities which, when expressed in Fourier sequence space, give rise to discrete convolution products. To carry out the nonlinear analysis, and in particular to construct the bounds $Y_0$, $Z_1$, and $Z_2$ required by the Newton--Kantorovich theorem of Section~\ref{sec : radii poly}, it is therefore necessary to obtain convolution estimates in the space $\ell^2_{\omega}$. The principal result of this subsection is Lemma~\ref{lem : conv_estimates}, which establishes that $\ell^2_{\omega}$ is a Banach algebra under convolution. We also derive Lemma~\ref{lem : 1 bounded by 2 nu}, which plays a key role in the construction of the $Z_2$ bound appearing in Theorem~\ref{th : radii polynomial theorem periodic}.

Given full Fourier coefficients $u = (u_n)_{n \in \mathbb{Z}^m}$ and $v= (v_n)_{n \in \mathbb{Z}^m}$ corresponding to the standard complex exponential expansion, we define the discrete convolution $\star$ by 
\[
(u\star v)_n = \sum_{k \in \mathbb{Z}^m} u_{n-k}v_k
\]
For symmetric sequences in $\ell^p_{\mathcal{G},\omega}$, we retain the notation $u \star v$. We also let $\Omega$ be a diagonal operator defined as
\begin{align}\label{def : V}
    \Omega_{n,m} \bydef \begin{cases}
        \sqrt{\omega_n} & n = m \\
        0 & \mathrm{else}.
    \end{cases}
\end{align}
We now focus our attention on $\ell^2_{\omega}$, which is what we will use throughout this paper. With this in mind, we define what it means for $\omega$ to be an admissible sequence of weights.
\begin{definition}\label{def:admissible_weights}
We say that a sequence of weights $\omega = (\omega_n)_{n \in \mathbb{Z}^m}$ is \emph{admissible} if 
\begin{itemize}
    \item[(i)] $\omega_n > 0, \qquad$ for all $n$ (positivity);
    \item[(ii)] $\omega_n  \le \omega_{n-k} \omega_{k}, \qquad \text{for all } k,n \in \mathbb{Z}^m$ (sub-multiplicativity);
    \item[(iii)] $\displaystyle \sup_{n \in \mathbb{Z}^m} \sum_{k \in \mathbb{Z}^m} \frac{\omega_n}{\omega_{n-k} \omega_k} < \infty$;
    \item[(iv)] $\displaystyle \sum_{n \in \mathbb{Z}^m} \frac{1}{\omega_n} < \infty$.
\end{itemize}
\end{definition}


\begin{lemma}[\bf Convolution Estimates] \label{lem : conv_estimates}
Assume that the sequence of weights $\omega=(\omega_n)_{n \in \mathbb{Z}^m}$ is admissible (in accordance with Definition~\ref{def:admissible_weights}). Then, we can define
\begin{equation} \label{eq:kappa_def_and_assumption}
\kappa \bydef \sqrt{\sup_{n \in \mathbb{Z}^m} \sum_{k \in \mathbb{Z}^m} \frac{\omega_n}{\omega_{n-k}\omega_k}} < \infty.
\end{equation}
Then, for any $u,v \in \ell^2_{\omega}$, it follows that
\begin{equation}
    \|u \star v\|_{2,\omega} \leq \kappa \|u\|_{2,\omega}\|v\|_{2,\omega}.
\end{equation}
\end{lemma}

\begin{proof}
To begin, observe that
\begin{align}
    \|u \star v\|_{2,\omega}^2 &= \sum_{n \in \mathbb{Z}^m} \left|\sum_{k \in \mathbb{Z}^m} u_{n-k} v_k \right|^2 \omega_n
    \\
    &= \sum_{n \in \mathbb{Z}^m} \left|\sum_{k \in \mathbb{Z}^m} \left(u_{n-k} \sqrt{\omega_{n-k}}\right) \left(v_k \sqrt{\omega_k}\right) \sqrt{\frac{\omega_n}{\omega_{n-k}\omega_k}}\right|^2 \\
    &\leq \sum_{n \in \mathbb{Z}^m} \left(\sum_{k \in \mathbb{Z}^m} |u_{n-k}|^2 \omega_{n-k} |v_k|^2 \omega_k\right)\left(\sum_{k \in \mathbb{Z}^m} \frac{\omega_n}{\omega_{n-k}\omega_k}\right),
\end{align}
where we used Cauchy--Schwarz for each fixed value of $n$. Now, recalling \eqref{eq:kappa_def_and_assumption},
\begin{align}
    \sum_{n \in \mathbb{Z}^m} \left(\sum_{k \in \mathbb{Z}^m} |u_{n-k}|^2 \omega_{n-k} |v_k|^2 \omega_k\right)\left(\sum_{k \in \mathbb{Z}^m} \frac{\omega_n}{\omega_{n-k}\omega_{k}}\right) &\leq \kappa^2 \sum_{n \in \mathbb{Z}^m} \sum_{k \in \mathbb{Z}^m} |u_{n-k}|^2 \omega_{n-k} |v_k|^2 \omega_{k} \\
    &= \kappa^2 \sum_{k \in \mathbb{Z}^m} |v_k|^2 \omega_{k} \sum_{n \in \mathbb{Z}^m} |u_{n-k}|^2 \omega_{n-k}
\end{align}
where we flipped the order of summation since every term is non-negative. Performing the index change $j = n - k$ gives
\begin{align}
    \kappa^2 \sum_{k \in \mathbb{Z}^m} |v_k|^2 \omega_k \sum_{n \in \mathbb{Z}^m} |u_{n-k}|^2 \omega_{n-k} &= \kappa^2 \sum_{k \in \mathbb{Z}^m} |v_k|^2 \omega_k \sum_{j \in \mathbb{Z}^m} |u_{j}|^2 \omega_{j} \\
    &= \kappa^2  \left(\sum_{j \in \mathbb{Z}^m} |u_{j}|^2 \omega_{j}\right)\left(\sum_{k \in \mathbb{Z}^m} |v_k|^2 \omega_k\right) \\
    &= \kappa^2 \|u\|_{2,\omega}^2 \|v\|_{2,\omega}^2.
\end{align}
Taking the square root of both sides concludes the proof.
\end{proof}

The following results play a key role in the construction of the $Z_2$ bound appearing in Theorem~\ref{th : radii polynomial theorem periodic}.
The first is the classical Young's Inequality
\begin{equation} \label{eq:youngs inequality}
    \|u \star v\|_{2} \leq \|u\|_{2} \|v\|_{1},
\end{equation}

which is used in the proof of Lemma~\ref{lem : gen young} below. 
Given the sequence of weights $\omega=(\omega_n)_{n \in \mathbb{Z}^m}$ satisfying Definition~\eqref{def:admissible_weights}, denote $\sqrt{\omega} \bydef (\sqrt{\omega_n})_{n \in \mathbb{Z}^m}$.

\begin{lemma}\label{lem : gen young}
Let $u \in \ell^2_{\omega}$ and $v \in \ell^1_{\sqrt{\omega}}$. Then, it follows that
\begin{align}
    \|u \star v\|_{2,\omega} \leq \|u\|_{2,\omega} \|v\|_{1,\sqrt{\omega}}.
\end{align}
\end{lemma}
\begin{proof}
Since $\omega$ is admissible (Definition~\eqref{def:admissible_weights}), we have $\omega_n \leq \omega_{n-k}\,\omega_k$ for all $n,k\in\mathbb{Z}^m$, hence $\sqrt{\omega_n}\leq\sqrt{\omega_{n-k}}\,\sqrt{\omega_k}$. Therefore, for each $n \in \mathbb{Z}^m$,
\begin{align*}
    \sqrt{\omega_n}\,|(u\star v)_n|
    \leq \sum_{k\in\mathbb{Z}^m} \sqrt{\omega_n}\,|u_{n-k}|\,|v_k|
    \leq \sum_{k\in\mathbb{Z}^m} |(\Omega u)_{n-k}|\,|(\Omega v)_k|
    = \bigl(|\Omega u|\star|\Omega v|\bigr)_n,
\end{align*}
where we used the triangle inequality and $(\Omega u)_n = \sqrt{\omega_n}\,u_n$, $(\Omega v)_n = \sqrt{\omega_n}\,v_n$. Taking $\ell^2$-norms and applying the classical Young inequality \eqref{eq:youngs inequality} gives
\begin{align}
    \|u\star v\|_{2,\omega}
    = \|\Omega(u\star v)\|_{2}
    \leq \bigl\||\Omega u|\star|\Omega v|\bigr\|_{2}
    \leq \|\Omega u\|_{2}\,\|\Omega v\|_{1}
    = \|u\|_{2,\omega}\,\|v\|_{1,\sqrt{\omega}}. \qquad \qedhere
\end{align}
\end{proof}

\begin{lemma}\label{lem : 1 bounded by 2 nu}
Assume that the sequence of weights $\omega=(\omega_n)_{n \in \mathbb{Z}^m}$ is admissible as in Definition~\eqref{def:admissible_weights}. Then, we have that
\begin{align}
    \kappa_0 \bydef \sqrt{\sum_{n\in\mathbb{Z}^m} \frac{1}{\omega_n}} < \infty.
\end{align}
Then, for any $v \in \ell^2_{\omega}$,
\begin{align}
    \|v\|_{1} \leq \kappa_0 \|v\|_{2,\omega}.
\end{align}
\end{lemma}
\begin{proof}
Applying Cauchy--Schwarz to $\|v\|_{1}$ for $v = (v_n)_{n \in \mathbb{Z}^m}$,
\begin{align}\label{kappa estimate for 1 norm}
\|v\|_{1} = \sum_{n\in\mathbb{Z}^m} |v_n| = \sum_{n\in\mathbb{Z}^m}\sqrt{\omega_n} \sqrt{\frac{1}{\omega_n}} |v_n| \leq \left(\sum_{n\in\mathbb{Z}^m} \frac{1}{\omega_n} \right)^{\frac{1}{2}} \left(\sum_{n\in\mathbb{Z}^m} \omega_n |v_n|^2\right)^{\frac{1}{2}} = \kappa_0 \|v\|_{2,\omega}. \qquad \qedhere
\end{align}
%

\end{proof}

Lemmas~\ref{lem : conv_estimates}--\ref{lem : 1 bounded by 2 nu} are stated on the full space $\ell^2_\omega$. They transfer verbatim to the symmetry-reduced space, as we now record.
\begin{corollary}\label{cor : conv estimates symmetric}
Let $\omega$ be admissible, assume $\alpha_{\mathscr{g}} \equiv 1$ (as holds for every group considered in this work, see the Remark following Corollary~\ref{corr : reduced_set}), and let $u,v \in \ell^2_{\mathcal{G},\omega}$. Then $u \star v \in \ell^2_{\mathcal{G},\omega}$ and
\[
\|u\star v\|_{2,\omega} \leq \kappa\,\|u\|_{2,\omega}\|v\|_{2,\omega}, \qquad \|v\|_1 \leq \kappa_0\|v\|_{2,\omega},
\]
where, consistently with \eqref{eq:ell_2_Gnu}, $\|v\|_1 \bydef \|\iota(v)\|_1$.
\end{corollary}
\begin{proof}
Since $\alpha_{\mathscr{g}} \equiv 1$, the unfolding map $\iota$ satisfies $\iota(u\star v) = \iota(u)\star\iota(v)$, and the convolution of two $\mathcal{G}$-symmetric sequences is again $\mathcal{G}$-symmetric. Hence, by the norm equality \eqref{eq : norm equality iota} and Lemma~\ref{lem : conv_estimates},
\[
\|u\star v\|_{2,\omega} = \|\iota(u)\star\iota(v)\|_{2,\omega} \leq \kappa\|\iota(u)\|_{2,\omega}\|\iota(v)\|_{2,\omega} = \kappa\|u\|_{2,\omega}\|v\|_{2,\omega},
\]
and similarly for the second estimate using Lemma~\ref{lem : 1 bounded by 2 nu}.
\end{proof}

With the convolution estimates on $\ell^2_{\mathcal{G},\omega}$ established, we now turn to the adjoint multiplication operator.
\begin{remark}
In order to apply Lemmas \ref{lem : 1 bounded by 2 nu} and \ref{lem : conv_estimates}, we will need to compute $\kappa$ and $\kappa_0$. This will be done explicitly in Appendix \ref{apen : algebra estimates} for the choice of $\omega_n \bydef \prod_{j=1}^m (1 + |n_j|)^s$ for some $s > 1$. 
\end{remark} 
\subsection{Adjoint multiplication in \texorpdfstring{\boldmath$\ell^2_{\mathcal{G},\omega}$\unboldmath}{ell2Gomega}} \label{sec : adjoint mult}

As we will now focus on $p = 2$, for given $\mathbf{u} = (u_1,\dots,u_{\mathscr{p}}) \in (\ell^2_{\mathcal{G}})^{\mathscr{p}}$ and $\mathbf{v} = (v_1,\dots,v_{\mathscr{p}}) \in (\ell^2_{\mathcal{G}})^{\mathscr{p}}$, we define the inner product in consideration as
\begin{align}\label{def : ell2 inner product}
    (\mathbf{u},\mathbf{v})_2 \bydef \sum_{j = 1}^{\mathscr{p}}\sum_{n \in \mathcal{Z}_{\mathrm{red}}(\mathcal{G})} (u_j)_n \mathrm{conj}((v_j)_n).
\end{align}
We emphasize that $(\cdot,\cdot)_2$ carries \emph{no} orbit weight; it is not the inner product inherited from $L^2_{\mathcal{G}}(\mathscr{D})$, and the adjoint $\cdot^*$ used in this subsection is always the adjoint with respect to \eqref{def : ell2 inner product}. The precise relation between $(\cdot,\cdot)_2$ and $(\cdot,\cdot)_{L^2}$, and hence between the two notions of adjoint, is established in Remark~\ref{rmk : L2 vs ell2} below; it is exactly the operator $\mathcal{O}_{\mathcal{G}}$ introduced in \eqref{def : D operator} that mediates between them.
Throughout this subsection, all sequences $u \in \ell^2_{\mathcal{G},\omega}$ are assumed to be \emph{real-valued}, meaning their Fourier coefficients satisfy $(\iota(u))_{-n} = \mathrm{conj}((\iota(u))_n)$ for all $n \in \mathbb{Z}^m$. On the real subspace so defined, $(\cdot,\cdot)_2$ is a symmetric $\mathbb{R}$-bilinear form, which is the property used when differentiating the cusp map in Section~\ref{sec : cusp map} and in Appendix~\ref{sec:proof_thm_cusp_map_nondegen}. We also assume that $\alpha_{\mathscr{g}}(n) = 1$ for all $\mathscr{g} \in \mathcal{G}$ and $n \in \mathbb{Z}^m$, so that the $\mathcal{G}$-symmetry condition reduces to $(\iota(u))_{\mathscr{g} \cdot n} = (\iota(u))_n$. Both conditions are satisfied for all symmetry groups and applications considered in this work.

As discussed in Section~\ref{sec:introduction}, and in particular in Lemma~\ref{lemma : Cusp bifurcation}, our approach to proving the existence of a cusp bifurcation relies on the computation of the adjoint operator $D_u f^*$, taken with respect to $(\cdot,\cdot)_2$ (the passage to the $L^2_{\mathcal{G}}(\mathscr{D})$-adjoint $D_uf^\dagger$ of Section~\ref{sec:introduction} is Remark~\ref{rmk : L2 vs ell2}). For reaction--diffusion systems with polynomial nonlinearities, the Fréchet derivative $D_u f$ involves multiplication operators, making the characterization of their adjoints a central issue. Since all computations and computer-assisted estimates are performed in the symmetry-reduced Fourier sequence space $\ell^2_{\mathcal{G},\omega}$ defined in \eqref{eq:ell_2_Gnu}, we must derive explicit formulas and norm estimates for the adjoints of multiplication operators. This is the objective of the present section. Doing so requires first introducing some notation and operators.

Let $\mathcal{B}(X,Y)$ be the set of bounded linear operators from $X \to Y$. We shorthand and write $\mathcal{B}(X)$ as the set of bounded linear operators from $X$ to itself. Given $u \in \ell^2_{\omega}$, we define the {\em multiplication operator}
\begin{equation} \label{def : discrete conv operator}
    \mathbb{M}_u : \ell^2_{\omega} \to \ell^2_{\omega}: v \mapsto \mathbb{M}_u v \bydef  u\star v,
\end{equation}
i.e., the discrete convolution operator associated with $u$. Note that Lemma~\ref{lem : conv_estimates} guarantees that $\mathbb{M}_u$ maps $\ell^2_{\omega}$ into itself.  
We also abuse notation and denote $\mathbb{M}_u$ the multiplication operator associated to $u \in \ell^2_{\mathcal{G},\omega}$.
Now, the key is the diagonal operator $\mathcal{O}_{\mathcal{G}} \in \mathcal{B}(\ell^2_{\mathcal{G},\omega})$ defined by
\begin{equation} \label{def : D operator}
    (\mathcal{O}_{\mathcal{G}})_{n_1,n_2} \bydef \begin{cases} |\mathrm{orb}_{\mathcal{G}}(n)| & n_1 = n_2 = n \\
    0 & \mathrm{else}
    \end{cases}
\end{equation}
where $\mathrm{orb}_{\mathcal{G}}(n)$ is defined as in \eqref{def : orbit}. 
In particular, notice that the diagonal entries of $\mathcal{O}_{\mathcal{G}}$ are the orbit sizes for each $n$. We also define the quantity $\mathscr{O}_{\mathrm{max}} \in \mathbb{R}$ as 
\begin{equation} \label{def : omax}
    \mathscr{O}_{\mathrm{max}} \bydef \max_{n \in \mathcal{Z}_{\mathrm{red}}(\mathcal{G})} |\mathrm{orb}_{\mathcal{G}}(n)|.
\end{equation}

In the rest of this section, the main results that are established are the four identities
\begin{equation*}
    \mathbb{M}_u^* = \mathcal{O}_{\mathcal{G}} \mathbb{M}_{u} \mathcal{O}_{\mathcal{G}}^{-1}, \qquad
    \mathbb{M}_u^* v = \mathcal{O}_{\mathcal{G}}\bigl(u\star (\mathcal{O}_{\mathcal{G}}^{-1} v)\bigr), \qquad
    D_u(\mathbb{M}_u^* v) = \mathbb{Q}_v, \qquad
    \mathbb{Q}_u v = \mathcal{O}_{\mathcal{G}}\bigl((\mathcal{O}_{\mathcal{G}}^{-1}u)\star v\bigr),
\end{equation*}
collected as \eqref{adjoint mult action 1}--\eqref{bbQ action} below, where $\mathbb{Q}_u \bydef \mathcal{O}_{\mathcal{G}}\mathbb{M}_{\mathcal{O}_{\mathcal{G}}^{-1}u}$. These identities are the key computational ingredients in the formulation of the cusp map in Section~\ref{sec : cusp map}.

We now establish the operator norm bounds for $\mathcal{O}_{\mathcal{G}}$ needed later.

\begin{lemma}\label{lem : O operator norm maximum}
Let $\mathcal{O}_{\mathcal{G}} \in \mathcal{B}(\ell^2_{\mathcal{G},\omega})$ be defined as in \eqref{def : D operator} and $\mathscr{O}_{\mathrm{max}}$ be defined as in \eqref{def : omax}. Then, it follows that
\begin{equation}
    \|\mathcal{O}_{\mathcal{G}}\|_{\mathcal{B}(\ell^2_{\mathcal{G},\omega})} = \mathscr{O}_{\mathrm{max}} \qquad \text{and} \qquad \|\mathcal{O}_{\mathcal{G}}^{-1}\|_{\mathcal{B}(\ell^2_{\mathcal{G},\omega})} = 1.
\end{equation}
\end{lemma}

\begin{proof}
By \eqref{eq : orbit weights}, the norm of $\ell^2_{\mathcal{G},\omega}$ is the weighted norm
\[
\|u\|_{2,\omega}^2 = \sum_{n\in\mathcal{Z}_{\mathrm{red}}(\mathcal{G})} W_n\,|u_n|^2, \qquad W_n = \sum_{j\in\mathrm{orb}_{\mathcal{G}}(n)}\omega_j > 0 ,
\]
on sequences indexed by $\mathcal{Z}_{\mathrm{red}}(\mathcal{G})$. (We emphasise that one should \emph{not} argue through the map $\Omega$ of \eqref{def : V}: $\Omega : \ell^2_{\mathcal{G},\omega} \to \ell^2_{\mathcal{G}}$ is an isometry if and only if $W_n = |\mathrm{orb}_{\mathcal{G}}(n)|\,\omega_n$ for every $n \in \mathcal{Z}_{\mathrm{red}}(\mathcal{G})$, which holds when $\omega$ is constant on $\mathcal{G}$-orbits but fails for $D_6$, see \eqref{eq : orbit weights}; the argument below needs no such assumption.)
For any diagonal operator $D$ with entries $d_n$ on a weighted $\ell^2$ space with weight $W_n$, the operator norm equals $\sup_n|d_n|$ since the weights cancel in $\|Da\|^2/\|a\|^2 = \sum_n W_n|d_n|^2|a_n|^2 / \sum_n W_n|a_n|^2$. Applying this to $\mathcal{O}_{\mathcal{G}}$ with entries $|\mathrm{orb}_{\mathcal{G}}(n)|$ gives $\|\mathcal{O}_{\mathcal{G}}\|_{\mathcal{B}(\ell^2_{\mathcal{G},\omega}}) = \max_{n}|\mathrm{orb}_{\mathcal{G}}(n)| = \mathscr{O}_{\mathrm{max}}$, establishing the first bound. For the second, we use that $|\mathrm{orb}_{\mathcal{G}}(n)| \ge 1$. This follows trivially from the fact that $n\in \mathrm{orb}_{\mathcal{G}}(n)$, and hence $|\mathrm{orb}_{\mathcal{G}}(n)| \ge 1$. Hence every orbit has size at least one, so $\mathcal{O}_{\mathcal{G}}$ is invertible with diagonal entries $\frac{1}{|\mathrm{orb}_{\mathcal{G}}(n)|}$. Note that for every group $\mathcal{G}$, we have $|\mathrm{orb}_{\mathcal{G}}(\mathbf{0})| = 1$, and $\mathbf{0} \in \mathcal{Z}_{\mathrm{red}}(\mathcal{G})$ (constant functions are $\mathcal{G}$-symmetric, so $\mathbf{0}\notin\mathcal{Z}_{\mathrm{triv}}(\mathcal{G})$), so that the minimum below is attained. Therefore,
\begin{align}
    \|\mathcal{O}_{\mathcal{G}}^{-1}\|_{\mathcal{B}(\ell^2_{\mathcal{G},\omega})} &= \max_{n \in \mathcal{Z}_{\mathrm{red}}(\mathcal{G})} \frac{1}{|\mathrm{orb}_{\mathcal{G}}(n)|} = 1
\end{align}
 as desired.
\end{proof}
Now, given $\mathcal{N} \in \mathbb{N}$, let us introduce the following projection operators 
 \begin{align}
(\Pi^{\leq\mathcal{N}}q)_n  =  \begin{cases}
          q_n,  & n \in I^{\mathcal{N}} \\
              0, &n \notin I^{\mathcal{N}}
    \end{cases} ~~ \text{ and } ~~
     (\Pi^{>\mathcal{N}}q)_n  =  \begin{cases}
          0,  & n \in I^{\mathcal{N}} \\
              q_n, &n \notin I^{\mathcal{N}}
    \end{cases}\label{def : piN and pisubN}
 \end{align}
where $I^{\mathcal{N}} \bydef \{n \in \mathcal{Z}_{\mathrm{red}}(\mathcal{G}) :  |n|_\infty \leq \mathcal{N}\}$ for all $q = (q_n)_{n \in  \mathcal{Z}_{\mathrm{red}}(\mathcal{G})} \in \ell^2_{\mathcal{G},\omega}.$ We now justify some results regarding the adjoint multiplication operator. Throughout this section, let $u, v \in \ell^2_{\mathcal{G},\omega}$. We now decompose $\mathbb{M}_u$ and write
\begin{align}
    \mathbb{M}_u \bydef 
        u_0 + u^* \mathcal{O}_{\mathcal{G}}\Pi^{> 0} +
        \Pi^{>0}u + \mathcal{M}_u,\label{mult decomp gen}
\end{align}
where $u^*$ is the linear functional defined by 
\begin{equation}\label{def : u star}
  u^* h \bydef (h,u)_2 = \sum_{n \in \mathcal{Z}_{\mathrm{red}}(\mathcal{G})} h_n\,\mathrm{conj}(u_n),
\end{equation}
which is the correct form of the $n=\mathbf{0}$ row of $\mathbb{M}_u$: grouping $(u\star v)_{\mathbf{0}} = \sum_{j\in\mathbb{Z}^m}(\iota(u))_{-j}(\iota(v))_j$ into orbits and using $-\mathrm{orb}_{\mathcal{G}}(k) = \mathrm{orb}_{\mathcal{G}}(-k)$ together with $(\iota(u))_{-k} = \mathrm{conj}((\iota(u))_k)$ gives $(u\star v)_{\mathbf{0}} = u_{\mathbf{0}}v_{\mathbf{0}} + \sum_{k\neq\mathbf{0}}|\mathrm{orb}_{\mathcal{G}}(k)|\,v_k\,\mathrm{conj}(u_k) = u_{\mathbf{0}}v_{\mathbf{0}} + u^*(\mathcal{O}_{\mathcal{G}}\Pi^{>0}v)$. In all the applications of this paper $\mathcal{G}$ contains $x\mapsto -x$, so that $\mathcal{G}$-symmetric real-valued functions are even, their Fourier coefficients are real, and $u^*h = (h,u)_2 = (u,h)_2$; this is also what makes $(\cdot,\cdot)_2$ a symmetric $\mathbb{R}$-bilinear form, as used from Section~\ref{sec : cusp map} onwards. Here $\mathcal{M}_u \bydef \Pi^{> 0}\mathbb{M}_{u}\Pi^{> 0}$. Hence, it follows that
\[
    \mathbb{M}_u^* = 
        u_0 + u^*\Pi^{> 0} +
        \Pi^{> 0}\mathcal{O}_{\mathcal{G}}u +\mathcal{M}_u^*.
\]
We also define
\begin{align}
    \mathbb{Q}_u \bydef \mathcal{O}_{\mathcal{G}} \mathbb{M}_{\mathcal{O}_{\mathcal{G}}^{-1} u}.\label{def : Qa}
\end{align}
With these definitions in place, we establish the four identities announced above.
\begin{prop}\label{prop : adjoint mult}
Let $u, v \in \ell^2_{\mathcal{G},\omega}$. Then
\begin{align}
    \mathbb{M}_u^* &= \mathcal{O}_{\mathcal{G}} \mathbb{M}_{u} \mathcal{O}_{\mathcal{G}}^{-1}, \label{adjoint mult action 1}\\
    \mathbb{M}_u^* v &= \mathcal{O}_{\mathcal{G}}\bigl(u\star (\mathcal{O}_{\mathcal{G}}^{-1} v)\bigr), \label{adjoint mult action 2}\\
    D_u(\mathbb{M}_u^* v) &= \mathbb{Q}_v, \label{adjoint mult action 3}\\
    \mathbb{Q}_u v &= \mathcal{O}_{\mathcal{G}}\bigl((\mathcal{O}_{\mathcal{G}}^{-1}u)\star v\bigr). \label{bbQ action}
\end{align}
\end{prop}
\begin{proof}
We begin the proof first by showing that $\mathcal{M}_u^* = \mathcal{O}_{\mathcal{G}}\mathcal{M}_u\mathcal{O}_{\mathcal{G}}^{-1}$.
For $n, k \in \mathcal{Z}_{\mathrm{red}}(\mathcal{G}) \setminus \{\mathbf{0}\}$, the $\mathcal{G}$-symmetry of $v$ gives $(u\star v)_n = \sum_{k \in \mathcal{Z}_{\mathrm{red}}(\mathcal{G})} v_k \sum_{j \in \mathrm{orb}_{\mathcal{G}}(k)} (\iota(u))_{n-j}$. Each element of $\mathrm{orb}_{\mathcal{G}}(k)$ is reached $|\mathrm{stab}_{\mathcal{G}}(k)|$ times as $\mathscr{g}$ ranges over $\mathcal{G}$. Since $\mathcal{G}$ is finite, we can use the Orbit-Stabilizer Theorem (cf. \cite{gallianalgebra}) and write
\begin{align}
    (\mathcal{M}_u)_{n,k} = \frac{|\mathrm{orb}_{\mathcal{G}}(k)|}{|\mathcal{G}|}\sum_{\mathscr{g} \in \mathcal{G}} (\iota(u))_{n - \mathscr{g} \cdot k}.
\end{align}
By definition of the adjoint and real-valuedness of $u$ (i.e., $\mathrm{conj}((\iota(u))_j) = (\iota(u))_{-j}$),
\begin{align}
    (\mathcal{M}_u^*)_{n,k} = \frac{|\mathrm{orb}_{\mathcal{G}}(n)|}{|\mathcal{G}|}\sum_{\mathscr{g} \in \mathcal{G}} \mathrm{conj}((\iota(u)))_{k - \mathscr{g} \cdot n}.
\end{align}
Expanding the conjugated operator, replacing $\mathscr{g}$ by $\mathscr{g}^{-1}$, and invoking $\mathcal{G}$-symmetry of $u$ (i.e. $(\iota(u))_{\mathscr{g}\cdot j} = (\iota(u))_j$) followed by real-valuedness,
\begin{align}
    (\mathcal{O}_{\mathcal{G}}\mathcal{M}_u\mathcal{O}_{\mathcal{G}}^{-1})_{n,k}
    &= \frac{|\mathrm{orb}_{\mathcal{G}}(n)|}{|\mathrm{orb}_{\mathcal{G}}(k)|}(\mathcal{M}_u)_{n,k}
    = \frac{|\mathrm{orb}_{\mathcal{G}}(n)|}{|\mathcal{G}|}\sum_{\mathscr{g}\in\mathcal{G}} (\iota(u))_{n-\mathscr{g}\cdot k}
    = \frac{|\mathrm{orb}_{\mathcal{G}}(n)|}{|\mathcal{G}|}\sum_{\mathscr{g}\in\mathcal{G}} (\iota(u))_{\mathscr{g}\cdot n - k} \\
    &= \frac{|\mathrm{orb}_{\mathcal{G}}(n)|}{|\mathcal{G}|}\sum_{\mathscr{g}\in\mathcal{G}} \mathrm{conj}((\iota(u)))_{k - \mathscr{g}\cdot n}
    = (\mathcal{M}_u^*)_{n,k}.
\end{align}

Second, we prove that \eqref{adjoint mult action 1} hold.
Writing $\mathcal{O}_{\mathcal{G}} = 1 + \Pi^{>0}\mathcal{O}_{\mathcal{G}}\Pi^{>0}$, expanding via the decomposition \eqref{mult decomp gen}, and using that $\mathcal{M}_u^* = \mathcal{O}_{\mathcal{G}}\mathcal{M}_u\mathcal{O}_{\mathcal{G}}^{-1}$,
\[
    \mathcal{O}_{\mathcal{G}}\mathbb{M}_u\mathcal{O}_{\mathcal{G}}^{-1}
    = u_0 + u^*\Pi^{>0} + \Pi^{>0}\mathcal{O}_{\mathcal{G}}u + \Pi^{>0}\mathcal{O}_{\mathcal{G}}\mathcal{M}_u\mathcal{O}_{\mathcal{G}}^{-1}\Pi^{>0}
    = u_0 + u^{*}\Pi^{> 0} + \Pi^{> 0}\mathcal{O}_{\mathcal{G}}u + \mathcal{M}_{u}^* = \mathbb{M}_{u}^*.
\]

The proof of \eqref{adjoint mult action 2} follows by applying \eqref{adjoint mult action 1} to $v$:
\begin{align}
    \mathbb{M}_u^* v = \mathcal{O}_{\mathcal{G}}\mathbb{M}_u(\mathcal{O}_{\mathcal{G}}^{-1}v) = \mathcal{O}_{\mathcal{G}}\bigl(u \star (\mathcal{O}_{\mathcal{G}}^{-1}v)\bigr).
\end{align}

Since the right-hand side of \eqref{adjoint mult action 2} is linear in $u$, differentiating with respect to $u$ gives
\begin{align}
    D_u(\mathbb{M}_u^* v) = \mathcal{O}_{\mathcal{G}}\mathbb{M}_{\mathcal{O}_{\mathcal{G}}^{-1}v} = \mathbb{Q}_v
\end{align}
which proves \eqref{adjoint mult action 3}. Finally, \eqref{bbQ action} follows since $\mathbb{Q}_u v = \mathcal{O}_{\mathcal{G}}\mathbb{M}_{\mathcal{O}_{\mathcal{G}}^{-1}u}v = \mathcal{O}_{\mathcal{G}}\bigl((\mathcal{O}_{\mathcal{G}}^{-1}u)\star v\bigr)$.
\end{proof}
With the adjoint multiplication operator characterized, we now introduce the cusp map.

\subsection{The Cusp Map} \label{sec : cusp map}

The conditions for a cusp bifurcation were collected in Lemma~\ref{lemma : Cusp bifurcation}: in addition to the equilibrium condition $f(\blambda,\mathbf{u})=0$ and the spectral hypothesis on $D_u f(\blambda,\mathbf{u})$, one must exhibit a right null vector $\mathbf{v}$ and a left null vector $\mathbf{w}$ (normalized so that $(\mathbf{w},\mathbf{v})_{L^2}=1$ in Lemma~\ref{lemma : Cusp bifurcation}; the cusp map below imposes instead $(\mathbf{v},\mathbf{w})_2=1$ on the rescaled vector $\mathcal{O}_{\mathcal{G}}\mathbf{w}$, which by Remark~\ref{rmk : L2 vs ell2} is the same condition), verify that the quadratic coefficient $b$ vanishes, and confirm that the cubic coefficient $c$ is nonzero. The strategy underlying computer-assisted proofs of bifurcations is to encode all of these requirements (except the non-degeneracy $c\neq 0$, which will be handled separately) as the zero of a single map, whose non-degenerate zeros can then be located and validated via a Newton--Kantorovich argument.

To this end, we follow \cite{cusp_jp} and introduce an augmented unknowns vector $\mathbf{x} \bydef (\mathbf{s},\boldsymbol{\lambda},\mathbf{u},\mathbf{v},\mathbf{w},\mathbf{h})$, where $\mathbf{s} \bydef (s_1,s_2) \in \mathbb{R}^2$ are scalar bordering parameters, $\boldsymbol{\lambda} \in \mathbb{R}^2$ is the bifurcation parameter, $\mathbf{u}$ is the equilibrium state, $\mathbf{v}$ and $\mathbf{w}$ are the right and left null vectors of $D_u f(\boldsymbol{\lambda},\mathbf{u})$, and $\mathbf{h}$ is the solution of the \emph{bordering system}
\begin{align}
    \begin{bmatrix}
        D_u f(\boldsymbol{\lambda},\mathbf{u})\mathbf{h} + s_2\mathbf{v} + D_{uu} f(\boldsymbol{\lambda},\mathbf{u})(\mathbf{v},\mathbf{v}) \\
        (\mathbf{w}, \mathbf{h})_2
    \end{bmatrix} = \begin{bmatrix}
        0 \\ 0
    \end{bmatrix}.
\end{align}
The parameter $s_2$ regularizes this system when $D_u f(\blambda,\mathbf{u})$ is singular (at a cusp $s_2 = 0$, and $\mathbf{h}$ reduces to the vector already appearing in Lemma~\ref{lemma : Cusp bifurcation}); $s_1$ plays an analogous role for the eigenvalue equation. The \emph{cusp map} $F$ is then defined component-wise as 
\begin{equation}\label{def : cusp map}
    F(\mathbf{x}) \bydef \begin{bmatrix}
        (\mathbf{v},\mathbf{v})_2 - 1 \\
        (\mathbf{v},\mathbf{w})_2 - 1 \\
        (\mathbf{w},\mathbf{h})_2 \\
        (D_{uu} f(\boldsymbol{\lambda},\mathbf{u})(\mathbf{v},\mathbf{v}),\mathbf{w})_2 \\
        f(\boldsymbol{\lambda},\mathbf{u}) \\
        D_u f(\boldsymbol{\lambda},\mathbf{u}) \mathbf{v} + s_1 \mathbf{v} \\
        D_u f(\boldsymbol{\lambda},\mathbf{u})^* \mathbf{w} \\
        D_u f(\boldsymbol{\lambda},\mathbf{u}) \mathbf{h} + s_2 \mathbf{v} + D_{uu} f(\boldsymbol{\lambda},\mathbf{u}) (\mathbf{v},\mathbf{v})
    \end{bmatrix}.
\end{equation}
The rows of $F$ encode, in order: the normalizations $(\mathbf{v},\mathbf{v})_2 = 1$ and $(\mathbf{v},\mathbf{w})_2 = 1$; the orthogonality condition $(\mathbf{w},\mathbf{h})_2 = 0$ from the bordering system; the vanishing of the quadratic coefficient $b = 0$; the equilibrium condition $f(\blambda,\mathbf{u}) = 0$; the right and left null-vector equations; and the bordering equation for $\mathbf{h}$. 

It is essential to observe that $F$ is written in terms of the sequence-space inner product $(\cdot,\cdot)_2$ of \eqref{def : ell2 inner product} and of the corresponding adjoint $\cdot^*$, whereas Lemma~\ref{lemma : Cusp bifurcation} is formulated in $L^2_{\mathcal{G}}(\mathscr{D})$ with the inner product \eqref{def : L2 inner product} and the adjoint $\cdot^\dagger$. These do \emph{not} coincide, and the following remark is the dictionary between them; it is the translation announced in Section~\ref{sec:introduction}.

\begin{remark}[\bf From $L^2_{\mathcal{G}}(\mathscr{D})$ to $\ell^2_{\mathcal{G}}$]\label{rmk : L2 vs ell2}
Identify $\mathcal{G}$-symmetric functions with their reduced Fourier coefficient sequences through Corollary~\ref{corr : reduced_set}. Substituting $y = \mathcal{L}^{-T}x$ in \eqref{def:mathscr_D} shows that the exponentials $e^{i\mathcal{L}\tilde{n}\cdot y}$ are orthogonal on $\mathscr{D}$ with $\int_{\mathscr{D}}|e^{i\mathcal{L}\tilde{n}\cdot y}|^2dy = |\mathscr{D}| = (2d)^m/|\det\mathcal{L}|$, so Parseval on the full lattice together with the constancy of $\iota(\mathbf{f})$ on $\mathcal{G}$-orbits gives, with $\mathcal{O}_{\mathcal{G}}$ as in \eqref{def : D operator},
\begin{equation}\label{eq : L2 vs ell2}
    (\mathbf{f},\mathbf{g})_{L^2} = \bigl(\mathcal{O}_{\mathcal{G}}\mathbf{f},\mathbf{g}\bigr)_2,
    \qquad\text{hence}\qquad
    A^\dagger = \mathcal{O}_{\mathcal{G}}^{-1}A^*\mathcal{O}_{\mathcal{G}}
\end{equation}
for every linear operator $A$ acting on such sequences ($\mathcal{O}_{\mathcal{G}}$ being diagonal with positive entries, hence self-adjoint for $(\cdot,\cdot)_2$). In particular $\mathbb{M}_u^\dagger = \mathbb{M}_u$: Proposition~\ref{prop : adjoint mult} says precisely that multiplication by a real $\mathcal{G}$-symmetric function is self-adjoint in $L^2_{\mathcal{G}}(\mathscr{D})$, as it must be.

Consequently, at a zero of $F$ the vector $\mathbf{w}^\flat \bydef \mathcal{O}_{\mathcal{G}}^{-1}\mathbf{w}$ satisfies $(\mathbf{w}^\flat,\mathbf{y})_{L^2} = (\mathbf{w},\mathbf{y})_2$ for all $\mathbf{y}$, so that rows 2, 3, 4 and 7 of \eqref{def : cusp map} are exactly the $L^2_{\mathcal{G}}(\mathscr{D})$ conditions $(\mathbf{w}^\flat,\mathbf{v})_{L^2}=1$, $(\mathbf{w}^\flat,\mathbf{h})_{L^2}=0$, $b=0$ and $D_uf(\blambda,\mathbf{u})^\dagger\mathbf{w}^\flat = 0$ of Lemma~\ref{lemma : Cusp bifurcation}; applying $(\mathbf{w},\cdot)_2$ to rows 6 and 8 gives $s_1=s_2=0$, whereupon rows 5, 6, 8 supply the remaining equations. Likewise $\mathcal{O}_{\mathcal{G}}^{-1}\mathbf{z}$ solves the $L^2$-equation \eqref{eq : z vector} whenever $D_uf^*\mathbf{z} = D_{uu}f^*(\mathbf{w},\mathbf{v})$, and $b$, $c$, $\sigma_i$, $\tau_i$, $\Delta$ take the \emph{same} values in the two conventions --- so all of them may be computed with $(\cdot,\cdot)_2$. Thus $F(\mathbf{x})=0$ delivers every hypothesis of Lemma~\ref{lemma : Cusp bifurcation} except $c\neq0$, $\Delta\neq0$ (Theorem~\ref{thm : cusp map nondegen}) and the spectral condition (Sections~\ref{sec : gershgorin swift hohenberg}--\ref{sec : gershgorin gray scott}). Two further consequences, used in Appendix~\ref{sec:proof_thm_cusp_map_nondegen}: $\ker D_uf(\blambda,\mathbf{u})^\dagger = \mathrm{span}(\mathbf{w}^\flat)$ and $\ker D_uf(\blambda,\mathbf{u})^* = \mathrm{span}(\mathbf{w})$; and, by elliptic regularity, $\mathbf{u},\mathbf{v},\mathbf{w},\mathbf{h} \in H^{2r}_{\mathcal{G}}(\mathscr{D})^{\mathscr{p}}$ (rewrite rows 5--8 as $l_{\blambda}(\cdot) = $ an element of $\ell^2_{\mathcal{G},\omega}$, which is a Banach algebra by Corollary~\ref{cor : conv estimates symmetric}), while conversely $H^{2r}_{\mathcal{G}}(\mathscr{D})^{\mathscr{p}}\subset\ell^2_{\mathcal{G},\omega}$ as soon as $ms\leq4r$, as in all our applications.

Finally, a word on notation: $\mathbf{w}$ denotes the left null vector in $L^2_{\mathcal{G}}(\mathscr{D})$ in Section~\ref{sec:introduction} and the one for $(\cdot,\cdot)_2$ from Section~\ref{sec : adjoint mult} onwards; the two differ by $\mathcal{O}_{\mathcal{G}}$, and the above says that this is the only adjustment ever needed.
\end{remark}

We note that Lemma~\ref{lemma : Cusp bifurcation} requires only the single normalization $(\mathbf{w}^\flat,\mathbf{v})_{L^2} = 1$; the additional condition $(\mathbf{v},\mathbf{v})_2 = 1$ (row~1) pins down the scale of $\mathbf{v}$ independently of $\mathbf{w}$, ensuring the augmented system is square and non-degenerate, as required by the Newton--Kantorovich argument of Section~\ref{sec : radii poly}. Note also that $s_1 = s_2 = 0$ is automatic at a zero of $F$ and need not be imposed.

We seek zeros $\mathbf{x} \bydef (\mathbf{s},\boldsymbol{\lambda},\mathbf{u},\mathbf{v},\mathbf{w},\mathbf{h})$ of $F$ in the Banach space 
\[
X_{\mathcal{G},\omega} \bydef \mathbb{R}^4 \times (\ell^2_{\mathcal{G},\omega})^{4\mathscr{p}},
\]
where $\mathscr{p}$ is the number of equations in $f(\blambda,\mathbf{u})$, equipped with the norm
\[
     \|\mathbf{x}\|_{X_{\mathcal{G},\omega}} \bydef |s_1| + |s_2| + |\lambda_1| + |\lambda_2| + \sum_{j = 1}^\mathscr{p} (\|u_j\|_{2,\omega} + \|v_j\|_{2,\omega} + \|w_j\|_{2,\omega} + \|h_j\|_{2,\omega}).
\]
For a tuple $\mathbf{z} \bydef (z_1,\dots,z_{p_0}) \in (\ell^2_{\mathcal{G},\omega})^{p_0}$ we also write
 \begin{align}
     \|\mathbf{z}\|_{2,\omega} \bydef \sum_{j = 1}^{p_0} \|z_j\|_{2,\omega},
 \end{align}
an abuse of notation (this is a sum of $\ell^2_{\mathcal{G},\omega}$-norms, not itself a $\|\cdot\|_{2,\omega}$-norm) that we retain for brevity.

With the functional setting in place, we now state the central property justifying the cusp map formulation: non-degenerate zeros of $F$ are cusp bifurcation points. The following is the infinite-dimensional analogue of Theorem~3.2 in \cite{cusp_jp}. 
\begin{theorem}[{\bf Non-degenerate zeros of the cusp map are cusps}] \label{thm : cusp map nondegen}
Let $f$ be of class $C^3$ in $\mathbf{u}$ and of class $C^2$ in $\blambda$. Suppose that $\mathbf{x} = (\mathbf{s},\blambda,\mathbf{u},\mathbf{v},\mathbf{w},\mathbf{h}) \in X_{\mathcal{G},\omega}$ is a non-degenerate zero of the cusp map $F$ (i.e.\ $F(\mathbf{x}) = 0$ and $DF(\mathbf{x})$ is boundedly invertible), and that $D_u f(\blambda,\mathbf{u})$ has exactly one eigenvalue with zero real part, counted with algebraic multiplicity (so that, $D_uf(\blambda,\mathbf{u})$ being Fredholm of index zero, $\ker D_uf(\blambda,\mathbf{u}) = \mathrm{span}(\mathbf{v})$ is one-dimensional), and that its remaining spectrum lies in $\{\mu\in\mathbb{C} : |\mathrm{Re}(\mu)|\geq\varepsilon\}$ for some $\varepsilon>0$ --- that is, exactly the spectral hypothesis of Lemma~\ref{lemma : Cusp bifurcation}, which is what the Gershgorin enclosures of Sections~\ref{sec : gershgorin swift hohenberg} and~\ref{sec : gershgorin gray scott} deliver. Then
\begin{align}\label{eq : c nonzero}
  c \bydef \frac{1}{6} \big( \mathbf{w} , D_{uuu}f(\blambda,\mathbf{u})(\mathbf{v},\mathbf{v},\mathbf{v}) + 3\,D_{uu}f(\blambda,\mathbf{u})(\mathbf{v},\mathbf{h}) \big)_{2} \neq 0,
\end{align}
and the transversality condition \eqref{eq : Delta} holds, that is $\Delta \neq 0$. Consequently, a cusp bifurcation occurs at $(\blambda,\mathbf{u})$ with normal form $\dot{y} = cy^3 + O(y^4)$ on the center manifold.
\end{theorem}
\begin{proof}
The proof can be found in Appendix~\ref{sec:proof_thm_cusp_map_nondegen}. 
\end{proof}

\begin{remark}\label{rmk : factorization}
The two non-degeneracy conditions of Theorem~\ref{thm : cusp map nondegen} are not independent hypotheses but the two factors of a single determinant. Indeed, the proof in Appendix~\ref{sec:proof_thm_cusp_map_nondegen} establishes the linear isomorphism
\begin{equation}\label{eq : kernel isomorphism}
    \ker DF(\mathbf{x}) \;\cong\; \ker \mathbb{T}, \qquad
    \mathbb{T} \bydef \begin{bmatrix}
        \sigma_1 & \sigma_2 & 0 \\
        \tilde{\tau}_1 & \tilde{\tau}_2 & 0 \\
        \theta_1 & \theta_2 & 6c
    \end{bmatrix} \in M_{3\times3}(\mathbb{R}), \qquad
    \det \mathbb{T} = 6\,c\,\Delta,
\end{equation}
where $\tilde{\tau}_i \bydef \tau_i + \beta\sigma_i$ with $\beta \bydef (\mathbf{v},\mathbf{z})_2$ (so that $\sigma_1\tilde{\tau}_2 - \sigma_2\tilde{\tau}_1 = \Delta$, cf.\ the discussion following \eqref{eq : Delta}), where $\theta \in \mathbb{R}^2$ depends only on $\mathbf{x}$ and is given explicitly in \eqref{eq : theta def}, and the isomorphism is given by $\widehat{\mathbf{X}} \mapsto (\hat{\lambda}_1,\hat{\lambda}_2,(\mathbf{w},\widehat{\mathbf{u}})_2)$. In particular $\ker DF(\mathbf{x}) = \{0\}$ if and only if $c\,\Delta\neq0$; this is the only implication used in the proof of Theorem~\ref{thm : cusp map nondegen}, and it requires no Fredholm property. If, in addition, $DF(\mathbf{x})$ is Fredholm of index zero, then $DF(\mathbf{x})$ is boundedly invertible if and only if $c\,\Delta\neq0$. We stress that this converse is \emph{not} used anywhere in this paper, and that it must be read on the pair of spaces on which $DF(\mathbf{x})$ is a bounded operator, namely
\[
  \mathcal{D} \bydef \mathbb{R}^4\times\bigl(H^{2r}_{\mathcal{G}}(\mathscr{D})^{\mathscr{p}}\bigr)^{4}
  \;\longrightarrow\;
  \mathcal{R} \bydef \mathbb{R}^4\times\bigl(L^2_{\mathcal{G}}(\mathscr{D})^{\mathscr{p}}\bigr)^{4},
\]
and not on $X_{\mathcal{G},\omega}$, on which $D_uf(\blambda,\mathbf{u})$ is unbounded. On that pair, $DF(\mathbf{x})$ is the block-diagonal operator $\mathrm{diag}\bigl(D_uf,D_uf,D_uf^*,D_uf\bigr)$, perturbed by off-diagonal terms which are multiplication operators, hence relatively compact by the compactness of $H^{2r}_{\mathcal{G}}(\mathscr{D}) \hookrightarrow L^2_{\mathcal{G}}(\mathscr{D})$, and bordered by four scalar unknowns and four scalar rows; neither operation changes the index, so $DF(\mathbf{x}):\mathcal{D}\to\mathcal{R}$ is Fredholm of index zero. Since $\mathcal{D}\subset X_{\mathcal{G},\omega}$ (Remark~\ref{rmk : L2 vs ell2}), injectivity on $X_{\mathcal{G},\omega}$ implies injectivity on $\mathcal{D}$, and $DF(\mathbf{x}):\mathcal{D}\to\mathcal{R}$ is boundedly invertible if and only if $c\,\Delta\neq0$.
\end{remark}

\begin{remark}
Theorem~\ref{thm : cusp map nondegen} shows that verifying the hypotheses of Lemma~\ref{lemma : Cusp bifurcation} reduces to two computationally accessible tasks: (i) proving that the cusp map $F$ has a non-degenerate zero (accomplished via Theorem~\ref{th : radii polynomial theorem periodic}), and (ii) verifying that $D_u f(\blambda,\mathbf{u})$ has exactly one eigenvalue on the imaginary axis (accomplished via the Gershgorin arguments of Sections~\ref{sec : gershgorin swift hohenberg} and \ref{sec : gershgorin gray scott}). In particular, neither the non-vanishing of the cubic normal form coefficient $c$ nor the transversality condition $\Delta \neq 0$ requires a separate computation: both are consequences of the non-degeneracy of the zero.
\end{remark}

With the cusp map and its key property established, it remains to state the contraction theorem used to locate its non-degenerate zeros.
 
\subsection{The Newton--Kantorovich Theorem}\label{sec : radii poly}
Our main computational tool is the following Newton--Kantorovich contraction theorem. Its hypotheses require three computable bounds, $Y_0$, $Z_1$, and $Z_2$, which are constructed explicitly for each application in Sections~\ref{sec : swift hohenberg} and \ref{sec : gray scott}. 
\begin{theorem}\label{th : radii polynomial theorem periodic}
Let $\overline{\mathbf{x}} \in X_{\mathcal{G},\omega}$. Let $A$ be an injective, bounded linear operator such that $AF : X_{\mathcal{G},\omega} \to X_{\mathcal{G},\omega}$. Moreover, let $Y_0, Z_1$ be non-negative constants and $Z_2(r) : (0,\infty) \to [0,\infty)$ a non-negative function such that
\begin{align}
    \|AF(\overline{\mathbf{x}})\|_{X_{\mathcal{G},\omega}} &\leq Y_0 \\
    \|I_d - ADF(\overline{\mathbf{x}})\|_{\mathcal{B}(X_{\mathcal{G},\omega})} &\leq Z_1  \\
\|A(DF(\mathbf{x})- DF(\overline{\mathbf{x}}))\|_{\mathcal{B}(X_{\mathcal{G},\omega})} &\leq Z_2(r)r, ~ \text{for all} ~ \mathbf{x} \in B_r(\overline{\mathbf{x}}).
\end{align}
If there exists $r>0$ such that
\begin{equation}\label{condition radii polynomial}
    \frac{1}{2}Z_2(r)r^2 - (1 - Z_1)r + Y_0 <0 \quad \text{and} \quad  Z_1 + Z_2(r)r < 1,
 \end{equation}
then there exists a unique $\widetilde{\mathbf{x}} \in \overline{B_r(\overline{\mathbf{x}})} \subset X_{\mathcal{G},\omega}$ such that $F(\widetilde{\mathbf{x}})=0$, where $B_r(\overline{\mathbf{x}})$ is the open ball of $X_{\mathcal{G},\omega}$ centered at $\overline{\mathbf{x}}$ with radius $r$. Moreover, $\widetilde{\mathbf{x}}$ is a non-degenerate zero of $F$, that is, $DF(\widetilde{\mathbf{x}})$ is boundedly invertible --- see Remark~\ref{rmk : NK nondegeneracy} for the precise sense in which this is to be read, and for the fact that only the injectivity of $DF(\widetilde{\mathbf{x}})$ is used in this paper.
 \end{theorem}
 \begin{proof}
The proof can be found in \cite{van2021spontaneous}.
 \end{proof}

\begin{remark}\label{rmk : NK nondegeneracy}
The last assertion of Theorem~\ref{th : radii polynomial theorem periodic} deserves a word. Since $\widetilde{\mathbf{x}} \in \overline{B_r(\overline{\mathbf{x}})}$,
\[
\|I_d - ADF(\widetilde{\mathbf{x}})\|_{\mathcal{B}(X_{\mathcal{G},\omega})}
\leq \|I_d - ADF(\overline{\mathbf{x}})\|_{\mathcal{B}(X_{\mathcal{G},\omega})} + \|A(DF(\overline{\mathbf{x}}) - DF(\widetilde{\mathbf{x}}))\|_{\mathcal{B}(X_{\mathcal{G},\omega})} \leq Z_1 + Z_2(r)r < 1 ,
\]
so $ADF(\widetilde{\mathbf{x}})$ is boundedly invertible by the Neumann series. Consequently $DF(\widetilde{\mathbf{x}})$ has the bounded left inverse $(ADF(\widetilde{\mathbf{x}}))^{-1}A$, and is in particular \emph{injective}. Injectivity is all that the proof of Theorem~\ref{thm : cusp map nondegen} uses. If moreover $DF(\widetilde{\mathbf{x}})$ is regarded as an operator $\mathcal{D}\to\mathcal{R}$ as in Remark~\ref{rmk : factorization}, where it is Fredholm of index zero, injectivity upgrades to bounded invertibility on that pair of spaces.
\end{remark}
 The construction of the numerical approximation $\overline{\mathbf{x}}$ and the approximate inverse $A$ is carried out separately for each application. We begin with the Swift-Hohenberg equation.

\subsection{Spectral enclosure by Gershgorin disks}\label{sec : gershgorin}

Both applications require the spectral hypothesis of Lemma~\ref{lemma : Cusp bifurcation} --- that $D_u f(\tilde{\blambda},\tilde{\mathbf{u}})$ has exactly one eigenvalue on the imaginary axis --- together with the number $k$ of eigenvalues with strictly positive real part needed in Lemma~\ref{lem : bistability}. We settle both here, once and for all, for a general system of $\mathscr{p}$ equations; Sections~\ref{sec : gershgorin swift hohenberg} and~\ref{sec : gershgorin gray scott} then only have to supply two equation-specific constants.

Assume Theorem~\ref{th : radii polynomial theorem periodic} has been applied with radius $r_0$, yielding $\tilde{\mathbf{x}} \in \overline{B_{r_0}(\overline{\mathbf{x}})}$. The idea is to first \emph{pseudo-diagonalize} $D_u f(\tilde{\blambda},\tilde{\mathbf{u}})$ and then localize its spectrum. Let
\[
P \bydef P^N + \Pi^{>N}, \qquad P^{-1} = (P^N)^{-1}+\Pi^{>N},
\]
where $P^N = \Pi^{\leq N}P^N\Pi^{\leq N}$ is constructed so that its columns approximate eigenvectors of the truncation $\Pi^{\leq N}D_uf(\overline{\blambda},\overline{\mathbf{u}})\Pi^{\leq N}$, and set
\begin{equation}\label{def : D gersh}
  \mathcal{D} \bydef P^{-1}D_u f(\tilde{\blambda},\tilde{\mathbf{u}})P,
  \qquad
  \overline{\mathcal{D}} \bydef P^{-1}D_u f(\overline{\blambda},\overline{\mathbf{u}})P .
\end{equation}
Being similar to $D_uf(\tilde{\blambda},\tilde{\mathbf{u}})$, the operator $\mathcal{D}$ has the same spectrum; $\overline{\mathcal{D}}$, unlike $\mathcal{D}$, is numerically accessible. We write $\mathcal{D}_{i,j}$ for the $(i,j)$ component block, each indexed by $\mathcal{Z}_{\mathrm{red}}(\mathcal{G})$, and use the column-wise Gershgorin theorem of \cite{gershgorin_bredencadiotzurek,gershgorin_statement} rather than the classical row-wise version of \cite{gershgorin_original}.

\begin{lemma}[{\bf Gershgorin enclosure}]\label{lem : gershgorin}
For $n \in \mathcal{Z}_{\mathrm{red}}(\mathcal{G})$ and $j \in \{1,\dots,\mathscr{p}\}$ set
\[
  \mathscr{r}_{n,j} \bydef \frac{1}{|\mathrm{orb}_{\mathcal{G}}(n)|}
  \Bigl(\sum_{k \neq n} |(\mathcal{D}_{j,j})_{k,n}|\,|\mathrm{orb}_{\mathcal{G}}(k)|
  + \sum_{i \neq j}\sum_{k} |(\mathcal{D}_{i,j})_{k,n}|\,|\mathrm{orb}_{\mathcal{G}}(k)|\Bigr),
\]
the sums running over $\mathcal{Z}_{\mathrm{red}}(\mathcal{G})$, and $B_{n,j} \bydef \{z \in \mathbb{C} : |z-(\mathcal{D}_{j,j})_{n,n}| \leq \mathscr{r}_{n,j}\}$. Then $\sigma\bigl(D_uf(\tilde{\blambda},\tilde{\mathbf{u}})\bigr) \subset \bigcup_{n,j}B_{n,j}$, and any sub-collection of $k$ disks whose union is disjoint from all remaining disks contains exactly $k$ eigenvalues, counted with algebraic multiplicity.
\end{lemma}

Since $\tilde{\mathbf{u}}$ is known only through its enclosure, we replace $\mathcal{D}$ by $\overline{\mathcal{D}}$ and inflate the radii. Denote by $\overline{\mathscr{r}}_{n,j}$ and $\overline{B}_{n,j}$ the quantities obtained from $\overline{\mathcal{D}}$, write $Dg \bydef D_{u}g(\overline{\blambda},\overline{\mathbf{u}})$ and assume, as holds for the polynomial nonlinearities considered here, that each block $(Dg)_{i,j} = \mathbb{M}_{\overline{q}_{i,j}}$ is a multiplication operator whose symbol is supported in $I^{2N}$ (true whenever $\overline{\mathbf{u}} = \obpi^{\leq N}\overline{\mathbf{u}}$ and $g$ is at most cubic, as in both applications); set
\begin{equation}\label{def : Qj}
  Q_j \bydef \sum_{i=1}^{\mathscr{p}} \|\overline{q}_{i,j}\|_1 .
\end{equation}
The single equation-specific ingredient is a constant $\varepsilon_{r_0}$ with
\begin{equation}\label{def : eps r0}
  \bigl\|D_uf(\tilde{\blambda},\tilde{\mathbf{u}}) - D_uf(\overline{\blambda},\overline{\mathbf{u}})\bigr\|_{\mathcal{B}(\ell^1_{\mathcal{G}})} \leq \varepsilon_{r_0}
  \qquad\text{for all } \tilde{\mathbf{x}} \in \overline{B_{r_0}(\overline{\mathbf{x}})} .
\end{equation}

\begin{lemma}[{\bf Computable Gershgorin radii}]\label{lem : computable radii}
Let $\mathscr{r}^{\infty} \bydef \max(1,\|P^N\|_{\mathcal{B}(\ell^1_{\mathcal{G}})})\max(1,\|(P^N)^{-1}\|_{\mathcal{B}(\ell^1_{\mathcal{G}})})\,\varepsilon_{r_0}$ and define, for $j\in\{1,\dots,\mathscr{p}\}$,
{\small\begin{align*}
  \overline{\mathscr{r}}^{\leq N}_{n,j} &\bydef \frac{1}{|\mathrm{orb}_{\mathcal{G}}(n)|}\Bigl(\sum_{k \in I^N,\, (k,i)\neq(n,j)} |(\overline{\mathcal{D}}_{i,j})_{k,n}|\,|\mathrm{orb}_{\mathcal{G}}(k)| + \sum_{k \in I^{3N}\setminus I^N}\sum_{i=1}^{\mathscr{p}} |((Dg\,P^N)_{i,j})_{k,n}|\,|\mathrm{orb}_{\mathcal{G}}(k)|\Bigr), & n &\in I^N,\\
  \overline{\mathscr{r}}^{3N\setminus N}_{n,j} &\bydef \frac{1}{|\mathrm{orb}_{\mathcal{G}}(n)|}\sum_{k \in I^N}\sum_{i=1}^{\mathscr{p}} |(((P^N)^{-1}Dg)_{i,j})_{k,n}|\,|\mathrm{orb}_{\mathcal{G}}(k)| \;+\; Q_j, & n &\in I^{3N}\setminus I^N,
\end{align*}}
where in the first line the index $i$ also runs over $\{1,\dots,\mathscr{p}\}$. Then $\mathscr{r}_{n,j} \leq \rho_{n,j}$, where
\[
  \rho_{n,j} \bydef
  \begin{cases}
    \overline{\mathscr{r}}^{\leq N}_{n,j} + \mathscr{r}^{\infty}, & n \in I^N,\\
    \overline{\mathscr{r}}^{3N\setminus N}_{n,j} + \mathscr{r}^{\infty}, & n \in I^{3N}\setminus I^N,\\
    Q_j + \mathscr{r}^{\infty}, & n \notin I^{3N},
  \end{cases}
\]
and consequently $B_{n,j} \subset \{z : |z - (\overline{\mathcal{D}}_{j,j})_{n,n}| \leq \rho_{n,j}\}$ for all $n$ and $j$.
\end{lemma}
\begin{proof}
See Appendix~\ref{apen : gershgorin}.
\end{proof}

All the quantities $\rho_{n,j}$ for $n \in I^{3N}$ are finite sums of numerically accessible terms, hence computable with interval arithmetic; the tail is covered by the single constant $Q_j+\mathscr{r}^\infty$. This yields the following criterion, which is what the applications verify.

\begin{prop}[{\bf Spectral enclosure}]\label{prop : gershgorin enclosure}
Suppose that
\begin{enumerate}[label=(\roman*),nosep,leftmargin=2.2em]
\item $|\mathrm{Re}\,(\overline{\mathcal{D}}_{j,j})_{n,n}| > \rho_{n,j}$ for every $n \in I^{3N}$ and every $j$, with exactly one exception $(n_0,j_0)$;
\item $\displaystyle\sup_{n \in \mathcal{Z}_{\mathrm{red}}(\mathcal{G})\setminus I^{3N}} \mathrm{Re}\,(\overline{\mathcal{D}}_{j,j})_{n,n} + Q_j + \mathscr{r}^{\infty} < 0$ for every $j$;
\item the disk $\{z : |z-(\overline{\mathcal{D}}_{j_0,j_0})_{n_0,n_0}| \leq \rho_{n_0,j_0}\}$ is disjoint from all the others.
\end{enumerate}
Then $D_uf(\tilde{\blambda},\tilde{\mathbf{u}})$ has exactly one eigenvalue with zero real part, namely the algebraically simple eigenvalue $0$, and every other eigenvalue lies in $\{z : \mathrm{Re}(z) \neq 0\}$. If, in addition,
\begin{enumerate}[label=(\roman*),nosep,leftmargin=2.2em,start=4]
\item $\mathrm{Re}\,(\overline{\mathcal{D}}_{j,j})_{n,n} + \rho_{n,j} < 0$ for every $n \in I^{3N}$ and every $j$ with $(n,j)\neq(n_0,j_0)$,
\end{enumerate}
then $k=0$. (Conditions (i)--(iii) alone leave open the possibility of disks lying in $\{\mathrm{Re}(z)>0\}$; in general $k$ is the number of eigenvalues contained in those disks, counted as in Lemma~\ref{lem : gershgorin}.)
\end{prop}
\begin{proof}
See Appendix~\ref{apen : gershgorin}.
\end{proof}

Condition (ii) is checked in practice by observing that $(\overline{\mathcal{D}}_{j,j})_{n,n}$ is decreasing in $|n|$ beyond $I^{3N}$ --- the symbol of $l_{\overline{\blambda}}$ dominating there --- so that the supremum is attained at the smallest such $|n|$; this holds as soon as $N$ is large enough, which is the case in all our applications.

\section{Application to the Swift-Hohenberg PDE}\label{sec : swift hohenberg}

We now apply the framework of Section~\ref{sec : set-up} to the Swift-Hohenberg (SH) equation
\begin{equation}\label{swift hohenberg}
\partial_t u = f(\boldsymbol{\lambda},u) \bydef l_{\boldsymbol{\lambda}} u + g(\boldsymbol{\lambda},u), \qquad
l_{\boldsymbol{\lambda}} \bydef -(I_d + \Delta)^2 - \lambda_1 I_d, \qquad
g(\boldsymbol{\lambda},u) \bydef \lambda_2 u^2 - u^3.
\end{equation}

The Swift-Hohenberg equation is among the most thoroughly studied models for pattern formation. It supports periodic patterns with diverse spatial symmetry groups---including square, hexagonal, and general dihedral symmetries~\cite{hexagon2021lloyd,hexagon2008,squareSakaguchi_1997,jason_review_paper,jason_spot_paper,jason_ring_paper}---and displays a rich multistability structure that makes it a natural benchmark for our cusp validation framework. Rigorous computer-assisted existence proofs for stationary periodic solutions have been developed in one, two, and three dimensions~\cite{lessard_2,lessard_1,dominic_sh_periodic,gabriel_snaking,sh_cadiot,symmetry_blanco_cadiot,olivier_radial,symmetry_blanco_cadiot}.

Specializing the general framework of Section~\ref{sec : set-up} to~\eqref{swift hohenberg}, note that this is a scalar equation ($\mathscr{p}=1$), so throughout this section we simplify notation and write $\mathbf{u} = u$, $\mathbf{v} = v$, $\mathbf{w} = w$, $\mathbf{h} = h$. One computes
\[
D_{uu}f(\boldsymbol{\lambda},u)(v,v) = (2\lambda_2 - 6u)v^2
\]
and, using Proposition~\ref{prop : adjoint mult},
\[
D_u f(\boldsymbol{\lambda},u)^* w = l_{\boldsymbol{\lambda}} w + \mathcal{O}_{\mathcal{G}}\bigl((2\lambda_2 u - 3u^2)\star (\mathcal{O}_{\mathcal{G}}^{-1}w)\bigr).
\]
Substituting these into the general cusp map~\eqref{def : cusp map} of Section~\ref{sec : cusp map} and representing the map in Fourier space yields
\begin{equation}\label{sh cusp map}
F(\mathbf{x}) \bydef \begin{bmatrix}
(v,v)_2 - 1 \\
(v,w)_2 - 1 \\
(h,w)_2 \\
\bigl(2\lambda_2 v^2 - 6u v^2,\, w\bigr)_2 \\
l_{\boldsymbol{\lambda}} u + \lambda_2 u^2 - u^3 \\
l_{\boldsymbol{\lambda}} v + (2\lambda_2 u - 3u^2)v + s_1 v \\
l_{\boldsymbol{\lambda}} w + \mathcal{O}_{\mathcal{G}}\bigl((2\lambda_2 u - 3u^2)\star(\mathcal{O}_{\mathcal{G}}^{-1}w)\bigr) \\
l_{\boldsymbol{\lambda}} h + (2\lambda_2 u - 3u^2)h + s_2 v + (2\lambda_2 - 6u)v^2
\end{bmatrix}.
\end{equation}
We seek zeros of~\eqref{sh cusp map} in $X_{\mathcal{G},\omega} = \mathbb{R}^4 \times (\ell^2_{\mathcal{G},\omega})^4$ by applying Theorem~\ref{th : radii polynomial theorem periodic}. The construction of $\overline{\mathbf{x}}$ and $A$ is described in Section~\ref{sec : numerical aspects}; the bounds $Y_0$, $Z_1$, $Z_2$ required by that theorem are computed in Section~\ref{sec : bounds sh}; the spectral hypothesis of Lemma~\ref{lemma : Cusp bifurcation} is verified in Section~\ref{sec : gershgorin swift hohenberg}; and the rigorous cusp proofs are assembled in Section~\ref{sec : caps sh}.

\subsection{Numerical approximation and approximate inverse}\label{sec : numerical aspects}

\noindent\textbf{Constructing $\overline{\mathbf{x}}$.}
The numerical approximation $\overline{\mathbf{x}} = (\bar{\mathbf{s}}, \bar{\boldsymbol{\lambda}}, \bar{u}, \bar{v}, \bar{w}, \bar{h})$ is obtained via a three-stage procedure. First, one-parameter pseudo-arclength continuation in $\lambda_1$ locates an approximate saddle-node bifurcation, which is refined using Newton's method on the saddle-node map~\cite{jp_saddle_node1,jp_saddle_node} to obtain accurate values of $(\bar{\lambda}_1, \bar{u}, \bar{v})$. Second, fixing these values, continuation in $\lambda_2$ traces the fold curve until a saddle-node of the saddle-node map is encountered, signalling a candidate cusp; see~\cite{cusp_jp} for the analogous finite-dimensional procedure. Third, Newton's method is applied to the full cusp map~\eqref{sh cusp map}, initialized from the data of the previous steps (with $\bar{w}$ computed from the approximate adjoint equation and $\bar{h}$ corrected by Newton iterations). The resulting $\overline{\mathbf{x}}$ satisfies $\bar{s}_1 = \bar{s}_2 = 0$.

Fix $N \in \mathbb{N}$ as the truncation index and recall the projections $\Pi^{\leq N}$, $\Pi^{>N}$ from~\eqref{def : piN and pisubN}. For $\mathbf{x} = (\mathbf{s},\boldsymbol{\lambda},u,v,w,h) \in X_{\mathcal{G},\omega}$, we define the tuple-level projections
{\small\begin{align}
    &\bpi^{\leq N}(u,v,w,h) \bydef (\Pi^{\leq N} u, \Pi^{\leq N} v, \Pi^{\leq N} w, \Pi^{\leq N} h), \quad
     \bpi^{>N}(u,v,w,h) \bydef (\Pi^{> N} u, \Pi^{> N} v, \Pi^{> N} w, \Pi^{> N} h), \\
    &\obpi^{\leq N}\mathbf{x} \bydef (\mathbf{s},\boldsymbol{\lambda},\Pi^{\leq N} u, \Pi^{\leq N} v, \Pi^{\leq N} w, \Pi^{\leq N} h), \quad
     \obpi^{>N}\mathbf{x} \bydef (0,0,0,0,\Pi^{> N} u, \Pi^{> N} v, \Pi^{> N} w, \Pi^{> N} h).
\end{align}}
We also introduce the shorthand $\mathcal{O}_{\mathcal{G},N} \bydef \Pi^{\leq N}\mathcal{O}_{\mathcal{G}}\Pi^{\leq N}$ for the truncation of the orbit-size operator~\eqref{def : D operator} to the first $N$ modes. We take $\overline{\mathbf{x}} \bydef \obpi^{\leq N}\overline{\mathbf{x}}$, so that $\overline{\mathbf{x}}$ has finitely many nonzero Fourier modes. In Fourier space, $l_{\boldsymbol{\lambda}}$ is diagonal with entries
\[
(l_{\bar{\boldsymbol{\lambda}}})_{n,n} = -(1 - |\mathcal{L}\tilde{n}|_2^2)^2 - \bar{\lambda}_1, \quad n \in \mathcal{Z}_{\mathrm{red}}(\mathcal{G}).
\]
For $|n|$ sufficiently large, $(l_{\bar{\boldsymbol{\lambda}}})_{n,n}$ is bounded away from zero, ensuring invertibility in the tail.

\noindent\textbf{Constructing \boldmath{$A$}.}
The approximate inverse is decomposed as $A = \obpi^{\leq N}A + \obpi^{>N}A$, where the finite part $A^N \approx (\obpi^{\leq N}DF(\overline{\mathbf{x}})\obpi^{\leq N})^{-1}$ is computed and stored as a matrix. The tail part inverts the dominant diagonal operator $l_{\bar{\boldsymbol{\lambda}}}$ on each sequence component independently: since $(l_{\bar{\boldsymbol{\lambda}}})_{n,n} = -(1-|\mathcal{L}\tilde{n}|_2^2)^2 - \bar{\lambda}_1$ is nonzero for all $n \in \mathcal{Z}_{\mathrm{red}}(\mathcal{G}) \setminus I^N$, we define, for each $z \in \{u,v,w,h\}$,
\[
(\Pi_A^{>N}z)_n \bydef \begin{cases}
0, & n \in I^N \\
\frac{z_n}{(l_{\bar{\boldsymbol{\lambda}}})_{n,n}}, &n \in \mathcal{Z}_{\mathrm{red}}(\mathcal{G}) \setminus I^N.
\end{cases}
\]
In full, for $\mathbf{x} = (\mathbf{s},\boldsymbol{\lambda},u,v,w,h) \in X_{\mathcal{G},\omega}$,
\[
\obpi^{>N}A\,\mathbf{x} \bydef \bigl(0,0,\;\Pi_A^{>N}u,\;\Pi_A^{>N}v,\;\Pi_A^{>N}w,\;\Pi_A^{>N}h\bigr).
\]
We define
\begin{equation}\label{def : Linfty}
\mathcal{L}_{\infty} \bydef \max_{n \in \mathcal{Z}_{\mathrm{red}}(\mathcal{G}) \setminus I^N} \frac{1}{|(1 - |\mathcal{L}\tilde{n}|_2^2)^2 + \bar{\lambda}_1|}.
\end{equation}

\begin{lemma}\label{lem : tailA}
$\|\obpi^{>N}A\|_{\mathcal{B}(X_{\mathcal{G},\omega})} \leq \mathcal{L}_{\infty}$.
\end{lemma}
\begin{proof}
For $\|\mathbf{x}\|_{X_{\mathcal{G},\omega}} = 1$, the definition of the tail action gives
\[
 \|\obpi^{> N} A \mathbf{x}\|_{X_{\mathcal{G},\omega}} 
    \leq \mathcal{L}_\infty \sum_{z \in \{u,v,w,h\}} \left\|\Pi^{> N}z\right\|_{2,\omega}
    \leq \mathcal{L}_\infty \|\mathbf{x}\|_{X_{\mathcal{G},\omega}} = \mathcal{L}_\infty. \qquad \qedhere
\]
\end{proof}

\subsection{Newton--Kantorovich bounds for Swift-Hohenberg}\label{sec : bounds sh}

With $\overline{\mathbf{x}}$ and $A$ in hand, we construct the three bounds required by Theorem~\ref{th : radii polynomial theorem periodic}. Throughout this section we write $a^{\mathcal{N}} \bydef \Pi^{\leq \mathcal{N}}a$ for any sequence $a \in \ell^2_{\mathcal{G},\omega}$ and some $\mathcal{N} \in \mathbb{N}$.

Because the nonlinearity in~\eqref{swift hohenberg} is cubic, $F(\overline{\mathbf{x}})$ has Fourier support contained in $I^{3N}$. This leads to the following result, which provides a computable bound for $Y_0$.

\begin{lemma}[\bf The \boldmath{$Y_0$} bound]\label{lem : Y0 SH}
Let $Y_0 \bydef \|\obpi^{\leq 3N} AF(\overline{\mathbf{x}})\|_{X_{\mathcal{G},\omega}}$. Then $\|AF(\overline{\mathbf{x}})\|_{X_{\mathcal{G},\omega}} \leq Y_0$.
\end{lemma}

The following lemma bounds the variation of $DF$ in a ball of radius $r$ around $\overline{\mathbf{x}}$. It relies on the convolution constants $\kappa$, $\kappa_0$ from Section~\ref{sec:conv_estimates} and the orbit-size bound $\mathscr{O}_{\mathrm{max}}$ from~\eqref{def : omax}.

\begin{lemma}[\bf The \boldmath{$Z_2$} bound]\label{lem : Z2}
Define
\begin{align}\label{def : Z2 SH}
Z_2(r) \bydef \bigl(\|A^N\|_{\mathcal{B}(X_{\mathcal{G},\omega})} + \mathcal{L}_{\infty}\bigr)\sum_{k=1}^{5} Z_{2,k}(r),
\end{align}
where
{\footnotesize\begin{align*}
Z_{2,1}(r) &\bydef \bigl(2\kappa r^2 + 4\|\overline{v}\|_{1}r + 2\|\overline{v}^2\|_{2} + (2\kappa r + 4\|\overline{v}\|_{1}) \|\overline{w}\|_{2}\bigr)r, \\
Z_{2,2}(r) &\bydef 6 + \kappa_0 (28\|\overline{v}\|_{1} + 22\kappa r + 4\|\overline{\lambda}_2 - 3\overline{u}\|_{1})\|\overline{w}\|_{2} + 4\kappa_0 \|\overline{\lambda}_2 \overline{v} - 3 \overline{u}\overline{v}\|_{2} + 4\kappa_0 \|\overline{\lambda}_2 - 3\overline{u}\|_{1}r + 6\|\overline{u}\overline{v}\|_{1} \\
    &+ 2(3\kappa \|\overline{u}\|_{1}r + (2|\overline{\lambda}_2| + 8\kappa r + 14\kappa_0 r)\|\overline{v}\|_{1} + 8\|\overline{v}^2\|_{1}+ 6\kappa_0\|\overline{v}^2\|_{2} + |\overline{\lambda}_2|\kappa r + (4\kappa^2 + \kappa_0(10 + 12\kappa)) r^2), \\
Z_{2,3}(r) &\bydef (7+2\mathscr{O}_{\mathrm{max}})\kappa r + 6 +  (6+2\mathscr{O}_{\mathrm{max}})\|\overline{u}\|_{1,\sqrt{\omega}}
 + 6 \|\overline{v}\|_{1,\sqrt{\omega}} + 2 \mathscr{O}_{\mathrm{max}}\|\mathcal{O}_{\mathcal{G},N}^{-1} \overline{w}\|_{1,\sqrt{\omega}} + 2 \|\overline{h}\|_{1,\sqrt{\omega}}, \\
Z_{2,4}(r) &\bydef (14 + 4\kappa \mathscr{O}_{\mathrm{max}})\kappa r + 6 +  (6+2\kappa \mathscr{O}_{\mathrm{max}})\|\overline{u}\|_{1,\sqrt{\omega}}
 + 6 \|\overline{v}\|_{1,\sqrt{\omega}} + 2\kappa \mathscr{O}_{\mathrm{max}}\|\mathcal{O}_{\mathcal{G},N}^{-1}\overline{w}\|_{1,\sqrt{\omega}} + 2 \|\overline{h}\|_{1,\sqrt{\omega}}, \\
Z_{2,5}(r) &\bydef (14\kappa + 4\mathscr{O}_{\mathrm{max}}) \|\overline{\lambda}_2 - 3\overline{u}\|_{1,\sqrt{\omega}}+ 30\kappa \|\overline{v}\|_{1,\sqrt{\omega}} + 6\kappa \|\overline{h}\|_{1,\sqrt{\omega}} + 6\kappa \mathscr{O}_{\mathrm{max}}\|\mathcal{O}_{\mathcal{G},N}^{-1} \overline{w}\|_{1,\sqrt{\omega}} + (39+9\mathscr{O}_{\mathrm{max}})\kappa^2 r.
\end{align*}}
Then $\|A(DF(\mathbf{x})- DF(\overline{\mathbf{x}}))\|_{\mathcal{B}(X_{\mathcal{G},\omega})} \leq Z_2(r)r$ for all $\mathbf{x} \in B_r(\overline{\mathbf{x}})$.
\end{lemma}
\begin{proof}
See Appendix~\ref{apen : Z2}.
\end{proof}

To control $\|I_d - ADF(\overline{\mathbf{x}})\|$, we introduce the auxiliary sequences
\begin{equation}\label{def : q notation sh}
\bar{q}_0 \bydef (2\bar{\lambda}_2 - 6\overline{u})\star(\mathcal{O}_{\mathcal{G},N}^{-1}\overline{w}), 
\quad
\bar{q}_1 \bydef 2\bar{\lambda}_2\bar{u} - 3\bar{u}^2,
\quad
\bar{q}_2 \bydef 2\bar{\lambda}_2\overline{v} - 6\overline{u} \star \overline{v}, \quad
\bar{q}_3 \bydef 2\bar{\lambda}_2\overline{h} -6\overline{u} \star \overline{h} - 6\overline{v}^2.
\end{equation}
The bound $Z_1$ decomposes into a finite-dimensional part $Z_{1,1}$ (the norm of a finite matrix, computable directly), an interaction term $Z_{1,2}$ between $A^N$ and the tail of $DF(\overline{\mathbf{x}})$, and a truncation error term $Z_\infty$ arising from the finite approximation $\overline{\mathbf{x}}$.

\begin{lemma}[\bf The \boldmath{$Z_1$} bound]\label{lem : Z1 full}
Let $\overline{L}_{\mathbf{s},\blambda}$ and $M^N(\overline{\mathbf{x}})$ be defined as in Appendix \ref{apen : Z1}. Let $Z_{1,1},Z_{1,2},Z_{\infty} > 0$ be defined as 
\begin{align}
    &Z_{1,1} \bydef \left\|\obpi^{\leq N} - A^N (\overline{L}_{\mathbf{0},\overline{\blambda}} + M^N(\overline{\mathbf{x}}))\obpi^{\leq 2N}\right\|_{\mathcal{B}(X_{\mathcal{G},\omega})},\label{def : Z11}\\
    &Z_{1,2} \bydef 2 \mathcal{L}_{\infty} \left(\left(3+\mathscr{O}_{\mathrm{max}}\right)\|\bar{q}_{1}^N\|_{1,\sqrt{\omega}} + 3\|\bar{q}_{2}^N\|_{1,\sqrt{\omega}} + \|\bar{q}_3^N\|_{1,\sqrt{\omega}} + \mathscr{O}_{\mathrm{max}}\|\bar{q}_0^N\|_{1,\sqrt{\omega}}\right),\label{def : Z12}
\end{align}
and
{\small\begin{align}
Z_{\infty} &\bydef  \left(3+\mathscr{O}_{\mathrm{max}}\right)\|\bar{q}_1 - \bar{q}_1^N\|_{1,\sqrt{\omega}}
+3 \|\bar{q}_2 - \bar{q}_2^N\|_{1,\sqrt{\omega}}
+ \|\bar{q}_3 - \bar{q}_3^N\|_{1,\sqrt{\omega}}
+ \mathscr{O}_{\mathrm{max}}\|\bar{q}_0 - \bar{q}_0^N\|_{1,\sqrt{\omega}} \nonumber\\
&\quad + \|(\Pi^{3N} - \Pi^{2N})(\bar{q}_2\star \overline{v})\|_{2}
+\|\overline{u}^2 - \Pi^{\leq N} \overline{u}^2\|_{2,\omega}
+ 2\|\overline{u}\star \overline{v} - \Pi^{\leq N}(\overline{u}\star \overline{v})\|_{2,\omega} \nonumber\\
&\quad + 2\|\mathcal{O}_{\mathcal{G},2N}\bigl(\overline{u}\star(\mathcal{O}_{\mathcal{G},N}^{-1}\overline{w}) - \Pi^{\leq N}(\overline{u}\star(\mathcal{O}_{\mathcal{G},N}^{-1}\overline{w}))\bigr)\|_{2,\omega}
+2\|\overline{u}\star \overline{h}+\overline{v}^2 - \Pi^{\leq N}(\overline{u}\star \overline{h}+\overline{v}^2)\|_{2,\omega} \nonumber\\
&\quad + 6\|(\Pi^{\leq 3N} - \Pi^{\leq 2N}) (\overline{v}^2 \star \overline{w})\|_{2} + 2\|(\Pi^{\leq 3N} - \Pi^{\leq 2N})  (\overline{q}_2 \star \overline{w})\|_{2}.\label{def : Zinf sh}
\end{align}}
Then, defining $Z_1 \bydef Z_{1,1} + Z_{1,2} + (\|A^N\|_{\mathcal{B}(X_{\mathcal{G},\omega})} + \mathcal{L}_{\infty})Z_{\infty}$,
\[
\|I_d - ADF(\overline{\mathbf{x}})\|_{\mathcal{B}(X_{\mathcal{G},\omega})} \leq Z_1.
\]
\end{lemma}
\begin{proof}
See Appendix~\ref{apen : Z1}.
\end{proof}

\begin{remark}
All three bounds can be enclosed rigorously using interval arithmetic. The quantities $Z_{1,1}$ and $\|A^N\|_{\mathcal{B}(X_{\mathcal{G},\omega})}$ are computed as norms of finite matrices; $Z_{1,2}$ and $Z_\infty$ involve norms of truncated Fourier sequences, controlled via the convolution estimates of Section~\ref{sec:conv_estimates}.
\end{remark}

\subsection{Gershgorin spectral enclosure for Swift-Hohenberg}\label{sec : gershgorin swift hohenberg}

We apply the framework of Section~\ref{sec : gershgorin} with $\mathscr{p}=1$, so that a single family of disks $B_n \bydef B_{n,1}$ is involved and $\mathcal{D} = P^{-1}D_uf(\tilde{\blambda},\tilde u)P$ is a scalar operator on $\ell^2_{\mathcal{G},\omega}$. Here $Dg(\overline{\lambda}_2,\overline{u}) = \mathbb{M}_{\overline{q}_1}$ with $\overline{q}_1 \bydef 2\overline{\lambda}_2\overline{u}-3\overline{u}^2$, so that \eqref{def : Qj} reads $Q_1 = \|\overline{q}_1\|_1$. It remains to produce the constant $\varepsilon_{r_0}$ of \eqref{def : eps r0}.

\begin{lemma}\label{lem : eps r0 sh}
For Swift--Hohenberg one may take
\[
  \varepsilon_{r_0} \bydef (2\kappa_0+3\kappa_0^2)r_0^2 + \bigl(1+2(1+3\kappa_0)\|\overline{u}\|_1 + 2|\overline{\lambda}_2|\kappa_0\bigr)r_0 .
\]
\end{lemma}
\begin{proof}
Write $\tilde u = \overline u + y$ and $\tilde{\blambda} = \overline{\blambda}+\boldsymbol{\eta}$ with $\|y\|_{2,\omega}\leq r_0$ and $|\boldsymbol{\eta}|_1\leq r_0$, so that $\|y\|_1 \leq \kappa_0 r_0$ by Corollary~\ref{cor : conv estimates symmetric}. Since $D_uf(\blambda,u) = l_{\blambda} + 2\lambda_2u - 3u^2$ and $l_{\blambda}$ depends on $\blambda$ only through the additive constant $-\lambda_1$,
\[
D_uf(\tilde{\blambda},\tilde u) - D_uf(\overline{\blambda},\overline u) = -\eta_1 + 2\bigl(\eta_2 y + \eta_2\overline u + \overline{\lambda}_2 y\bigr) - 3\bigl(y^2+2y\overline u\bigr),
\]
a multiplication operator, whose $\mathcal{B}(\ell^1_{\mathcal{G}})$-norm is bounded by the $\ell^1$-norm of its symbol; estimating each term with $|\eta_j|\leq r_0$ and $\|y\|_1\leq\kappa_0r_0$ gives the stated $\varepsilon_{r_0}$.
\end{proof}

The tail condition (ii) of Proposition~\ref{prop : gershgorin enclosure} is verified through the symbol of $l_{\overline{\blambda}}$: since $\overline{\mathcal{D}}_{n,n} = -(1-\tfrac{\pi^2}{d^2}|\mathcal{L}n|^2)^2-\overline{\lambda}_1+\sum_{j\in\mathrm{orb}_{\mathcal{G}}(n)}(\iota(\overline{q}_1))_{n-j}$ for $n \notin I^{3N}$ (the diagonal entry of $\mathbb{M}_{\overline{q}_1}$, cf.\ the formula for $(\mathcal{M}_u)_{n,k}$ in the proof of Proposition~\ref{prop : adjoint mult}), the assumption
\begin{equation}\label{tail assumption sh}
  \min_{n \in \mathcal{Z}_{\mathrm{red}}(\mathcal{G})\setminus I^{3N}}\Bigl(1-\frac{\pi^2}{d^2}|\mathcal{L}n|^2\Bigr)^2 > \|\overline{q}_1\|_1 - \overline{\lambda}_1
\end{equation}
gives $\overline{\mathcal{D}}_{n,n} < 0$ there and makes $n \mapsto \overline{\mathcal{D}}_{n,n}$ decreasing, so that the supremum in (ii) is attained at the smallest $|n|$ outside $I^{3N}$ and (ii) becomes a single verifiable inequality. Condition~\eqref{tail assumption sh} holds whenever $N$ is large enough, which is always the case in our applications. Conditions (i)--(iv) of Proposition~\ref{prop : gershgorin enclosure} are then checked by interval arithmetic, certifying that all eigenvalues of $D_uf(\tilde{\blambda},\tilde u)$ other than the simple eigenvalue $0$ lie in the open left half-plane, i.e.\ $k=0$.

\subsection{Computer-assisted proofs for Swift-Hohenberg}\label{sec : caps sh}

Combining Theorem~\ref{th : radii polynomial theorem periodic} (a non-degenerate zero $\tilde{\mathbf{x}}$ of~\eqref{sh cusp map}), Theorem~\ref{thm : cusp map nondegen} (which shows $c \neq 0$ from non-degeneracy), and the Gershgorin spectral enclosure of Section~\ref{sec : gershgorin swift hohenberg} (exactly one zero eigenvalue), all hypotheses of Lemma~\ref{lemma : Cusp bifurcation} are simultaneously verified. Specializing~\eqref{eq:normal_form_coeff_c} to~\eqref{swift hohenberg} (using $D_{uuu}f = -6$ and $D_{uu}f(v,h) = (2\lambda_2 - 6u)vh$) yields the normal form coefficient
\begin{equation}\label{c : SH}
c = \bigl(w,v^3 + 3uvh\bigr)_2,
\end{equation}
whose sign is certified by interval arithmetic applied to the rigorous enclosure of $\tilde{\mathbf{x}}$. With $k$ and $\mathrm{sign}(c)$ both certified, Lemma~\ref{lem : bistability} determines the stability structure of the three nearby solutions.

We prove cusp bifurcations in three distinct settings: the one-dimensional problem with $\mathbb{Z}_2$-symmetry, and the two-dimensional problem with square ($D_4$) and hexagonal ($D_6$) symmetry.

\begin{theorem}[$\mathbb{Z}_2$ cusp bifurcation in 1D Swift-Hohenberg]\label{th : SH 1D cusp}
Set $r_0 \bydef 9 \times 10^{-11}$ and $\mathcal{L} = 1$. Let $(\omega_n)_{n \in \mathbb{Z}} \bydef (1 + |n|)^{1.01})_{n \in \mathbb{Z}}$. When $\boldsymbol{\lambda} \approx (0.05293,\, 1.22174)$, there exists a unique $\tilde{\mathbf{x}} \in \overline{B_{r_0}(\overline{\mathbf{x}})} \subset X_{\mathbb{Z}_2,\omega}$ satisfying $F(\tilde{\mathbf{x}}) = 0$ in~\eqref{sh cusp map}. Moreover, $(\tilde{\boldsymbol{\lambda}}, \tilde{u})$ is a cusp bifurcation with $k = 0$ and $c < -3.114 $; in particular, the three coexisting solutions exhibit bistability.
\end{theorem}
\begin{proof}
Take $N = 170$ and $d = 50$. The numerical approximation $\overline{\mathbf{x}}$ and the approximate inverse $A^N$ are constructed as described in Section~\ref{sec : numerical aspects}. Using the Julia package~\cite{CuspProofs.jl} with interval arithmetic, Lemmas~\ref{lem : Y0 SH}--\ref{lem : Z1 full} yield
\[
\|A^N\|_{\mathcal{B}(X_{\mathbb{Z}_2,\omega})} \leq 56548.1, \quad
Y_0 \leq 1.215 \times 10^{-11}, \quad
Z_2(r_0) \leq 1.993 \times 10^9, \quad
Z_1 \leq 0.2873.
\]
These values satisfy the conditions of Theorem~\ref{th : radii polynomial theorem periodic}, establishing the unique zero $\tilde{\mathbf{x}}$. The Gershgorin argument of Section~\ref{sec : gershgorin swift hohenberg} confirms that all eigenvalues of $D_u f(\tilde{\boldsymbol{\lambda}}, \tilde{u})$ other than zero have negative real part, so $k = 0$. Interval arithmetic applied to~\eqref{c : SH} certifies $c < -3.114 < 0$, and Lemma~\ref{lem : bistability} yields bistability.
\end{proof}
\begin{figure}[H]
\centering
 \begin{minipage}{.24\linewidth}
  \centering\epsfig{figure=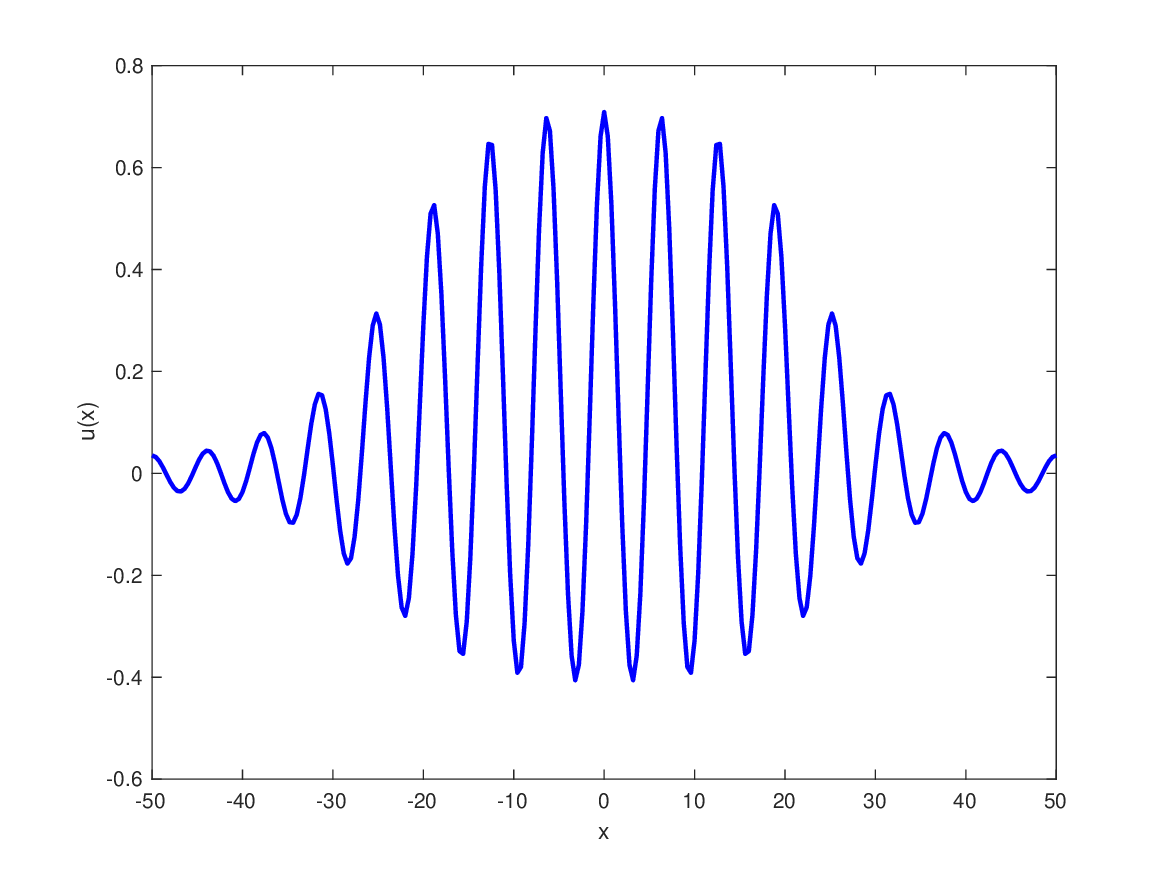,width=\linewidth}
  \end{minipage}%
 \begin{minipage}{.24\linewidth}
  \centering\epsfig{figure=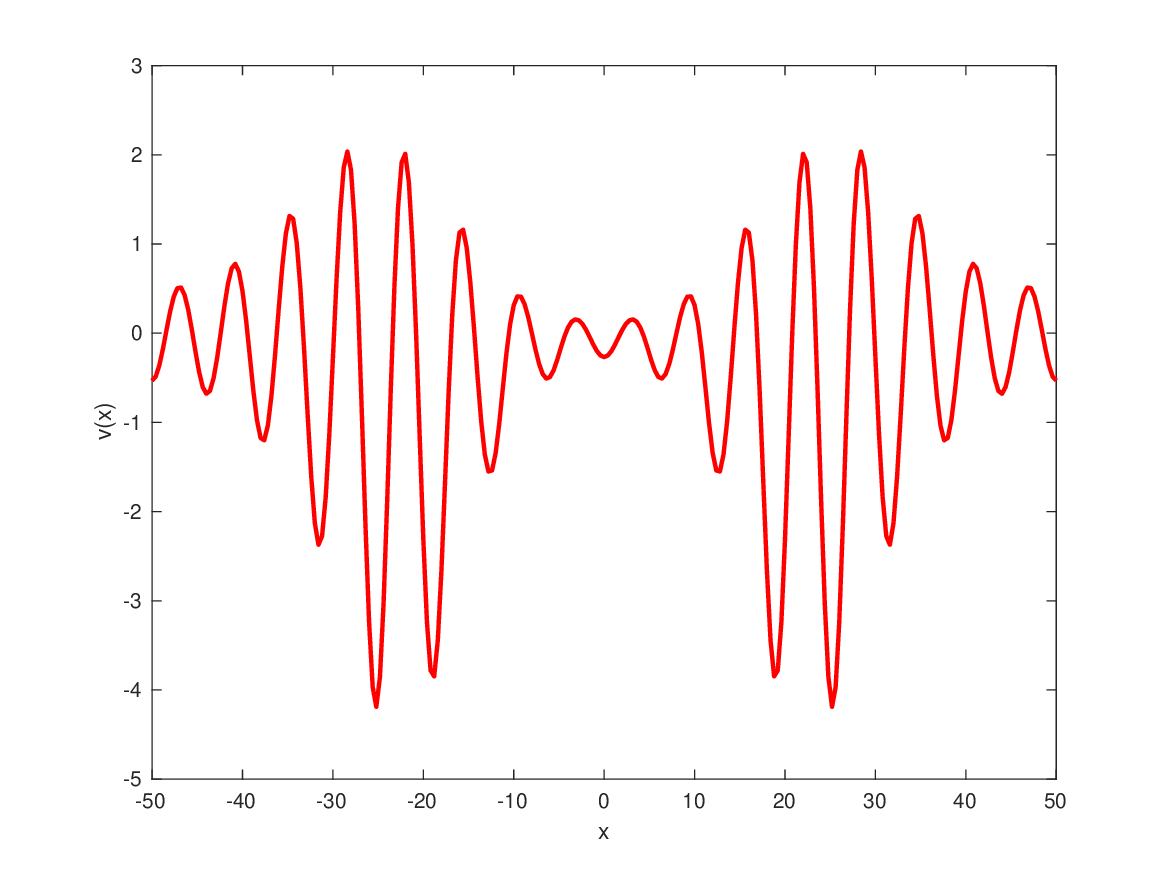,width=\linewidth}
 \end{minipage}%
  \begin{minipage}{.24\linewidth}
  \centering\epsfig{figure=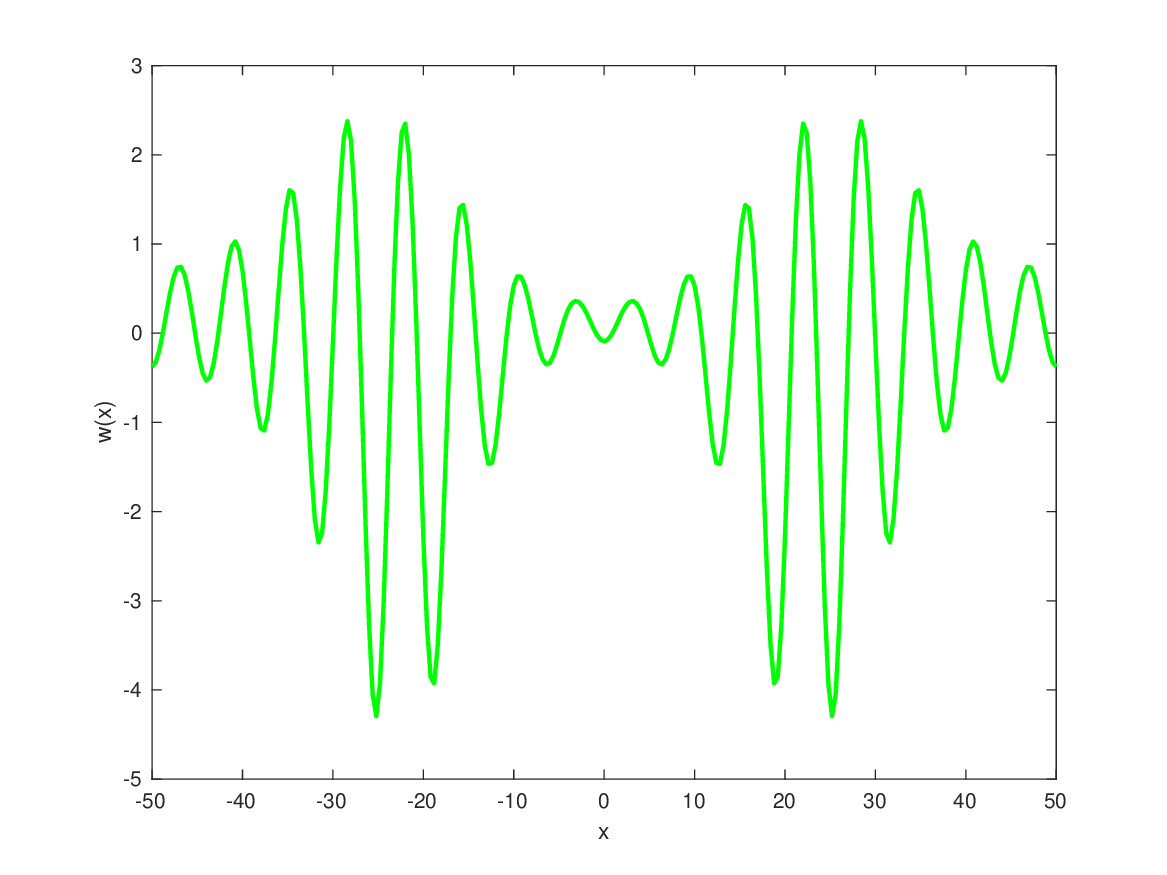,width=\linewidth}
 \end{minipage}%
  \begin{minipage}{.24\linewidth}
  \centering\epsfig{figure=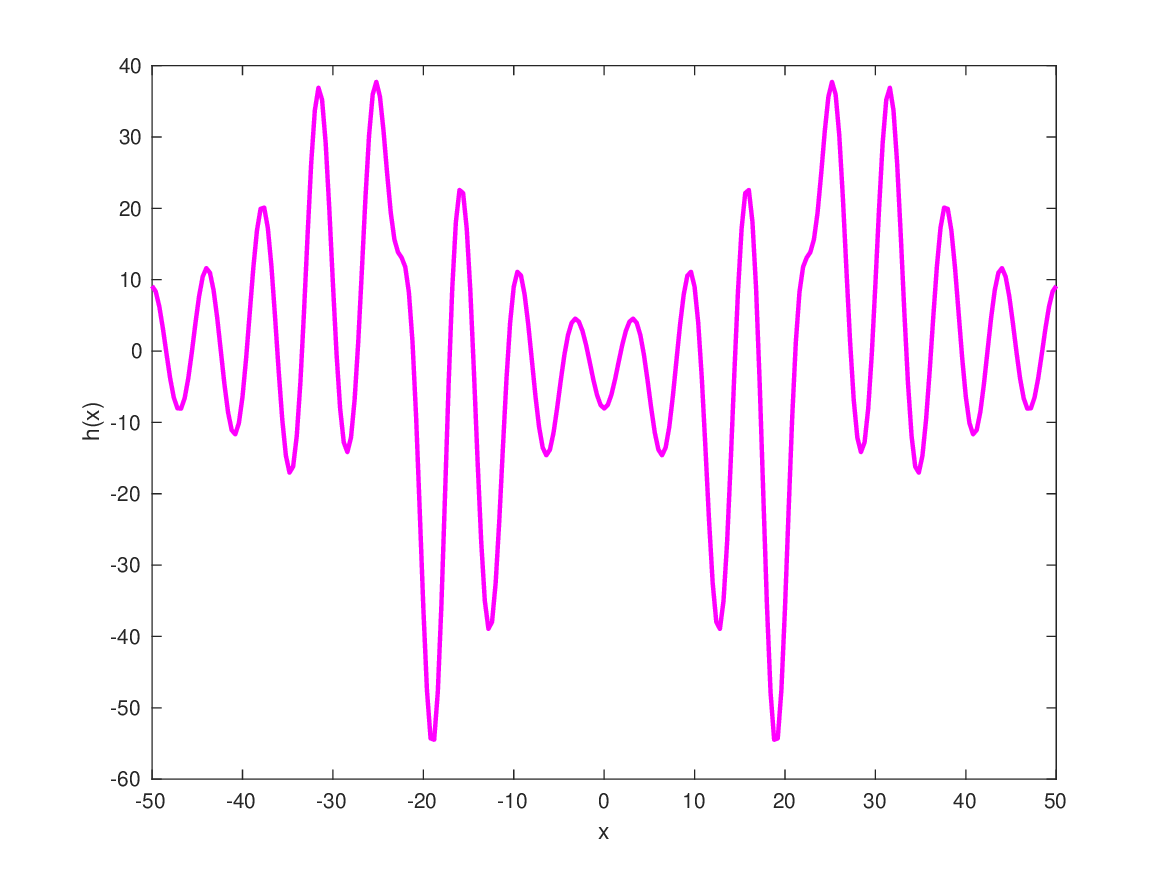,width=\linewidth}
 \end{minipage} 
 \caption{Plot of $ \overline{u}$ (L), $\overline{v}$ (ML), $\overline{w}$ (MR), and $\overline{h}$ (R) on $(-50,50)$ used in the proof of Theorem \ref{th : SH 1D cusp}.}\label{fig : Th311}
 \end{figure}%
\begin{theorem}[$D_4$ cusp bifurcation in 2D Swift-Hohenberg]\label{th : SH 2D cusp D4}
Set $r_0 = 4 \times 10^{-11}$ and $\mathcal{L} = I_2$. Let $(\omega_n)_{n = (n_1,n_2) \in \mathbb{Z}^2} = ((1 + |n_1|)^{1.01}(1 + |n_2|)^{1.01})_{n = (n_1,n_2) \in \mathbb{Z}^2}$. When $\boldsymbol{\lambda} \approx (-0.95266,\, 0)$, there exists a unique $\tilde{\mathbf{x}} \in \overline{B_{r_0}(\overline{\mathbf{x}})} \subset X_{D_4,\omega}$ satisfying $F(\tilde{\mathbf{x}}) = 0$ in~\eqref{sh cusp map}. Moreover, $(\tilde{\boldsymbol{\lambda}}, \tilde{u})$ is a cusp bifurcation with square $D_4$-symmetry. Additionally, $(\tilde{\blambda},\tilde{u})$ is a cusp bifurcation with $k = 0$ and $c > 3.77881 $; in particular, the three coexisting solutions exhibit monostability.
\end{theorem}
\begin{proof}
Take $N = 15$ and $d = 5$. The $D_4$-symmetric numerical approximation is constructed using~\cite{dominic_dihedral_julia}. Applying~\cite{CuspProofs.jl} with interval arithmetic yields
\[
\|A^N\|_{\mathcal{B}(X_{D_4,\omega})} \leq 167.987, \quad
Y_0 \leq 2.562 \times 10^{-11}, \quad
Z_2(r_0) \leq 3.5655 \times 10^7, \quad
Z_1 \leq 0.245.
\]
The spectral and stability conclusions follow as in Theorem~\ref{th : SH 1D cusp}.
\end{proof}
\begin{figure}[H]
\centering
 \begin{minipage}{.24\linewidth}
  \centering\epsfig{figure=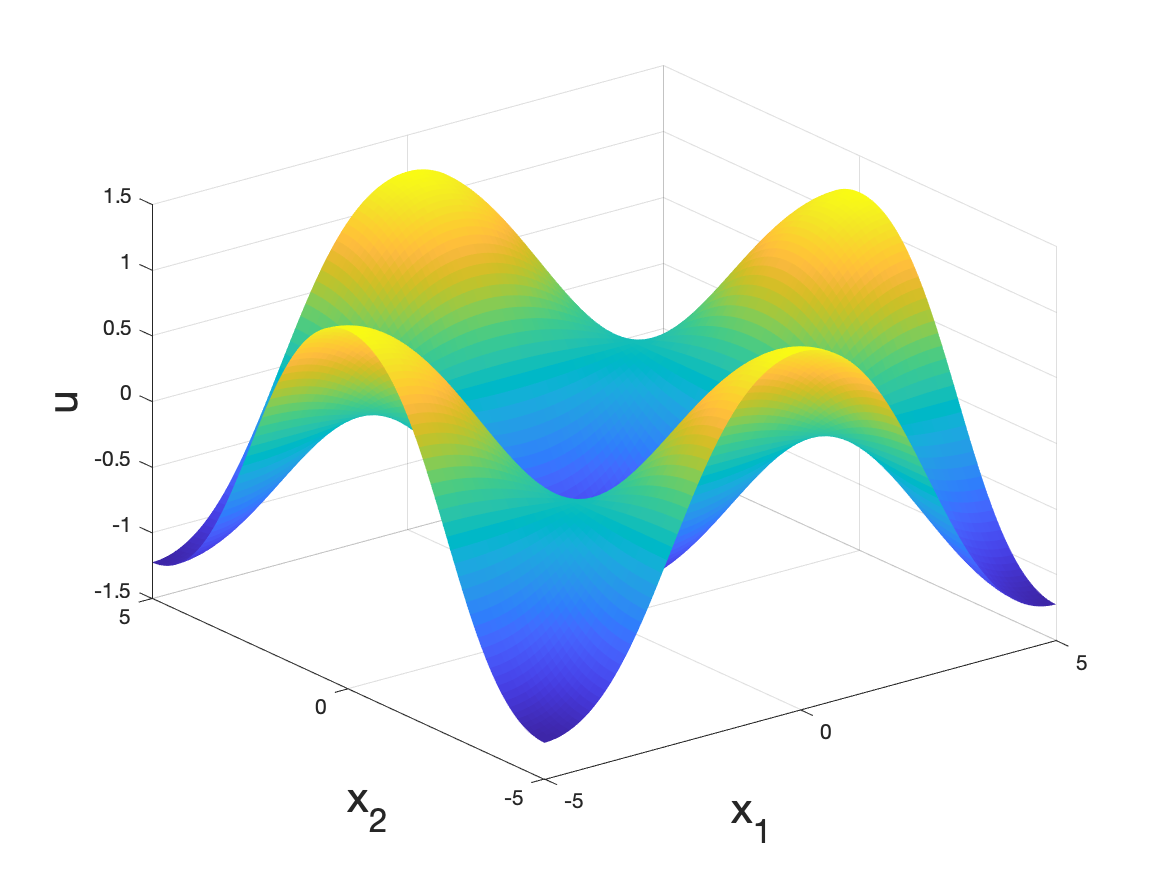,width=\linewidth}
  \end{minipage}%
 \begin{minipage}{.24\linewidth}
  \centering\epsfig{figure=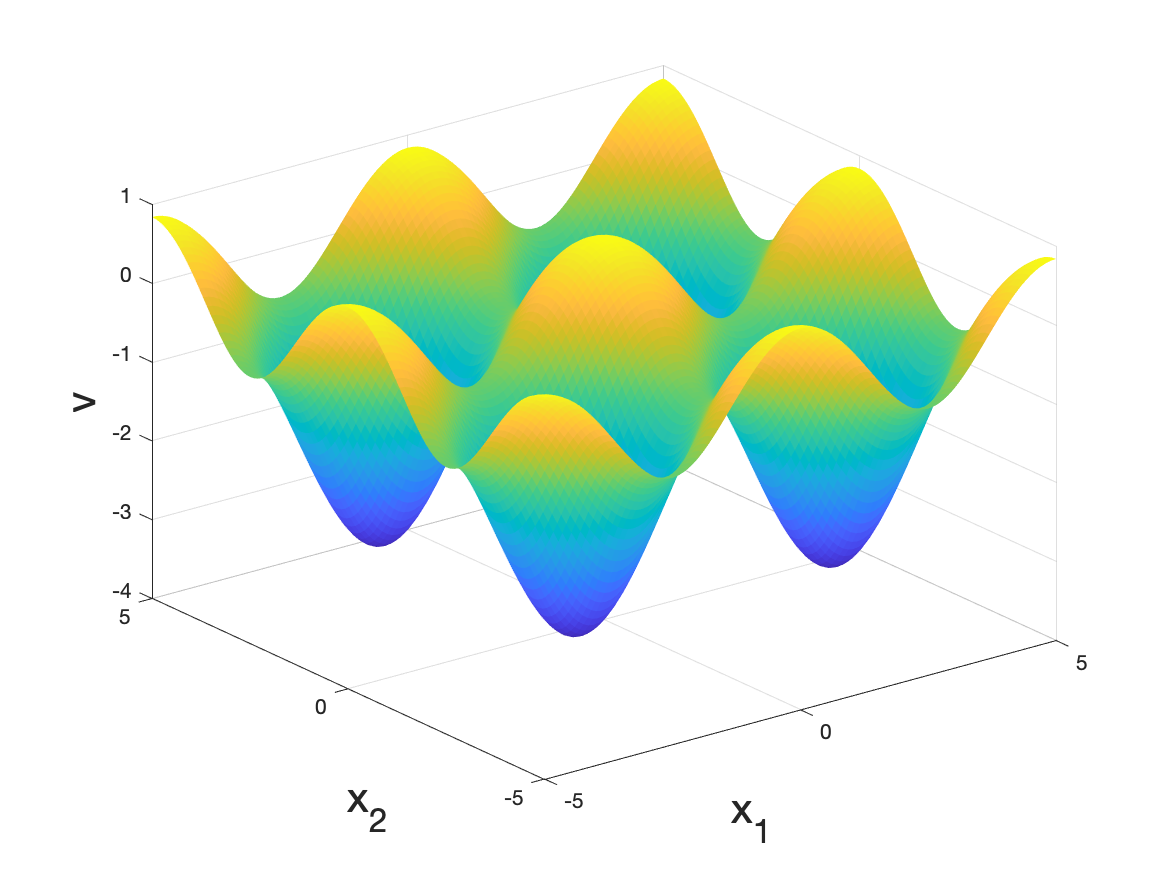,width=\linewidth}
 \end{minipage}%
  \begin{minipage}{.24\linewidth}
  \centering\epsfig{figure=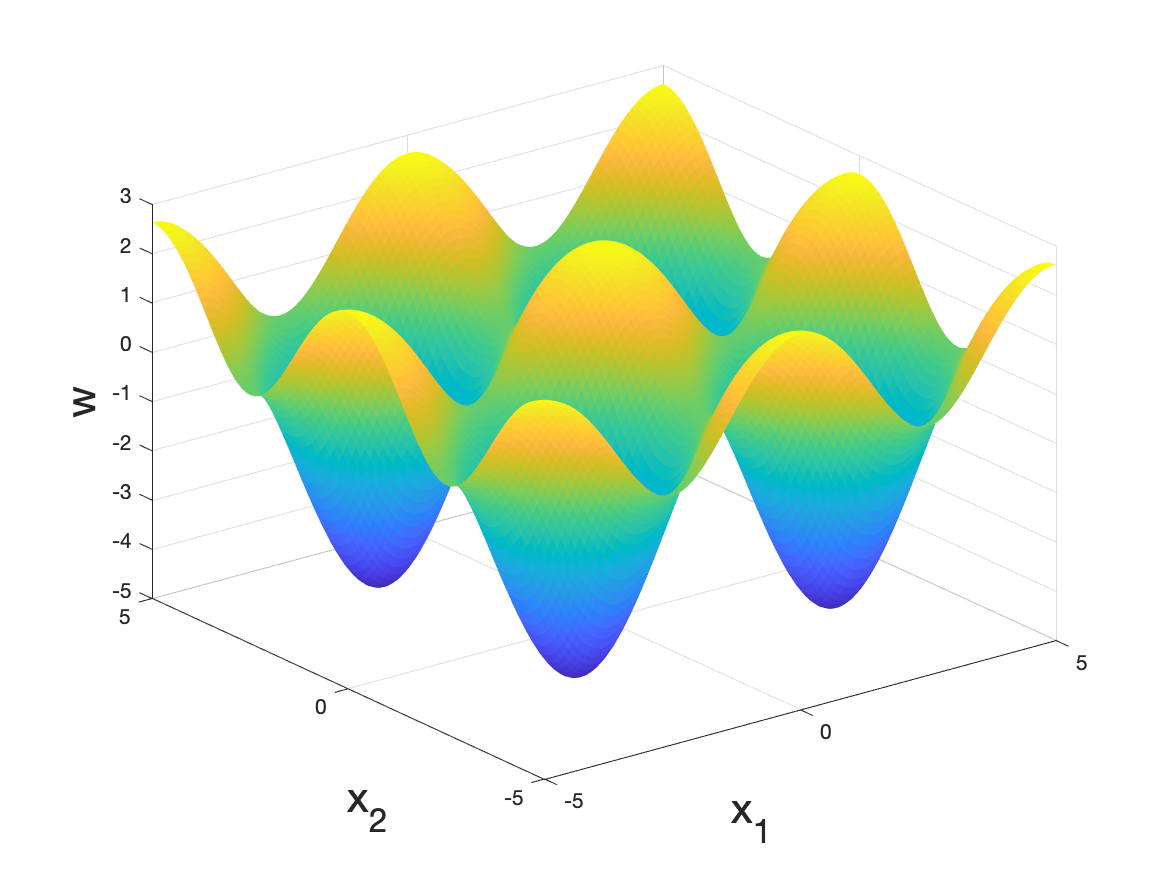,width=\linewidth}
 \end{minipage}%
  \begin{minipage}{.24\linewidth}
  \centering\epsfig{figure=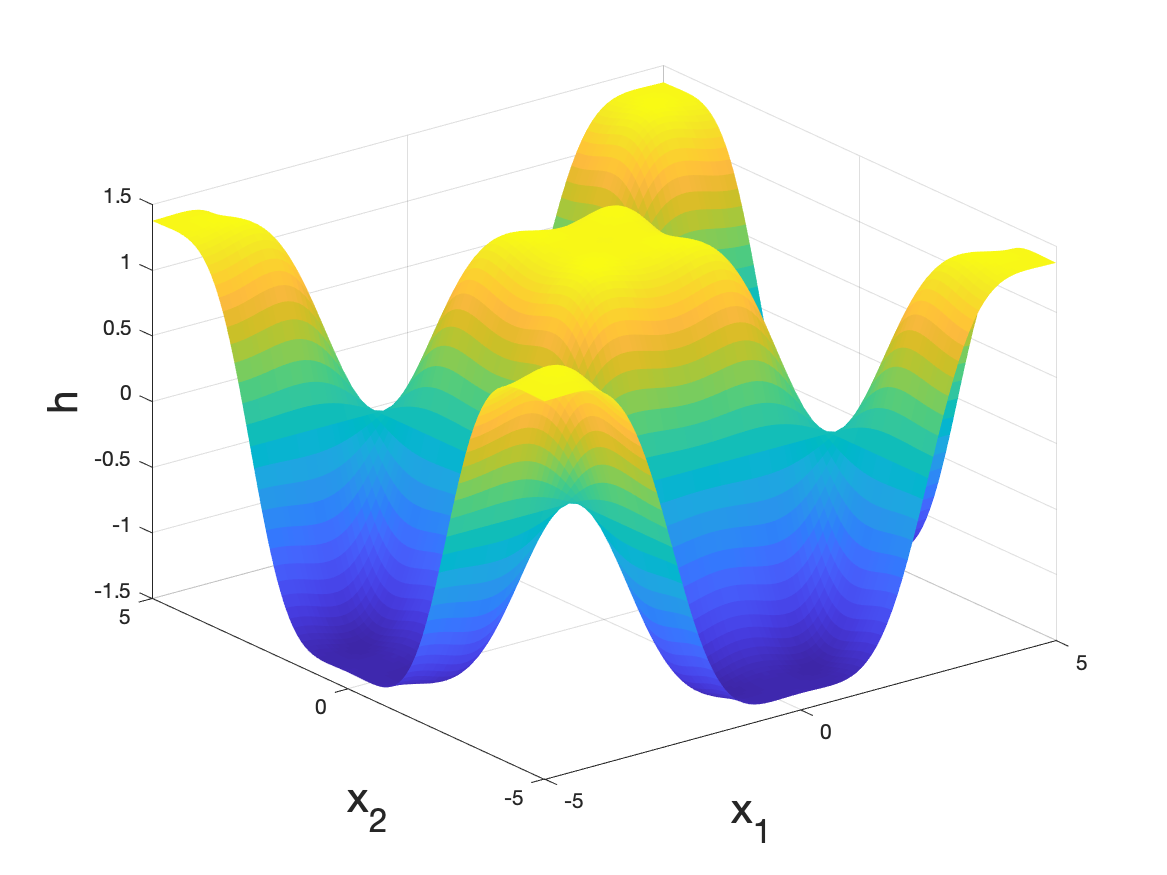,width=\linewidth}
 \end{minipage} 
 \caption{Plot of $ \overline{u}$ (L), $\overline{v}$ (ML), $\overline{w}$ (MR), and $\overline{h}$ (R) on $(-5,5)^2$ used in the proof of Theorem \ref{th : SH 2D cusp D4}.}\label{fig : Th312}
 \end{figure}%
\begin{theorem}[$D_6$ cusp bifurcation in 2D Swift-Hohenberg]\label{th : SH 2D cusp D6}
Set $r_0 \bydef 3 \times 10^{-10}$ and $\mathcal{L} = \begin{bmatrix} 1 & -\frac{1}{2} \\ 0 &\frac{\sqrt{3}}{2}  \end{bmatrix}$. Let $(\omega_n)_{n = (n_1,n_2) \in \mathbb{Z}^2} = ((1 + |n_1|)^{1.01}(1 + |n_2|)^{1.01})_{n = (n_1,n_2) \in \mathbb{Z}^2}$. When we have $\boldsymbol{\lambda} \approx (-0.32732,\, 1.04163)$, there exists a unique $\tilde{\mathbf{x}} \in \overline{B_{r_0}(\overline{\mathbf{x}})} \subset X_{D_6,\omega}$ satisfying $F(\tilde{\mathbf{x}}) = 0$ in~\eqref{sh cusp map}. Moreover, $(\tilde{\boldsymbol{\lambda}}, \tilde{u})$ is a cusp bifurcation with hexagonal $D_6$-symmetry; by Theorem~2.3 of~\cite{dominic_sh_periodic}, the solution $\tilde{u}$ is periodic on a hexagon. Additionally, $(\tilde{\blambda},\tilde{u})$ is a cusp bifurcation with $k = 0$ and $c > 0.3695 $; in particular, the three coexisting solutions exhibit monostability.
\end{theorem}
\begin{proof}
Take $N = 30$ and $d = 5$. The $D_6$-symmetric numerical approximation is constructed using~\cite{dominic_dihedral_julia}. Applying~\cite{CuspProofs.jl} with interval arithmetic yields
\[
\|A^N\|_{\mathcal{B}(X_{D_6,\omega})} \leq 1435.53, \quad
Y_0 \leq 1.668 \times 10^{-10}, \quad
Z_2(r_0) \leq 1.05 \times 10^9, \quad
Z_1 \leq 0.0772.
\]
The spectral and stability conclusions follow as in Theorem~\ref{th : SH 1D cusp}.
\end{proof}
\begin{figure}[H]
\centering
 \begin{minipage}{.24\linewidth}
  \centering\epsfig{figure=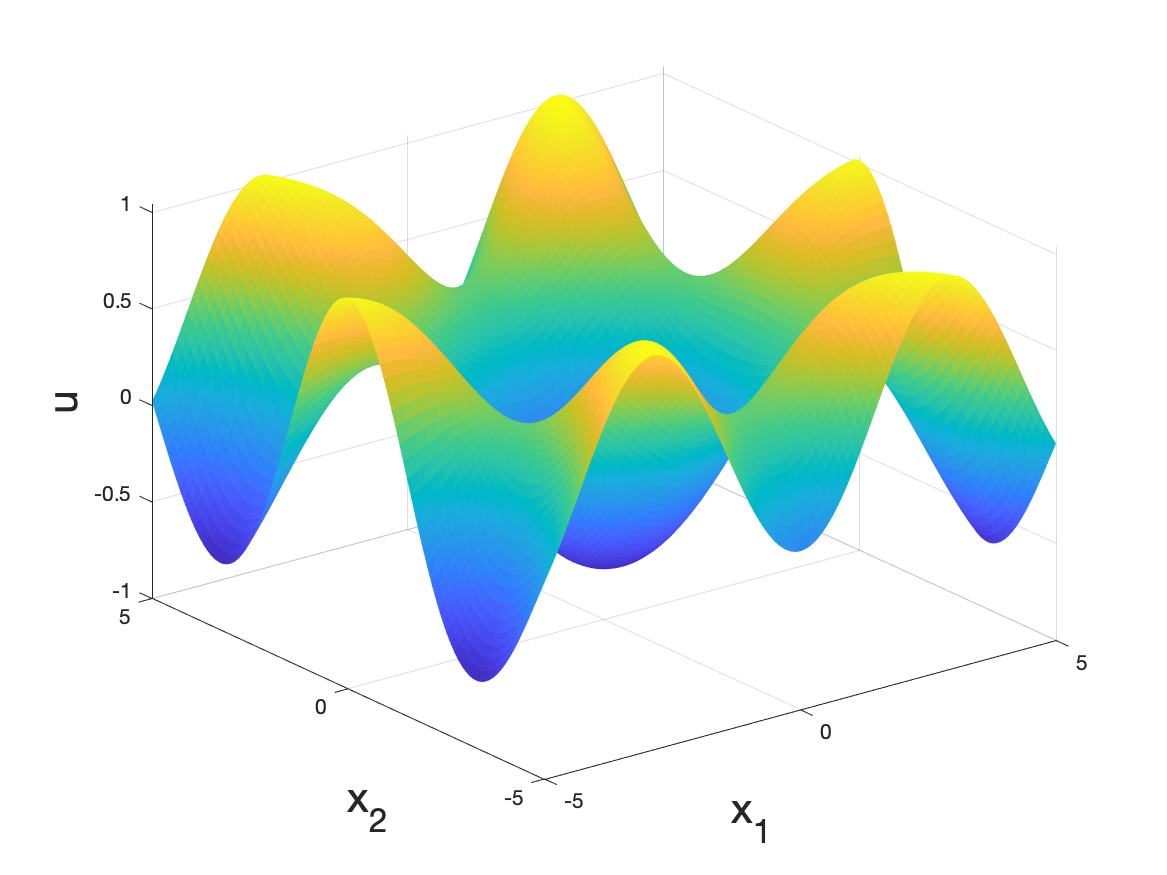,width=\linewidth}
  \end{minipage}%
 \begin{minipage}{.24\linewidth}
  \centering\epsfig{figure=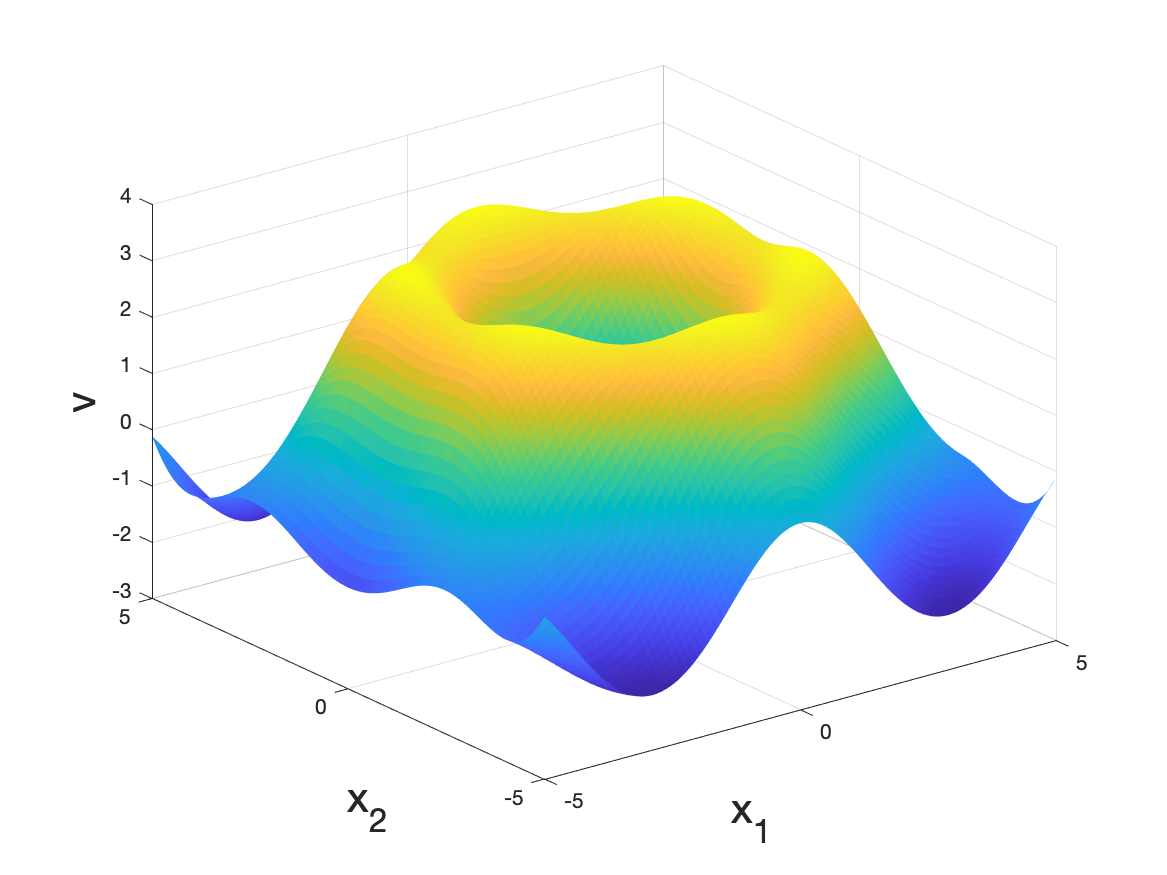,width=\linewidth}
 \end{minipage}%
  \begin{minipage}{.24\linewidth}
  \centering\epsfig{figure=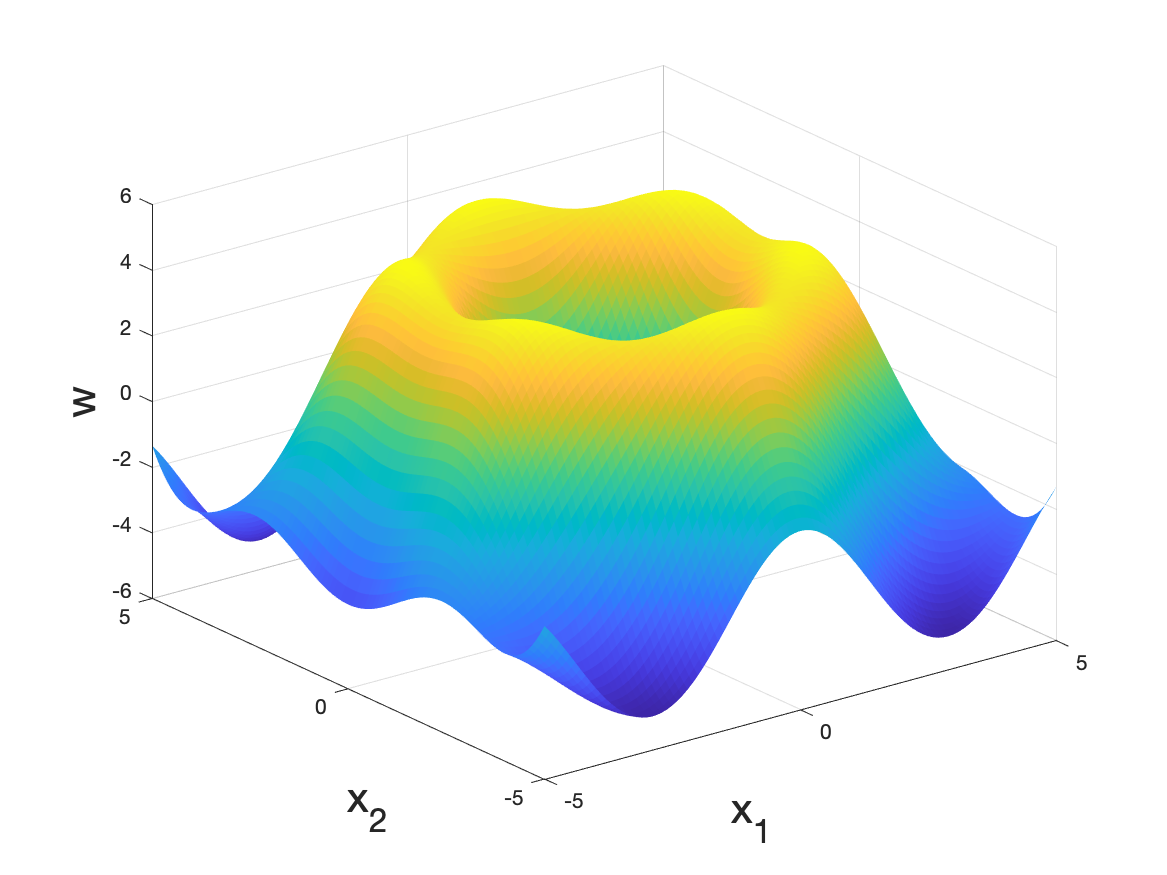,width=\linewidth}
 \end{minipage}%
  \begin{minipage}{.24\linewidth}
  \centering\epsfig{figure=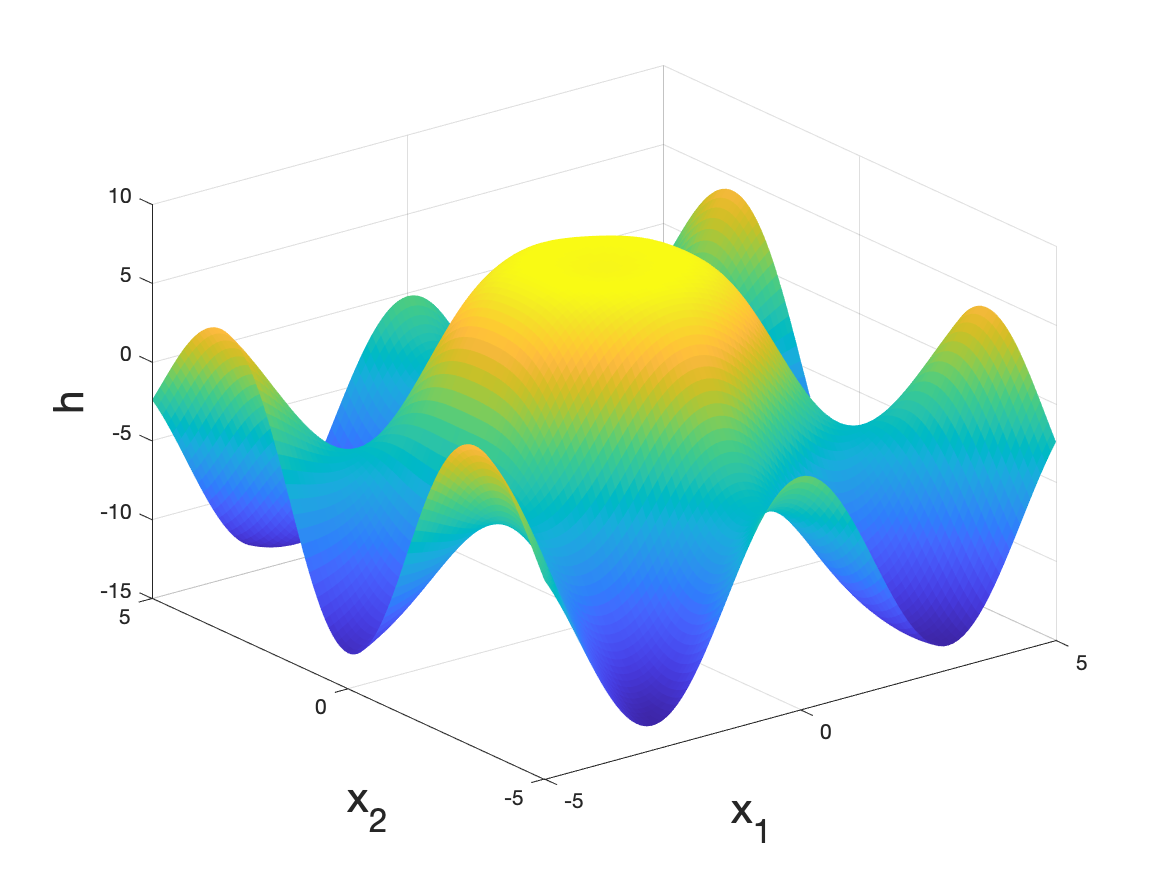,width=\linewidth}
 \end{minipage} 
 \caption{Plot of $ \overline{u}$ (L), $\overline{v}$ (ML), $\overline{w}$ (MR), and $\overline{h}$ (R) on $(-5,5)^2$ used in the proof of Theorem \ref{th : SH 2D cusp D6}.}\label{fig : Th313}
 \end{figure}%

\section{Application to the Gray-Scott system of PDEs}\label{sec : gray scott}
In this section, we consider the problem of looking for equilibria in the Gray-Scott model, namely
\begin{equation}\label{eq : gray_scott}
\begin{aligned}
    \frac{\partial u_1}{\partial t} =  \delta_{1} \Delta u_1 - \lambda_1 u_1 + (u_2 + 1) u_1^2 &= 0 \\
    \frac{\partial u_2}{\partial t} =  \delta_{2} \Delta u_2 - \lambda_2 u_2 - (u_2 + 1) u_1^2 &= 0 .
\end{aligned}
\end{equation}
As we now have a system, $\mathscr{p} = 2$ and we let $\mathbf{u} \bydef (u_1,u_2)$. We define
\begin{align}
    f(\blambda,\mathbf{u}) \bydef l_{\blambda}\mathbf{u} + g(\mathbf{u}), \qquad~
    l_{\blambda} \bydef \begin{bmatrix}
        \delta_1 \Delta - \lambda_1 I_d & 0 \\
        0 & \delta_2 \Delta - \lambda_2 I_d
    \end{bmatrix},\qquad~ g(\mathbf{u}) \bydef \begin{bmatrix}
        (u_2 + 1) u_1^2 \\
        -(u_2 + 1) u_1^2
    \end{bmatrix}.
\end{align}
First introduced by Peter Gray and S.K. Scott in \cite{gs_original} as a chemical equation, the Gray-Scott model \eqref{eq : gray_scott} is a prototypical reaction-diffusion equation of the activator-inhibitor type which models a pair of reactions featuring a cubic autocatalysis. The equations presented in \eqref{eq : gray_scott} are the resulting reaction-diffusion equations corresponding to the chemical reaction proposed in \cite{gs_original}. The pioneered work of  John Pearson in \cite{pearson_gs} allowed to discover and classify a wide variety of patterns, including Turing patterns. Many works have followed his classification on both Turing and localized patterns (moving and stationary). These include \cite{mcgough_riley_gs}, \cite{munafo_gs}, \cite{autosolitons_2D}, \cite{pearson_spots}, \cite{turing_3d}, and \cite{ueyama_numerics} all of which found and classified additional numerical solutions to \eqref{eq : gray_scott}. A more diverse classification which is related to a general class of activator inhibitor models was presented by the authors of \cite{gs_alsaadi}. Here, the authors introduced a homotopy parameter to transform the Glycolysis model studied in \cite{gly_1d} into Gray-Scott. Some rigorous results include the work of \cite{doleman_pulse}, where the authors considered a specific parameter regime allowing them to prove the existence of single pulse solutions. In a different, but still specific parameter regime, exact solutions were studied by the authors of \cite{exact_sol1D_Hale}. These results were then extended beyond this regime using a computer-assisted approach by the authors of \cite{jay_jp_gs}. In 2D, a variety of work using asymptotic expansions in specific parameter regimes exists. We refer the interested reader to the works \cite{wei_spike_2d_unbounded,wei_ring_2d_small_epsilon_unbounded,wei_multi_spike_2d,wei_asym_spike_2d,wei_breakup,thesis_localized_but_moving}. As far as approaches using computer-assisted proofs in 2D, there is the work by the authors of \cite{gs_cadiot_blanco} which is for localized patterns and \cite{castelli_gs} for stationary periodic patterns. 
\par As before, we let $\mathbf{v} \bydef (v_1,v_2)$ be a left null vector of $D_{u}f(\blambda,\mathbf{u})$, $\mathbf{w} \bydef (w_1,w_2)$ be a right null vector of $D_{u}f(\blambda,\mathbf{u})^*$, and $\mathbf{h} \bydef (h_1,h_2)$ solve the bordering system for Gray-Scott. Since $\mathscr{p} = 2$, we have the space $X_{\mathcal{G},\omega} \bydef \mathbb{R}^4 \times (\ell^2_{\mathcal{G},\omega})^8$. We then redefine $\mathbf{x} \bydef (\mathbf{s},\blambda,\mathbf{u},\mathbf{v},\mathbf{w},\mathbf{h}) \in X_{\mathcal{G},\omega}$. We then compute
\begin{align}
    &D_{u} f(\blambda,\mathbf{u})\mathbf{v} = \begin{bmatrix}
        \delta_1 \Delta v_1 - \lambda_1 v_1 + 2(u_2 + 1)u_1 v_1 + u_1^2 v_2  \\
         \delta_2 \Delta v_2 - \lambda_2 v_2 - 2(u_2 + 1) u_1 v_1 - u_1^2 v_2
    \end{bmatrix},\\
    &D_{uu}f (\blambda,\mathbf{u})(\mathbf{v},\mathbf{v}) = \begin{bmatrix}
        2(u_2 + 1) v_1^2 + 4u_1 v_1 v_2 \\
        -2(u_2 + 1) v_1^2 - 4u_1 v_1 v_2
    \end{bmatrix},
\end{align}
and, using Proposition \ref{prop : adjoint mult}, 
\begin{align}
    D_u f(\blambda,\mathbf{u})^* \mathbf{w} =\begin{bmatrix}
        \delta_1 \Delta w_1 - \lambda_1 w_1 + \mathcal{O}_{\mathcal{G}}((2(u_2 + 1)u_1)\star  (\mathcal{O}_{\mathcal{G}}^{-1} (w_1-w_2)))\\
         \delta_2 \Delta w_2 - \lambda_2 w_2 +\mathcal{O}_{\mathcal{G}}(u_1^2\star ( \mathcal{O}_{\mathcal{G}}^{-1} (w_1-w_2)))
    \end{bmatrix}.
\end{align}
Substituting these into the general cusp map \eqref{def : cusp map} of Section \ref{sec : cusp map} and representing the map in Fourier space yields
\begin{align}\label{gs cusp map}
    F(\mathbf{x}) \bydef \begin{bmatrix}
        (\mathbf{v},\mathbf{v})_{2} - 1 \\
        (\mathbf{v},\mathbf{w})_{2} - 1 \\
        (\mathbf{w},\mathbf{h})_{2} \\
        (2(u_2 + 1)v_1^2 + 4 u_1 v_1 v_2,w_1-w_2)_{2} \\
        \delta_1 \Delta u_1 - \lambda_1 u_1 + (u_2 + 1)u_1^2 \\
        \delta_2 \Delta u_2 - \lambda_2 u_2 - (u_2 + 1)u_1^2 \\
        \delta_1 \Delta v_1 - \lambda_1 v_1 + 2(u_2 + 1)u_1 v_1 + u_1^2 v_2 + s_1 v_1 \\
        \delta_2 \Delta v_2 - \lambda_2 v_2 - 2(u_2 + 1) u_1 v_1 - u_1^2 v_2 + s_1 v_2 \\
        \delta_1 \Delta w_1 - \lambda_1 w_1 + \mathcal{O}_{\mathcal{G}}((2(u_2 + 1)u_1)\star  (\mathcal{O}_{\mathcal{G}}^{-1} (w_1-w_2)))\\
        \delta_2 \Delta w_2 - \lambda_2 w_2 +\mathcal{O}_{\mathcal{G}}(u_1^2\star ( \mathcal{O}_{\mathcal{G}}^{-1} (w_1-w_2))) \\
        \delta_1 \Delta h_1 - \lambda_1 h_1 + 2(u_2 + 1)u_1 h_1 + u_1^2 h_2 + s_2 v_1 + 2(u_2 + 1)v_1^2 + 4u_1v_1v_2 \\
        \delta_2 \Delta h_2 - \lambda_2 h_2 - 2(u_2 + 1)u_1 h_1 - u_1^2 h_2 + s_2 v_1 - 2(u_2 + 1)v_1^2 - 4u_1v_1v_2
    \end{bmatrix}.
\end{align}

We now redefine
{\small\begin{align}
    &\bpi^{\leq N}(\mathbf{u},\mathbf{v},\mathbf{w},\mathbf{h}) \bydef (\Pi^{\leq N} u_1,\Pi^{\leq N} u_2, \Pi^{\leq N} v_1,\Pi^{\leq N} v_2, \Pi^{\leq N} w_1,\Pi^{\leq N} w_2, \Pi^{\leq N} h_1,\Pi^{\leq N} h_2), \\
    &\bpi^{>N}(\mathbf{u},\mathbf{v},\mathbf{w},\mathbf{h}) \bydef (\Pi^{> N} u_1,\Pi^{> N} u_2, \Pi^{> N} v_1,\Pi^{> N} v_2, \Pi^{> N} w_1,\Pi^{> N} w_2, \Pi^{> N} h_1, \Pi^{> N} h_2)\\
    &\obpi^{\leq N}\mathbf{x} \bydef (\mathbf{s},\boldsymbol{\lambda},\Pi^{\leq N} u_1,\Pi^{\leq N} u_2, \Pi^{\leq N} v_1,\Pi^{\leq N} v_2, \Pi^{\leq N} w_1,\Pi^{\leq N} w_2, \Pi^{\leq N} h_1,\Pi^{\leq N} h_2) \\
    &\obpi^{>N}\mathbf{x} \bydef (0,0,0,0,\Pi^{> N} u_1,\Pi^{> N} u_2, \Pi^{> N} v_1,\Pi^{> N} v_2, \Pi^{> N} w_1,\Pi^{> N} w_2, \Pi^{> N} h_1, \Pi^{> N} h_2)
\end{align}}As in Section \ref{sec : numerical aspects}, we construct $\overline{\mathbf{x}}$ using our strategy of locating a solution, a saddle-node, and then a saddle-node of the saddle-node map. We construct $A$ using a similar strategy as well. More specifically, we let $\mathbf{z} = (z_1,z_2)$ where $\mathbf{z} \in \{\mathbf{u},\mathbf{v},\mathbf{w},\mathbf{h}\}$. Then, we write
\begin{align}
    (\Pi^{> N} A \mathbf{z})_n \bydef \begin{cases}
        0 & n \in I^N \\
        \begin{bmatrix}
            \frac{(z_1)_n}{-\delta_1 |\mathcal{L}\tilde{n}|_2^2 - \overline{\lambda}_1} \\
            \frac{(z_2)_n}{- \delta_2|\mathcal{L}\tilde{n}|_2^2 - \overline{\lambda}_2}
        \end{bmatrix} & n \in \mathcal{Z}_{\mathrm{red}}(\mathcal{G}) \setminus I^N
    \end{cases}
\end{align}
while the scalar parameters $\mathbf{s}$ and $\blambda$ are left unchanged as before. We then define \begin{align}
    \mathcal{L}_{\infty} \bydef \max_{n \in \mathcal{Z}_{\mathrm{red}}(\mathcal{G}) \setminus I^N}\left\{\frac{1}{|\delta_1 |\mathcal{L}\tilde{n}|_2^2 + \overline{\lambda}_1|},\frac{1}{| \delta_2|\mathcal{L}\tilde{n}|_2^2 + \overline{\lambda}_2|}\right\},
\end{align}
and perform a similar proof to that of Lemma \ref{lem : tailA} to obtain that $\|\obpi^{> N} A\|_{\mathcal{B}(X_{\mathcal{G},\omega})} \leq \mathcal{L}_{\infty}$.

\subsection{Newton--Kantorovich bounds for Gray--Scott}\label{sec : bounds gs}
In this section, we compute the bounds required to apply Theorem \ref{th : radii polynomial theorem periodic}. That is, we must compute the constants $Y_0, Z_1$ and the function $Z_2(r)$. As the highest power in Gray-Scott is a cubic, the result for $Y_0$ is the same as Lemma \ref{lem : Y0 SH}. We now move to the $Z_2(r)$ bound.
\begin{lemma}\label{lem : Z2 GS}
Let $Z_2(r) : (0,\infty) \to [0,\infty)$ be defined as 
\begin{align}
    Z_2(r) \bydef (\|A^N\|_{\mathcal{B}(X_{\mathcal{G},\omega})} + \mathcal{L}_{\infty}) (24 + Z_{2,1}(r) + Z_{2,2}(r))
\end{align}
where
\begin{align}
    &Z_{2,1}(r) \bydef 12 + \kappa_0 (18\kappa r + 16\|\overline{v}_1\|_{1} + 8\|\overline{v}_2\|_{1} + 8\|\overline{u}_1\|_{1} + 4\|\overline{u}_2 + 1\|_{1})\|\overline{w}_1 - \overline{w}_2\|_{2}+ (36\kappa_0\kappa + 30\kappa^2) r^2\\
    &\hspace{+0.2cm} + 8\kappa_0\|\overline{v}_1 \overline{v}_2\|_{2} + 4\|\overline{v}_1\overline{v}_2\|_{1} + (16\kappa + 32\kappa_0)\|\overline{v}_1\|_{1}r + (8\kappa + 16\kappa_0) \|\overline{v}_2\|_{1}r + 4\kappa_0\|\overline{v}_1^2\|_{2} + 4\|\overline{v}_1^2\|_{1}\\
    &\hspace{+0.7cm} + 8\kappa_0\|(\overline{u}_2 + 1)\overline{v}_1\|_{2} + 8\|(\overline{u}_2 + 1)\overline{v}_1\|_{1} + (4\kappa + 8\kappa_0)\|\overline{u}_2 + 1\|_{1}r + (8\kappa_0 + 8)\|\overline{u}_1 \overline{v}_2\|_{1}\\
    &\hspace{+1.0cm} + (16\kappa_0 + 12\kappa)\|\overline{u}_1\|_{1}r + (8\kappa_0+8)\|\overline{u}_1 \overline{v}_1\|_{1}, \\
    &Z_{2,2}(r) \bydef \kappa \biggl((78+24\mathscr{O}_{\mathrm{max}}) \kappa r + (56 + 24\mathscr{O}_{\mathrm{max}}) \|\overline{u}_1\|_{1,\sqrt{\omega}} + (28 + 12\mathscr{O}_{\mathrm{max}}) \|\overline{u}_2+1\|_{1,\sqrt{\omega}} \\
    &\hspace{+1cm}+ 40 \|\overline{v}_1\|_{1,\sqrt{\omega}}+ 20 \|\overline{v}_2\|_{1,\sqrt{\omega}}+ 8 \|\overline{h}_1\|_{1,\sqrt{\omega}} + 4 \|\overline{h}_2\|_{1,\sqrt{\omega}} +6 \mathscr{O}_{\mathrm{max}} \|\mathcal{O}_{\mathcal{G},N}^{-1}(\overline{w}_1 - \overline{w}_2)\|_{1,\sqrt{\omega}}\biggr).
\end{align}
Then, it follows that $\|A(DF(\mathbf{x}) - DF(\overline{\mathbf{x}}))\|_{\mathcal{B}(X_{\mathcal{G},\omega})} \leq Z_2(r)r.$
\end{lemma}
\begin{proof}
The proof can be found in Appendix \ref{apen : Z2 GS}.
\end{proof}
Let us now discuss the $Z_1$ bound. First, we introduce the necessary notation.
\begin{align}
    &\overline{q}_{-1} \bydef (2(\overline{u}_2 + 1)\star (\mathcal{O}_{\mathcal{G},N}^{-1} (\overline{w}_1 - \overline{w}_2))),~\overline{q}_{0} \bydef (2\overline{u}_1,\star (\mathcal{O}_{\mathcal{G},N}^{-1} (\overline{w}_1 - \overline{w}_2))), ~\overline{q}_1 \bydef 2\overline{u}_1(\overline{u}_2 + 1), \\
    &\overline{q}_2 \bydef \overline{u}_1^2,~\overline{q}_3 \bydef 2(\overline{u}_2 + 1)\overline{v}_1 + 2\overline{u}_1 \overline{v}_2, ~\overline{q}_4 \bydef 2\overline{u}_1 \overline{v}_1, ~\overline{q}_5 \bydef 2(\overline{u}_2 + 1)\overline{h}_1 + 2\overline{u}_1 \overline{h}_2 + 4\overline{v}_1 \overline{v}_2 \\
    &\overline{q}_6 \bydef 2\overline{u}_1 \overline{h}_1 + 2\overline{v}_1^2,~\overline{q}_7 \bydef 2(u_2 + 1)v_1^2 + 4u_1v_1v_2, ~\text{and}~ \overline{q}_k^N \bydef \Pi^{\leq N} \overline{q}_k ~ \text{for} ~ k \in \{-1,0,1,2,3,4,5,6,7\}.
\end{align}
We now state the lemma for the $Z_1$ bound.
\begin{lemma}\label{lem : Z1 full GS}
Let $Z_{1,1},Z_{1,2},Z_{\infty} > 0$ be defined as 
\begin{align}
    &Z_{1,1} \bydef \left\|\obpi^{\leq N} - A^NM^N(\overline{\mathbf{x}})\obpi^{\leq 2N}\right\|_{\mathcal{B}(X_{\mathcal{G},\omega})}, \\
    &Z_{1,2} \bydef 2 \mathcal{L}_{\infty}\biggl( (6+2\mathscr{O}_{\mathrm{max}})\| \overline{q}_1^N\|_{1,\sqrt{\omega}} + (6+2\mathscr{O}_{\mathrm{max}})\|\overline{q}_2^N\|_{1,\sqrt{\omega}} + 6\| \overline{q}_3^N\|_{1,\sqrt{\omega}} + 6\| \overline{q}_4^N\|_{1,\sqrt{\omega}} \\
    &\hspace{+1.5cm}+ 2\| \overline{q}_5^N\|_{1,\sqrt{\omega}} + 2\| \overline{q}_6^N\|_{1,\sqrt{\omega}}+ \mathscr{O}_{\mathrm{max}}\|\overline{q}_{-1}^N\|_{1,\sqrt{\omega}} +2\mathscr{O}_{\mathrm{max}}\|\overline{q}_0^N\|_{1,\sqrt{\omega}}\biggr), 
\end{align}
and 
\begin{align}
    Z_{\infty} &\bydef 2\|(\Pi^{3N} - \Pi^{2N})\overline{q}_7 - \overline{q}_7^N\|_{2}+ (6 + 2\mathscr{O}_{\mathrm{max}}) \|\overline{q}_1 - \overline{q}_1^N\|_{1,\sqrt{\omega}}+  (6 + 2\mathscr{O}_{\mathrm{max}})\|\overline{q}_2 - \overline{q}_2^N\|_{1,\sqrt{\omega}}\\
    &+ 6 \|\overline{q}_3 - \overline{q}_3^N\|_{1,\sqrt{\omega}}+ 6 \|\overline{q}_4 - \overline{q}_4^N\|_{1,\sqrt{\omega}}+ 2 \|\overline{q}_5 - \overline{q}_5^N\|_{1,\sqrt{\omega}}+ 2 \|\overline{q}_6 - \overline{q}_6^N\|_{1,\sqrt{\omega}} \\
    &+  \mathscr{O}_{\mathrm{max}} \|\overline{q}_{-1} - \overline{q}_{-1}^N\|_{1,\sqrt{\omega}} + 2 \mathscr{O}_{\mathrm{max}}\|\overline{q}_0 - \overline{q}_0^N\|_{1,\sqrt{\omega}}  \\
    &+ 4\|(\Pi^{3N} - \Pi^{2N}) (\overline{v}_1 \overline{v}_2 \star (\overline{w}_1 - \overline{w}_2))\|_{2}+ 2\|(\Pi^{3N} - \Pi^{2N}) (\overline{v}_1^2 \star (\overline{w}_1 - \overline{w}_2))\|_{2}\\
    &+ 4\|(\Pi^{3N} - \Pi^{2N}) ([(\overline{u}_2 + 1)\overline{v}_1 + \overline{u}_1 \overline{v}_2] \star (\overline{w}_1 - \overline{w}_2))\|_{2} + 4\|(\Pi^{3N} - \Pi^{2N}) (\overline{u}_1 \overline{v}_1 \star (\overline{w}_1 - \overline{w}_2))\|_{2}.
\end{align}
Then, defining
\begin{align}
    Z_1 \bydef Z_{1,1} + Z_{1,2} + (\|A^N\|_{\mathcal{B}(X_{\mathcal{G},\omega})} + \mathcal{L}_{\infty}) Z_{\infty},
\end{align}
it follows that $\|I_d - ADF(\overline{\mathbf{x}})\|_{\mathcal{B}(X_{\mathcal{G},\omega})} \leq Z_1$.
\end{lemma}
\begin{proof}
The proof can be found in Appendix \ref{apen : Z1 GS}.
\end{proof}
\subsection{Enclosing the spectrum with Gershgorin disks for Gray Scott}\label{sec : gershgorin gray scott}

We apply the framework of Section~\ref{sec : gershgorin} with $\mathscr{p}=2$; the disks now come in two families $B_{n,1}, B_{n,2}$, one per component. The blocks of $Dg(\overline{\mathbf{u}})$ are the multiplication operators
\[
  (Dg)_{1,1} = \mathbb{M}_{\overline{q}_1}, \quad (Dg)_{2,1} = -\mathbb{M}_{\overline{q}_1},
  \qquad
  (Dg)_{1,2} = \mathbb{M}_{\overline{q}_2}, \quad (Dg)_{2,2} = -\mathbb{M}_{\overline{q}_2},
\]
with $\overline{q}_1 \bydef 2(\overline{u}_2+1)\overline{u}_1$ and $\overline{q}_2 \bydef \overline{u}_1^2$, so that \eqref{def : Qj} gives $Q_j = 2\|\overline{q}_j\|_1$ for $j=1,2$. The constant $\varepsilon_{r_0}$ of \eqref{def : eps r0} is given by the following lemma, whose proof is the computation of Lemma~\ref{lem : eps r0 sh}, carried out on $g(\mathbf{u}) = ((u_2+1)u_1^2,-(u_2+1)u_1^2)$ and using $\|y_i\|_1\leq\kappa_0r_0$ for each component.

\begin{lemma}\label{lem : eps r0 gs}
For Gray--Scott one may take
\[
  \varepsilon_{r_0} \bydef r_0 + 2\kappa_0 r_0\bigl(4\|\overline{u}_1\|_1 + 2\|\overline{u}_2+1\|_1 + 3\kappa_0 r_0\bigr).
\]
\end{lemma}

For the tail condition (ii) of Proposition~\ref{prop : gershgorin enclosure}, the diagonal blocks of $l_{\overline{\blambda}}$ have symbols $-\delta_j\frac{\pi^2}{d^2}|\mathcal{L}n|^2 - \overline{\lambda}_j$, so requiring $\delta_j\frac{\pi^2}{d^2}|\mathcal{L}n|^2 > Q_j - \overline{\lambda}_j$ for both $j$ amounts to
\begin{equation}\label{tail assumption}
  \min_{n \in \mathcal{Z}_{\mathrm{red}}(\mathcal{G})\setminus I^{3N}} |\mathcal{L}n|^2
  \;>\; \max_{j\in\{1,2\}} \frac{d^2\bigl(2\|\overline{q}_j\|_1-\overline{\lambda}_j\bigr)}{\pi^2\delta_j} .
\end{equation}
This makes $\mathrm{Re}\,(\overline{\mathcal{D}}_{j,j})_{n,n}$ negative and decreasing beyond $I^{3N}$ for both $j$, reducing (ii) to two verifiable inequalities; as before it holds once $N$ is large enough. Conditions (i)--(iv) are then checked by interval arithmetic, giving exactly one eigenvalue on the imaginary axis and $k=0$.

\subsection{Computer-Assisted Proofs in Gray-Scott}
In this section, we present computer-assisted proofs of cusp bifurcations in the Gray Scott system of PDEs. Note that in order to apply Lemma \ref{lem : bistability}, we must compute the value of $c$. In the case of Gray-Scott, observe that the normal form coefficient is
\begin{align}
    c = (w_1 - w_2,v_1^2 v_2 + (u_2 + 1)v_1 h_1 +  u_1 v_2 h_1 + u_1 v_1 h_2)_{2}.\label{c : GS}
\end{align}
We prove cusp bifurcations in two settings: the one-dimensional problem with $\mathbb{Z}_2$-symmetry, and the two dimensional problem with rectangle ($D_2$) symmetry.
\begin{theorem}[Even ($\mathbb{Z}_2$) Cusp Bifurcation in 1D Gray-Scott]\label{th : GS 1D cusp}
Set $r_0 \bydef  5 \times 10^{-11}$ and $\mathcal{L} = 1$. Let $ (\omega_n)_{n \in \mathbb{Z}} \bydef ((1 + |n|)^{1.01})_{n \in \mathbb{Z}}$. When $\blambda \approx (1,4.476425525999964)$, there exists a unique  $\tilde{\mathbf{x}} \in \overline{B_{r_0}(\overline{\mathbf{x}})} \subset X_{\mathbb{Z}_2,\omega}$ satisfying $F(\overline{\mathbf{x}}) = 0$ in \eqref{gs cusp map}. Moreover,  $(\tilde{\blambda},\tilde{\mathbf{u}})$ is a cusp bifurcation with $k = 0$ and $c > 1.99 \times 10^{-3}$; in particular, the three coexisting solutions exhibit monostability.
\end{theorem}
\begin{proof}
Take $N = 100$ and $d = 5$. Applying \cite{CuspProofs.jl} with interval arithmetic yields
\begin{align}
    \|A^N\|_{\mathcal{B}(X_{\mathbb{Z}_2,\omega})} \leq 1.62 \times 10^4 \text{,}~Y_0 \bydef 9.24 \times 10^{-12}  \text{,}~Z_{2}(r_0) \bydef 2.031 \times 10^8    \text{,}~Z_1 \bydef 0.756.
    \end{align}
The stability and spectral conclusions follow as in Theorem \ref{th : SH 1D cusp}.
\end{proof}
 \begin{figure}[H]
    \centering
    \begin{subfigure}[b]{0.24\textwidth}
        \centering
        \epsfig{figure=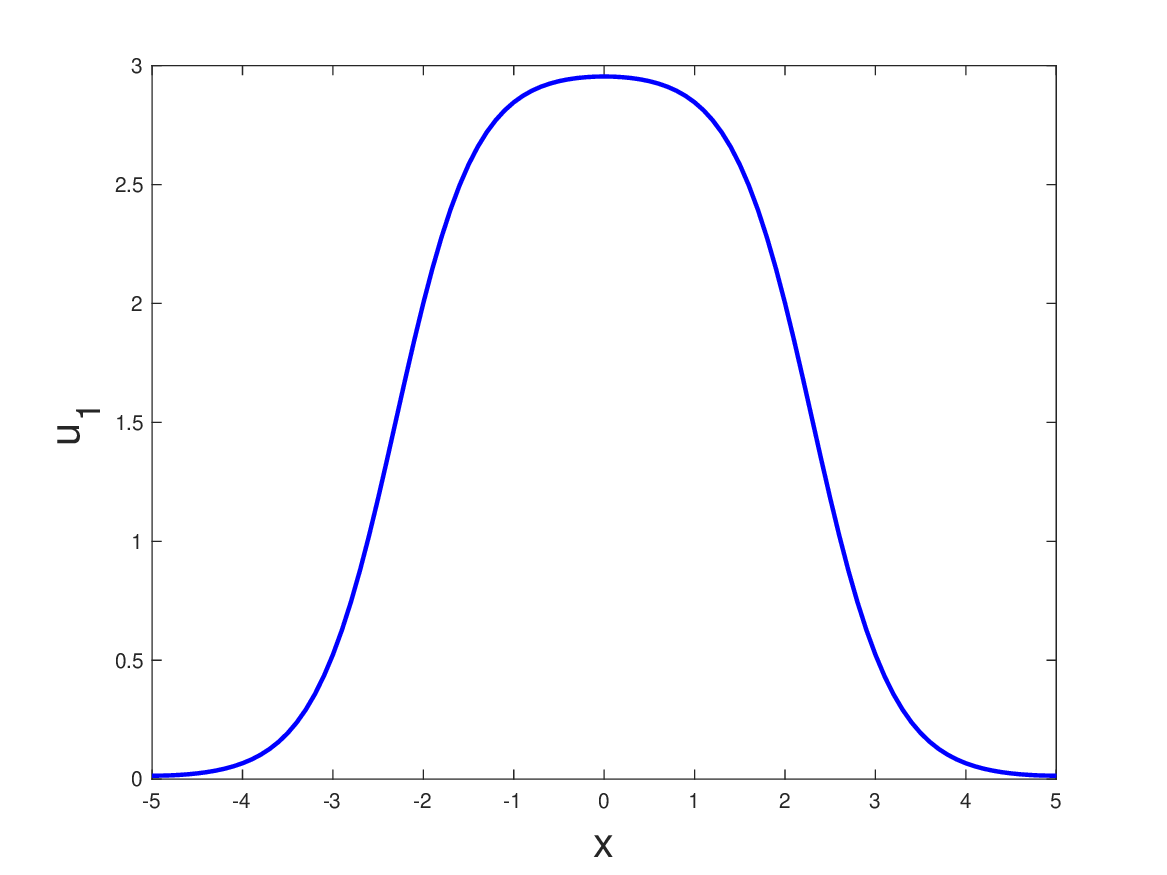, width=\textwidth}
        \caption{$\overline{u}_1$ approximate solution in the Gray-Scott system equation}\label{fig : Z2 u1}
    \end{subfigure}
    \hfill
    \begin{subfigure}[b]{0.24\textwidth}
    \centering
        \epsfig{figure=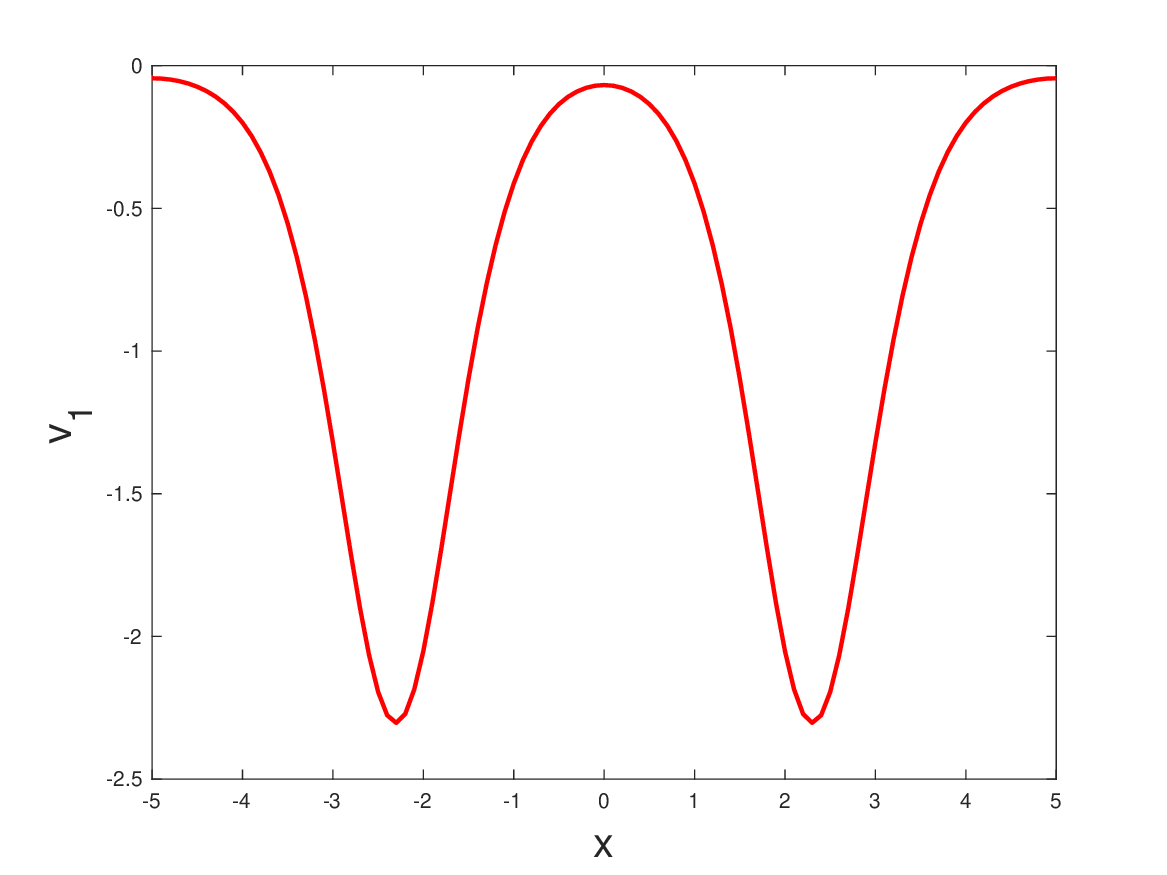, width=\textwidth}
        \caption{$\overline{v}_1$ approximate solution in the Gray-Scott system equation}\label{fig : Z2 v1}
    \end{subfigure}
    \hfill
    \begin{subfigure}[b]{0.24\textwidth}
    \centering
        \epsfig{figure=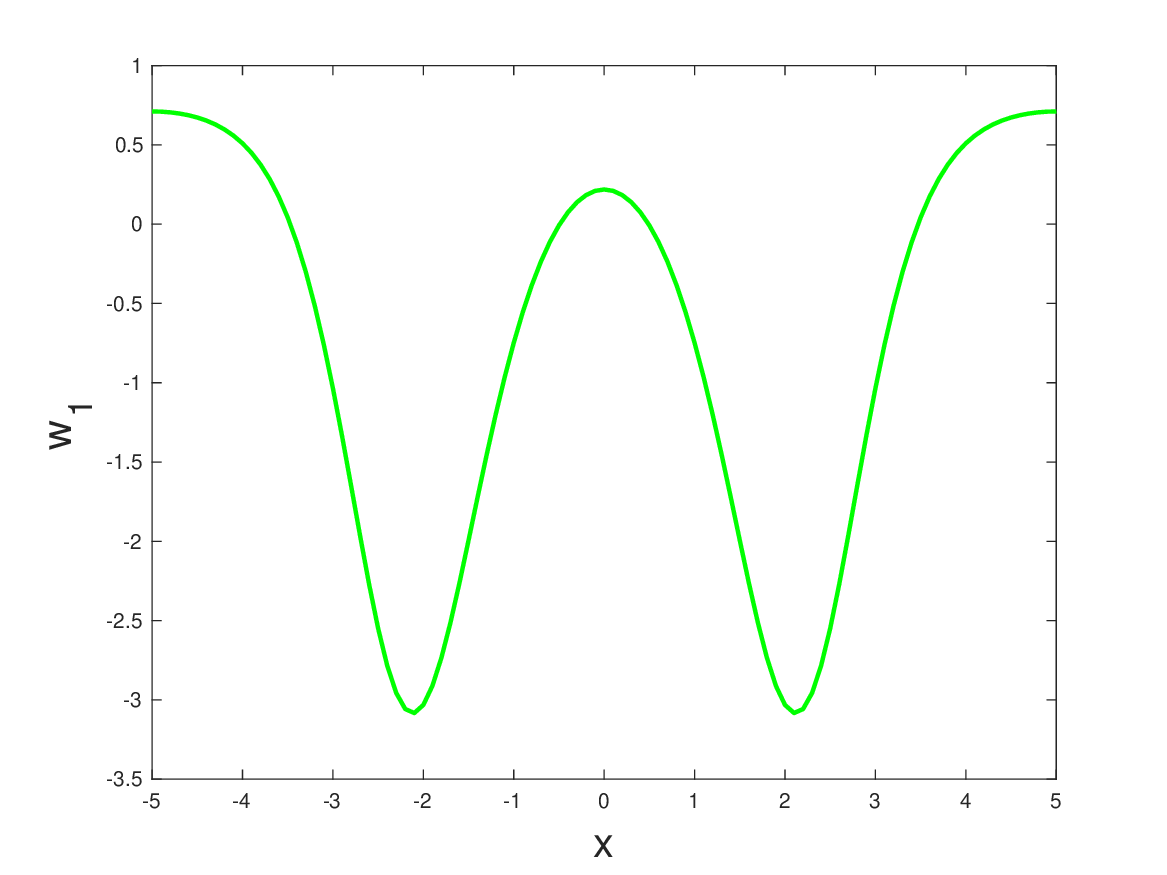, width=\textwidth}
        \caption{$\overline{w}_1$ approximate solution in the Gray-Scott system equation}\label{fig : Z2 w1}
    \end{subfigure}
    \hfill
    \begin{subfigure}[b]{0.24\textwidth}
    \centering
        \epsfig{figure=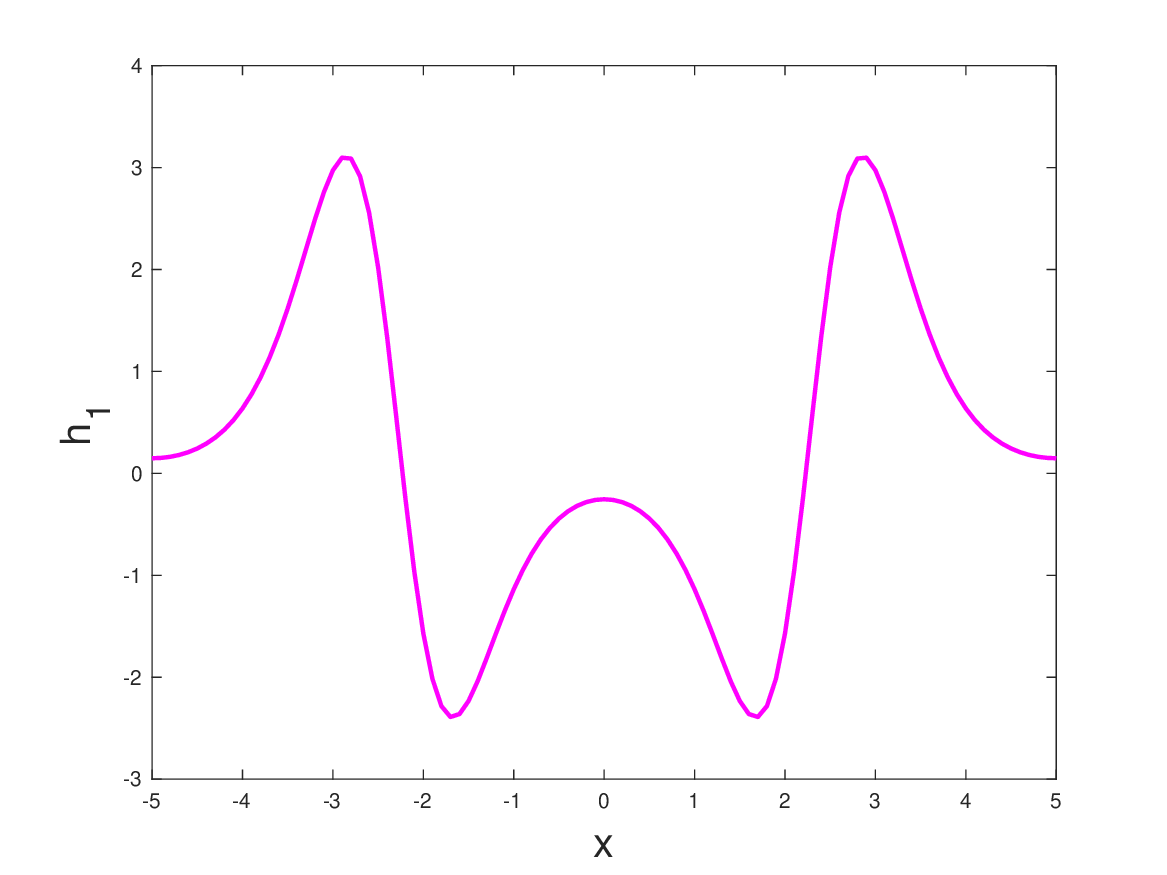, width=\textwidth}
        \caption{$\overline{h}_1$ approximate solution in the Gray-Scott system equation}\label{fig : Z2 h1}
    \end{subfigure}
    \\
    \begin{subfigure}[b]{0.24\textwidth}
        \centering
        \epsfig{figure=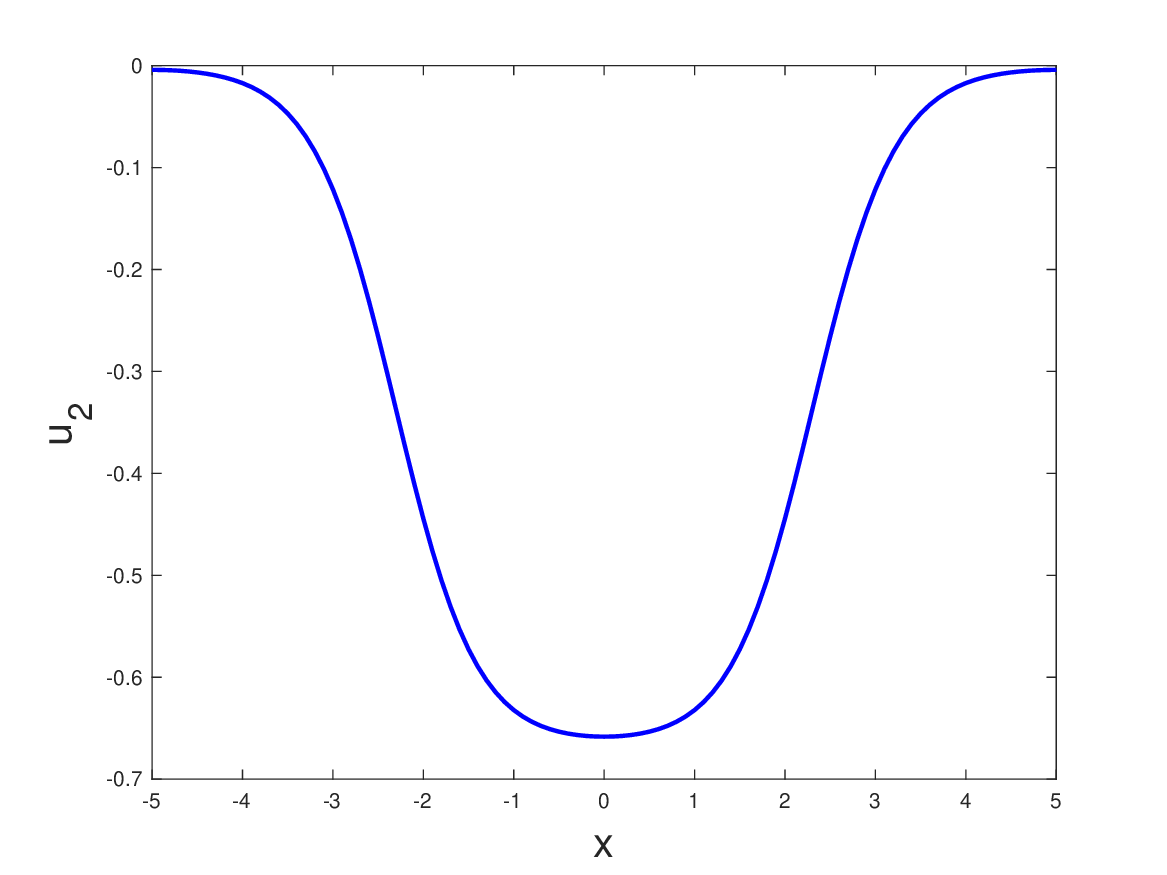, width=\textwidth}
        \caption{$\overline{u}_2$ approximate solution in the Gray-Scott system equation}\label{fig : Z2 u2}
    \end{subfigure}
    \hfill
    \begin{subfigure}[b]{0.24\textwidth}
    \centering
        \epsfig{figure=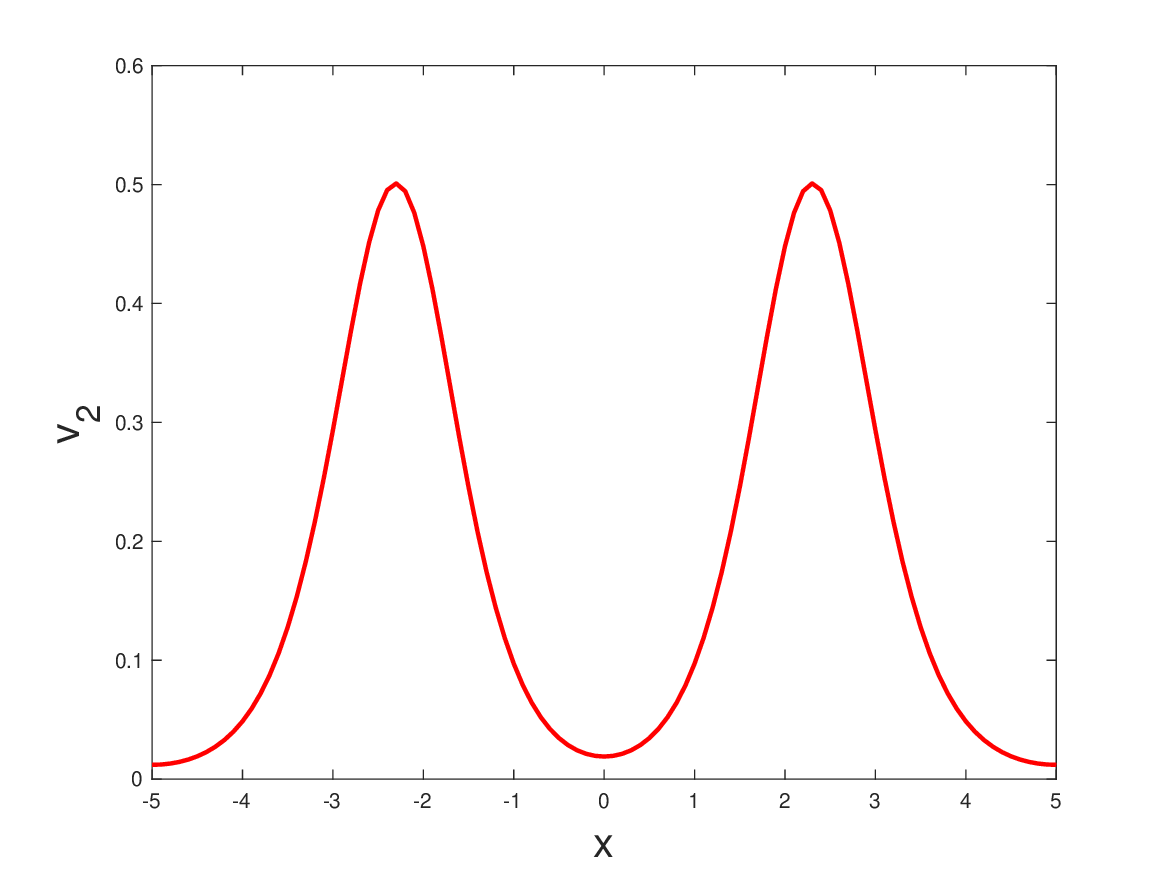, width=\textwidth}
        \caption{$\overline{v}_2$ approximate solution in the Gray-Scott system equation}\label{fig : Z2 v2}
    \end{subfigure}
    \hfill
    \begin{subfigure}[b]{0.24\textwidth}
    \centering
        \epsfig{figure=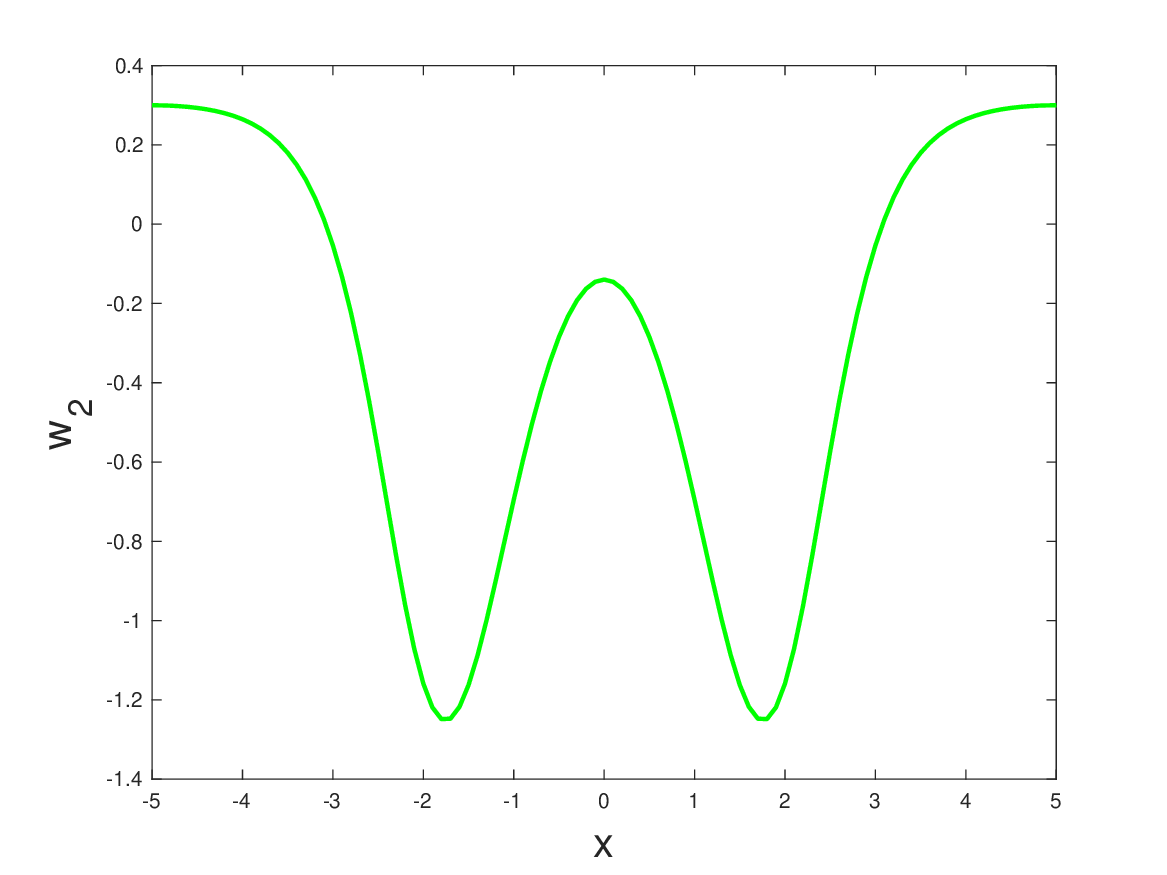, width=\textwidth}
        \caption{$\overline{w}_2$ approximate solution in the Gray-Scott system equation}\label{fig : Z2 w2}
    \end{subfigure}
    \hfill
    \begin{subfigure}[b]{0.24\textwidth}
    \centering
        \epsfig{figure=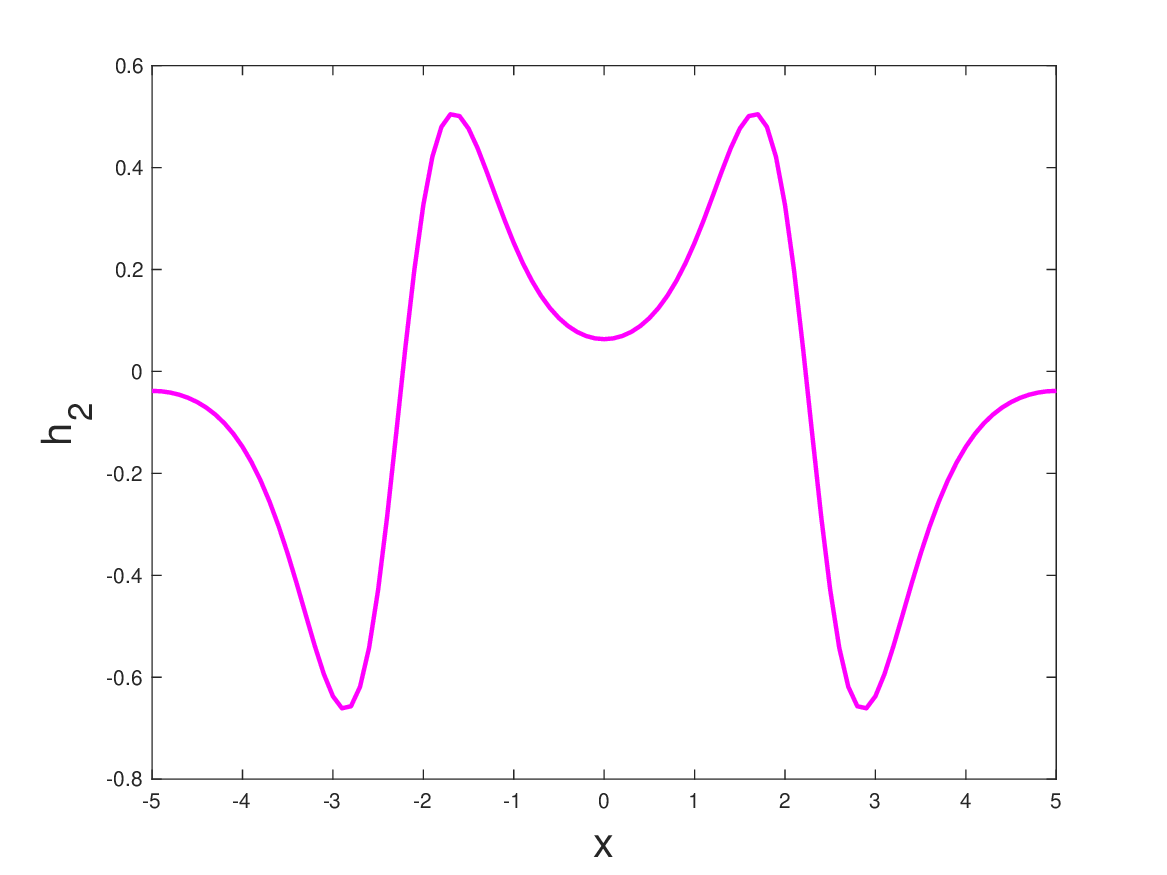, width=\textwidth}
        \caption{$\overline{h}_2$ approximate solution in the Gray-Scott system equation}\label{fig : Z2 h2}
    \end{subfigure}
    \caption{Approximation of each component of $\overline{\mathbf{x}}$ plotted on $(-5,5)$ used in Theorem \ref{th : GS 1D cusp}.}
\end{figure}
\begin{theorem}[$D_2$ Cusp Bifurcation in 2D Gray-Scott]\label{th : GS 2D cusp}
Set $r_0 \bydef 3 \times 10^{-12}$ and choose $\mathcal{L} = I_d$. Let $(\omega_n)_{n = (n_1,n_2) \in \mathbb{Z}^2} \bydef ((1 + |n_1|)^{1.01} (1 + |n_2|)^{1.01})_{n = (n_1,n_2) \in \mathbb{Z}^2}, \delta_1 = 0.08583846158364704,$ and $\delta_2 = 1$. When $\blambda \approx (1,4.183827930189122)$, there exists a unique  $\tilde{\mathbf{x}}\in \overline{B_{r_0}(\overline{\mathbf{x}})} \subset X_{D_2,\omega}$ satisfying $F(\overline{\mathbf{x}}) = 0$ as in \eqref{gs cusp map}. Moreover, $(\tilde{\blambda},\tilde{\mathbf{u}})$ is a cusp bifurcation with $k = 0$ and $c > 0.1796$; in particular, the three coexisting solutions exhibit monostability.
\end{theorem}
\begin{proof}
Take $N = 50$ and  $d = 1.5$. Appyling \cite{CuspProofs.jl} with interval arithemtic yields
\begin{align}
    \|A^N\|_{\mathcal{B}(X_{D_2,\omega})} \leq  382.464\text{,}~Y_0 \bydef 1.389 \times 10^{-13}  \text{,}~Z_{2}(r_0) \bydef  1.13 \times 10^8   \text{,}~Z_1 \bydef 0.9143.
    \end{align}
The spectral and stability conclusions follow as in Theorem \ref{th : SH 1D cusp}.
\end{proof}
 \begin{figure}[H]
    \centering
    \begin{subfigure}[b]{0.24\textwidth}
        \centering
        \epsfig{figure=u1GSD2cusp.eps, width=\textwidth}
        \caption{$\overline{u}_1$ approximate solution in the Gray-Scott system equation}\label{fig : D2 u1}
    \end{subfigure}
    \hfill
    \begin{subfigure}[b]{0.24\textwidth}
    \centering
        \epsfig{figure=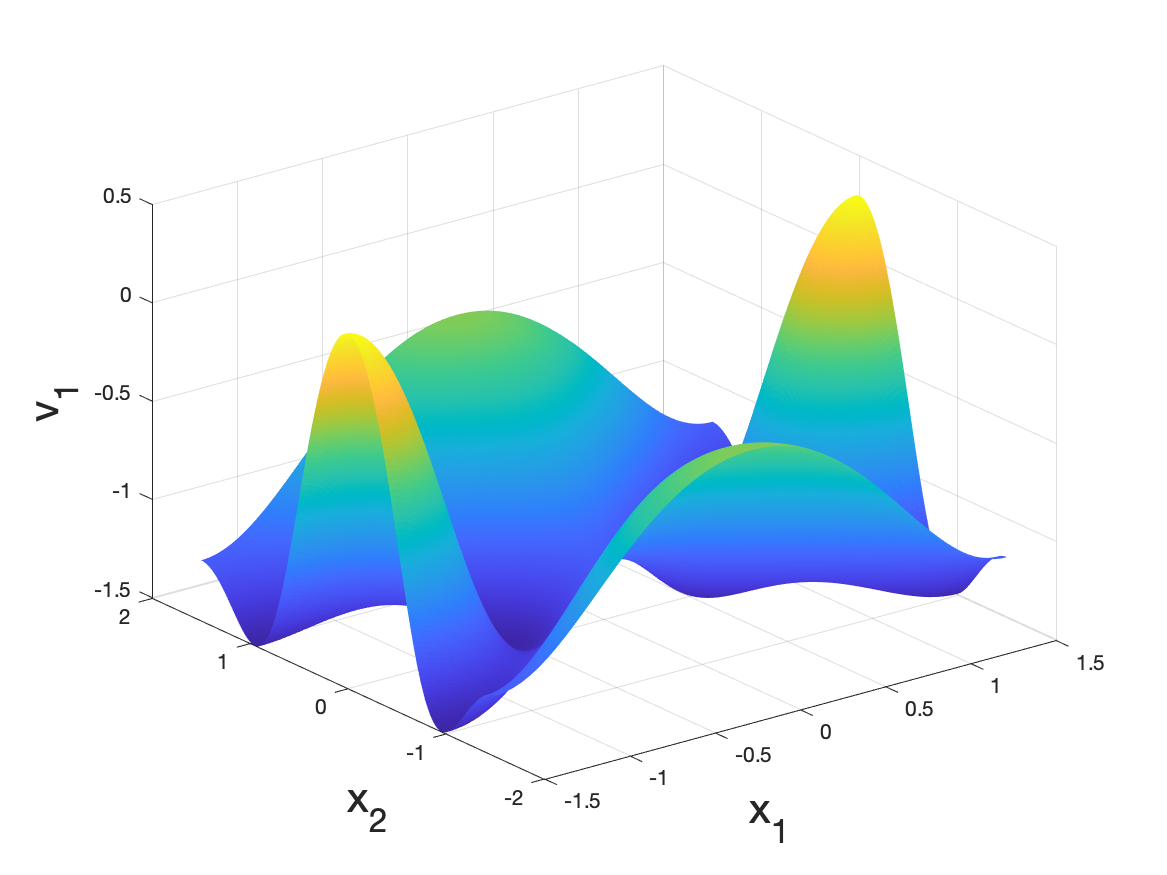, width=\textwidth}
        \caption{$\overline{v}_1$ approximate solution in the Gray-Scott system equation}\label{fig : D2 v1}
    \end{subfigure}
    \hfill
    \begin{subfigure}[b]{0.24\textwidth}
    \centering
        \epsfig{figure=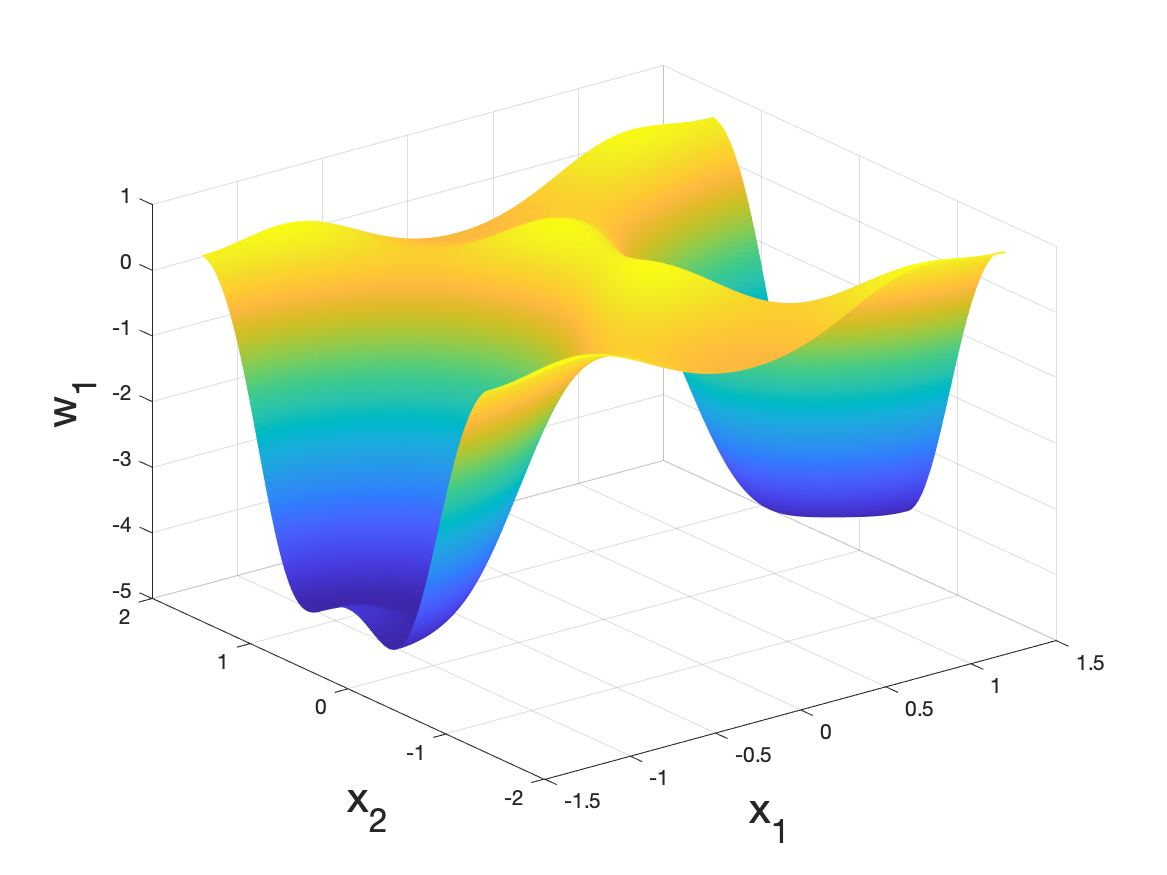, width=\textwidth}
        \caption{$\overline{w}_1$ approximate solution in the Gray-Scott system equation}\label{fig : D2 w1}
    \end{subfigure}
    \hfill
    \begin{subfigure}[b]{0.24\textwidth}
    \centering
        \epsfig{figure=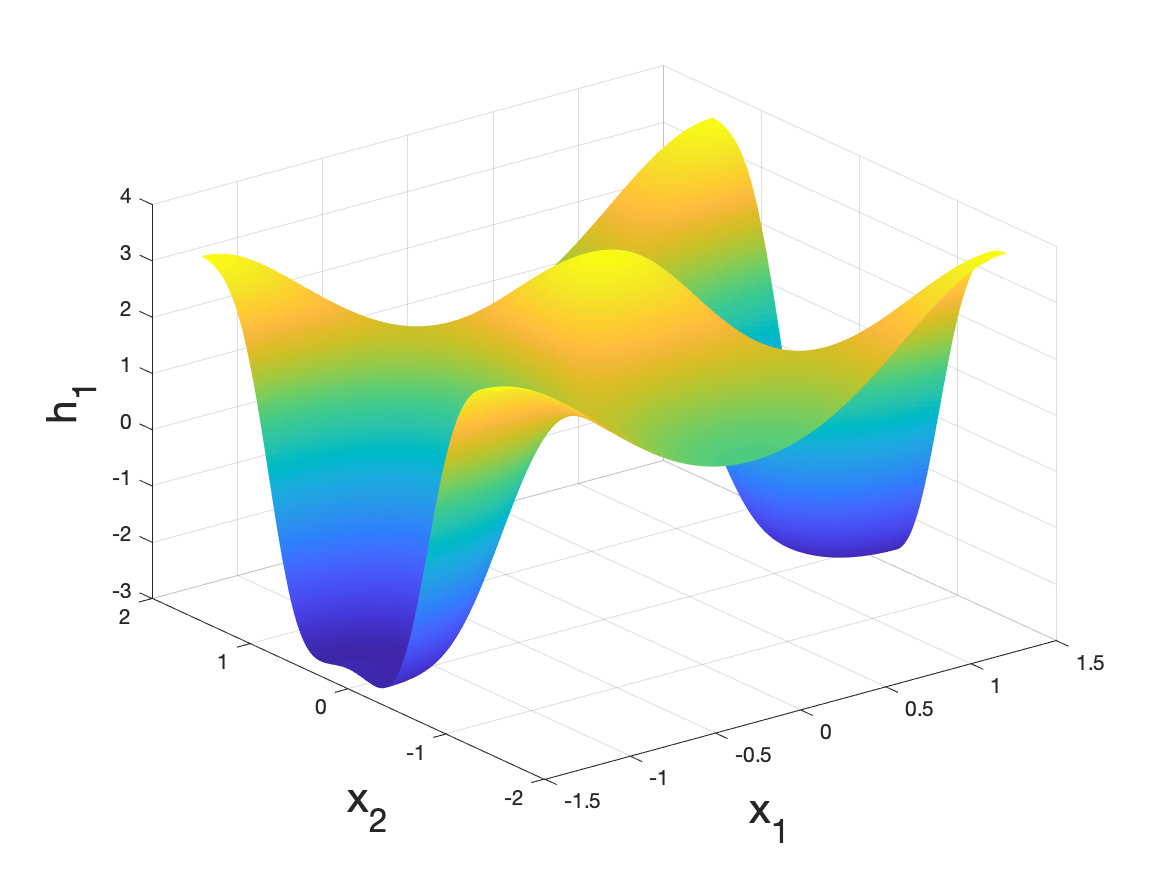, width=\textwidth}
        \caption{$\overline{h}_1$ approximate solution in the Gray-Scott system equation}\label{fig : D2 h1}
    \end{subfigure}
    \\
    \begin{subfigure}[b]{0.24\textwidth}
        \centering
        \epsfig{figure=u2GSD2cusp.eps, width=\textwidth}
        \caption{$\overline{u}_2$ approximate solution in the Gray-Scott system equation}\label{fig : D2 u2}
    \end{subfigure}
    \hfill
    \begin{subfigure}[b]{0.24\textwidth}
    \centering
        \epsfig{figure=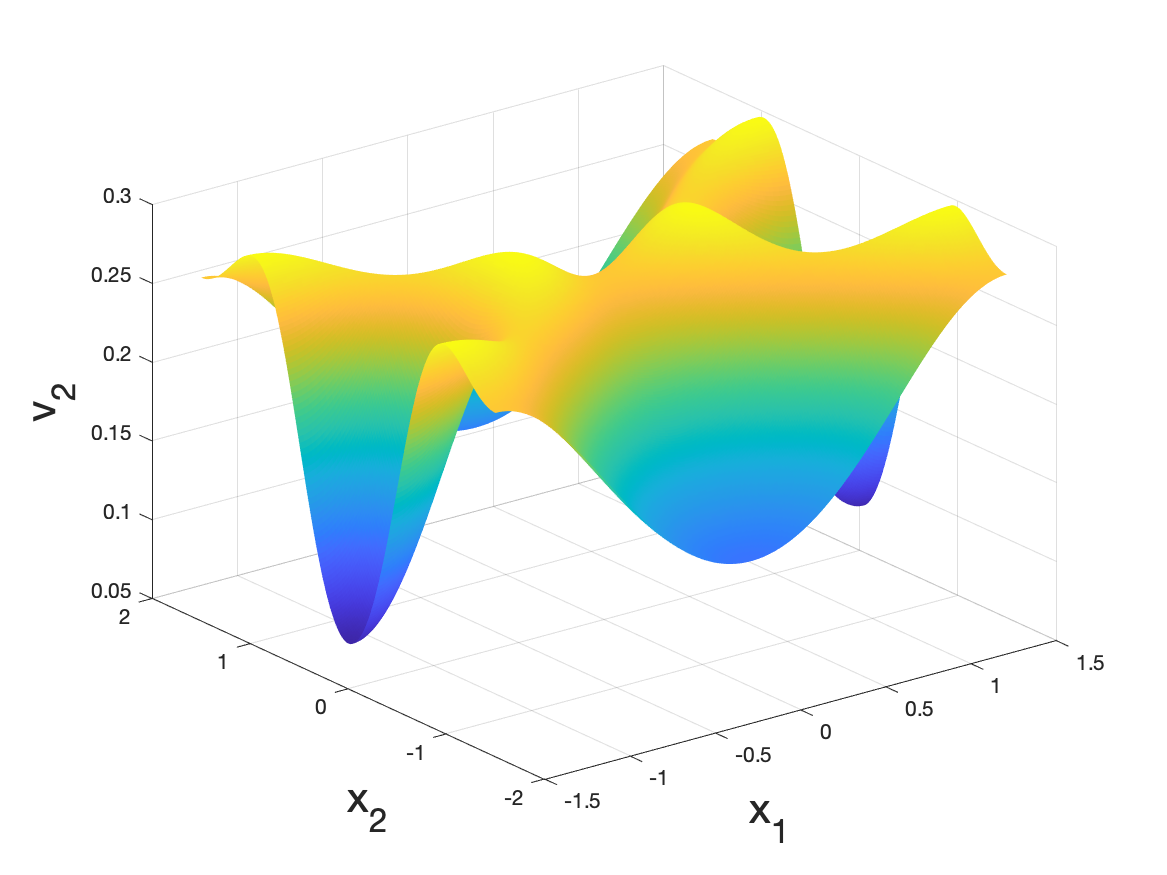, width=\textwidth}
        \caption{$\overline{v}_2$ approximate solution in the Gray-Scott system equation}\label{fig : D2 v2}
    \end{subfigure}
    \hfill
    \begin{subfigure}[b]{0.24\textwidth}
    \centering
        \epsfig{figure=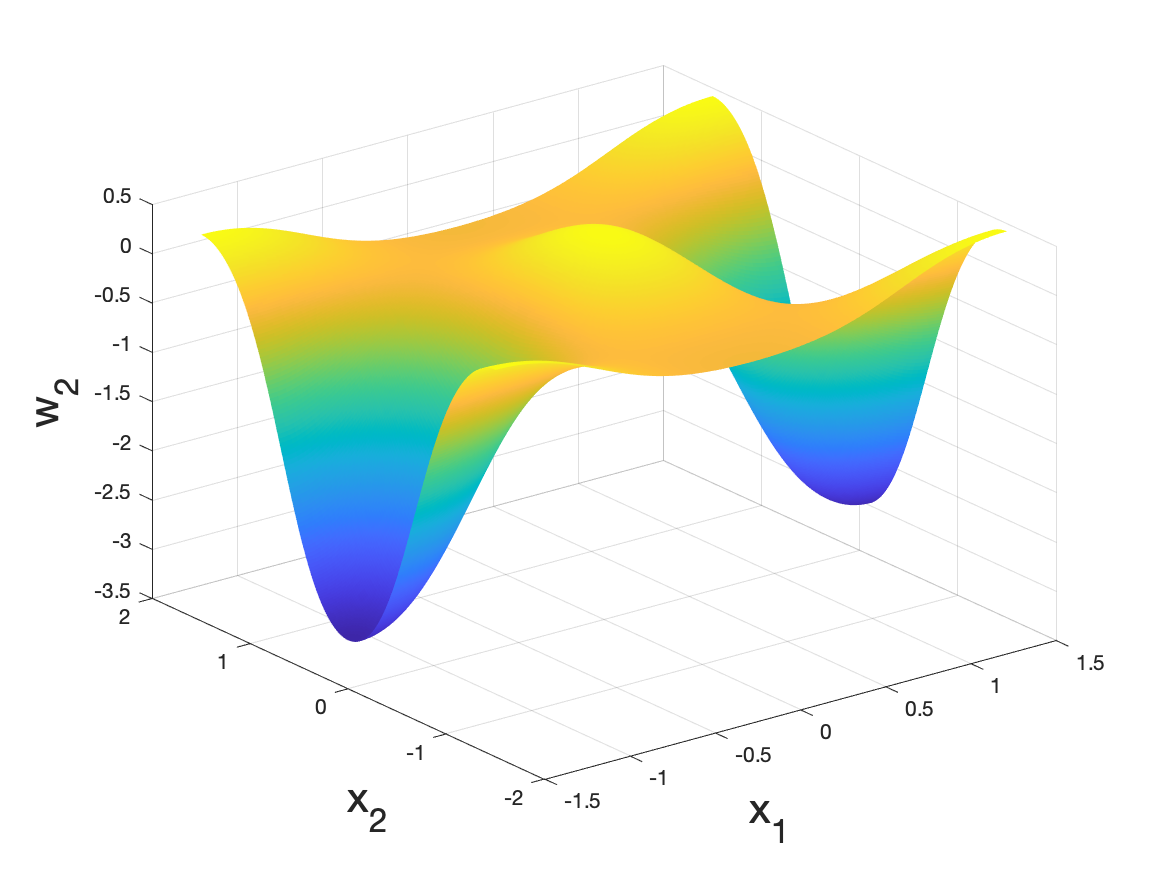, width=\textwidth}
        \caption{$\overline{w}_2$ approximate solution in the Gray-Scott system equation}\label{fig : D2 w2}
    \end{subfigure}
    \hfill
    \begin{subfigure}[b]{0.24\textwidth}
    \centering
        \epsfig{figure=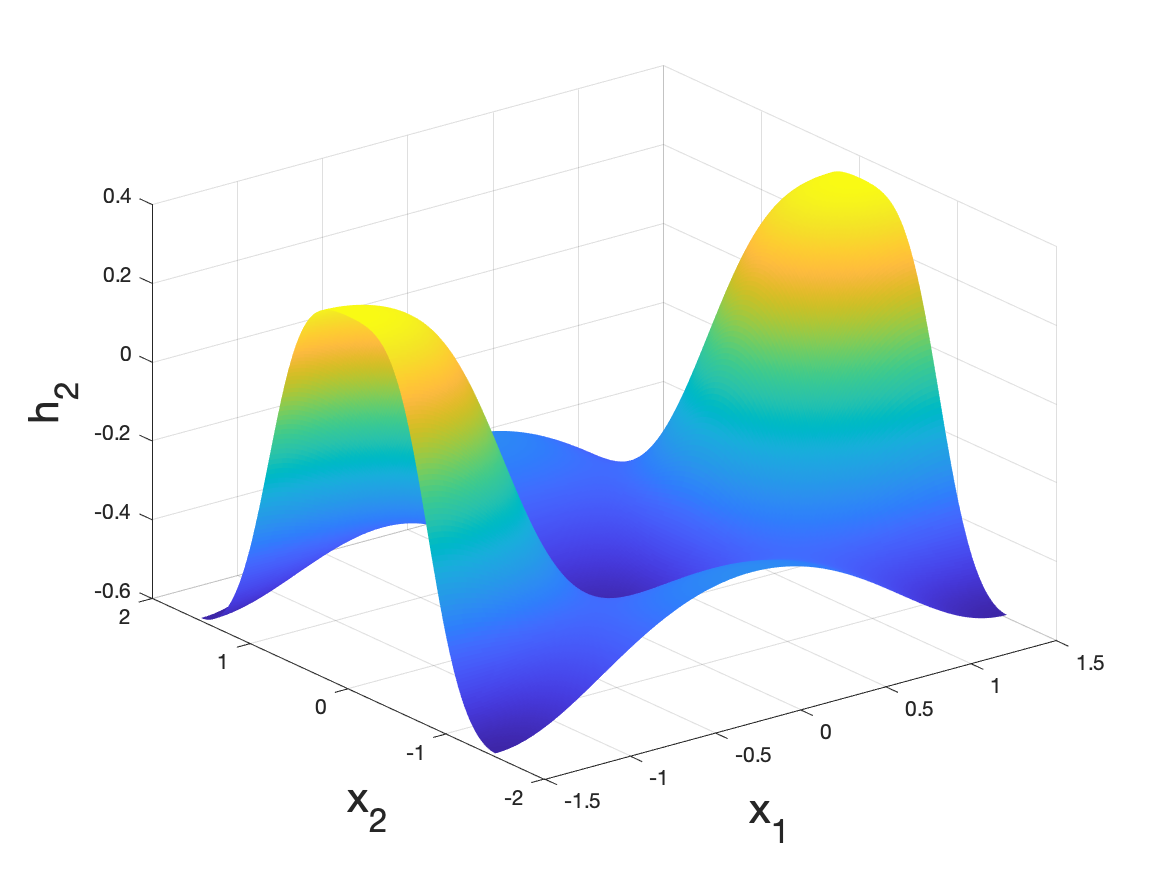, width=\textwidth}
        \caption{$\overline{h}_2$ approximate solution in the Gray-Scott system equation}\label{fig : D2 h2}
    \end{subfigure}
    \caption{Approximation of each component of $\overline{\mathbf{x}}$ plotted on $(-1.5,1.5)^2$ used in Theorem \ref{th : GS 2D cusp}.}
\end{figure}
\section{Conclusion}\label{sec : conclusion}
In this manuscript, we have presented the necessary tools to rigorously compute cusp bifurcations of PDEs. We applied the theory developed in \cite{cusp_jp} to PDEs, which naturally fits our CAP setup. We then demonstrate our approach on the Swift-Hohenberg PDE and the Gray-Scott PDE both in 1D and 2D. We provide the needed bounds and the computational details on how to estimate them. Using these bounds, we proved five cusp bifurcations of various symmetries. Additionally, in the case of Gray-Scott, we have provided the details for rigorously enclosing the spectrum of $D_{u}f(\tilde{\mathbf{u}})$. We do so by using Gershgorin circles, and show how we can use such an argument after completing our rigorous proofs. 

One future goal of ours would be to extend this approach to prove solutions on infinite domains. More specifically, we would wish to apply the method developed by the authors of \cite{unbounded_domain_cadiot}, which allows for one to prove localized patterns of PDEs. In fact, when we began this project, we initially searched for a cusp bifurcation that had the potential to be a localized pattern. We were unsuccessful in locating such a solution in Swift-Hohenberg. In Gray-Scott, we found promising candidates; however, upon extending the size of the domain $d$, the cusp bifurcation would no longer be present. As a result, the first challenge in such a project would be to locate a cusp bifurcation on an unbounded domain. We remark that such solutions appear to be conjectured in 2D Swift-Hohenberg by the authors of \cite{hexagon2008}. Additionally, the authors of \cite{thomas_1d} identify a cusp point in the 1D Thomas model. As their study primarily concerns localized patterns, it is possible that a candidate could exist in this model. The challenge behind studying the Thomas model is the presence of the nonpolynomial nonlinearity. A possible solution could be a similar approach to what was used by the author of \cite{dominic_thomas} who studied the Thomas model. The approach was based on the work of \cite{maxime_paper_continuation,olivier_kevin_paper}. Along with identifying a candidate solution, one will have to handle the adjoint multiplication operator. In this paper, we focused on stationary periodic patterns where we were able to perform our analysis on (Fourier) sequences. When applying the method of \cite{unbounded_domain_cadiot}, one must perform estimates on the function. This would lead to further complications with the adjoint multiplication operator. The study of cusp bifurcations on unbounded domains requires additional attention, and we consider this to be a future project.

Our approach increases the applicability of CAPs when it comes to performing bifurcation analysis. More specifically, our use of an augmented system to locate bifurcations is not new in this realm. As mentioned in the introduction, there are a variety of methods to study bifurcations using CAPs (saddle-nodes, symmetry-breaking, etc). Our approach, based on that of \cite{cusp_jp}, for cusp bifurcations follows the aforementioned approaches in spirit as it relies on Lemma \ref{lemma : Cusp bifurcation}. A more general question is how much further this idea can be stretched, and if there is a way to set-up such augmented zero-finding problems in general. To explain this, a saddle-node bifurcation is co-dimension one and a cusp bifurcation is co-dimension two. Is there a way to generalize the derivation and set-up augmented zero finding problems for bifurcations of co-dimension $k$ for every $k$? Answering such a question would further improve the applicability of CAPs to perform bifurcation analysis, and we consider its investigation a possible future project. 
\appendix
\renewcommand{\theequation}{A.\arabic{equation}}
\setcounter{equation}{0}

\section{Proofs of Lemmas~\ref{lemma : Cusp bifurcation} and \ref{lem : bistability}, and of Theorem~\ref{thm : cusp map nondegen}}

\subsection{Proofs of Lemma~\ref{lemma : Cusp bifurcation}} \label{sec:proof_lemma_cusp}

\begin{proof}
The operator $D_u f(\blambda,\mathbf{u})$, acting from its Sobolev domain $H^{2r}_{\mathcal{G}}(\mathscr{D})^\mathscr{p}$ (where $r$ denotes half the order of the leading differential operator in $l_{\blambda}$) into $L^2_{\mathcal{G}}(\mathscr{D})^\mathscr{p}$, is Fredholm of index zero: the principal symbol of $l_{\blambda}$ is elliptic on the periodic domain $\mathscr{D}$, and the index vanishes by the formal self-adjointness of the leading-order part together with compact Sobolev embeddings. Its kernel is spanned by $\mathbf{v}$, and that of its $L^2_{\mathcal{G}}(\mathscr{D})$-adjoint $D_u f(\blambda,\mathbf{u})^\dagger$ by $\mathbf{w}$, so the Lyapunov--Schmidt reduction applies \cite{chow_hale}. (Throughout Appendices~\ref{sec:proof_lemma_cusp} and~\ref{sec:proof_lemma_stability} we write $\dagger$ for the $L^2_{\mathcal{G}}(\mathscr{D})$-adjoint, reserving $\cdot^*$ for the $(\cdot,\cdot)_2$-adjoint of Section~\ref{sec : adjoint mult}; see Remark~\ref{rmk : L2 vs ell2}.)

Let $P\boldsymbol{\psi} \bydef (\mathbf{w},\boldsymbol{\psi})_{L^2}\mathbf{v}$, a projection since $(\mathbf{w},\mathbf{v})_{L^2}=1$, and $Q \bydef I-P$; every nearby state is uniquely $\mathbf{u}+\xi\mathbf{v}+\boldsymbol{\psi}$ with $\xi\in\R$ and $\boldsymbol{\psi}\in\ker P$. Since $\mathrm{ran}\,D_uf(\blambda,\mathbf{u}) = (\ker D_uf(\blambda,\mathbf{u})^\dagger)^\perp = \ker P$, on which $Q$ is the identity, and $\ker D_uf(\blambda,\mathbf{u})\cap\ker P = \mathrm{span}(\mathbf{v})\cap\ker P = \{0\}$, the operator $QD_uf(\blambda,\mathbf{u})|_{\ker P}$ is invertible. The implicit function theorem applied to the range equation $Qf(\tilde{\blambda},\mathbf{u}+\xi\mathbf{v}+\boldsymbol{\psi})=0$ therefore gives a unique smooth $\Psi(\xi,\tilde{\blambda})\in\ker P$ near $(0,\blambda)$ with $\Psi(0,\blambda)=0$, and the bifurcation equation becomes the scalar \emph{bifurcation function}
\[
  \phi(\xi,\tilde{\blambda}) \bydef \left(\mathbf{w}, f\!\bigl(\tilde{\blambda}, \mathbf{u} + \xi\mathbf{v} + \Psi(\xi,\tilde{\blambda})\bigr) \right)_{L^2},
\]
whose zeros correspond bijectively to the solutions $\mathbf{u}+\xi\mathbf{v}+\Psi(\xi,\tilde{\blambda})$ of $f(\tilde{\blambda},\cdot)=0$ near $(\blambda,\mathbf{u})$ --- these are the states \eqref{eq : three solutions}. In particular $\phi(0,\blambda) = (\mathbf{w},f(\blambda,\mathbf{u}))_{L^2}=0$ is the equilibrium condition; carrying the base point $\mathbf{u}$ along is what makes this so.

Differentiating the range equation once at $(0,\blambda)$ gives $QD_uf(\blambda,\mathbf{u})(\mathbf{v}+\partial_\xi\Psi(0,\blambda))=0$, hence $\partial_\xi\Psi(0,\blambda)=0$ by the invertibility just noted. Differentiating twice then gives $QD_uf(\blambda,\mathbf{u})\partial_\xi^2\Psi(0,\blambda) = -QD_{uu}f(\blambda,\mathbf{u})(\mathbf{v},\mathbf{v}) = -D_{uu}f(\blambda,\mathbf{u})(\mathbf{v},\mathbf{v})$, the last equality because $PD_{uu}f(\blambda,\mathbf{u})(\mathbf{v},\mathbf{v}) = 2b\,\mathbf{v}=0$; hence $\partial_\xi^2\Psi(0,\blambda)=\mathbf{h}$, normalised by $(\mathbf{w},\mathbf{h})_{L^2}=0$ (the value of $c$ does not depend on this choice, see Section~\ref{sec:introduction}). By the chain rule,
\[
  \partial_\xi\phi(0,\blambda)=0, \qquad \partial_\xi^2\phi(0,\blambda)=2b=0, \qquad \partial_\xi^3\phi(0,\blambda)=6c,
\]
so that $\phi(\xi,\blambda) = c\xi^3+O(\xi^4)$ with $c\neq0$.

It remains to check that the two parameters unfold this degeneracy transversally, which is where \eqref{eq : Delta} enters. Differentiating $\phi$ with respect to $\lambda_i$ at $(0,\blambda)$ and using $D_uf(\blambda,\mathbf{u})^\dagger\mathbf{w}=0$ to kill the term carrying $\partial_{\lambda_i}\Psi$ gives $\partial_{\lambda_i}\phi(0,\blambda) = \sigma_i$. Differentiating once more in $\xi$ and using $\partial_\xi\Psi(0,\blambda)=0$,
\[
  \partial_{\lambda_i}\partial_\xi\phi(0,\blambda)
  = \bigl(\mathbf{w},\partial_{\lambda_i}[D_uf]\mathbf{v}\bigr)_{L^2} + \bigl(\mathbf{w},D_{uu}f(\mathbf{v},\Theta_i)\bigr)_{L^2},
  \qquad \Theta_i \bydef \partial_{\lambda_i}\Psi(0,\blambda) \in \ker P ,
\]
and since $\Theta_i$ solves $D_uf\,\Theta_i = -\partial_{\lambda_i}f + \sigma_i\mathbf{v}$ while $D_{uu}f^\dagger(\mathbf{w},\mathbf{v}) = D_uf^\dagger\mathbf{z}$ by \eqref{eq : z vector},
\[
  \bigl(\mathbf{w},D_{uu}f(\mathbf{v},\Theta_i)\bigr)_{L^2} = \bigl(\mathbf{z},D_uf\,\Theta_i\bigr)_{L^2} = -\bigl(\mathbf{z},\partial_{\lambda_i}f\bigr)_{L^2} + \beta\,\sigma_i,
  \qquad \beta\bydef(\mathbf{z},\mathbf{v})_{L^2}.
\]
Thus $\partial_{\lambda_i}\partial_\xi\phi(0,\blambda) = \tau_i+\beta\sigma_i$ --- the shift already announced after \eqref{eq : Delta} --- and
\begin{equation}\label{eq : versality determinant}
  \det \frac{\partial(\phi,\partial_\xi\phi)}{\partial(\lambda_1,\lambda_2)}(0,\blambda)
  = \sigma_1(\tau_2+\beta\sigma_2)-\sigma_2(\tau_1+\beta\sigma_1) = \Delta \neq 0 .
\end{equation}
So $\phi$ has at $(0,\blambda)$ the singularity $\phi=\partial_\xi\phi=\partial_\xi^2\phi=0$, $\partial_\xi^3\phi = 6c\neq0$, unfolded transversally by $(\lambda_1,\lambda_2)$. By the versal-unfolding (contact-equivalence) theorem for the cusp \cite{kuznetsov_book,chow_hale,kuznetsov_2}, there are a smooth change of coordinate $\xi=\Xi(y,\tilde{\blambda})$ with $\partial_y\Xi>0$, a reparametrisation $\tilde{\blambda}\mapsto(\beta_1,\beta_2)$ which is a local diffeomorphism precisely by \eqref{eq : versality determinant}, and a smooth $\varrho>0$, such that
\[
  \phi\bigl(\Xi(y,\tilde{\blambda}),\tilde{\blambda}\bigr) = \varrho(y,\tilde{\blambda})\bigl[\beta_1+\beta_2 y + \mathrm{sign}(c)\,y^3\bigr] .
\]
This is the assertion of the lemma.
\end{proof}

\subsection{Proofs of Lemma~\ref{lem : bistability}} \label{sec:proof_lemma_stability}

\begin{proof}
Let $\tilde{\blambda}$ be close to $\blambda$ and inside the cusp region, so that $\phi(\cdot,\tilde{\blambda}) = 0$ has exactly three roots $\xi_1 < \xi_2 < \xi_3$ near $\xi = 0$, and the three solutions $\tilde{\mathbf{u}}_j = \mathbf{u} + \xi_j\mathbf{v} + \Psi(\xi_j,\tilde{\blambda})$ ($j=1,2,3$) vary continuously with $\tilde{\blambda}$.

\textbf{Step 1 (Persistence of non-zero eigenvalues).} Set $A_0\bydef D_uf(\blambda,\mathbf{u})$ and $A_j\bydef D_uf(\tilde{\blambda},\tilde{\mathbf{u}}_j)$, closed on $L^2_{\mathcal{G}}(\mathscr{D})^{\mathscr{p}}$ with the common domain $H^{2r}_{\mathcal{G}}(\mathscr{D})^{\mathscr{p}}$; then $A_0$ is sectorial with compact resolvent and $\|A_j-A_0\|_{\mathcal{B}(H^{2r},L^2)}\to0$ as $\tilde{\blambda}\to\blambda$. Consequently $\sigma(A_0)$ is discrete and lies in a left sector, so the spectral gap hypothesis of Lemma~\ref{lemma : Cusp bifurcation} makes $\sigma(A_0)\cap\{\mathrm{Re}\,\mu>-\varepsilon\}$ finite: the simple eigenvalue $0$ and the $k<\infty$ eigenvalues with positive real part. Analytic perturbation theory \cite[Thm.~IV.3.16]{kato2013perturbation} then gives, for $\tilde{\blambda}$ close to $\blambda$, that $A_j$ has exactly $k$ eigenvalues with positive real part, exactly one eigenvalue $\mu_j$ near $0$ --- real and algebraically simple, with left eigenvector $\mathbf{w}_j\to\mathbf{w}$ normalised by $(\mathbf{w}_j,\mathbf{v})_{L^2}=1$ --- and all others in $\{\mathrm{Re}\,\mu<-\varepsilon/2\}$. Thus the only eigenvalue that can change half-plane is the one whose real part is zero at $(\blambda,\mathbf{u})$.

\textbf{Step 2 (Exchange of stability).} By construction $Qf(\tilde{\blambda},\mathbf{u}+\xi\mathbf{v}+\Psi(\xi,\tilde{\blambda}))=0$ and $Pf(\tilde{\blambda},\mathbf{u}+\xi\mathbf{v}+\Psi(\xi,\tilde{\blambda})) = \phi(\xi,\tilde{\blambda})\mathbf{v}$, so that $f(\tilde{\blambda},\mathbf{u}+\xi\mathbf{v}+\Psi(\xi,\tilde{\blambda})) = \phi(\xi,\tilde{\blambda})\,\mathbf{v}$ identically in $\xi$. Differentiating at $\xi=\xi_j$ gives
\[
  A_j\mathbf{y}_j = \phi'(\xi_j,\tilde{\blambda})\,\mathbf{v},
  \qquad \mathbf{y}_j \bydef \mathbf{v}+\partial_\xi\Psi(\xi_j,\tilde{\blambda}) \;\longrightarrow\; \mathbf{v},
\]
where $\phi' = \partial_\xi\phi$. Pairing with $\mathbf{w}_j$ and using $A_j^\dagger\mathbf{w}_j = \mu_j\mathbf{w}_j$ with $\mu_j$ real yields the exchange-of-stability relation
\begin{align}\label{eq : exchange stability}
  \mu_j\,\mathcal{Q}(\xi_j,\tilde{\blambda}) = \phi'(\xi_j,\tilde{\blambda}),
  \qquad
  \mathcal{Q}(\xi_j,\tilde{\blambda}) \bydef \bigl(\mathbf{w}_j,\mathbf{y}_j\bigr)_{L^2} \;\longrightarrow\; (\mathbf{w},\mathbf{v})_{L^2} = 1 .
\end{align}
In particular $\mathcal{Q}>0$, hence $\mathrm{sign}(\mu_j)=\mathrm{sign}(\phi'(\xi_j,\tilde{\blambda}))$, for $\tilde{\blambda}$ close enough to $\blambda$. (In the scalar model $\dot u = f(\blambda,u)$ one has $\phi=f$ and $\mathcal{Q}\equiv1$, consistent with $\mu = D_uf = \phi'$.)

\textbf{Step 3 (Sign of $\phi'$ at each root).} We argue on $\phi$ itself, avoiding any truncation (cf.\ Remark~\ref{rmk : discriminant}). Since $\phi(\xi,\blambda) = c\xi^3(1+O(\xi))$ with $c\neq0$, we may fix $\xi_*>0$ with $\phi(\cdot,\blambda)\neq0$ on $0<|\xi|\leq\xi_*$ and $\mathrm{sign}\,\phi(\xi_*,\blambda)=\mathrm{sign}(c)$, and read ``near $\xi=0$'' as ``in $(-\delta,\delta)$'' for some $\delta\in(0,\xi_*)$. As $\phi$ is uniformly continuous and $|\phi(\cdot,\blambda)|$ is bounded below on $\delta\leq|\xi|\leq\xi_*$, for $\tilde{\blambda}$ close enough $\phi(\cdot,\tilde{\blambda})$ has no zero there and still satisfies $\mathrm{sign}\,\phi(\xi_*,\tilde{\blambda})=\mathrm{sign}(c)$; hence $\xi_1,\xi_2,\xi_3$ are its only zeros in $[-\xi_*,\xi_*]$, and they are simple because the cusp region avoids the fold set $\mathcal{F}$. Now $\phi(\cdot,\tilde{\blambda})$ has constant sign $\mathrm{sign}(c)$ on $(\xi_3,\xi_*]$, so $\mathrm{sign}\,\phi'(\xi_3,\tilde{\blambda})=\mathrm{sign}(c)$; and the signs of $\phi'$ at consecutive simple zeros alternate. Therefore
\[
  \mathrm{sign}\,\phi'(\xi_1,\tilde{\blambda}) = \mathrm{sign}\,\phi'(\xi_3,\tilde{\blambda}) = \mathrm{sign}(c),
  \qquad
  \mathrm{sign}\,\phi'(\xi_2,\tilde{\blambda}) = -\mathrm{sign}(c) .
\]

\textbf{Step 4 (Counting unstable eigenvalues).} By Steps 1--3, the total number of eigenvalues of $D_u f(\tilde{\blambda},\tilde{\mathbf{u}}_j)$ with positive real part equals $k$ plus the contribution of the near-zero eigenvalue:
\begin{itemize}
  \item If $c < 0$: $\phi'(\xi_{1,3}) < 0$ so $\mu_{1},\mu_3 < 0$, contributing $0$ to the unstable count; $\phi'(\xi_2) > 0$ so $\mu_2 > 0$, contributing $1$. Totals: $k$ at $\tilde{\mathbf{u}}_{1,3}$, $k+1$ at $\tilde{\mathbf{u}}_2$.
  \item If $c > 0$: $\phi'(\xi_{1,3}) > 0$ so $\mu_{1},\mu_3 > 0$, contributing $1$; $\phi'(\xi_2) < 0$ so $\mu_2 < 0$, contributing $0$. Totals: $k+1$ at $\tilde{\mathbf{u}}_{1,3}$, $k$ at $\tilde{\mathbf{u}}_2$.
\end{itemize}
The bistability ($k=0$, $c<0$), monostability ($k=0$, $c>0$), and complete instability ($k\geq 1$) claims follow immediately, by the principle of linearized stability for the analytic semigroup generated by $D_uf(\tilde{\blambda},\tilde{\mathbf{u}}_j)$ \cite{henry_geometric,lunardi_analytic}, the required spectral gap being provided by Step~1.
\end{proof}

\begin{remark}[\bf The unfolded cubic]\label{rmk : discriminant}
Step~3 is often presented on the truncated unfolding $P(\xi) \bydef c\xi^3+\beta_2\xi+\beta_1$, obtained from $\phi$ after the near-identity shift removing the quadratic term \cite{kuznetsov_book} (we write $\beta_1,\beta_2$ rather than $\mu_1,\mu_2$ so as not to clash with the eigenvalues $\mu_j$). Two warnings. First, the discriminant of $P$ is $-4c\beta_2^3-27c^2\beta_1^2$, so three distinct real roots means
\begin{align}\label{eq : discriminant}
  4c\,\beta_2^3 + 27c^2\beta_1^2 < 0
\end{align}
--- \emph{not} $4\beta_2^3+27c^2\beta_1^2<0$, which is not homogeneous in $c$ and fails for $-\xi^3+\xi$ (roots $0,\pm1$, yet $4\beta_2^3+27c^2\beta_1^2 = 4>0$). From \eqref{eq : discriminant} one gets $c\beta_2^3<0$, hence $c\beta_2<0$, and with $\rho\bydef-\beta_2/(3c)>0$, $P'(\xi) = 3c(\xi^2-\rho)$ gives the signs of Step~3 at once; note that one may not multiply $\xi_{1,3}^2>\rho$ by $3c$ without distinguishing the sign of $c$. Second, these are statements about $P$ and not about $\phi$, which carries a remainder $O(\xi^4)$; this is why Step~3 argues on $\phi$ directly.
\end{remark}

\subsection{Proof of Theorem~\ref{thm : cusp map nondegen}} \label{sec:proof_thm_cusp_map_nondegen}

\begin{proof}
By Remark~\ref{rmk : L2 vs ell2}, $F(\mathbf{x})=0$ forces $s_1=s_2=0$ and delivers every hypothesis of Lemma~\ref{lemma : Cusp bifurcation} except $c\neq0$ and $\Delta\neq0$. We obtain both at once by establishing the kernel isomorphism \eqref{eq : kernel isomorphism} of Remark~\ref{rmk : factorization}: $DF(\mathbf{x})$ being invertible, hence injective, $\ker\mathbb{T}=\{0\}$ and $\det\mathbb{T}=6c\Delta\neq0$. (Only injectivity is used. The argument of \cite{cusp_jp}, recalled in Remark~\ref{rmk : old kernel element}, gives only ``$c=0\Rightarrow\ker DF(\mathbf{x})\neq\{0\}$''; transversality needs the parameter directions.)

\smallskip\noindent\textbf{Notation.} Suppressing $(\blambda,\mathbf{u})$, put $\mathcal{A}\bydef D_uf$, $\mathcal{B}\bydef D_{uu}f$, $\mathcal{C}\bydef D_{uuu}f$ and, for $\hat{\blambda}\in\R^2$, $g_{\hat{\blambda}}\bydef\sum_i\hat\lambda_i\partial_{\lambda_i}f$, $G_{\hat{\blambda}}\bydef\sum_i\hat\lambda_i\partial_{\lambda_i}[D_uf]$, $H_{\hat{\blambda}}\bydef\sum_i\hat\lambda_i\partial_{\lambda_i}[D_{uu}f]$; also $\gamma\bydef(\mathbf{v},\mathbf{h})_2$, $\beta\bydef(\mathbf{v},\mathbf{z})_2$ and $\tilde\tau_i\bydef\tau_i+\beta\sigma_i$. We write $\mathcal{B}^*(\hat{\mathbf{u}},\mathbf{w})$ for the variation of $\mathcal{A}^*\mathbf{w}$ in the direction $\hat{\mathbf{u}}$, so $(\mathcal{B}^*(\hat{\mathbf{u}},\mathbf{w}),\mathbf{y})_2 = (\mathbf{w},\mathcal{B}(\hat{\mathbf{u}},\mathbf{y}))_2$: the direction sits in the first slot, the reverse of \eqref{eq : z vector}. Thus $\mathcal{A}^*\mathbf{z} = \mathcal{B}^*(\mathbf{v},\mathbf{w})$ (a different vector from $\mathcal{B}^*(\mathbf{w},\mathbf{v})$, which is not used), and $\mathbf{z}$ exists by Step~0 since $(\mathbf{v},\mathcal{B}^*(\mathbf{v},\mathbf{w}))_2 = (\mathbf{w},\mathcal{B}(\mathbf{v},\mathbf{v}))_2 = 2b = 0$. Here $\sigma_i,\tau_i,\Delta,\mathbf{z}$ are read in the $(\cdot,\cdot)_2$ convention, which by Remark~\ref{rmk : L2 vs ell2} gives the same numbers as \eqref{eq : z vector}--\eqref{eq : Delta}. With $\hat{\mathbf{X}} = (\hat s_1,\hat s_2,\hat{\blambda},\hat{\mathbf{u}},\hat{\mathbf{v}},\hat{\mathbf{w}},\hat{\mathbf{h}})$, the equation $DF(\mathbf{x})\hat{\mathbf{X}}=0$ reads
\begin{equation*}
\begin{aligned}
&\text{(1)} \ 2(\hat{\mathbf{v}},\mathbf{v})_2 = 0, \quad
\text{(2)} \ (\hat{\mathbf{v}},\mathbf{w})_2 + (\mathbf{v},\hat{\mathbf{w}})_2 = 0, \quad
\text{(3)} \ (\hat{\mathbf{w}},\mathbf{h})_2 + (\mathbf{w},\hat{\mathbf{h}})_2 = 0, \quad
\text{(5)} \ \mathcal{A}\hat{\mathbf{u}} + g_{\hat{\blambda}} = 0,\\
&\text{(4)} \ \bigl(\mathcal{C}(\hat{\mathbf{u}},\mathbf{v},\mathbf{v}) + H_{\hat{\blambda}}(\mathbf{v},\mathbf{v}) + 2\mathcal{B}(\hat{\mathbf{v}},\mathbf{v}),\mathbf{w}\bigr)_2 + \bigl(\mathcal{B}(\mathbf{v},\mathbf{v}),\hat{\mathbf{w}}\bigr)_2 = 0,\\
&\text{(6)} \ \mathcal{B}(\hat{\mathbf{u}},\mathbf{v}) + G_{\hat{\blambda}}\mathbf{v} + \mathcal{A}\hat{\mathbf{v}} + \hat{s}_1\mathbf{v} = 0, \qquad
\text{(7)} \ \mathcal{B}^*(\hat{\mathbf{u}},\mathbf{w}) + G_{\hat{\blambda}}^*\mathbf{w} + \mathcal{A}^*\hat{\mathbf{w}} = 0,\\
&\text{(8)} \ \mathcal{B}(\hat{\mathbf{u}},\mathbf{h}) + G_{\hat{\blambda}}\mathbf{h} + \mathcal{A}\hat{\mathbf{h}} + \hat{s}_2\mathbf{v} + \mathcal{C}(\hat{\mathbf{u}},\mathbf{v},\mathbf{v}) + H_{\hat{\blambda}}(\mathbf{v},\mathbf{v}) + 2\mathcal{B}(\hat{\mathbf{v}},\mathbf{v}) = 0 .
\end{aligned}
\end{equation*}

\smallskip\noindent\textbf{Step 0 (Solvability).} By Remark~\ref{rmk : L2 vs ell2}, $\ker\mathcal{A}=\mathrm{span}(\mathbf{v})$ and $\ker\mathcal{A}^*=\mathrm{span}(\mathbf{w})$; since $\mathcal{A}$ is Fredholm of index zero, the orthogonality relations $\mathrm{ran}\,\mathcal{A} = (\ker\mathcal{A}^\dagger)^{\perp_{L^2}}$ and $\mathrm{ran}\,\mathcal{A}^* = \mathcal{O}_{\mathcal{G}}(\ker\mathcal{A})^{\perp_{L^2}}$ become, by \eqref{eq : L2 vs ell2},
\[
\mathrm{ran}\,\mathcal{A} = \{\mathbf{r} : (\mathbf{w},\mathbf{r})_2 = 0\},
\qquad
\mathrm{ran}\,\mathcal{A}^* = \{\mathbf{r} : (\mathbf{v},\mathbf{r})_2 = 0\}
\]
(readable in $X_{\mathcal{G},\omega}$, again by that remark). Since $(\mathbf{w},\mathbf{v})_2=(\mathbf{v},\mathbf{v})_2=1$, it follows that each of
\[
(\mathbf{y},\alpha)\mapsto\bigl(\mathcal{A}\mathbf{y}+\alpha\mathbf{v},(\mathbf{w},\mathbf{y})_2\bigr),\quad
(\mathbf{y},\alpha)\mapsto\bigl(\mathcal{A}\mathbf{y}+\alpha\mathbf{v},(\mathbf{y},\mathbf{v})_2\bigr),\quad
(\mathbf{y},\alpha)\mapsto\bigl(\mathcal{A}^*\mathbf{y}+\alpha\mathbf{w},(\mathbf{v},\mathbf{y})_2\bigr)
\]
is a bijection: pairing the first component with $\mathbf{w}$, $\mathbf{w}$, $\mathbf{v}$ respectively forces $\alpha=0$, after which the border forces $\mathbf{y}=0$; and for surjectivity one takes $\alpha=(\mathbf{w},\mathbf{r})_2$, $(\mathbf{w},\mathbf{r})_2$, $(\mathbf{v},\mathbf{r})_2$ respectively, then fixes the free multiple of $\mathbf{v}$, resp.\ $\mathbf{w}$, by the border. These are the only solvability facts used.

\smallskip\noindent\textbf{Step 1 ($\Phi$ is injective on $\ker DF(\mathbf{x})$).} Set $\Phi(\hat{\mathbf{X}})\bydef(\hat\lambda_1,\hat\lambda_2,(\mathbf{w},\hat{\mathbf{u}})_2)$ and let $\hat{\mathbf{X}}\in\ker DF(\mathbf{x})$ with $\Phi(\hat{\mathbf{X}})=0$. Then $\hat{\blambda}=0$, so (5) gives $\hat{\mathbf{u}}\in\ker\mathcal{A}$ and $(\mathbf{w},\hat{\mathbf{u}})_2=0$ gives $\hat{\mathbf{u}}=0$. Now (6) reads $\mathcal{A}\hat{\mathbf{v}}+\hat s_1\mathbf{v}=0$, so pairing with $\mathbf{w}$ gives $\hat s_1=0$ and (1) gives $\hat{\mathbf{v}}=0$; (7) reads $\mathcal{A}^*\hat{\mathbf{w}}=0$, and (2) gives $\hat{\mathbf{w}}=0$; (8) reads $\mathcal{A}\hat{\mathbf{h}}+\hat s_2\mathbf{v}=0$, so pairing with $\mathbf{w}$ gives $\hat s_2=0$ and (3) gives $\hat{\mathbf{h}}=0$.

\smallskip\noindent\textbf{Step 2 (Reconstruction).} Fix $(\hat{\blambda},t)\in\R^3$. By Step~0 there are unique $(\hat{\mathbf{u}}_0,\alpha_1)$, $(\hat{\mathbf{v}},\hat s_1)$, $(\hat{\mathbf{w}},\alpha_2)$, $(\hat{\mathbf{h}},\hat s_2)$, depending linearly on $(\hat{\blambda},t)$, with
\[
\begin{array}{ll}
\mathcal{A}\hat{\mathbf{u}}_0+\alpha_1\mathbf{v} = -g_{\hat{\blambda}}, \ (\mathbf{w},\hat{\mathbf{u}}_0)_2=0,
&\quad
\mathcal{A}\hat{\mathbf{v}}+\hat s_1\mathbf{v} = -\mathcal{B}(\hat{\mathbf{u}},\mathbf{v})-G_{\hat{\blambda}}\mathbf{v}, \ (\hat{\mathbf{v}},\mathbf{v})_2=0,
\\[2pt]
\mathcal{A}^*\hat{\mathbf{w}}+\alpha_2\mathbf{w} = -\mathcal{B}^*(\hat{\mathbf{u}},\mathbf{w})-G^*_{\hat{\blambda}}\mathbf{w}, \ (\mathbf{v},\hat{\mathbf{w}})_2 = -(\hat{\mathbf{v}},\mathbf{w})_2,
&\quad
\text{(8) with } (\mathbf{w},\hat{\mathbf{h}})_2 = -(\hat{\mathbf{w}},\mathbf{h})_2,
\end{array}
\]
where $\hat{\mathbf{u}}\bydef\hat{\mathbf{u}}_0+t\mathbf{v}$, so $(\mathbf{w},\hat{\mathbf{u}})_2=t$. Rows (6) and (8) carry the free scalars $\hat s_1,\hat s_2$ and hold by construction; row (7) has none, and is bordered only so that the construction --- hence $\theta$ below --- is defined for every $\hat{\blambda}$. Pairing the first system with $\mathbf{w}$ gives $\alpha_1=-\sigma\cdot\hat{\blambda}$, so (5) holds iff $\sigma\cdot\hat{\blambda}=0$. Pairing the second with $\mathbf{w}$ and using $\mathcal{A}^*\mathbf{z}=\mathcal{B}^*(\mathbf{v},\mathbf{w})$ in the form $(\mathbf{w},\mathcal{B}(\hat{\mathbf{u}},\mathbf{v}))_2 = (\mathbf{z},\mathcal{A}\hat{\mathbf{u}})_2 = -(\mathbf{z},g_{\hat{\blambda}})_2-\alpha_1(\mathbf{z},\mathbf{v})_2$ gives $\hat s_1 = -\tilde\tau\cdot\hat{\blambda}$; pairing the third with $\mathbf{v}$ gives $\alpha_2 = -(\mathbf{w},\mathcal{B}(\hat{\mathbf{u}},\mathbf{v})+G_{\hat{\blambda}}\mathbf{v})_2 = \hat s_1$, so (7) holds iff $\tilde\tau\cdot\hat{\blambda}=0$. Conversely every $\hat{\mathbf{X}}\in\ker DF(\mathbf{x})$ arises this way with $t=(\mathbf{w},\hat{\mathbf{u}})_2$, by uniqueness in Step~0.

\smallskip\noindent\textbf{Step 3 (Row (4) and the coefficient $6c$).} Writing $\hat{\mathbf{u}}_0,\hat{\mathbf{v}}_0,\hat{\mathbf{w}}_0$ for the objects of Step~2 with $t=0$, set
\begin{equation}\label{eq : theta def}
  \theta\cdot\hat{\blambda} \bydef \bigl(\mathcal{C}(\hat{\mathbf{u}}_0,\mathbf{v},\mathbf{v}) + H_{\hat{\blambda}}(\mathbf{v},\mathbf{v}) + 2\mathcal{B}(\hat{\mathbf{v}}_0,\mathbf{v}),\mathbf{w}\bigr)_2 + \bigl(\mathcal{B}(\mathbf{v},\mathbf{v}),\hat{\mathbf{w}}_0\bigr)_2 ,
\end{equation}
a linear functional on $\R^2$ depending only on $\mathbf{x}$; we claim row (4) reads $\theta\cdot\hat{\blambda}+6ct=0$. Adding $t\mathbf{v}$ to $\hat{\mathbf{u}}$ leaves $\hat s_1$ and $\alpha_2$ unchanged (their variation is $-t(\mathbf{w},\mathcal{B}(\mathbf{v},\mathbf{v}))_2=-2tb=0$) and changes $\hat{\mathbf{v}}$ by $t\boldsymbol{\delta}$ and $\hat{\mathbf{w}}$ by $t\boldsymbol{\zeta}$, where $\mathcal{A}\boldsymbol{\delta}=-\mathcal{B}(\mathbf{v},\mathbf{v})=\mathcal{A}\mathbf{h}$ with $(\boldsymbol{\delta},\mathbf{v})_2=0$, i.e.\ $\boldsymbol{\delta}=\mathbf{h}-\gamma\mathbf{v}$, and $\mathcal{A}^*\boldsymbol{\zeta}=-\mathcal{A}^*\mathbf{z}$, i.e.\ $\boldsymbol{\zeta}=-\mathbf{z}+\varsigma\mathbf{w}$ with $\varsigma=\gamma+\beta$ forced by (2), since $(\boldsymbol{\delta},\mathbf{w})_2=-\gamma$. Row (4) therefore varies by
\[
t\Bigl[\bigl(\mathcal{C}(\mathbf{v},\mathbf{v},\mathbf{v}),\mathbf{w}\bigr)_2 + 2\bigl(\mathcal{B}(\boldsymbol{\delta},\mathbf{v}),\mathbf{w}\bigr)_2 + \bigl(\mathcal{B}(\mathbf{v},\mathbf{v}),\boldsymbol{\zeta}\bigr)_2\Bigr] = 6ct ,
\]
since $b=0$ gives $2(\mathcal{B}(\boldsymbol{\delta},\mathbf{v}),\mathbf{w})_2 = 2(\mathcal{B}(\mathbf{v},\mathbf{h}),\mathbf{w})_2$ and $(\mathcal{B}(\mathbf{v},\mathbf{v}),\boldsymbol{\zeta})_2 = -(\mathcal{B}(\mathbf{v},\mathbf{v}),\mathbf{z})_2 = (\mathcal{B}(\mathbf{v},\mathbf{h}),\mathbf{w})_2$, the last equality from $(\mathcal{B}(\mathbf{v},\mathbf{v}),\mathbf{z})_2 = -(\mathcal{A}\mathbf{h},\mathbf{z})_2 = -(\mathbf{h},\mathcal{A}^*\mathbf{z})_2$.

\smallskip\noindent\textbf{Step 4 (Conclusion).} By Steps 1--3, $\Phi$ maps $\ker DF(\mathbf{x})$ isomorphically onto $\{(\hat\lambda_1,\hat\lambda_2,t) : \sigma\cdot\hat{\blambda}=0,\ \tilde\tau\cdot\hat{\blambda}=0,\ \theta\cdot\hat{\blambda}+6ct=0\}=\ker\mathbb{T}$. Expanding $\det\mathbb{T}$ along the third column and using $\tilde\tau=\tau+\beta\sigma$ gives $\det\mathbb{T}=6c\Delta$, which is therefore non-zero: $c\neq0$ and $\Delta\neq0$, and Lemma~\ref{lemma : Cusp bifurcation} applies.
\end{proof}

\begin{remark}\label{rmk : old kernel element}
The classical argument of \cite{cusp_jp} is the case $\Phi^{-1}(0,0,-1)$: if $c=0$, Step~0 provides $\mathbf{u}_2$ with $D_uf^*\mathbf{u}_2 = D_{uu}f^*(\mathbf{w},\mathbf{v})$, $(\mathbf{v},\mathbf{u}_2)_2=-\gamma$, and $\mathbf{u}_3$ with $D_uf\,\mathbf{u}_3 = D_{uuu}f(\mathbf{v},\mathbf{v},\mathbf{v})+3D_{uu}f(\mathbf{v},\mathbf{h})-2\gamma D_{uu}f(\mathbf{v},\mathbf{v})$, $(\mathbf{w},\mathbf{u}_3)_2=-(\mathbf{h},\mathbf{u}_2)_2$ (solvable by $b=0$, resp.\ $c=0$); then $\hat{\mathbf{X}} = (\mathbf{0},\mathbf{0},-\mathbf{v},-\mathbf{h}+\gamma\mathbf{v},\mathbf{u}_2,\mathbf{u}_3)$ is a non-zero element of $\ker DF(\mathbf{x})$, whence $c\neq0$. This is Step~2 with $\hat{\blambda}=0$, $t=-1$, for which $\boldsymbol{\delta}=\mathbf{h}-\gamma\mathbf{v}$ and $\boldsymbol{\zeta}=-\mathbf{u}_2$, and Step~3 reduces row (4) to $-6c=0$. Producing only kernel elements with $\hat{\blambda}=0$, it says nothing about $\Delta$.
\end{remark}

\renewcommand{\theequation}{B.\arabic{equation}}
\setcounter{equation}{0}
\section{Computational details for Swift-Hohenberg}
For the remaining appendices, we use two convenient notations for the discrete convolution. These are
\begin{align}
    uv = \mathrm{conv}(u,v) = u \star v,
\end{align}
for $u,v \in \ell^p_{\mathcal{G},\omega}$. In addition, we introduce the sequence $e_k$
\begin{align}
    e_k \bydef ((e_k)_n)_{n \in \mathcal{Z}_{\mathrm{red}}(\mathcal{G})} = \begin{cases}
    1 & n = k \\
    0 & \mathrm{else}
\end{cases}.\label{def : ek}
\end{align}
In this appendix, we provide the computational details for the bounds required for the Swift-Hohenberg PDE. That is, we prove Lemmas \ref{lem : Z2} and \ref{lem : Z1 full}. We begin with Lemma \ref{lem : Z2}.
\subsection{Proof of Lemma \ref{lem : Z2}}\label{apen : Z2}
We now define
{\footnotesize\begin{align}
    L_{\mathbf{s},\boldsymbol{\lambda}} \bydef \begin{bmatrix}
        l_{\boldsymbol{\lambda}} & 0 & 0 & 0 \\
        0 & l_{\boldsymbol{\lambda}} + s_1 I_d & 0 & 0 \\
        0 & 0 & l_{\boldsymbol{\lambda}} & 0 \\
        0 & s_2 I_d & 0 & l_{\boldsymbol{\lambda}}
    \end{bmatrix}, ~ G(\mathbf{x}) \bydef \begin{bmatrix}
        \lambda_2 u^2 - u^3 \\
        (2\lambda_2 u - 3u^2)v \\
        \mathcal{O}_{\mathcal{G}} ((2\lambda_2 u - 3u^2)\star (\mathcal{O}_{\mathcal{G}}^{-1}w)) \\
        (2\lambda_2 u - 3u^2)h + (2\lambda_2 - 6u)v^2
    \end{bmatrix}, ~ \varphi(\mathbf{x}) \bydef \begin{bmatrix}
        (v,v)_{2} - 1 \\
        (v,w)_{2} - 1 \\
        (h,w)_{2} \\
        (2\lambda_2v^2 - 6uv^2,w)_{2}
    \end{bmatrix}.\label{def L G and Phi}
\end{align}}
We also write $F$ in the form
\begin{align}
    F(\mathbf{x}) = \begin{bmatrix}
        \varphi(\mathbf{x}) \\
        L_{\mathbf{s},\boldsymbol{\lambda}} (u,v,w,h) + G(\mathbf{x})
    \end{bmatrix}.
\end{align}
We will now prove Lemma \ref{lem : Z2}.
\begin{proof}
First, we simplify the computation by 
\begin{align}
    \|A(DF(\mathbf{x}) - DF(\overline{\mathbf{x}}))\|_{\mathcal{B}(X_{\mathcal{G},\omega})} &\leq \|A\|_{\mathcal{B}(X_{\mathcal{G},\omega})} \|DF(\mathbf{x}) - DF(\overline{\mathbf{x}})\|_{\mathcal{B}(X_{\mathcal{G},\omega})} \\
    &\leq (\|A^N\|_{\mathcal{B}(X_{\mathcal{G},\omega})} + \mathcal{L}_{\infty})\|DF(\mathbf{x}) - DF(\overline{\mathbf{x}})\|_{\mathcal{B}(X_{\mathcal{G},\omega})}.
\end{align}
First, for simplicity, let $\mathcal{F}(\mathbf{x}) \bydef L_{\mathbf{s},\blambda}(u,v,w,h) + G(\mathbf{x})$
Now, observe that
{\small\begin{align}
    &\|DF(\mathbf{x}) - DF(\overline{\mathbf{x}})\|_{\mathcal{B}(X_{\mathcal{G},\omega})} = \left\| \begin{bmatrix}
        \partial_{(\mathbf{s},\blambda)} \varphi(\mathbf{x}) - \partial_{(\mathbf{s},\blambda)}\varphi(\overline{\mathbf{x}}) & D\varphi(\mathbf{x}) - D\varphi(\overline{\mathbf{x}}) \\
        \partial_{(\mathbf{s},\blambda)}\mathcal{F}(\mathbf{x}) - \partial_{(\mathbf{s},\blambda)}\mathcal{F}(\overline{\mathbf{x}}) & L_{\mathbf{s},\blambda} + DG(\mathbf{x}) - L_{\mathbf{0},\overline{\blambda}} - DG(\overline{\mathbf{x}})
    \end{bmatrix}\right\|_{\mathcal{B}(X_{\mathcal{G},\omega})} \\
    &\leq \|\partial_{(\mathbf{s},\blambda)}\varphi(\mathbf{x}) - \partial_{(\mathbf{s},\blambda)}\varphi(\overline{\mathbf{x}})\|_{\mathcal{B}(\mathbb{R}^4)} + \|D\varphi(\mathbf{x}) - D\varphi(\overline{\mathbf{x}})\|_{\mathcal{B}((\ell^2_{\mathcal{G},\omega})^4,\mathbb{R}^4)} \\
    &\hspace{+0.1cm}+ \|\partial_{(\mathbf{s},\blambda)}\mathcal{F}(\mathbf{x}) - \partial_{(\mathbf{s},\blambda)}\mathcal{F}(\overline{\mathbf{x}})\|_{\mathcal{B}(\mathbb{R}^4,(\ell^2_{\mathcal{G},\omega})^4)} + \|L_{\mathbf{s},\blambda} + DG(\mathbf{x}) - L_{\mathbf{0},\overline{\blambda}} - DG(\overline{\mathbf{x}})\|_{\mathcal{B}((\ell^2_{\mathcal{G},\omega})^4)} \\
    &\leq|(2v^2,w)_{2} - (2\overline{v}^2,\overline{w})_{2}| + \|D\varphi(\mathbf{x}) - D\varphi(\overline{\mathbf{x}})\|_{\mathcal{B}((\ell^2_{\mathcal{G},\omega})^4,\mathbb{R}^4)} \\
    &\hspace{+0.1cm}+ \|\partial_{(\mathbf{s},\blambda)}\mathcal{F}(\mathbf{x}) - \partial_{(\mathbf{s},\blambda)}\mathcal{F}(\overline{\mathbf{x}})\|_{\mathcal{B}(\mathbb{R}^4,(\ell^2_{\mathcal{G},\omega})^4)} + \|L_{\mathbf{s},\blambda} - L_{\mathbf{0},\overline{\blambda}}+DG(\mathbf{x}) - DG(0,\overline{\lambda_2},u,v,w,h)\|_{\mathcal{B}((\ell^2_{\mathcal{G},\omega})^4)}\\
    &\hspace{+0.3cm}+ \|DG(0,\overline{\lambda_2},u,v,w,h)- DG(\overline{\mathbf{x}})\|_{\mathcal{B}((\ell^2_{\mathcal{G},\omega})^4)}.\label{Z2 split up}
\end{align}}
Let us now estimate each quantity of \eqref{Z2 split up}. Firstly, we write $\mathbf{x} = \mathbf{t} + \overline{\mathbf{x}}$ for some $\mathbf{t} = (t_1,t_2,\dots,t_8) \in B_r(0)$. Therefore, 
\begin{align}
    |(2v^2,w)_{2} - (2\overline{v}^2,\overline{w})_{2}| &= |(2t_6^2 + 4t_6 \overline{v} + 2\overline{v}^2,t_7 + \overline{w})_{2} - (2\overline{v}^2,\overline{w})_{2}| \\
    &= |(2t_6^2 + 4t_6 \overline{v} + 2\overline{v}^2,t_7)_{2} + (2t_6^2 + 4t_6 \overline{v} + 2\overline{v}^2,\overline{w})_{2}- (2\overline{v}^2,\overline{w})_{2}| \\
    &= |(2t_6^2 + 4t_6 \overline{v} + 2\overline{v}^2,t_7)_{2} + (2t_6^2 + 4t_6 \overline{v},\overline{w})_{2}|
\end{align}
where we used bilinearity of the inner product. Now, we use the following inequalities
\begin{align}\label{ineq list}
    &\|a+b\|_{2,\omega} \leq \|a\|_{2,\omega} + \|b\|_{2,\omega}~ \text{Triangle inequality} \\
    &|(a,b)_{2}| \leq \|a\|_{2} \|b\|_{2} ~ \text{Cauchy-Schwarz} \\
    &\|ab\|_{2} \leq \|a\|_{2} \|b\|_{1} ~ \text{Young's Inequality for Convolution} \\
    &\|a\|_{2} \leq \|a\|_{2,\omega} ~ \text{follows from} ~ \ell^2_{\omega} \subset \ell^2\\
    &\|ab\|_{2} \leq \|a\|_{2,\omega} \|b\|_{1} ~\text{Combining Young's Inequality with} ~\ell^2_{\omega} \subset \ell^2\\
    &\|ab\|_{2,\omega} \leq \kappa \|a\|_{2,\omega} \|b\|_{2,\omega} ~ \text{from Lemma \ref{lem : conv_estimates}}.
\end{align}
These six inequalities provide us with
\begin{align}
    &|(2t_6^2 + 4t_6 \overline{v} + 2\overline{v}^2,t_7)_{2} + (2t_6^2 + 4t_6 \overline{v},\overline{w})_{2}| \leq \|2t_6^2 + 4t_6 \overline{v} + 2\overline{v}^2\|_{2} \|t_7\|_{2} + \|2t_6^2 + 4t_6 \overline{v}\|_{2} \|\overline{w}\|_{2} \\
    &\hspace{+0.2cm}\leq (2\kappa\|t_6\|_{2,\omega}^2 + 4 \|t_6\|_{2,\omega} \|\overline{v}\|_{1} + 2\|\overline{v}^2\|_{2}) \|t_7\|_{2,\omega} + 2\kappa\|t_6\|_{2,\omega}^2\|\overline{w}\|_{2} + 4\|t_6\|_{2,\omega} \|\overline{v}\|_{1} \|\overline{w}\|_{2} \\
    &\hspace{+0.2cm}\leq (2\kappa r^2 + 4\|\overline{v}\|_{1}r + 2\|\overline{v}^2\|_{2}) r + 2\kappa r^2\|\overline{w}\|_{2} + 4\|\overline{v}\|_{1} \|\overline{w}\|_{2} r \\
    &\hspace{+0.2cm}= (2\kappa r^2 + 4\|\overline{v}\|_{1}r + 2\|\overline{v}^2\|_{2} + (2\kappa r + 4\|\overline{v}\|_{1}) \|\overline{w}\|_{2})r\\
    &\hspace{+2cm}\bydef Z_{2,1}(r) r.\label{def : Z21}
\end{align}
We now move to the second term of \eqref{Z2 split up}. Observe that
\begin{align}
    D\varphi(\mathbf{t} + \overline{\mathbf{x}}) - D\varphi(\overline{\mathbf{x}}) &= \begin{bmatrix}
        0 & 2t_6^* & 0 & 0 \\
        0 & t_7^* & t_6^* & 0 \\
        0 & 0 & t_8^* & t_7^* \\
        \mathscr{Z}_1^* & \mathscr{Z}_2^* & \mathscr{Z}_3^*  & 0
    \end{bmatrix},
\end{align}
where
\begin{align}
    &(\mathscr{Z}_1)_k \bydef -6(2\mathrm{conv}(e_k, \overline{v}t_6) + \mathrm{conv}(e_k,t_6^2),\overline{w})_{2} - 6(\mathrm{conv}(e_k,(\overline{v} + t_6)^2),t_7)_{2} \end{align}
    \vspace{-0.3cm}{\tiny\begin{align}
    &(\mathscr{Z}_2)_k \bydef 4(\mathrm{conv}(e_k,(\overline{\lambda}_2 t_6 - 3\overline{u}t_6 + (t_4 - 3t_5)\overline{v} + t_4t_6 - 3t_5 t_6)),\overline{w})_{2} + 4(\mathrm{conv}(e_k,((\overline{\lambda}_2+t_4)(\overline{v}+t_6) - 3(\overline{u}+t_5)(\overline{v}+t_6))),t_7)_{2} \end{align}}
    \vspace{-0.5cm}\begin{align}&\mathscr{Z}_3 \bydef 4\overline{\lambda}_2 \overline{v} t_6 - 12\overline{u}\overline{v} t_6 + (2t_4 - 6t_5) \overline{v}^2 + 2\overline{\lambda}_2 t_6^2 - 6\overline{u} t_6^2 + (4t_4 t_6 - 12t_5 t_6)\overline{v} + 2t_4t_6^2 - 6t_5 t_6^2.
\end{align}
Let $\mathbf{p} = (p_1,p_2,p_3,p_4) \in (\ell^2_{\mathcal{G},\omega})^4, \|\mathbf{p}\|_{2,\omega} = 1$. We now rely on the inequalities from \eqref{ineq list} and two additional inequalities 
\begin{align}\label{ineq list 2}
    &|(a,b)_{2}| \leq \|a\|_{\infty} \|b\|_{1} ~ \text{Holder's Inequality} \\
    &\|a\|_{1} \leq \kappa_0 \|a\|_{2,\omega} , ~ \text{from Lemma \ref{lem : 1 bounded by 2 nu}} 
\end{align}
Then,

\begin{align}
    &\|D\varphi(\mathbf{x}) - D\varphi(\overline{\mathbf{x}})\|_{\mathcal{B}((\ell^2_{\mathcal{G},\omega})^4,\mathbb{R}^4)} \\
    &= |2(p_2,t_6)_{2}| +
        |(p_2,t_7)_{2} + (p_3,t_6)_{2}| +
        |(p_3,t_8)_{2} + (p_4,t_7)_{2}| +
        |(p_1,\mathscr{Z}_1)_{2} + (p_2,\mathscr{Z}_2)_{2} + (p_3,\mathscr{Z}_3)_{2}| \\
        &\leq 6r + \|\mathscr{Z}_1\|_{\infty}\|p_1\|_{1} + \|\mathscr{Z}_2\|_{\infty}\|p_2\|_{1} + \|\mathscr{Z}_3\|_{2}\|p_3\|_{2}
        \\
        &\leq 6r + \kappa_0 \|\mathscr{Z}_1\|_{\infty}\|p_1\|_{2,\omega} + \kappa_0 \|\mathscr{Z}_2\|_{\infty}\|p_2\|_{2,\omega} + \|\mathscr{Z}_3\|_{2}\|p_3\|_{2,\omega} \\
        &\leq 6r + \kappa_0 \|\mathscr{Z}_1\|_{\infty} + \kappa_0 \|\mathscr{Z}_2\|_{\infty} + \|\mathscr{Z}_3\|_{2}.  
\end{align}
 For $\mathscr{Z}_1$ and $\mathscr{Z}_2$, we bound each coefficient as
\begin{align}
    |(\mathscr{Z}_1)_k| &\leq 6\|2\mathrm{conv}(e_k,\overline{v}t_6) + \mathrm{conv}(e_k,t_6^2)\|_{2} \|\overline{w}\|_{2} + 6\|\mathrm{conv}(e_k,\overline{v}^2 +2\overline{v}t_6+ t_6^2)\|_{2} \|t_7\|_{2} \\
    &\leq 6(2\|\mathrm{conv}(e_k,\overline{v})\|_{1} \|t_6\|_{2,\omega} + \|e_k\|_{1} \kappa \|t_6\|_{2,\omega}^2)\|\overline{w}\|_{2} + 6(\|\mathrm{conv}(e_k,\overline{v}^2)\|_{2} \\
    &\hspace{+1cm}+2\|\mathrm{conv}(e_k,\overline{v})\|_{1} \|t_6\|_{2,\omega} + \|e_k\|_{1}\kappa \|t_6\|_{2,\omega}^2)\|t_7\|_{2,\omega} \\
    &\leq 6(2\|\mathrm{conv}(e_k, \overline{v})\|_{1} r + \kappa r^2)\|\overline{w}\|_{2} + 6(\|\mathrm{conv}(e_k,\overline{v}^2)\|_{2} + 2\|\mathrm{conv}(e_k,\overline{v})\|_{1} r + \kappa r^2)r \\
    &=[6(2\|\mathrm{conv}(e_k, \overline{v})\|_{1} + \kappa r)\|\overline{w}\|_{2} + 6(\|\mathrm{conv}(e_k,\overline{v}^2)\|_{2} + 2\|\mathrm{conv}(e_k,\overline{v})\|_{1} r + \kappa r^2)]r \\
    &=[6(2\|\overline{v}\|_{1} + \kappa r)\|\overline{w}\|_{2} + 6( \|\overline{v}^2\|_{2} + 2\|\overline{v}\|_{1} r + \kappa r^2)]r.
\end{align}
where we used that 
\begin{align}
    &\|\mathrm{conv}(e_k,a)\|_{1} \leq \|a\|_{1} \\
    &\|\mathrm{conv}(e_k,a)\|_{2} \leq \|e_k\|_{1} \|a\|_{2} \leq \|a\|_{2}
\end{align}
for some sequence $a$. 
Since we bounded each $(\mathscr{Z}_1)_k$ with the same bound, it follows that
\begin{align}
    \|\mathscr{Z}_1\|_{\infty} \leq [6(2\|\overline{v}\|_{1} + \kappa r)\|\overline{w}\|_{2} + 6(\|\overline{v}^2\|_{2} + 2\|\overline{v}\|_{1} r + \kappa r^2)]r.
\end{align}
We use a similar process for $\mathscr{Z}_2$.
{\footnotesize\begin{align}
    |(\mathscr{Z}_2)_k| &\leq 4\|\mathrm{conv}(e_k,\overline{\lambda}_2 t_6 - 3\overline{u}t_6 + (t_4 - 3t_5)\overline{v} + t_4t_6 - 3t_5 t_6)\|_{2} \|\overline{w}\|_{2} \\
    &\hspace{+0.5cm}+ 4\|\mathrm{conv}(e_k,\overline{\lambda}_2\overline{v} -3\overline{u}\overline{v} + \overline{\lambda}_2 t_6 - 3\overline{u}t_6 + (t_4 - 3t_5)\overline{v} + t_4t_6 - 3t_5 t_6) \|_{2} \|t_7\|_{2} \\
    &\leq 4\biggl( (\|\mathrm{conv}(e_k,\overline{\lambda}_2 - 3\overline{u})\|_{1} \|t_6\|_{2,\omega} + \|\mathrm{conv}(e_k,\overline{v})\|_{1}\|t_4-3t_5\|_{2,\omega} + \|e_k\|_{1} \|t_4 t_6 - 3t_5 t_6\|_{2,\omega})\|\overline{w}\|_{2} \\
    &\hspace{+0.5cm}+ (\|\mathrm{conv}(e_k,\overline{\lambda}_2\overline{v}-3\overline{u}\overline{v})\|_{2} + \|\mathrm{conv}(e_k, \overline{\lambda}_2 - 3\overline{u})\|_{1} \|t_6\|_{2,\omega} \\
    &\hspace{+1cm}+ \|\mathrm{conv}(e_k,\overline{v})\|_{1}\|t_4-3t_5\|_{2,\omega} + \|e_k\|_{1} \|t_4 t_6 - 3t_5 t_6\|_{2,\omega})\|t_7\|_{2,\omega}\biggr) \\
    &\leq 4\biggl( (\|\mathrm{conv}(e_k, \overline{\lambda}_2 - 3\overline{u})\|_{1} r + \|\mathrm{conv}(e_k,\overline{v})\|_{1}(4r) + (r^2 + 3\kappa r^2)\|\overline{w}\|_{2} \\
    &\hspace{+0.5cm}+ (\|\mathrm{conv}(e_k,\overline{\lambda}_2\overline{v}-3\overline{u}\overline{v})\|_{2} + \|\mathrm{conv}(e_k, \overline{\lambda}_2 - 3\overline{u})\|_{1} r + \|\mathrm{conv}(e_k,\overline{v})\|_{1}(4r) + (1 + 3\kappa)r^2)r\biggr) \\
    &=4\biggl( (\|\mathrm{conv}(e_k ,\overline{\lambda}_2 - 3\overline{u})\|_{1}  + 4\|\mathrm{conv}(e_k,\overline{v})\|_{1} + 4\kappa r)\|\overline{w}\|_{2} \\
    &\hspace{+0.5cm}+ (\|\mathrm{conv}(e_k,\overline{\lambda}_2\overline{v}-3\overline{u}\overline{v})\|_{2} + \|\mathrm{conv}(e_k, \overline{\lambda}_2 - 3\overline{u})\|_{1} r + 4\|\mathrm{conv}(e_k,\overline{v})\|_{1}r + (1 + 3\kappa) r^2)\biggr)r \\&=4\biggl( (\| \overline{\lambda}_2 - 3\overline{u}\|_{1}  + 4\|\overline{v}\|_{1} + 4\kappa r)\|\overline{w}\|_{2} + (\|\overline{\lambda}_2\overline{v}-3\overline{u}\overline{v}\|_{2} + \|\overline{\lambda}_2 - 3\overline{u}\|_{1} r + 4\|\overline{v}\|_{1}r + (1 + 3\kappa) r^2)\biggr)r.
\end{align}}
Hence,
{\small\begin{align}
    \|\mathscr{Z}_2\|_{\infty} \leq 4\biggl( (\| \overline{\lambda}_2 - 3\overline{u}\|_{1}  + 4\|\overline{v}\|_{1} + 4\kappa r)\|\overline{w}\|_{2} + (\|\overline{\lambda}_2\overline{v}-3\overline{u}\overline{v}\|_{2} + \|\overline{\lambda}_2 - 3\overline{u}\|_{1} r + 4\|\overline{v}\|_{1}r + (1 + 3\kappa) r^2)\biggr)r.
\end{align}}
For $\mathscr{Z}_3$, we return to a more standard estimate.
{\small\begin{align}
    \|\mathscr{Z}_3\|_{2} &\leq 4|\overline{\lambda}_2| \|\overline{v}\|_{1} \|t_6\|_{2,\omega} + 12\|\overline{u}\overline{v}\|_{1} \|t_6\|_{2,\omega} + \|(2t_4 - 6t_5)\|_{2,\omega} \|\overline{v}^2\|_{1} + 2\kappa |\overline{\lambda}_2| \|t_6\|_{2,\omega}^2 + 6\kappa \|\overline{u}\|_{1} \|t_6\|_{2,\omega}^2 \\
    &\hspace{+0.5cm}+ \|(4t_4 t_6 - 12t_5 t_6)\|_{2,\omega}\|\overline{v}\|_{1} + \|2t_4t_6^2 - 6t_5 t_6^2\|_{2,\omega} \\
    &\leq 4|\overline{\lambda}_2| \|\overline{v}\|_{1} r + 12\|\overline{u}\overline{v}\|_{1} r + 8r \|\overline{v}^2\|_{1} + 2\kappa |\overline{\lambda}_2| r^2 + 6\kappa \|\overline{u}\|_{1} r^2+ 16\kappa r^2\|\overline{v}\|_{1} + 8\kappa^2 r^3 \\
    &\leq 2(3\kappa \|\overline{u}\|_{1}r + (2|\overline{\lambda}_2| + 8\kappa r)\|\overline{v}\|_{1} + 6\|\overline{u}\overline{v}\|_{1} + 8\|\overline{v}^2\|_{1} + \kappa |\overline{\lambda}_2|r + 4\kappa^2 r^2)r.
\end{align}}
Hence, we obtain
\begin{align}
    &\|D\varphi(\mathbf{x}) - D\varphi(\overline{\mathbf{x}})\|_{\mathcal{B}((\ell^2_{\mathcal{G},\omega})^4,\mathbb{R}^4)} \\
    &\leq (6 + \kappa_0 (28\|\overline{v}\|_{1} + 22\kappa r + 4\|\overline{\lambda}_2 - 3\overline{u}\|_{1})\|\overline{w}\|_{2} + 4\kappa_0 \|\overline{\lambda}_2 \overline{v} - 3 \overline{u}\overline{v}\|_{2} + 4\kappa_0 \|\overline{\lambda}_2 - 3\overline{u}\|_{1}r \\
    &\hspace{+1cm}+ 2(3\kappa \|\overline{u}\|_{1}r + (2|\overline{\lambda}_2| + 8\kappa r + 14\kappa_0 r)\|\overline{v}\|_{1} + 6\|\overline{u}\overline{v}\|_{1} + 8\|\overline{v}^2\|_{1} + 6\kappa_0\|\overline{v}^2\|_{2} + |\overline{\lambda}_2|\kappa r \\
    &\hspace{+1.1cm}+ (4\kappa^2 + \kappa_0(10 + 12\kappa)) r^2))r \\
    &\bydef Z_{2,2}(r)r.
\end{align}
We now move to the third term of \eqref{Z2 split up}. Observe that
{\small\begin{align}
   &\|\partial_{(\mathbf{s},\blambda)}\mathcal{F}(\mathbf{x}) - \partial_{(\mathbf{s},\blambda)}\mathcal{F}(\overline{\mathbf{x}})\|_{\mathcal{B}(\mathbb{R}^4,(\ell^2_{\mathcal{G},\omega})^4)} \\
   &= \left\| \begin{bmatrix}
        0 & 0 & -t_5 & t_5^2 + 2t_5 \overline{u} \\
        t_6 & 0 & -t_6 & 2t_5 t_6 + 2t_5 \overline{v} + 2t_6 \overline{u} \\
        0 & 0 & -t_7 & 2\mathcal{O}_{\mathcal{G}} \mathrm{conv}(t_5, \mathcal{O}_{\mathcal{G}}^{-1} t_7) + 2\mathcal{O}_{\mathcal{G}}\mathrm{conv}(\overline{u}, \mathcal{O}_{\mathcal{G}}^{-1}t_7) + 2\mathcal{O}_{\mathcal{G}} \mathrm{conv}(t_5,\mathcal{O}_{\mathcal{G},N}^{-1} \overline{w}) \\
        0 & t_6 & -t_8 & 2t_5 t_8 + 2t_5 \overline{h} + 2t_8 \overline{u} + 2t_6^2 + 4t_6 \overline{v}
    \end{bmatrix} \right\|_{\mathcal{B}(\mathbb{R}^4,(\ell^2_{\mathcal{G},\omega})^4)} \\
    &\leq 6r + \|t_5^2 + 2t_5 \overline{u}\|_{2,\omega} + \|2t_5 t_6 + 2t_5 \overline{v} + 2t_6 \overline{u}\|_{2,\omega}+ \|2t_5 t_8 + 2t_5 \overline{h} + 2t_8 \overline{u} + 2t_6^2 + 4t_6 \overline{v}\|_{2,\omega} \\
    &\hspace{+0.5cm}+ \|2\mathcal{O}_{\mathcal{G}} \mathrm{conv}(t_5,\mathcal{O}_{\mathcal{G}}^{-1} t_7) + 2\mathcal{O}_{\mathcal{G}}\mathrm{conv}(\overline{u},\mathcal{O}_{\mathcal{G}}^{-1}t_7) + 2\mathcal{O}_{\mathcal{G}} \mathrm{conv}(t_5,\mathcal{O}_{\mathcal{G},N}^{-1} \overline{w})\|_{2,\omega}.\label{before using banach algebra Z23}
\end{align}}
Now we use Lemmas \ref{lem : conv_estimates} and \ref{lem : gen young} to obtain
\begin{align}
    &\|t_5^2 + 2t_5 \overline{u}\|_{2,\omega} \leq \kappa r^2 + 2  \|\overline{u}\|_{1,\sqrt{\omega}}r \\
    &\|2t_5 t_6 + 2t_5 \overline{v} + 2t_6 \overline{u}\|_{2,\omega} \leq 2\kappa r^2 + 2(\|\overline{v}\|_{1,\sqrt{\omega}} + \|\overline{u}\|_{1,\sqrt{\omega}})r \\
    &\|2t_5 t_8 + 2t_5 \overline{h} + 2t_8 \overline{u} + 2t_6^2 + 4t_6 \overline{v}\|_{2,\omega} \leq 4\kappa r^2 + 2(\|\overline{h}\|_{1,\sqrt{\omega}} + \|\overline{u}\|_{1,\sqrt{\omega}} + 2\|\overline{v}\|_{1,\sqrt{\omega}})r,
\end{align}
\begin{align}
    &\hspace{-1cm}\|2\mathcal{O}_{\mathcal{G}} \mathrm{conv}(t_5, \mathcal{O}_{\mathcal{G}}^{-1} t_7) + 2\mathcal{O}_{\mathcal{G}}\mathrm{conv}(\overline{u},\mathcal{O}_{\mathcal{G}}^{-1}t_7) + 2\mathcal{O}_{\mathcal{G}} \mathrm{conv}(t_5,\mathcal{O}_{\mathcal{G},N}^{-1} \overline{w})\|_{2,\omega} \\
    &\leq 2\|\mathcal{O}_{\mathcal{G}}\|_{\mathcal{B}(\ell^2_{\mathcal{G},\omega})} (\|\mathrm{conv}(t_5,\mathcal{O}_{\mathcal{G}}^{-1} t_7)\|_{2,\omega} + \|\mathrm{conv}(\overline{u},\mathcal{O}_{\mathcal{G}}^{-1} t_7)\|_{2,\omega} + \|\mathrm{conv}(t_5, \mathcal{O}_{\mathcal{G},N}^{-1}\overline{w})\|_{2,\omega}) \\
    &\leq 2\mathscr{O}_{\mathrm{max}}(\kappa \|t_5\|_{2,\omega} \|\mathcal{O}_{\mathcal{G}}^{-1} t_7\|_{2,\omega} +  \|\overline{u}\|_{1,\sqrt{\omega}} \|\mathcal{O}_{\mathcal{G}}^{-1} t_7\|_{2,\omega} +  \|t_5\|_{2,\omega} \|\mathcal{O}_{\mathcal{G},N}^{-1} \overline{w}\|_{1,\sqrt{\omega}}) \\
    &\leq 2\mathscr{O}_{\mathrm{max}}(\kappa \|\mathcal{O}_{\mathcal{G}}^{-1}\|_{\mathcal{B}(\ell^2_{\mathcal{G},\omega})} \|t_5\|_{2,\omega} \|t_7\|_{2,\omega} +  \|\mathcal{O}_{\mathcal{G}}^{-1}\|_{\mathcal{B}(\ell^2_{\mathcal{G},\omega})}\|\overline{u}\|_{1,\sqrt{\omega}} \|t_7\|_{2,\omega} +  \|\mathcal{O}_{\mathcal{G},N}^{-1} \overline{w}\|_{1,\sqrt{\omega}}\|t_5\|_{2,\omega}) \\
    &\leq 2\mathscr{O}_{\mathrm{max}}\left(\kappa r^2 +  (\|\overline{u}\|_{1,\sqrt{\omega}}  +  \|\mathcal{O}_{\mathcal{G},N}^{-1} \overline{w}\|_{1,\sqrt{\omega}})r\right).
\end{align}
Returning to \eqref{before using banach algebra Z23}, we obtain
\begin{align}
    &\|\partial_{(\mathbf{s},\blambda)}\mathcal{F}(\mathbf{x}) - \partial_{(\mathbf{s},\blambda)}\mathcal{F}(\overline{\mathbf{x}})\|_{\mathcal{B}(\mathbb{R}^4,(\ell^2_{\mathcal{G},\omega})^4)} \\
    &\leq (7+2\mathscr{O}_{\mathrm{max}})\kappa r^2 + 6 +  (8+2\mathscr{O}_{\mathrm{max}})\|\overline{u}\|_{1,\sqrt{\omega}} + 6 \|\overline{v}\|_{1,\sqrt{\omega}} + 2 \mathscr{O}_{\mathrm{max}}\|\mathcal{O}_{\mathcal{G},N}^{-1} \overline{w}\|_{1,\sqrt{\omega}} + 2 \|\overline{h}\|_{1,\sqrt{\omega}})r \\
    &= ((7+2\mathscr{O}_{\mathrm{max}})\kappa r + 6 +  (6+2\mathscr{O}_{\mathrm{max}})\|\overline{u}\|_{1,\sqrt{\omega}} + 6\|\overline{v}\|_{1,\sqrt{\omega}} + 2 \mathscr{O}_{\mathrm{max}}\|\mathcal{O}_{\mathcal{G},N}^{-1} \overline{w}\|_{1,\sqrt{\omega}} + 2 \|\overline{h}\|_{1,\sqrt{\omega}})r
    \\
    &\bydef Z_{2,3}(r) r.\label{def : Z23}
\end{align}
We now move to the fourth term of \eqref{Z2 split up}. Let $\mathbf{p}$ be defined as earlier. Observe that
\begin{align}
    &\|L_{\mathbf{s},\blambda} - L_{\mathbf{0},\overline{\blambda}} + DG(\mathbf{x}) - DG(0,\overline{\lambda_2},u,v,w,h)\|_{\mathcal{B}((\ell^2_{\mathcal{G},\omega})^4)} \\
    &= \left\| \begin{bmatrix}
        -t_3 & 2t_4 \mathbb{M}_{u} & 0 & 0 \\
        2t_4 \mathbb{M}_v & -t_3 + 2t_4 \mathbb{M}_u + t_1 & 0 & 0 \\
        \mathbb{Q}_{\mathcal{O}_{\mathcal{G}}\mathrm{conv}(2t_4e_0,\mathcal{O}_{\mathcal{G}}^{-1}w)} & 0 & -t_3 + 2t_4\mathbb{M}_u^* & 0 \\
        2t_4 \mathbb{M}_h & t_2 + 4t_4 \mathbb{M}_v & 0 & -t_3 + 2t_4 \mathbb{M}_u
    \end{bmatrix}\mathbf{p}\right\|_{2,\omega} \\
    &= \left\| \begin{bmatrix}
        -t_3 p_1 + 2t_4 \mathrm{conv}(u,p_2) \\
        2t_4 \mathrm{conv}(v, p_1) + \mathrm{conv}(-t_3 + 2t_4 u + t_1,p_2) \\
        \mathcal{O}_{\mathcal{G}}\mathrm{conv}(\mathrm{conv}(2t_4e_0,\mathcal{O}_{\mathcal{G}}^{-1} w), p_1) + -t_3 p_3 + 2t_4 \mathcal{O}_{\mathcal{G}} \mathrm{conv}(u,\mathcal{O}_{\mathcal{G}}^{-1}p_3) \\
        2t_4 hp_1 + (t_2 + 4t_4 v)p_2 + (-t_3 + 2t_4 u)p_4
    \end{bmatrix}\right\|_{2,\omega}.
\end{align}
We now use Banach algebra to obtain
{\footnotesize\begin{align}
    &\|-t_3 p_1 + 2t_4up_2\|_{2,\omega} \leq r + 2 r \|u\|_{1,\sqrt{\omega}} \leq r + 2 r(\|\overline{u}\|_{1,\sqrt{\omega}} + r), \\
    &\|2t_4vp_1 + (-t_3 + 2t_4u + t_1)p_2\|_{2,\omega} \leq 2 r(\|\overline{v}\|_{1,\sqrt{\omega}} + r) + 2r + 2 r(\|\overline{u}\|_{1,\sqrt{\omega}} + r), \\
    &\|2t_4hp_1 + (t_2 + 4t_4v)p_2 + (-t_3 + 2t_4u)p_4\|_{2,\omega} \leq 2 r(\|\overline{h}\|_{1,\sqrt{\omega}} + r) + 2r + 4 r(\|\overline{v}\|_{1,\sqrt{\omega}} + r) + 2 r(\|\overline{u}\|_{1,\sqrt{\omega}} + r),
\end{align}}
\begin{align}
    &\|\mathcal{O}_{\mathcal{G}}\mathrm{conv}(\mathrm{conv}(2t_4e_0,\mathcal{O}_{\mathcal{G}}^{-1} w), p_1) + -t_3 p_3 + 2t_4 \mathcal{O}_{\mathcal{G}} \mathrm{conv}(u,\mathcal{O}_{\mathcal{G}}^{-1}p_3)\|_{2,\omega} \\&\leq \kappa \|\mathcal{O}_{\mathcal{G}}\|_{\mathcal{B}(\ell^2_{\mathcal{G},\omega})}\|\mathrm{conv}(2t_4e_0,\mathcal{O}_{\mathcal{G}}^{-1} w))\|_{2,\omega}\|p_1\|_{2,\omega} + r + 2 r\|\mathcal{O}_{\mathcal{G}}\|_{\mathcal{B}(\ell^2_{\mathcal{G},\omega})} \|\mathrm{conv}(t_5,\mathcal{O}_{\mathcal{G}}^{-1} p_3)\|_{2,\omega}  \\
    &\hspace{+1cm}+2 r\|\mathcal{O}_{\mathcal{G}}\|_{\mathcal{B}(\ell^2_{\mathcal{G},\omega})}\|\mathrm{conv}(\overline{u},\mathcal{O}_{\mathcal{G}}^{-1} p_3)\|_{2,\omega} \\
    &\leq 2\kappa  \mathscr{O}_{\mathrm{max}}r(\|\mathcal{O}_{\mathcal{G},N}^{-1} \overline{w}\|_{1,\sqrt{\omega}}+\kappa r) + r + 2\kappa  \mathscr{O}_{\mathrm{max}}r(\|\overline{u}\|_{1,\sqrt{\omega}} + \kappa r).
\end{align}
Hence, we obtain
{\small\begin{align}
    &\|L_{\mathbf{s},\blambda} - L_{\mathbf{0},\overline{\blambda}} +DG(\mathbf{x}) - DG(0,\overline{\lambda_2},u,v,w,h)\|_{\mathcal{B}((\ell^2_{\mathcal{G},\omega})^4)} \\
    &\leq (14 + 4\kappa \mathscr{O}_{\mathrm{max}})r^2 + [6 +  (6+2\kappa \mathscr{O}_{\mathrm{max}})\|\overline{u}\|_{1,\sqrt{\omega}} + 6 \|\overline{v}\|_{1,\sqrt{\omega}} + 2\kappa  \mathscr{O}_{\mathrm{max}}\|\mathcal{O}_{\mathcal{G},N}^{-1}\overline{w}\|_{1,\sqrt{\omega}} + 2 \|\overline{h}\|_{1,\sqrt{\omega}}]r \\
    &= ((14 + 4\kappa \mathscr{O}_{\mathrm{max}})r + 6 +  (6+2\kappa \mathscr{O}_{\mathrm{max}})\|\overline{u}\|_{1,\sqrt{\omega}} + 6 \|\overline{v}\|_{1,\sqrt{\omega}} + 2\kappa \mathscr{O}_{\mathrm{max}}\|\mathcal{O}_{\mathcal{G},N}\overline{w}\|_{1,\sqrt{\omega}} + 2 \|\overline{h}\|_{1,\sqrt{\omega}})r \\
    &\bydef Z_{2,4}(r)r.\label{def : Z24}
\end{align}}
Let us now examine the fifth term of \eqref{Z2 split up}. Using $\mathbf{p}$ as before, observe that
{\tiny\begin{align}
&\|DG(0,\overline{\lambda_2},u,v,w,h)- DG(\overline{\mathbf{x}})\|_{\mathcal{B}((\ell^2_{\mathcal{G},\omega})^4)} \\
&= \left\| \begin{bmatrix}
   (2\overline{\lambda_2} t_5 - 6t_5 \overline{u}- 3t_5^2 )p_1 \\
   (2\overline{\lambda_2} t_6 - 6t_5\overline{v} - 6t_6\overline{u} - 6t_5t_6)p_1 + (2\overline{\lambda}_2 t_5 - 6t_5 \overline{u} - 3t_5^2)p_2 \\ \mathcal{O}_{\mathcal{G}}\mathrm{conv}(\mathrm{conv}(2\overline{\lambda}_2 - 6\overline{u},\mathcal{O}_{\mathcal{G}}^{-1}t_7) - 6\mathrm{conv}(t_5 ,\mathcal{O}_{\mathcal{G},N}^{-1} \overline{w}) - 6 \mathrm{conv}(t_5,\mathcal{O}_{\mathcal{G}}^{-1}t_7),p_1) + \mathcal{O}_{\mathcal{G}}\mathrm{conv}(2\overline{\lambda_2} t_5 - 6t_5 \overline{u}- 3t_5^2 ,\mathcal{O}_{\mathcal{G}}^{-1}p_3) \\
   (2\overline{\lambda}_2 t_8 - 6\overline{u} t_8 - 6t_5 \overline{h} -6t_5 t_8 - 12t_6\overline{v} - 6t_6^2)p_1 + 2(2\overline{\lambda}_2 t_6 - 6t_5\overline{v} - 6t_6\overline{u} - 6t_5 t_6)p_2 + (2\overline{\lambda}_2 t_5 - 6t_5 \overline{u} - 3t_5^2)p_4
\end{bmatrix}\right\|_{2,\omega}
\end{align}}
We now use Lemmas \ref{lem : conv_estimates} and \ref{lem : gen young} to estimate each term
{\footnotesize\begin{align}
    &\|(2\overline{\lambda_2} t_5 - 6t_5 \overline{u}- 3t_5^2 )p_1\|_{2,\omega} \leq 2\kappa \|\overline{\lambda}_2 - 3\overline{u}\|_{1,\sqrt{\omega}} r + 3\kappa^2 r^2, \\
    &\|(2\overline{\lambda_2} t_6 - 6t_5\overline{v} - 6t_6\overline{u} - 6t_5t_6)p_1 + (2\overline{\lambda}_2 t_5 - 6t_5 \overline{u} - 3t_5^2)p_2\|_{2,\omega} \leq  6\kappa \|\overline{v}\|_{1,\sqrt{\omega}}r +  4\kappa \|\overline{\lambda}_2 - 3\overline{u}\|_{1,\sqrt{\omega}}r + 9\kappa^2 r^2. \end{align}}
We now move to the fourth term.
    {\footnotesize\begin{align}
    &\|(2\overline{\lambda}_2 t_8 - 6\overline{u} t_8 - 6t_5 \overline{h} -6t_5 t_8 - 12t_6\overline{v} - 6t_6^2)p_1 + 2(2\overline{\lambda}_2 t_6 - 6t_5\overline{v} - 6t_6\overline{u} - 6t_5 t_6)p_2 + (2\overline{\lambda}_2 t_5 - 6t_5 \overline{u} - 3t_5^2)p_4\|_{2,\omega} \\
    &\leq 8\kappa  \|\overline{\lambda}_2 - 3\overline{u}\|_{1,\sqrt{\omega}}r + 24\kappa \|\overline{v}\|_{1,\sqrt{\omega}}r+ 6\kappa \|\overline{h}\|_{1,\sqrt{\omega}}r + 27\kappa^2 r^2.
\end{align}}
Finally we estimate the last term, which requires us to use here $\mathcal{O}_{\mathcal{G}}$.
{\scriptsize\begin{align}
    &\| \mathcal{O}_{\mathcal{G}}\mathrm{conv}(\mathrm{conv}(2\overline{\lambda}_2 - 6\overline{u},\mathcal{O}_{\mathcal{G}}^{-1}t_7) - 6\mathrm{conv}(t_5 ,\mathcal{O}_{\mathcal{G},N}^{-1} \overline{w}) - 6 \mathrm{conv}(t_5,\mathcal{O}_{\mathcal{G}}^{-1}t_7),p_1) + \mathcal{O}_{\mathcal{G}}\mathrm{conv}(2\overline{\lambda_2} t_5 - 6t_5 \overline{u}- 3t_5^2 ,\mathcal{O}_{\mathcal{G}}^{-1}p_3)\|_{2,\omega} \\
    &\leq \kappa \|\mathcal{O}_{\mathcal{G}}\|_{\mathcal{B}(\ell^2_{\mathcal{G},\omega})} ( \|2\overline{\lambda}_2 - 6\overline{u}\|_{1,\sqrt{\omega}} \|\mathcal{O}_{\mathcal{G}}^{-1} t_7\|_{2,\omega} + 6 \|t_5\|_{2,\omega} \|\mathcal{O}_{\mathcal{G},N}^{-1} \overline{w}\|_{1,\sqrt{\omega}} + 6\kappa \|t_5\|_{2,\omega} \|\mathcal{O}_{\mathcal{G}}^{-1} t_7\|_{2,\omega}) \|p_1\|_{2,\omega} \\
    &\hspace{+3cm}+ \kappa\|\mathcal{O}_{\mathcal{G}}\|_{\mathcal{B}(\ell^2_{\mathcal{G},\omega})}( \|2\overline{\lambda}_2 - 6\overline{u}\|_{1,\sqrt{\omega}}r + 3\kappa r^2) \|\mathcal{O}_{\mathcal{G}}^{-1} p_3\|_{2,\omega} \\
    &\leq \kappa \mathscr{O}_{\mathrm{max}}(4\|\overline{\lambda}_2 - 3\overline{u}\|_{1,\sqrt{\omega}} r + 6\|\mathcal{O}_{\mathcal{G},N}^{-1} \overline{w}\|_{1,\sqrt{\omega}} r + 9\kappa r^2).
\end{align}}
Hence, we obtain
{\footnotesize\begin{align}
    &\|DG(0,\overline{\lambda_2},u,v,w,h)- DG(\overline{\mathbf{x}})\|_{\mathcal{B}((\ell^2_{\mathcal{G},\omega})^4)} \\
    &\leq (14\kappa + 4 \mathscr{O}_{\mathrm{max}})  \|\overline{\lambda}_2 - 3\overline{u}\|_{1,\sqrt{\omega}}r+ 30\kappa \|\overline{v}\|_{1,\sqrt{\omega}}r + 6\kappa \|\overline{h}\|_{1,\sqrt{\omega}}r+ 6\kappa \mathscr{O}_{\mathrm{max}}\|\mathcal{O}_{\mathcal{G},N}^{-1} \overline{w}\|_{1,\sqrt{\omega}}r + (39+9\mathscr{O}_{\mathrm{max}})\kappa^2 r^2 \\
    &= ((14\kappa + 4\mathscr{O}_{\mathrm{max}}) \|\overline{\lambda}_2 - 3\overline{u}\|_{1,\sqrt{\omega}}+ 30\kappa \|\overline{v}\|_{1,\sqrt{\omega}} + 6\kappa \|\overline{h}\|_{1,\sqrt{\omega}} + 6\kappa \mathscr{O}_{\mathrm{max}}\|\mathcal{O}_{\mathcal{G},N}^{-1} \overline{w}\|_{1,\sqrt{\omega}} + (39+9\mathscr{O}_{\mathrm{max}})\kappa^2 r)r \\
    &\bydef Z_{2,5}(r)r.
\end{align}}
Combining everything we estimated from \eqref{Z2 split up}, we obtain
\begin{align}
    \|A(DF(\mathbf{x}) - DF(\overline{\mathbf{x}}))\|_{\mathcal{B}(X_{\mathcal{G},\omega})} &\leq (\|A^N\|_{\mathcal{B}(X_{\mathcal{G},\omega})} + \mathcal{L}_{\infty}) \sum_{k = 1}^5 Z_{2,k}(r)r \bydef Z_2(r) r
\end{align}
as desired.
\end{proof}
\subsection{Proof of Lemma \ref{lem : Z1 full}}\label{apen : Z1}
We first introduce
{\small\begin{align}
&\overline{q}_4 \bydef \overline{q}_2 \overline{v},~(\mathscr{z}_1)_k \bydef (e_k\star -6\overline{v}^2),\overline{w})_{2}, ~ (\mathscr{z}_2)_k \bydef ((e_k\star (4\overline{\lambda}_2 \overline{v} - 12\overline{u}\overline{v})),\overline{w})_{2},
\end{align}}
where $e_k$ is defined as in \eqref{def : ek}.
We also observe an alternative expression for $DG(\overline{\mathbf{x}})$ and write
\begin{align}
    DG(\overline{\mathbf{x}}) = \begin{bmatrix}
        \mathbb{M}_{\overline{q}_1} & 0 & 0 & 0 \\
        \mathbb{M}_{\overline{q}_2} & \mathbb{M}_{\overline{q}_1} & 0 & 0 \\
        \mathbb{Q}_{\mathcal{O}_{\mathcal{G},2N}\overline{q}_0} & 0 & \mathbb{M}_{\overline{q}_1}^* & 0 \\
        \mathbb{M}_{\overline{q}_3} & 2\mathbb{M}_{\overline{q}_2} & 0 & \mathbb{M}_{\overline{q}_1}
    \end{bmatrix}, ~ DG^N(\overline{\mathbf{x}}) \bydef \begin{bmatrix}
        \mathbb{M}_{\overline{q}_1^N} & 0 & 0 & 0 \\
        \mathbb{M}_{\overline{q}_2^N} & \mathbb{M}_{\overline{q}_1^N} & 0 & 0 \\
        \mathbb{Q}_{\mathcal{O}_{\mathcal{G},N}\overline{q}_0^N} & 0 & \mathbb{M}_{\overline{q}_1^N}^* & 0 \\
        \mathbb{M}_{\overline{q}_3^N} & 2\mathbb{M}_{\overline{q}_2^N} & 0 & \mathbb{M}_{\overline{q}_1^N}
    \end{bmatrix}.
\end{align}
We also define
\begin{align}
    &M(\overline{\mathbf{x}}) \bydef \begin{bmatrix}
        0 & 0 & 0 & 0 & 0 & 2\overline{v}^* & 0 & 0 \\
        0 & 0 & 0 & 0 & 0 & \overline{w}^* & \overline{v}^* & 0 \\
        0 & 0 & 0 & 0 & 0 & 0 & \overline{h}^* & \overline{w}^* \\
        0 & 0 & 0 & (2\overline{v}^2,\overline{w})_{2} & \mathscr{z}_1^* & \mathscr{z}_2^* &  \overline{q}_4^* & 0 \\
        0 & 0 & -\overline{u} &  \overline{u}^2 & \mathbb{M}_{\overline{q}_1} & 0 & 0 & 0 \\
        \overline{v} & 0 & -\overline{v} & 2 \overline{u}\overline{v} & \mathbb{M}_{\overline{q}_2} &   \mathbb{M}_{\overline{q}_1} & 0 & 0 \\
        0 & 0 & -\overline{w} & 2\mathcal{O}_{\mathcal{G},2N}(\overline{u}\star (\mathcal{O}_{\mathcal{G},N}^{-1}\overline{w})) & \mathbb{Q}_{\mathcal{O}_{\mathcal{G},2N}\overline{q}_0} & 0 & \mathbb{M}_{\overline{q}_1}^* & 0 \\
        0 & \overline{v} & -\overline{h} & 2\overline{u}\overline{h} + 2\overline{v}^2 & \mathbb{M}_{\overline{q}_3} & 2\mathbb{M}_{\overline{q}_2} & 0 & \mathbb{M}_{\overline{q}_1}
    \end{bmatrix}, \end{align}
{\footnotesize\begin{align}    &M^N(\overline{\mathbf{x}}) \bydef \begin{bmatrix}
        0 & 0 & 0 & 0 & 0 & 2(\overline{v})^* & 0 & 0 \\
        0 & 0 & 0 & 0 & 0 & (\overline{w})^* & (\overline{v})^* & 0 \\
        0 & 0 & 0 & 0 & 0 & 0 & (\overline{h})^* & (\overline{w})^* \\
        0 & 0 & 0 & (2\overline{v}^2,\overline{w})_{2} & (\mathscr{z}_1^{2N})^* & (\mathscr{z}_2^{2N})^* &  (\overline{q}_4^{2N})^* & 0 \\
        0 & 0 & -\overline{u} &  \Pi^{\leq N} \overline{u}^2 & \mathbb{M}_{\overline{q}_1^N} & 0 & 0 & 0 \\
        \overline{v} & 0 & -\overline{v} & 2\Pi^{\leq N} \overline{u}\overline{v} & \mathbb{M}_{\overline{q}_2^N} &   \mathbb{M}_{\overline{q}_1^N} & 0 & 0 \\
        0 & 0 & -\overline{w} & 2\Pi^{\leq N}\mathcal{O}_{\mathcal{G},2N}(\overline{u}\star (\mathcal{O}_{\mathcal{G},N}^{-1}\overline{w}))& \mathbb{Q}_{\mathcal{O}_{\mathcal{G},N}\overline{q}_0^N} & 0 & \mathbb{M}_{\overline{q}_1^N}^* & 0 \\
        0 & \overline{v} & -\overline{h} & \Pi^{\leq N} (2\overline{u}\overline{h} + 2\overline{v}^2) & \mathbb{M}_{\overline{q}_3^N} & 2\mathbb{M}_{\overline{q}_2^N} & 0 & \mathbb{M}_{\overline{q}_1^N}
    \end{bmatrix},
\end{align}}
Note that $\mathscr{z}_j = \Pi^{\leq 3N} \mathscr{z}_j$ and $\overline{q}_4 = \Pi^{\leq 3N} \overline{q}_4$ by definition.
We now prove Lemma \ref{lem : Z1 full}.
\begin{proof}
First, we let
\begin{align}
    \overline{L}_{\mathbf{0},\overline{\blambda}} \bydef \begin{bmatrix}
        0 & 0 \\
        0 & L_{\mathbf{0},\overline{\blambda}}
    \end{bmatrix}.\label{L bar}
\end{align}
To begin, we introduce $M(\overline{\mathbf{x}})$ and $M^N(\overline{\mathbf{x}})$,
\begin{align}
    \|I_d - ADF(\overline{\mathbf{x}})\|_{\mathcal{B}(X_{\mathcal{G},\omega})} &= \|I_d - A(\overline{L}_{\mathbf{0},\overline{\blambda}} + M(\overline{\mathbf{x}}))\|_{\mathcal{B}(X_{\mathcal{G},\omega})} \\
    &\leq \|I_d - A(\overline{L}_{\mathbf{0},\overline{\blambda}} + M^N(\overline{\mathbf{x}}))\|_{\mathcal{B}(X_{\mathcal{G},\omega})} + \|A(M(\overline{\mathbf{x}}) - M^N(\overline{\mathbf{x}}))\|_{\mathcal{B}(X_{\mathcal{G},\omega})} \\
    &\hspace{-2.8cm}\leq \|I_d - A(\overline{L}_{\mathbf{0},\overline{\blambda}} + M^N(\overline{\mathbf{x}}))\|_{\mathcal{B}(X_{\mathcal{G},\omega})}  + (\|A^N\|_{\mathcal{B}(X_{\mathcal{G},\omega})} + \mathcal{L}_{\infty})\|M(\overline{\mathbf{x}}) - M^N(\overline{\mathbf{x}})\|_{\mathcal{B}(X_{\mathcal{G},\omega})}\label{Split into Z1 and Zinfty}
\end{align}
where the last step followed from the properties of $A$. We now further estimate the second term. Observe that
{\scriptsize\begin{align}
    &\|M(\overline{\mathbf{x}}) - M^N(\overline{\mathbf{x}})\|_{\mathcal{B}(X_{\mathcal{G},\omega})} \leq 
     \|\overline{q}_4 - \overline{q}_4^{2N}\|_{2} + \|\overline{u}^2 - \Pi^{\leq N} \overline{u}^2\|_{2,\omega} + 2\|\overline{u}\overline{v} - \Pi^{\leq N} \overline{u} \overline{v}\|_{2,\omega}  \\&\hspace{+0.5cm}+2\|\mathcal{O}_{\mathcal{G},2N}(\mathrm{conv}(\overline{u}, \mathcal{O}_{\mathcal{G},N}^{-1}\overline{w}) - \Pi^{\leq N} \mathrm{conv}(\overline{u},\mathcal{O}_{\mathcal{G},N}^{-1} \overline{w}))\|_{2,\omega} + 2\|\overline{u}\overline{h} + \overline{v}^2 - \Pi^{\leq N} (\overline{u} \overline{h} + \overline{v}^2)\|_{2,\omega} + 3\|\mathbb{M}_{\overline{q}_1} - \mathbb{M}_{\overline{q}_1^N}\|_{\mathcal{B}(\ell^2_{\mathcal{G},\omega})} \\
&\hspace{+1cm}+3\|\mathbb{M}_{\overline{q}_2} - \mathbb{M}_{\overline{q}_2^N}\|_{\mathcal{B}(\ell^2_{\mathcal{G},\omega})} + \|\mathbb{M}_{\overline{q}_3} - \mathbb{M}_{\overline{q}_3^N}\|_{\mathcal{B}(\ell^2_{\mathcal{G},\omega})} + \|\mathbb{M}_{\overline{q}_1}^* - \mathbb{M}_{\overline{q}_1^N}^*\|_{\mathcal{B}(\ell^2_{\mathcal{G},\omega})} + \|\mathbb{Q}_{\mathcal{O}_{\mathcal{G},2N}\overline{q}_0} - \mathbb{Q}_{\mathcal{O}_{\mathcal{G},N}\overline{q}_0^N}\|_{\mathcal{B}(\ell^2_{\mathcal{G},\omega})} \\
&\hspace{+1.5cm} + \|\mathscr{z}_1 - \mathscr{z}_1^{2N}\|_{2} + \|\mathscr{z}_2 - \mathscr{z}_2^{2N}\|_{2}.
\end{align}}
Note that the $\|\overline{q}_4 - \overline{q}_4^{2N}\|_{2}, \|\mathscr{z}_1 - \mathscr{z}_1^{2N}\|_{2},$ and $\|\mathscr{z}_2 - \mathscr{z}_2^{2N}\|_{2}$ entries are $2$-norms because it would be multiplied by an element in $\ell^2_{\mathcal{G},\omega}$ resulting in a dot product. Hence, we use Cauchy-Schwarz and get only a $2$-norm rather than a $2,\omega$-norm.
The estimate can be further simplified. Let $t \in \ell^2_{\mathcal{G},\omega}$ such that $\|t\|_{2,\omega} = 1$. Then, using Lemma \ref{lem : gen young}, we estimate
\begin{align}
    \|\mathbb{M}_{\overline{q}_k} - \mathbb{M}_{\overline{q}_k^N}\|_{\mathcal{B}(\ell^2_{\mathcal{G},\omega})} =  \|\mathrm{conv}(\overline{q}_k - \overline{q}_k^N,t)\|_{2,\omega} \leq  \|t\|_{2,\omega} \|\overline{q}_k - \overline{q}_k^N\|_{1,\sqrt{\omega}}=  \|\overline{q}_k - \overline{q}_k^N\|_{1,\sqrt{\omega}}
\end{align}
for $k = 1,2,3$. Also, 
\begin{align}
    \|\mathbb{M}_{\overline{q}_1}^* - \mathbb{M}_{\overline{q}_1^N}^*\|_{\mathcal{B}(\ell^2_{\mathcal{G},\omega})} = \|\mathcal{O}_{\mathcal{G}}(\mathbb{M}_{\overline{q}_1} - \mathbb{M}_{\overline{q}_1^N})\mathcal{O}_{\mathcal{G}}^{-1}\|_{\mathcal{B}(\ell^2_{\mathcal{G},\omega})} &\leq \|\mathcal{O}_\mathcal{G}^{-1}\|_{\mathcal{B}(\ell^2_{\mathcal{G},\omega})} \|\mathcal{O}_{\mathcal{G}}\|_{\mathcal{B}(\ell^2_{\mathcal{G},\omega})} \|\mathbb{M}_{\overline{q}_1} - \mathbb{M}_{\overline{q}_1^N}\|_{\mathcal{B}(\ell^2_{\mathcal{G},\omega})} \\
    &\leq \kappa \mathscr{O}_{\mathrm{max}}\|\overline{q}_1 - \overline{q}_1^N\|_{2,\omega}.
\end{align}
We use a similar idea and \eqref{bbQ action} to get 
\begin{align}
    \|\mathbb{Q}_{\mathcal{O}_{\mathcal{G},2N}\overline{q}_0} - \mathbb{Q}_{\mathcal{O}_{\mathcal{G},N}\overline{q}_0^N}\|_{\mathcal{B}(\ell^2_{\mathcal{G},\omega})} &= \|\mathcal{O}_{\mathcal{G}}\mathrm{conv}(\mathcal{O}_{\mathcal{G},2N}^{-1}\mathcal{O}_{\mathcal{G},2N}(\overline{q}_0 - \overline{q}_0^N),t)\|_{2,\omega} \\
    &\leq \|\mathcal{O}_{\mathcal{G}}\|_{\mathcal{B}(\ell^2_{\mathcal{G},\omega})} \|\mathrm{conv}(\overline{q}_0 - \overline{q}_0^N,t)\|_{2,\omega} \\
    &\leq  \|\mathcal{O}_{\mathcal{G}}\|_{\mathcal{B}(\ell^2_{\mathcal{G},\omega})} \|\overline{q}_0 - \overline{q}_0^N\|_{1,\sqrt{\omega}} \|t\|_{2,\omega} \\
    &\leq  \mathscr{O}_{\mathrm{max}}\|\overline{q}_0 -\overline{q}_0^N\|_{1,\sqrt{\omega}}
\end{align}
Hence, we find that
\begin{align}
    \|M^N(\overline{\mathbf{x}}) - M(\overline{\mathbf{x}})\|_{\mathcal{B}(X_{\mathcal{G},\omega})} \leq  Z_{\infty}.
\end{align}
We can now return to \eqref{Split into Z1 and Zinfty} and get
\begin{align}
    \|I_d - ADF(\overline{\mathbf{x}})\|_{\mathcal{B}(X_{\mathcal{G},\omega})} &\leq \|I_d - A(\overline{L}_{\mathbf{0},\overline{\blambda}} + M^N(\overline{\mathbf{x}}))\|_{\mathcal{B}(X_{\mathcal{G},\omega})} + Z_{\infty}.
\end{align}
We now let $\mathcal{M} = \overline{L}_{\mathbf{0},\overline{\blambda}} + M^N(\overline{\mathbf{x}})$. We then introduce projections{\scriptsize\begin{align}
    I_d - A\mathcal{M}&= \obpi^{\leq N} (I_d - A\mathcal{M})\obpi^{\leq 2N} + \obpi^{\leq N}(I_d - A\mathcal{M})\obpi^{>2N} + \obpi^{>N} (I_d - A\mathcal{M})\obpi^{\leq 2N} + \obpi^{> N}(I_d - A\mathcal{M})\obpi^{> 2N} \\
    &= \obpi^{\leq N}(I_d - A\mathcal{M})\obpi^{\leq 2N} - A^N \mathcal{M} \obpi^{> 2N} + \obpi^{>N} \obpi^{\leq 2N} - \obpi^{>N} A \mathcal{M} \obpi^{ \leq 2N} + \obpi^{>N} - \obpi^{>N} A\mathcal{M} \obpi^{> 2N} \\
    &= \obpi^{\leq N}(I_d - A\mathcal{M})\obpi^{\leq 2N} - A^N \mathcal{M} \obpi^{> 2N} + \obpi^{\leq 2N} - \obpi^{\leq N} - \obpi^{>N} A \mathcal{M} \obpi^{\leq 2N} + \obpi^{>N} - \obpi^{>N} A\mathcal{M} \obpi^{>2N}.
\end{align}}
The first term in the above is $Z_{1,1}$ by definition. For the second term, since $\obpi^{\leq N} \overline{L}_{\mathbf{0},\overline{\blambda}} \obpi^{>2N} = 0$, observe that
{\scriptsize\begin{align}
    \|-A^N \mathcal{M} \obpi^{>2N}\|_{\mathcal{B}(X_{\mathcal{G},\omega})} = \|A^N (L_{\mathbf{0},\overline{\blambda}} + M^N(\overline{\mathbf{x}}))\obpi^{>2N}\|_{\mathcal{B}(X_{\mathcal{G},\omega})}
    &= \|A^N M^N(\overline{\mathbf{x}})\obpi^{>2N}\|_{\mathcal{B}(X_{\mathcal{G},\omega})} \\
    &\leq \|A^N\|_{\mathcal{B}(X_{\mathcal{G},\omega})} \|\obpi^{\leq N}M^N(\overline{\mathbf{x}})\obpi^{>2N}\|_{\mathcal{B}(X_{\mathcal{G},\omega})} \\
    &=\|A^N\|_{\mathcal{B}(X_{\mathcal{G},\omega})} \|\bpi^{\leq N}DG^N(\overline{\mathbf{x}})\bpi^{>2N}\|_{\mathcal{B}((\ell^2_{\mathcal{G},\omega})^4)}.
\end{align}}
Now, let $\mathbf{t} = (t_1,t_2,t_3,t_4) \in (\ell^2_{\mathcal{G},\omega})^4, \|\mathbf{t}\|_{2,\omega} = 1$. Then, observe that
{\small\begin{align}
    \|\bpi^{\leq N}DG^N(\overline{\mathbf{x}})\bpi^{>2N}\|_{\mathcal{B}((\ell^2_{\mathcal{G},\omega})^4)} &= \left\| \begin{bmatrix}
        \Pi^{\leq N} \mathbb{M}_{\overline{q}_1^N} \Pi^{> 2N}t_1 \\
        \Pi^{\leq N} \mathbb{M}_{\overline{q}_2^N} \Pi^{> 2N}t_1 + \Pi^{\leq N} \mathbb{M}_{\overline{q}_1^N}  \Pi^{> 2N}t_2 \\
        \Pi^{\leq N} \mathbb{Q}_{\mathcal{O}_{\mathcal{G},N}\overline{q}_0^N} \Pi^{> 2N}t_1 + \Pi^{\leq N} \mathbb{M}_{\overline{q}_1^N}^* \Pi^{> 2N}t_3 \\
        \Pi^{\leq N} \mathbb{M}_{\overline{q}_3^N} \Pi^{> 2N}t_1 + 2\Pi^{\leq N} \mathbb{M}_{\overline{q}_2^N}  \Pi^{> 2N}t_2 + \Pi^{\leq N} \mathbb{M}_{\overline{q}_1^N}  \Pi^{> 2N}t_4
    \end{bmatrix}\right\|_{2,\omega}.\label{applying the derivative to try to show 0 in Z12}
\end{align}}
Now, let us show that each type of term in \eqref{applying the derivative to try to show 0 in Z12} is $0$. To begin,
\begin{align}
    &\Pi^{\leq N} \mathbb{M}_{\overline{q}_k^N} \Pi^{> 2N}t_j = \Pi^{\leq N} (\overline{q}_k^N \star \Pi^{> 2N} t_j) = \Pi^{\leq N} \Pi^{> N} (\overline{q}_k^N \star t_j) = 0 \end{align}
Next, we examine the adjoint multiplication term. We use the definition and the fact that $\Pi^{\leq N} \mathcal{O}_{\mathcal{G}}^{-1} = \mathcal{O}_{\mathcal{G},N}^{-1} \Pi^{\leq N}$ since $\mathcal{O}_{\mathcal{G}}$ is diagonal to get
\begin{align}
    \Pi^{\leq N} \mathbb{M}_{\overline{q}_1^N}^* \Pi^{> 2N}t_3 =  \Pi^{\leq N}\mathcal{O}_{\mathcal{G}} \mathrm{conv}(\overline{q}_1^N,\mathcal{O}_{\mathcal{G}}^{-1} \Pi^{> 2N}t_3) &= \mathcal{O}_{\mathcal{G},N} \Pi^{ \leq N} \mathrm{conv}(\overline{q}_1^N, \Pi^{> 2N} \mathcal{O}_{\mathcal{G}}^{-1} t_3) \\
    &= \mathcal{O}_{\mathcal{G},N} \Pi^{\leq N} \Pi^{> N} \mathrm{conv}(\overline{q}_1^N,\mathcal{O}_{\mathcal{G}}^{-1} t_3) \\
    &= 0.
\end{align}
Finally, we examine the action of $\mathbb{Q}$. Using \eqref{bbQ action}, notice that
\begin{align}
    \Pi^{\leq N} \mathbb{Q}_{\mathcal{O}_{\mathcal{G},N}\overline{q}_0^N} \Pi^{> 2N}t_1 = \Pi^{\leq N} \mathcal{O}_{\mathcal{G}} \mathrm{conv}(\overline{q}_0^N, \Pi^{> 2N} t_1) 
    &= \mathcal{O}_{\mathcal{G},N} \Pi^{\leq N} \mathrm{conv}(\overline{q}_0^N,\Pi^{> 2N} t_1) \\
    &= \mathcal{O}_{\mathcal{G},N}\Pi^{\leq N} \Pi^{> N}\mathrm{conv} (\overline{q}_0^N,t_1) \\
    &= 0.
\end{align}
Hence, we have shown that \eqref{applying the derivative to try to show 0 in Z12} is $0$.
We now move to the third term where
{\footnotesize\begin{align}
    \|\obpi^{\leq 2N} - \obpi^{\leq N} - \obpi^{> N} A \mathcal{M}\obpi^{\leq 2N}\|_{\mathcal{B}(X_{\mathcal{G},\omega})}
    &= \|\obpi^{\leq 2N} - \obpi^{\leq N} - \obpi^{>N} A(\overline{L}_{\mathbf{0},\overline{\blambda}} + M^N(\overline{\mathbf{x}}))\obpi^{\leq 2N}\|_{\mathcal{B}(X_{\mathcal{G},\omega})} \\
    &\hspace{-1cm}= \|\obpi^{\leq 2N} - \obpi^{\leq N} - \obpi^{> N}\obpi^{\leq 2N} - \obpi^{>N} AM^N(\overline{\mathbf{x}})\obpi^{\leq 2N}\|_{\mathcal{B}(X_{\mathcal{G},\omega})} \\
    &\hspace{-1.5cm}= \|\obpi^{\leq 2N} - \obpi^{\leq N} - (\obpi^{\leq 2N} - \obpi^{\leq N}) - \obpi^{>N} AM^N(\overline{\mathbf{x}})\obpi^{\leq 2N}\|_{\mathcal{B}(X_{\mathcal{G},\omega})}\\
    &= \|\obpi^{>N} AM^N(\overline{\mathbf{x}})\obpi^{\leq 2N}\|_{\mathcal{B}(X_{\mathcal{G},\omega})} \end{align}}
We now perform further estimates to remove the $\obpi^{> N} A$. 
{\footnotesize\begin{align}
    \|\obpi^{>N} AM^N(\overline{\mathbf{x}})\obpi^{\leq 2N}\|_{\mathcal{B}(X_{\mathcal{G},\omega})}&\leq \|\obpi^{> N} A\|_{\mathcal{B}(X_{\mathcal{G},\omega})} \|\obpi^{>N}M^N(\overline{\mathbf{x}}) \obpi^{\leq 2N}\|_{\mathcal{B}(X_{\mathcal{G},\omega})}
    \\
    &\leq \mathcal{L}_{\infty} \|DG^N(\overline{\mathbf{x}})\|_{\mathcal{B}((\ell^2_{\mathcal{G},\omega})^4)} \\
    &\hspace{-0.5cm}\leq  \mathcal{L}_{\infty} (3 \|\overline{q}_{1}^N\|_{1,\sqrt{\omega}} + 3 \|\overline{q}_{2}^N\|_{1,\sqrt{\omega}} +  \|\overline{q}_3^N\|_{1,\sqrt{\omega}} + \|\mathbb{M}_{\overline{q}_1^N}^*\|_{\mathcal{B}(\ell^2_{\mathcal{G},\omega})} + \|\mathbb{Q}_{\mathcal{O}_{\mathcal{G},N}\overline{q}_0^N}\|_{\mathcal{B}(\ell^2_{\mathcal{G},\omega})})\\&\leq  \mathcal{L}_{\infty} \left(\left(3+\mathscr{O}_{\mathrm{max}}\right)\|\overline{q}_{1}^N\|_{1,\sqrt{\omega}} + 3\|\overline{q}_{2}^N\|_{1,\sqrt{\omega}} + \|\overline{q}_3^N\|_{1,\sqrt{\omega}} + \mathscr{O}_{\mathrm{max}}\|\overline{q}_0^N\|_{1,\sqrt{\omega}}\right)\\
    &\bydef \frac{Z_{1,2}}{2}.\label{first Z12N}
\end{align}}
We now move to the final term and estimate.
\begin{align}
    \|\obpi^{>N} - \obpi^{>N} A \mathcal{M} \obpi^{> 2N} \|_{\mathcal{B}(\ell^2_{\mathcal{G},\omega})} &= \|\obpi^{>N} - \obpi^{>N} A (\overline{L}_{\mathbf{0},\overline{\blambda}} + M^N(\overline{\mathbf{x}})) \obpi^{>2N} \|_{\mathcal{B}(X_{\mathcal{G},\omega})} \\
    &=\|\obpi^{>N} - \obpi^{>N}\obpi^{>2N} - \obpi^{>N}A M^N(\overline{\mathbf{x}}) \obpi^{>2N} \|_{\mathcal{B}(X_{\mathcal{G},\omega})} \\
    &= \|\obpi^{>N} A M^N(\overline{\mathbf{x}})\obpi^{>2N}\|_{\mathcal{B}(X_{\mathcal{G},\omega})}.
\end{align}
We now perform similar steps as those used in \eqref{first Z12N}. That is,
{\footnotesize\begin{align}
    \|\obpi^{>N} A M^N(\overline{\mathbf{x}})\obpi^{>2N}\|_{\mathcal{B}(X_{\mathcal{G},\omega})} \leq \|\obpi^{>N} A \|_{\mathcal{B}(X_{\mathcal{G},\omega})} \|\obpi^{>N} M^N(\overline{\mathbf{x}}) \obpi^{>2N}\|_{\mathcal{B}(X_{\mathcal{G},\omega})}
    &\leq \mathcal{L}_{\infty} \|DG^N(\overline{\mathbf{x}})\|_{\mathcal{B}((\ell^2_{\mathcal{G},\omega})^4)} \\
    &\hspace{-1cm}\leq \frac{Z_{1,2}}{2}.\label{second Z12N}
    \end{align}}
This concludes the proof.
\end{proof}
\renewcommand{\theequation}{C.\arabic{equation}}
\setcounter{equation}{0}
\section{Proofs for the Gershgorin enclosure}\label{apen : gershgorin}

We prove Lemma~\ref{lem : computable radii} and Proposition~\ref{prop : gershgorin enclosure} of Section~\ref{sec : gershgorin}, for a general system of $\mathscr{p}$ equations. Throughout, sums over $k$ run over $\mathcal{Z}_{\mathrm{red}}(\mathcal{G})$ and $i,j$ over $\{1,\dots,\mathscr{p}\}$.

\subsection{Proof of Lemma~\ref{lem : computable radii}}\label{apen : computable radii}
\begin{proof}
\textbf{(a) The inflation constant.} Fix $n$ and $j$. Adding and subtracting $\overline{\mathcal{D}}$ termwise,
\[
\mathscr{r}_{n,j} + \bigl|(\mathcal{D}_{j,j})_{n,n}-(\overline{\mathcal{D}}_{j,j})_{n,n}\bigr|
\leq \overline{\mathscr{r}}_{n,j} + \frac{1}{|\mathrm{orb}_{\mathcal{G}}(n)|}\sum_{i}\sum_{k}\bigl|(\mathcal{D}_{i,j})_{k,n}-(\overline{\mathcal{D}}_{i,j})_{k,n}\bigr|\,|\mathrm{orb}_{\mathcal{G}}(k)| ,
\]
and the last sum is the $\ell^1_{\mathcal{G}}$-norm of the $(n,j)$ column of $\mathcal{D}-\overline{\mathcal{D}}$, divided by $|\mathrm{orb}_{\mathcal{G}}(n)|$; it is therefore bounded by $\|\mathcal{D}-\overline{\mathcal{D}}\|_{\mathcal{B}(\ell^1_{\mathcal{G}})}$. By \eqref{def : D gersh} and \eqref{def : eps r0},
\[
\|\mathcal{D}-\overline{\mathcal{D}}\|_{\mathcal{B}(\ell^1_{\mathcal{G}})}
= \bigl\|P^{-1}\bigl(D_uf(\tilde{\blambda},\tilde{\mathbf{u}})-D_uf(\overline{\blambda},\overline{\mathbf{u}})\bigr)P\bigr\|_{\mathcal{B}(\ell^1_{\mathcal{G}})}
\leq \max(1,\|(P^N)^{-1}\|_{\mathcal{B}(\ell^1_{\mathcal{G}})})\max(1,\|P^N\|_{\mathcal{B}(\ell^1_{\mathcal{G}})})\,\varepsilon_{r_0} = \mathscr{r}^{\infty},
\]
using $P = P^N+\Pi^{>N}$ and $P^{-1}=(P^N)^{-1}+\Pi^{>N}$. Hence $\mathscr{r}_{n,j}\leq\overline{\mathscr{r}}_{n,j}+\mathscr{r}^{\infty}$ and, by the triangle inequality, $B_{n,j}\subset\{z : |z-(\overline{\mathcal{D}}_{j,j})_{n,n}|\leq\overline{\mathscr{r}}_{n,j}+\mathscr{r}^{\infty}\}$. It remains to bound $\overline{\mathscr{r}}_{n,j}$ on each of the three index regions.

\textbf{(b) $n \in I^N$.} Splitting the sum at $I^N$ and using $\overline{\mathcal{D}} = P^{-1}(l_{\overline{\blambda}}+Dg)P$ together with $\Pi^{>N}P^N=0$, the terms with $k \notin I^N$ reduce to those of $Dg\,P^N$; since $\overline{\mathbf{u}} = \obpi^{\leq N}\overline{\mathbf{u}}$, the sequences $\overline{q}_{i,j}$ are supported in $I^{2N}$, so $(Dg\,P^N)_{k,n}=0$ once $k \notin I^{3N}$ and $n \in I^N$. This leaves exactly $\overline{\mathscr{r}}^{\leq N}_{n,j}$.

\textbf{(c) $n \in I^{3N}\setminus I^N$.} Here $P^{-1}$ acts as $(P^N)^{-1}$ on $I^N$ and as the identity elsewhere, so
\[
\overline{\mathscr{r}}_{n,j} = \frac{1}{|\mathrm{orb}_{\mathcal{G}}(n)|}\Bigl(\sum_{k\in I^N}\sum_i|(((P^N)^{-1}Dg)_{i,j})_{k,n}|\,|\mathrm{orb}_{\mathcal{G}}(k)| + \sum_{k \notin I^N,\,(k,i)\neq(n,j)}|((Dg)_{i,j})_{k,n}|\,|\mathrm{orb}_{\mathcal{G}}(k)|\Bigr).
\]
For the second sum we use that, for a bounded operator $A$ on $\ell^1_{\mathcal{G}}$, the unfolded basis vector $e_n$ has $\|e_n\|_1 = |\mathrm{orb}_{\mathcal{G}}(n)|$, so that the $\ell^1_{\mathcal{G}}$-operator norm is the supremum of the \emph{normalised} column sums:
\begin{equation}\label{eq : column norm}
  \|A\|_{\mathcal{B}(\ell^1_{\mathcal{G}})} = \sup_{n}\frac{1}{|\mathrm{orb}_{\mathcal{G}}(n)|}\sum_{k}|A_{k,n}|\,|\mathrm{orb}_{\mathcal{G}}(k)| .
\end{equation}
Applying this to $A = \mathbb{M}_{\overline{q}_{i,j}}$, whose norm is $\|\overline{q}_{i,j}\|_1$, and summing over $i$ gives
\[
\frac{1}{|\mathrm{orb}_{\mathcal{G}}(n)|}\sum_{i}\sum_{k}\bigl|((Dg)_{i,j})_{k,n}\bigr|\,|\mathrm{orb}_{\mathcal{G}}(k)| \;\leq\; \sum_i\|\overline{q}_{i,j}\|_1 = Q_j ,
\]
and discarding the (non-negative) terms already counted over $I^N$ yields $\overline{\mathscr{r}}_{n,j}\leq\overline{\mathscr{r}}^{3N\setminus N}_{n,j}$. We stress that it is the \emph{normalised} column sum that is bounded by $Q_j$: the unnormalised one can exceed it by as much as a factor $\mathscr{O}_{\mathrm{max}}$, which is why $Q_j$ appears in $\overline{\mathscr{r}}^{3N\setminus N}_{n,j}$ outside the factor $1/|\mathrm{orb}_{\mathcal{G}}(n)|$.

\textbf{(d) $n \notin I^{3N}$.} Here $P^{-1}$ and $P$ act as the identity, so $\overline{\mathcal{D}}_{i,j} = (l_{\overline{\blambda}})_{i,j}+(Dg)_{i,j}$ on the relevant columns and \eqref{eq : column norm} gives directly $\overline{\mathscr{r}}_{n,j}\leq Q_j$, the diagonal part $l_{\overline{\blambda}}$ contributing nothing off the diagonal.
\end{proof}

\subsection{Proof of Proposition~\ref{prop : gershgorin enclosure}}\label{apen : gershgorin enclosure}
\begin{proof}
Write $\overline{B}_{n,j} \bydef \{z : |z-(\overline{\mathcal{D}}_{j,j})_{n,n}|\leq\rho_{n,j}\}$, so that $B_{n,j}\subset\overline{B}_{n,j}$ by Lemma~\ref{lem : computable radii}; by Lemma~\ref{lem : gershgorin} it suffices to argue with the disks $\overline{B}_{n,j}$.

A disk $\overline{B}_{n,j}$ meets the imaginary axis if and only if its centre is at distance at most $\rho_{n,j}$ from that axis: minimising $|z-(\overline{\mathcal{D}}_{j,j})_{n,n}|$ over $z = i\varsigma$, $\varsigma\in\R$, at $\varsigma = \mathrm{Im}\,(\overline{\mathcal{D}}_{j,j})_{n,n}$ gives the distance $|\mathrm{Re}\,(\overline{\mathcal{D}}_{j,j})_{n,n}|$. Hypothesis (i) therefore says that, among the disks with $n \in I^{3N}$, only $\overline{B}_{n_0,j_0}$ can meet the imaginary axis.

For $n \notin I^{3N}$, every $z \in \overline{B}_{n,j}$ satisfies $\mathrm{Re}(z) \leq \mathrm{Re}\,(\overline{\mathcal{D}}_{j,j})_{n,n} + \rho_{n,j} = \mathrm{Re}\,(\overline{\mathcal{D}}_{j,j})_{n,n} + Q_j + \mathscr{r}^{\infty}$, which is negative by hypothesis (ii); so all tail disks lie in the open left half-plane.

Hence the imaginary axis meets $\overline{B}_{n_0,j_0}$ only. By hypothesis (iii) that disk is disjoint from all the others, so by Lemma~\ref{lem : gershgorin} it contains exactly one eigenvalue of $D_uf(\tilde{\blambda},\tilde{\mathbf{u}})$, counted with algebraic multiplicity. Since $D_uf(\tilde{\blambda},\tilde{\mathbf{u}})\mathbf{v}=0$ with $\mathbf{v}\neq0$, the value $0$ is an eigenvalue, and it must be the one in $\overline{B}_{n_0,j_0}$; it is algebraically simple, and every other eigenvalue lies off the imaginary axis. Finally, each remaining disk lies entirely in one open half-plane, so the eigenvalues with positive real part are exactly those contained in the disks lying in $\{\mathrm{Re}(z)>0\}$. Under (iv), every $z$ in a disk with $n \in I^{3N}$, $(n,j)\neq(n_0,j_0)$, has $\mathrm{Re}(z) \leq \mathrm{Re}\,(\overline{\mathcal{D}}_{j,j})_{n,n}+\rho_{n,j}<0$, and the tail disks are already in the left half-plane by (ii); hence no disk meets $\{\mathrm{Re}(z)>0\}$ and $k=0$.
\end{proof}

\section{Computational details for Gray-Scott}
In this appendix, we provide the computational details for the bounds required for the Gray-Scott system of PDEs. That is, we prove Lemmas \ref{lem : Z2 GS} and \ref{lem : Z1 full GS}. We begin with Lemma \ref{lem : Z2 GS}.
\subsection{Proof of Lemma \ref{lem : Z2 GS}}\label{apen : Z2 GS}
We define $L_{\mathbf{s},\boldsymbol{\lambda}}$ as in \eqref{def L G and Phi} but with the linear part for Gray-Scott, and 
{\footnotesize\begin{align}
    G(\mathbf{x}) \bydef \begin{bmatrix}
        (u_2 + 1)u_1^2 \\
        -(u_2 + 1) u_1^2 \\
        2(u_2 + 1)u_1 v_1 + u_1^2 v_2 \\
        -2(u_2 + 1)u_1 v_1 - u_1^2 v_2 \\
        \mathcal{O}_{\mathcal{G}}((2(u_2 + 1)u_1)\star (\mathcal{O}_{\mathcal{G}}^{-1}(w_1 - w_2))) \\
        \mathcal{O}_{\mathcal{G}}(u_1^2\star (\mathcal{O}_{\mathcal{G}}^{-1} (w_1 - w_2))) \\
        2(u_2 + 1) u_1 h_1 + u_1^2 h_2 + 2(u_2 + 1)v_1^2 + 4u_1 v_1 v_2 \\
        -2(u_2 + 1) u_1 h_1 - u_1^2 h_2 - 2(u_2 + 1)v_1^2 - 4u_1 v_1 v_2
    \end{bmatrix}, ~ \varphi(\mathbf{x}) \bydef \begin{bmatrix}
        (\mathbf{v},\mathbf{v})_{2} - 1 \\
        (\mathbf{v},\mathbf{w})_{2} - 1 \\
        (\mathbf{h},\mathbf{w})_{2} \\
        (2(u_2 + 1)v_1^2 + 4 u_1 v_1 v_2,w_1-w_2)_{2}
    \end{bmatrix},
\end{align}}
so that
\begin{align}
    F(\mathbf{x}) = \begin{bmatrix}
        \varphi(\mathbf{x}) \\
        L_{\mathbf{s},\boldsymbol{\lambda}} (u,v,w,h) + G(\mathbf{x})
    \end{bmatrix}.
\end{align}
We will now prove Lemma \ref{lem : Z2 GS}.
\begin{proof}
Similarly as to what was done in Lemma \ref{lem : Z2}, we simplify the computation by \begin{align}
    \|A(DF(\mathbf{x}) - DF(\overline{\mathbf{x}}))\|_{\mathcal{B}(X_{\mathcal{G},\omega})} &\leq \|A\|_{\mathcal{B}(X_{\mathcal{G},\omega})} \|DF(\mathbf{x}) - DF(\overline{\mathbf{x}})\|_{\mathcal{B}(X_{\mathcal{G},\omega})} \\
    &\leq (\|A^N\|_{\mathcal{B}(X_{\mathcal{G},\omega})} + \mathcal{L}_{\infty})\|DF(\mathbf{x}) - DF(\overline{\mathbf{x}})\|_{\mathcal{B}(X_{\mathcal{G},\omega})}.
\end{align}
First, for simplicity, let $\mathcal{F}(\mathbf{x}) \bydef L_{\mathbf{s},\blambda}(u,v,w,h) + G(\mathbf{x})$
Now, observe that
{\small\begin{align}
    &\|DF(\mathbf{x}) - DF(\overline{\mathbf{x}})\|_{\mathcal{B}(X_{\mathcal{G},\omega})} = \left\| \begin{bmatrix}
        0 & D\varphi(\mathbf{x}) - D\varphi(\overline{\mathbf{x}}) \\
        \partial_{(\mathbf{s},\blambda)}\mathcal{F}(\mathbf{x}) - \partial_{(\mathbf{s},\blambda)}\mathcal{F}(\overline{\mathbf{x}}) & L_{\mathbf{s},\blambda} + DG(\mathbf{x}) - L_{\mathbf{0},\overline{\blambda}} - DG(\overline{\mathbf{x}})
    \end{bmatrix}\right\|_{\mathcal{B}(X_{\mathcal{G},\omega})} \\
    &\leq  \|D\varphi(\mathbf{x}) - D\varphi(\overline{\mathbf{x}})\|_{\mathcal{B}((\ell^2_{\mathcal{G},\omega})^8,\mathbb{R}^4)} + \|\partial_{(\mathbf{s},\blambda)}\mathcal{F}(\mathbf{x}) - \partial_{(\mathbf{s},\blambda)}\mathcal{F}(\overline{\mathbf{x}})\|_{\mathcal{B}(\mathbb{R}^4,(\ell^2_{\mathcal{G},\omega})^8)} \\
    &\hspace{+1cm}+ \|L_{\mathbf{s},\blambda} + DG(\mathbf{x}) - L_{\mathbf{0},\overline{\blambda}} - DG(\overline{\mathbf{x}})\|_{\mathcal{B}((\ell^2_{\mathcal{G},\omega})^8)} \\
    &\leq \|D\varphi(\mathbf{x}) - D\varphi(\overline{\mathbf{x}})\|_{\mathcal{B}((\ell^2_{\mathcal{G},\omega})^8,\mathbb{R}^4)} + \|\partial_{(\mathbf{s},\blambda)}\mathcal{F}(\mathbf{x}) - \partial_{(\mathbf{s},\blambda)}\mathcal{F}(\overline{\mathbf{x}})\|_{\mathcal{B}(\mathbb{R}^4,(\ell^2_{\mathcal{G},\omega})^8)} \\
    &\hspace{+1cm}+ \|L_{\mathbf{s},\blambda} - L_{\mathbf{0},\overline{\blambda}}\|_{\mathcal{B}((\ell^2_{\mathcal{G},\omega})^8)} + \|DG(u,v,w,h) - DG(\overline{u},\overline{v},\overline{w},\overline{h})\|_{\mathcal{B}((\ell^2_{\mathcal{G},\omega})^8)},\label{Z2 split up GS}
\end{align}}
where we used the fact that $DG(\mathbf{x}) = DG(u,v,w,h)$ as there is no dependency on the parameters. We now examine each term individually. To begin, we write $\mathbf{x} = \mathbf{t} + \overline{\mathbf{x}}$ for some $\mathbf{t} = (t_1,t_2,\dots,t_{12}) \in B_r(0)$. Then, observe that
\begin{align}
    D\varphi(\mathbf{t} + \overline{\mathbf{x}}) - D\varphi(\overline{\mathbf{x}}) = \begin{bmatrix}
        0 & 0 & 2t_7^* & 2t_8^* & 0 & 0 & 0 & 0 \\
        0 & 0 & t_9^* & t_{10}^* & t_7^* & t_8^* & 0 & 0 \\
        0 &0 & 0 & 0 & t_{11}^* & t_{12}^* & t_{9}^* & t_{10}^* \\
        \mathscr{Z}_1^* & \mathscr{Z}_2^* & \mathscr{Z}_3^* & \mathscr{Z}_4^* & \mathscr{Z}_5^* & -\mathscr{Z}_5^* & 0 & 0 
    \end{bmatrix},
\end{align}
where
{\footnotesize\begin{align}
    &(\mathscr{Z}_1)_k \bydef 4(\mathrm{conv}(e_k, \overline{v}_1 t_8 + \overline{v}_2 t_7 + t_7 t_8),\overline{w}_1 - \overline{w}_2)_{2} + 4(\mathrm{conv}(e_k, \overline{v}_1\overline{v}_2 + \overline{v}_1 t_8 + \overline{v}_2 t_7 + t_7t_8),t_9-t_{10})_{2}, \\
    &(\mathscr{Z}_2)_k \bydef 2(\mathrm{conv}(e_k,2\overline{v}_1t_7 + t_7^2),\overline{w}_1-\overline{w}_2)_{2} + 2(\mathrm{conv}(e_k,\overline{v}_1^2 + 2\overline{v}_1 t_7 + t_7^2),t_{9} - t_{10})_{2}, \\
    &(\mathscr{Z}_3)_k \bydef 4(\mathrm{conv}(e_k,(\overline{u}_2+1)t_7 + \overline{v}_1 t_6 + t_6 t_7 + \overline{u}_1 t_8 + \overline{v}_2 t_5 + t_5 t_8),\overline{w}_1 - \overline{w}_2)_{2} \\
    &\hspace{+1cm}+ 4(\mathrm{conv}(e_k,(\overline{u}_2 + 1)\overline{v}_1 + (\overline{u}_2+1)t_7 + \overline{v}_1 t_6 + t_6 t_7 + \overline{u}_1\overline{v}_2 + \overline{u}_1 t_8 + \overline{v}_2 t_5 + t_5 t_8),t_9 - t_{10})_{2}, \\
    &(\mathscr{Z}_4)_k \bydef 4(\mathrm{conv}(e_k,\overline{u}_1t_7 + \overline{v}_1 t_5 + t_5 t_7),\overline{w}_1 - \overline{w}_2)_{2} + 4(\mathrm{conv}(e_k,\overline{u}_1\overline{v}_1 + \overline{u}_1 t_7 + \overline{v}_1 t_5 + t_5 t_7),t_9 - t_{10})_{2}, \\
    &\mathscr{Z}_5 \bydef 4(\overline{u}_2 + 1)\overline{v}_1 t_7 + 2\overline{v}_1^2 t_6 + 4\overline{u}_1 \overline{v}_1 t_8 + 4\overline{u}_1 \overline{v}_2 t_7 + 4\overline{v}_1 \overline{v}_2 t_5 +2(\overline{u}_2 + 1)t_7^2 + 4\overline{v}_1 t_6 t_7 + 4\overline{u}_1 t_7 t_8 + 4\overline{v}_1 t_5 t_8 + 4\overline{v}_2 t_5 t_7\\
    &\hspace{+1cm}+ 2t_6 t_7^2 + 4t_5 t_7 t_8.
\end{align}}
For the operator norm, suppose we have $\mathbf{p} = (p_1,p_2,\dots,p_8) \in (\ell^2_{\mathcal{G},\omega})^8, \|\mathbf{p}\|_{2,\omega} = 1$ such that the equality holds. We now rely on the inequalities from \eqref{ineq list} and \eqref{ineq list 2} to get
{\scriptsize\begin{align}
    &\|D\varphi(\mathbf{x}) - D\varphi(\overline{\mathbf{x}})\|_{\mathcal{B}((\ell^2_{\mathcal{G},\omega})^8,\mathbb{R}^4)} = 2|(p_3,t_7)_{2}+(p_4,t_8)| + |(p_3,t_9)_{2} + (p_4,t_{10})_{2} + (p_5,t_7)_{2} + (p_6,t_8)_{2}| \\
    &\hspace{+5cm}+ |(p_5,t_{11})_{2} + (p_6,t_{12})_{2} + (p_7,t_9)_{2} + (p_8,t_{10})_{2}|\\
    &\hspace{+5.2cm}+|(p_1,\mathscr{Z}_1)_{2} + (p_2,\mathscr{Z}_2)_{2} + (p_3,\mathscr{Z}_3)_{2} + (p_4,\mathscr{Z}_4)_{2} + (p_5-p_6,\mathscr{Z}_5)_{2}|\\
    &\leq 4r + 4r + 4r + \kappa_0 \|\mathscr{Z}_1\|_{\infty}\|p_1\|_{2,\omega} + \kappa_0 \|\mathscr{Z}_2\|_{\infty}\|p_2\|_{2,\omega} + \kappa_0 \|\mathscr{Z}_3\|_{\infty}\|p_3\|_{2,\omega} + \kappa_0 \|\mathscr{Z}_4\|_{\infty}\|p_4\|_{2,\omega} + \|\mathscr{Z}_5\|_{2}(\|p_5\|_{2} + \|p_6\|_{2}) \\
    &=12r + \kappa_0\|\mathscr{Z}_1\|_{\infty} + \kappa_0\|\mathscr{Z}_2\|_{\infty} + \kappa_0\|\mathscr{Z}_3 \|_{\infty} + \kappa_0 \|\mathscr{Z}_4\|_{\infty} + 2\|\mathscr{Z}_5\|_{2}.\label{GS before scrZ estimates}
\end{align}}
Then, observe that
{\small\begin{align}
    |(\mathscr{Z}_1)_k| &\leq 4\|\mathrm{conv}(e_k,\overline{v}_1t_8 + \overline{v}_2 t_7 + t_7 t_8)\|_{2} \|\overline{w}_1 - \overline{w}_2\|_{2} + 4\|\mathrm{conv}(e_k,\overline{v}_1 \overline{v}_2 + \overline{v}_1 t_8 + \overline{v}_2 t_7 + t_7 t_8)\|_{2} \|t_9 - t_{10}\|_{2} \\
    &\leq 4(\|\overline{v}_1\|_{1}r + \|\overline{v}_2\|_{1} r + \kappa r^2) \|\overline{w}_1 - \overline{w}_2 \|_{2} + 4(\|\overline{v}_1 \overline{v}_2\|_{2} + \|\overline{v}_1\|_{1}r + \|\overline{v}_2\|_{1}r + \kappa r^2)(r + r) \\
    &= \left(4(\|\overline{v}_1\|_{1} + \|\overline{v}_2\|_{1} + \kappa r)\|\overline{w}_1 - \overline{w}_2\|_{2} + 8(\|\overline{v}_1\overline{v}_2\|_{2} + \|\overline{v}_1\|_{1} r + \|\overline{v}_2\|_{1}r + \kappa r^2)\right)r.
\end{align}}
Hence, 
\begin{align}
    \|\mathscr{Z}_1\|_{\infty} \leq \left((4\|\overline{v}_1\|_{1} + 4\|\overline{v}_2\|_{1} + 4\kappa r)\|\overline{w}_1 - \overline{w}_2\|_{2} + 8\|\overline{v}_1\overline{v}_2\|_{2} + 8\|\overline{v}_1\|_{1} r + 8\|\overline{v}_2\|_{1}r + 8\kappa r^2\right)r.
\end{align}
Moving to $\mathscr{Z}_2$, we get
\begin{align}
    |(\mathscr{Z}_2)_k| &\leq 2\|\mathrm{conv}(e_k,2\overline{v}_1 t_7 + t_7^2)\|_{2} \|\overline{w}_1 - \overline{w}_2\|_{2} + 2\|\mathrm{conv}(e_k,\overline{v}_1^2 + 2\overline{v}_1 t_7 + t_7^2)\|_{2} \|t_9 - t_{10}\|_{2} \\
    &\leq 2\left((2\|\overline{v}_1\|_{1} r + \kappa r^2)\|\overline{w}_1 - \overline{w}_2\|_{2} + (\|\overline{v}_1^2\|_{2} + 2\|\overline{v}_1\|_{1} r + \kappa r^2)(r+r)\right) \\
    &= 2\left((2\|\overline{v}_1\|_{1} + \kappa r)\|\overline{w}_1 - \overline{w}_2\|_{2} + 2(\|\overline{v}_1^2\|_{2} + 2\|\overline{v}_1\|_{1}r + \kappa r^2)\right)r.
\end{align}
Hence,
\begin{align}
    \|\mathscr{Z}_{2}\|_{\infty} &\leq \left((4\|\overline{v}_1\|_{1} + 2\kappa r)\|\overline{w}_1 - \overline{w}_2\|_{2} + 4\|\overline{v}_1^2\|_{2} + 8\|\overline{v}_1\|_{1}r + 4\kappa r^2\right)r.
\end{align}
Moving to $\mathscr{Z}_3$, we get
{\small\begin{align}
    |(\mathscr{Z}_3)_k| &\leq 4\biggl( \|\mathrm{conv}(e_k,(\overline{u}_2 + 1)t_7 + \overline{v}_1 t_6 + t_6 t_7 + \overline{u}_1 t_8 + \overline{v}_2 t_5 + t_5 t_8)\|_{2} \|\overline{w}_1 - \overline{w}_2\|_{2} \\
    &\hspace{+0.1cm}+\|\mathrm{conv}(e_k,(\overline{u}_2 + 1)\overline{v}_1 + (\overline{u}_2 + 1)t_7 + \overline{v}_1 t_6 + t_6 t_7 + \overline{u}_1 \overline{v}_2+ \overline{u}_1 t_8 + \overline{v}_2 t_5 + t_5 t_8)\|_{2} \|t_9 - t_{10}\|_{2}\biggr) \\
    &\leq 4\biggl( (\|\overline{u}_2 + 1\|_{1} r + \|\overline{v}_1\|_{1}r + \kappa r^2 + \|\overline{u}_1\|_{1}r + \|\overline{v}_2\|_{1} r + \kappa r^2)\|\overline{w}_1 - \overline{w}_2\|_{2} \\
    &\hspace{+0.1cm}+ (\|(\overline{u}_2 + 1)\overline{v}_1\|_{2} + \|\overline{u}_2 + 1\|_{1} r + \|\overline{v}_1\|_{1}r + \kappa r^2 + \|\overline{u}_1 \overline{v}_2\|_{1} + \|\overline{u}_1\|_{1}r + \|\overline{v}_2\|_{1}r + \kappa r^2)(r + r)\biggr) \\
    &\leq 4\biggl( (\|\overline{u}_2 + 1\|_{1}  + \|\overline{v}_1\|_{1} + 2\kappa r + \|\overline{u}_1\|_{1} + \|\overline{v}_2\|_{1} )\|\overline{w}_1 - \overline{w}_2\|_{2} \\
    &\hspace{+0.1cm}+ 2(\|(\overline{u}_2 + 1)\overline{v}_1\|_{2} + \|\overline{u}_2 + 1\|_{1} r + \|\overline{v}_1\|_{1}r + 2\kappa r^2 + \|\overline{u}_1 \overline{v}_2\|_{1} + \|\overline{u}_1\|_{1}r + \|\overline{v}_2\|_{1}r)\biggr)r.
\end{align}}
Hence,
{\small\begin{align}
    \|\mathscr{Z}_3\|_{\infty} &\leq \biggl( (4\|\overline{u}_2 + 1\|_{1}  + 4\|\overline{v}_1\|_{1} + 8\kappa r + 4\|\overline{u}_1\|_{1} + 4\|\overline{v}_2\|_{1} )\|\overline{w}_1 - \overline{w}_2\|_{2}\\
    &\hspace{+0.5cm}+ 8\|(\overline{u}_2 + 1)\overline{v}_1\|_{2} + 8\|\overline{u}_2 + 1\|_{1} r + 8\|\overline{v}_1\|_{1}r + 16\kappa r^2 + 8\|\overline{u}_1 \overline{v}_2\|_{1} + 8\|\overline{u}_1\|_{1}r + 8\|\overline{v}_2\|_{1}r\biggr)r.
\end{align}}
Moving to $\mathscr{Z}_4$, we get
{\footnotesize\begin{align}
    |(\mathscr{Z}_4)_k| &\leq 4\|\mathrm{conv}(e_k,\overline{u}_1 t_7 + \overline{v}_1 t_5 + t_5t_7)\|_{2} \|\overline{w}_1 - \overline{w}_2\|_{2} + 4\|\mathrm{conv}(e_k,\overline{u}_1 \overline{v}_1 + \overline{u}_1 t_7 + \overline{v}_1 t_5 + t_5 t_7)\|_{2} \|t_9 - t_{10}\|_{2} \\
    &\leq 4\left((\|\overline{u}_1\|_{1}r + \|\overline{v}_1\|_{1}r + \kappa r^2)\|\overline{w}_1 - \overline{w}_2\|_{2} + (\|\overline{u}_1 \overline{v}_1\|_{2} + \|\overline{u}_1\|_{1}r + \|\overline{v}_1\|_{1}r + \kappa r^2)(r + r)\right) \\
    &= 4\left((\|\overline{u}_1\|_{1} + \|\overline{v}_1\|_{1} + \kappa r)\|\overline{w}_1 - \overline{w}_2\|_{2} + 2(\|\overline{u}_1 \overline{v}_1\|_{2} + \|\overline{u}_1\|_{1}r + \|\overline{v}_1\|_{1}r + \kappa r^2)\right)r.
\end{align}}
Hence,
\begin{align}
    \|\mathscr{Z}_4\|_{\infty} &\leq \left((4\|\overline{u}_1\|_{1} + 4\|\overline{v}_1\|_{1} + 4\kappa r)\|\overline{w}_1 - \overline{w}_2\|_{2} + 8\|\overline{u}_1 \overline{v}_1\|_{2} + 8\|\overline{u}_1\|_{1}r + 8\|\overline{v}_1\|_{1}r + 8\kappa r^2\right)r.
\end{align}
Finally, $\mathscr{Z}_5$ is a more classical estimate.
\begin{align}
    \|\mathscr{Z}_5\|_{2} &\leq 4\|(\overline{u}_2 + 1)\overline{v}_1\|_{1}  r + 2\|\overline{v}_1^2\|_{1}  r + 4\|\overline{u}_1 \overline{v}_1\|_{1} r + 4\|\overline{u}_1 \overline{v}_2\|_{1} r + 4\|\overline{v}_1 \overline{v}_2\|_{1} r + 2\|\overline{u}_2 + 1\|_{1} \kappa r^2 \\
    &\hspace{+0.1cm}+ 4\|\overline{v}_1\|_{1}\kappa r^2 + 4\|\overline{u}_1\|_{1}\kappa r^2 + 4\|\overline{v}_1\|_{1}\kappa r^2 + 4\|\overline{v}_2\|_{1} \kappa r^2 + 2\kappa^2 r^3 + 4\kappa^2 r^3 \\
    &\leq \biggl(4\|(\overline{u}_2 + 1)\overline{v}_1\|_{1} + 2\|\overline{v}_1^2\|_{1} + 4\|\overline{u}_1 \overline{v}_1\|_{1} + 4\|\overline{u}_1 \overline{v}_2\|_{1} + 4\|\overline{v}_1 \overline{v}_2\|_{1} + 2\|\overline{u}_2 + 1\|_{1} \kappa r \\
    &\hspace{+0.1cm}+ 8\|\overline{v}_1\|_{1}\kappa r +4\|\overline{v}_2\|_{1}\kappa r + 4\|\overline{u}_1\|_{1}\kappa r +   6\kappa^2 r^2\biggr)r.
\end{align}
Returning to \eqref{GS before scrZ estimates}, we estimate

{\small\begin{align}
    &\|D\varphi(\mathbf{x}) - D\varphi(\overline{\mathbf{x}})\|_{\mathcal{B}((\ell^2_{\mathcal{G},\omega})^8,\mathbb{R}^4)} \leq \biggl(12 + \kappa_0 (18\kappa r + 16\|\overline{v}_1\|_{1} + 8\|\overline{v}_2\|_{1} + 8\|\overline{u}_1\|_{1} + 4\|\overline{u}_2 + 1\|_{1})\|\overline{w}_1 - \overline{w}_2\|_{2}\\
    &\hspace{+0.2cm}+ (36\kappa_0\kappa + 30\kappa^2) r^2 + 8\kappa_0\|\overline{v}_1 \overline{v}_2\|_{2} + 4\|\overline{v}_1\overline{v}_2\|_{1} + (16\kappa + 32\kappa_0)\|\overline{v}_1\|_{1}r + (8\kappa + 16\kappa_0) \|\overline{v}_2\|_{1}r + 4\kappa_0\|\overline{v}_1^2\|_{2} \\
    &\hspace{+0.7cm}+ 4\|\overline{v}_1^2\|_{1} + 8\kappa_0\|(\overline{u}_2 + 1)\overline{v}_1\|_{2} + 8\|(\overline{u}_2 + 1)\overline{v}_1\|_{1} + (4\kappa + 8\kappa_0)\|\overline{u}_2 + 1\|_{1}r \\
    &\hspace{+1.0cm}+ (8\kappa_0 + 8)\|\overline{u}_1 \overline{v}_2\|_{1} + (16\kappa_0 + 12\kappa)\|\overline{u}_1\|_{1}r + (8\kappa_0+8)\|\overline{u}_1 \overline{v}_1\|_{1}\biggr)r \\
    &\bydef Z_{2,1}(r)r.
\end{align}}
Moving to the second term of \eqref{Z2 split up GS}, we have
\begin{align}
    &\|\partial_{(\mathbf{s},\blambda)} \mathcal{F}(\mathbf{x}) - \partial_{(\mathbf{s},\blambda)} \mathcal{F}(\overline{\mathbf{x}})\|_{\mathcal{B}(\mathbb{R}^4,(\ell^2_{\mathcal{G},\omega})^8)} &= \left\| \begin{bmatrix}
    0 & 0 & -t_5 & 0 \\
    0 & 0 & 0 & -t_6 \\
    t_7 & 0 & -t_7 & 0 \\
    t_8 & 0 & 0 & -t_8 \\
    0 & 0 & - t_9 & 0 \\
    0 & 0 & 0 & -t_{10} \\
    0 & t_7 & -t_{11} & 0 \\
    0 & t_8 & 0 & -t_{12}
    \end{bmatrix}\right\|_{\mathcal{B}(\mathbb{R}^4,(\ell^2_{\mathcal{G},\omega})^8)} \leq 12r.
\end{align}
For the third term of \eqref{Z2 split up GS}, 
\begin{align}
    \|L_{\mathbf{s},\blambda} - L_{\mathbf{0},\overline{\blambda}}\|_{\mathcal{B}((\ell^2_{\mathcal{G},\omega})^8)} &= \left\|\begin{bmatrix}
        -t_3 & 0 & 0 & 0 & 0 & 0 & 0 & 0 \\
        0 & -t_4 & 0 & 0 & 0 & 0 & 0 & 0 \\
        0 & 0 & t_1 - t_3 & 0 & 0 & 0 & 0 & 0 \\
        0 & 0 & 0 & t_1 - t_4 & 0 & 0 & 0 & 0 \\
        0 & 0 & 0 & 0 & -t_3 & 0 & 0 & 0 \\
        0 & 0 & 0 & 0 & 0 & -t_4 & 0 & 0 \\
        0 & 0 & t_2 & 0 & 0 & 0 & -t_3 & 0 \\
        0 & 0 & 0 & t_2 & 0 & 0 & 0 & -t_4
    \end{bmatrix}\right\|_{\mathcal{B}((\ell^2_{\mathcal{G},\omega})^8)} \leq 12r.
\end{align}
We now move to the fourth and final term of \eqref{Z2 split up GS}. Observe that
{\footnotesize\begin{align}
    &\|DG(\mathbf{x}) - DG(\overline{\mathbf{x}})\|_{\mathcal{B}((\ell^2_{\mathcal{G},\omega})^8)} = \|(DG(\mathbf{x}) - DG(\overline{\mathbf{x}}))\mathbf{p}\|_{2,\omega}
     \\
    &\leq 6\kappa \|2u_1(u_2 + 1) - 2\overline{u}_1(\overline{u}_2 + 1)\|_{2,\omega} + 6\kappa \|u_1^2 - \overline{u}_1^2\|_{2,\omega} \\
    &\hspace{+0.1cm}+ 6\kappa \|2(u_2 + 1)v_1 + 2u_1 v_2 - 2(\overline{u}_2 + 1)\overline{v}_1 - 2\overline{u}_1 \overline{v}_2\|_{2,\omega} + 12\kappa \|u_1 v_1 - \overline{u}_1  \overline{v}_1\|_{2,\omega}\\
    &\hspace{+0.2cm}+ 2\kappa \|2(u_2 + 1)h_1 + 2u_1 h_2 + 4v_1 v_2 - 2(\overline{u}_2 + 1)\overline{h}_1 - 2\overline{u}_1 \overline{h}_2 - 4\overline{v}_1 \overline{v}_2\|_{2,\omega} + 2\kappa \|2u_1 h_1 + 2v_1^2 - 2\overline{u}_1 \overline{h}_1 - 2\overline{v}_1^2\|_{2,\omega} \\
    &\hspace{+0.2cm}+2\|\mathbb{M}_{2u_1(u_2 + 1) - 2\overline{u}_1(\overline{u}_2 + 1)}^*(p_5 - p_6)\|_{2,\omega} + 2\|\mathbb{M}_{u_1^2 - \overline{u}_1^2}^*(p_5 - p_6)\|_{2,\omega} \\
    &\hspace{+0.3cm}+\|\mathbb{Q}_{\mathcal{O}_{\mathcal{G}}\mathrm{conv}(2(u_2 + 1),\mathcal{O}_{\mathcal{G}}^{-1}(w_1-w_2)) - \mathcal{O}_{\mathcal{G}}\mathrm{conv}(2(\overline{u}_2 + 1),\mathcal{O}_{\mathcal{G}}^{-1}(\overline{w}_1-\overline{w}_2))}p_1\|_{2,\omega} \\
    &\hspace{+0.4cm}+ \|\mathbb{Q}_{\mathcal{O}_{\mathcal{G}}\mathrm{conv}(2u_1, \mathcal{O}_{\mathcal{G}}^{-1}(w_1 - w_2)) - \mathcal{O}_{\mathcal{G}}\mathrm{conv}(2\overline{u}_1,\mathcal{O}_{\mathcal{G}}^{-1}(\overline{w}_1 - \overline{w}_2))}p_2\|_{2,\omega} \\
    &\hspace{+0.5cm}+ \|\mathbb{Q}_{\mathcal{O}_{\mathcal{G}}\mathrm{conv}(2u_1, \mathcal{O}_{\mathcal{G}}^{-1}(w_1 - w_2)) - \mathcal{O}_{\mathcal{G}}\mathrm{conv}(2\overline{u}_1, \mathcal{O}_{\mathcal{G}}^{-1}(\overline{w}_1 - \overline{w}_2))}p_1\|_{2,\omega}.
\end{align}}
Now, observe that for some $a,b,c \in \ell^2_{\mathcal{G},\omega}$ and $\|b\|_{2,\omega},\|c\|_{2,\omega} \leq 1$, we have
\begin{align}
    \|\mathbb{M}_a^* (b-c)\|_{2,\omega} &= \|\mathcal{O}_{\mathcal{G}} \mathrm{conv}(a, \mathcal{O}^{-1} (b-c))\|_{2,\omega} \\
    &\leq \|\mathcal{O}_{\mathcal{G}}\|_{\mathcal{B}(\ell^2_{\mathcal{G},\omega})} \|\mathrm{conv}(a, \mathcal{O}^{-1} (b-c))\|_{2,\omega} \\
    &\leq \kappa \mathscr{O}_{\mathrm{max}} \|a\|_{2,\omega} \|\mathcal{O}^{-1} (b-c)\|_{2,\omega} \\
    &\leq \kappa \mathscr{O}_{\mathrm{max}} \|a\|_{2,\omega} \|\mathcal{O}_{\mathcal{G}}^{-1}\|_{\mathcal{B}(\ell^2_{\mathcal{G},\omega})}\| b-c\|_{2,\omega} \\
    &\leq \kappa \mathscr{O}_{\mathrm{max}} \|a\|_{2,\omega}(\|b\|_{2,\omega} + \|c\|_{2,\omega}) \\
    &= 2\kappa \mathscr{O}_{\mathrm{max}} \|a\|_{2,\omega}.
\end{align}
Also we have
\begin{align}
    \|\mathbb{Q}_{\mathcal{O}_{\mathcal{G}}a} b\|_{2,\omega} &= \|\mathcal{O}_{\mathcal{G}}\mathrm{conv}(\mathcal{O}_{\mathcal{G}}^{-1}\mathcal{O}_{\mathcal{G}}a, b)\|_{2,\omega} \\
    &\leq \|\mathscr{O}_{\mathcal{G}}\|_{\mathcal{B}(\ell^2_{\mathcal{G},\omega})} \|\mathrm{conv}(a,b)\|_{2,\omega} \\
    &\leq \kappa \mathscr{O}_{\mathrm{max}} \|a\|_{2,\omega} \|b\|_{2,\omega} \\
    &= \kappa \mathscr{O}_{\mathrm{max}} \|a\|_{2,\omega}
\end{align}
Hence, we must estimate
\begin{align}
    &\|DG(\mathbf{x}) - DG(\overline{\mathbf{x}})\|_{\mathcal{B}((\ell^2_{\mathcal{G},\omega})^8)} \\
    &\leq \kappa\biggl( (12+8\mathscr{O}_{\mathrm{max}})\|u_1(u_2 + 1) - \overline{u}_1(\overline{u}_2 + 1)\|_{2,\omega} +  (6+4\mathscr{O}_{\mathrm{max}})\|u_1^2 - \overline{u}_1^2\|_{2,\omega} \\
    &\hspace{+0.1cm}+ 12\|(u_2 + 1)v_1 + u_1 v_2 - (\overline{u}_2 + 1)\overline{v}_1 - \overline{u}_1 \overline{v}_2\|_{2,\omega} + 12\|u_1 v_1 - \overline{u}_1  \overline{v}_1\|_{2,\omega}\\
    &\hspace{+0.2cm}+ 4\|(u_2 + 1)h_1 + u_1 h_2 + 2v_1 v_2 - (\overline{u}_2 + 1)\overline{h}_1 - \overline{u}_1 \overline{h}_2 - 2\overline{v}_1 \overline{v}_2\|_{2,\omega} \\
    &\hspace{+0.3cm}+ 4\|u_1 h_1 + v_1^2 - \overline{u}_1 \overline{h}_1 - \overline{v}_1^2\|_{2,\omega} \\
    &\hspace{+0.4cm}+2\mathscr{O}_{\mathrm{max}}\|\mathrm{conv}(u_2 + 1, \mathcal{O}_{\mathcal{G}}^{-1}(w_1-w_2)) - \mathrm{conv}(\overline{u}_2 + 1, \mathcal{O}_{\mathcal{G},N}^{-1}(\overline{w}_1-\overline{w}_2))\|_{2,\omega} \\
    &\hspace{+0.5cm}+ 4\mathscr{O}_{\mathrm{max}}\|\mathrm{conv}(u_1, \mathcal{O}_{\mathcal{G}}^{-1}(w_1 - w_2)) - \mathrm{conv}(\overline{u}_1,\mathcal{O}_{\mathcal{G},N}^{-1}(\overline{w}_1 - \overline{w}_2))\|_{2,\omega}\biggr).
\end{align}
We now estimate term by term. The first term,
\begin{align}
    \|u_1(u_2 + 1) - \overline{u}_1(\overline{u}_2 + 1)\|_{2,\omega} &= \| \overline{u}_1 t_6 +  (\overline{u}_2+1) t_5 + t_5 t_6\|_{2,\omega} \\
    &\leq  \| \overline{u}_1\|_{1,\sqrt{\omega}} \|t_6\|_{2,\omega} +  \|(\overline{u}_2+1)\|_{1,\sqrt{\omega}} \|t_5\|_{2,\omega} + \kappa \|t_5\|_{2,\omega}\| t_6\|_{2,\omega}\\
    &\leq  \|\overline{u}_1\|_{1,\sqrt{\omega}}r + \|\overline{u}_2+1\|_{1,\sqrt{\omega}}r + \kappa r^2. \end{align}
   
The second term is next,
\begin{align}
    &\|u_1^2 - \overline{u}_1^2\|_{2,\omega} = \|2\overline{u}_1t_5 + t_5^2\|_{2,\omega} \leq 2\|\overline{u}_1\|_{1,\sqrt{\omega}} \|t_5\|_{2,\omega} + \kappa \|t_5\|_{2,\omega}^2 \leq  2\|\overline{u}_1\|_{1,\sqrt{\omega}}r + \kappa r^2.\end{align}
Now for the third term, 
{\tiny\begin{align}
    \|(u_2 + 1)v_1 + u_1 v_2 - (\overline{u}_2 + 1)\overline{v}_1 - \overline{u}_1 \overline{v}_2\|_{2,\omega} &= \|(\overline{u}_2+1) t_7 + \overline{v}_1 t_6 + t_6 t_7 + \overline{u}_1 t_8 + \overline{v}_2 t_5 + t_5 t_8\|_{2,\omega}  \\&\hspace{-5cm}\leq \|(\overline{u}_2+1)\|_{1,\sqrt{\omega}} \|t_7\|_{2,\omega} + \|\overline{v}_1\|_{1,\sqrt{\omega}} \|t_6\|_{2,\omega} + \kappa \|t_6\|_{2,\omega} \|t_7\|_{2,\omega} + \|\overline{u}_1\|_{1,\sqrt{\omega}} \|t_8\|_{2,\omega} + \|\overline{v}_2\|_{1,\sqrt{\omega}}\|t_5\|_{2,\omega} + \kappa \|t_5\|_{2,\omega} \|t_8\|_{2,\omega} \\
    &\leq   \|\overline{u}_2+1\|_{1,\sqrt{\omega}}r +  \|\overline{v}_1\|_{1,\sqrt{\omega}} r + 2\kappa r^2 +  \|\overline{u}_1\|_{1,\sqrt{\omega}}r +  \|\overline{v}_2\|_{1,\sqrt{\omega}}r.
\end{align}}
Moving to the fourth term,
\begin{align}
    \|u_1v_1 - \overline{u}_1 \overline{v}_1\|_{2,\omega} &= \|\overline{u}_1 t_7 + \overline{v}_1 t_5 + t_5 t_7\|_{2,\omega} \\
    &\leq \|\overline{u}_1\|_{1,\sqrt{\omega}} \|t_7\|_{2,\omega} + \|\overline{v}_1\|_{1,\sqrt{\omega}} \|t_5\|_{2,\omega} + \kappa \|t_5\|_{2,\omega} \|t_7\|_{2,\omega} \\
    &\leq  \|\overline{u}_1\|_{1,\sqrt{\omega}}r +  \|\overline{v}_1\|_{1,\sqrt{\omega}}r + \kappa r^2.
\end{align}
We examine the fifth term next.
\begin{align}
    &\|(u_2 + 1)h_1 + u_1 h_2 + 2v_1v_2 - (\overline{u}_2 + 1)\overline{h}_1 - \overline{u}_1 \overline{h}_2 - 2\overline{v}_1 \overline{v}_2\|_{2,\omega} \\
    &= \|t_{11} + \overline{u}_2 t_{11} + \overline{h}_1 t_6 + t_6 t_{11} + \overline{u}_1 t_{12} + \overline{h}_2 t_5 + t_5 t_{12} + 2\overline{v}_1 t_8 + 2\overline{v}_2 t_7 + 2t_7t_8\|_{2,\omega} \\
    &\leq   \|\overline{u}_2+1\|_{1,\sqrt{\omega}} r +  \|\overline{h}_1\|_{1,\sqrt{\omega}}r + 4\kappa r^2 +  \|\overline{u}_1\|_{1,\sqrt{\omega}}r +  \|\overline{h}_2\|_{1,\sqrt{\omega}} r + 2 \|\overline{v}_1\|_{1,\sqrt{\omega}}r + 2\|\overline{v}_2\|_{1,\sqrt{\omega}}r.
\end{align}
The sixth term will follow,
{\scriptsize\begin{align}
    \|u_1h_1 + v_1^2 - \overline{u}_1 \overline{h}_1 - \overline{v}_1^2\|_{2,\omega} &= \|\overline{u}_1 t_{11} + \overline{h}_1 t_5 + t_5 t_{11} + 2\overline{v}_1 t_7 + t_7^2\|_{2,\omega} \leq  \|\overline{u}_1\|_{1,\sqrt{\omega}}r +  \|\overline{h}_1\|_{1,\sqrt{\omega}}r + 2\|\overline{v}_1\|_{1,\sqrt{\omega}} + 2\kappa r^2.
\end{align}}
Now we come to the seventh term, where we must handle the operator $\mathcal{O}_{\mathcal{G}}^{-1}$.
{\small\begin{align}
    &\|\mathrm{conv}(u_2 + 1, \mathcal{O}_{\mathcal{G}}^{-1}(w_1-w_2)) - \mathrm{conv}(\overline{u}_2 + 1,\mathcal{O}_{\mathcal{G},N}^{-1}(\overline{w}_1-\overline{w}_2))\|_{2,\omega} \\
    &= \|\mathrm{conv}(\overline{u}_2 + t_6 + 1, \mathcal{O}_{\mathcal{G}}^{-1}(\overline{w}_1 - \overline{w}_2 + t_9 - t_{10})) - \mathrm{conv}(\overline{u}_2 + 1,\mathcal{O}_{\mathcal{G},N}^{-1}(\overline{w}_1-\overline{w}_2))\|_{2,\omega} \\
    &= \|\mathrm{conv}(t_6,\mathcal{O}_{\mathcal{G},N}^{-1}(\overline{w}_1-\overline{w}_2)) + \mathrm{conv}(\overline{u}_2 + 1,\mathcal{O}_{\mathcal{G}}^{-1}(t_9 - t_{10})) + \mathrm{conv}(t_6,\mathcal{O}_{\mathcal{G}}^{-1}(t_9 - t_{10}))\|_{2,\omega} \\
    &\leq \|t_6\|_{2,\omega}\|\mathcal{O}_{\mathcal{G},N}^{-1}(\overline{w}_1-\overline{w}_2)\|_{1,\sqrt{\omega}} + \|\overline{u}_2 + 1\|_{1,\sqrt{\omega}} \|\mathcal{O}_{\mathcal{G}}^{-1}(t_9 - t_{10})\|_{2,\omega} + \kappa \|t_6\|_{2,\omega} \|\mathcal{O}_{\mathcal{G}}^{-1}(t_9 - t_{10})\|_{2,\omega} \\
    &\leq  \|\mathcal{O}_{\mathcal{G},N}^{-1}(\overline{w}_1-\overline{w}_2)\|_{1,\sqrt{\omega}}r + \|\overline{u}_2 + 1\|_{1,\sqrt{\omega}} \|\mathcal{O}_{\mathcal{G}}^{-1}\|_{\mathcal{B}(\ell^2_{\mathcal{G},\omega})}\|(t_9 - t_{10})\|_{2,\omega} +   \kappa  \|\mathcal{O}_{\mathcal{G}}^{-1}\|_{\mathcal{B}(\ell^2_{\mathcal{G},\omega})}\|(t_9 - t_{10})\|_{2,\omega}r \\
    &\leq \|\mathcal{O}_{\mathcal{G},N}^{-1}(\overline{w}_1-\overline{w}_2)\|_{1,\sqrt{\omega}}r + \|\overline{u}_2 + 1\|_{1,\sqrt{\omega}}( r+ r) +  \kappa r( r+ r) \\
    &= \|\mathcal{O}_{\mathcal{G},N}^{-1}(\overline{w}_1-\overline{w}_2)\|_{1,\sqrt{\omega}}r + 2 \|\overline{u}_2 + 1\|_{1,\sqrt{\omega}}r + 2\kappa r^2.
\end{align}}
The eighth and final term will also require us to handle $\mathcal{O}_{\mathcal{G}}^{-1}$, which we do similarly to what was done with the seventh term,
\begin{align}
    &\|\mathrm{conv}(u_1, \mathcal{O}_{\mathcal{G}}^{-1}(w_1 - w_2)) - \mathrm{conv}(\overline{u}_1,\mathcal{O}_{\mathcal{G},N}^{-1}(\overline{w}_1 - \overline{w}_2))\|_{2,\omega} \\
    &= \|\mathrm{conv}(\overline{u}_1 + t_5,\mathcal{O}_{\mathcal{G}}^{-1}(\overline{w}_1 - \overline{w}_2 + t_9 - t_{10})) - \mathrm{conv}(\overline{u}_1,\mathcal{O}_{\mathcal{G},N}^{-1}(\overline{w}_1 - \overline{w}_2))\|_{2,\omega} \\
    &= \|\mathrm{conv}(t_5, \mathcal{O}_{\mathcal{G},N}^{-1}(\overline{w}_1 - \overline{w}_2)) + \mathrm{conv}(\overline{u}_1,\mathcal{O}_{\mathcal{G}}^{-1}(t_9  - t_{10})) + \mathrm{conv}(t_5, \mathcal{O}_{\mathcal{G}}^{-1}(t_9 - t_{10}))\|_{2,\omega} \\
    &\leq \|t_5\|_{2,\omega} \|\mathcal{O}_{\mathcal{G},N}^{-1}(\overline{w}_1 - \overline{w}_2)\|_{1,\sqrt{\omega}} + \|\overline{u}_1\|_{1,\sqrt{\omega}} \|\mathcal{O}_{\mathcal{G}}^{-1}(t_9  - t_{10})\|_{2,\omega} + \kappa \|t_5\|_{2,\omega} \|\mathcal{O}_{\mathcal{G}}^{-1}(t_9  - t_{10})\|_{2,\omega} \\
    &\leq \|\mathcal{O}_{\mathcal{G},N}^{-1}(\overline{w}_1 - \overline{w}_2)\|_{1,\sqrt{\omega}}r + \|\overline{u}_1\|_{1,\sqrt{\omega}} \|\mathcal{O}_{\mathcal{G}}^{-1}\|_{\mathcal{B}(\ell^2_{\mathcal{G},\omega})} \|t_9 - t_{10}\|_{2,\omega} +\kappa \|\mathcal{O}_{\mathcal{G}}^{-1}\|_{\mathcal{B}(\ell^2_{\mathcal{G},\omega})} \|t_9 - t_{10}\|_{2,\omega}r \\
    &\leq  \|\mathcal{O}_{\mathcal{G},N}^{-1}(\overline{w}_1 - \overline{w}_2)\|_{1,\sqrt{\omega}}r + \|\overline{u}_1\|_{1,\sqrt{\omega}} ( r +  r) +\kappa r( r+ r) \\
    &=\|\mathcal{O}_{\mathcal{G},N}^{-1}(\overline{w}_1 - \overline{w}_2)\|_{1,\sqrt{\omega}}r + 2 \|\overline{u}_1\|_{1,\sqrt{\omega}} r +2\kappa r^2.
\end{align}
Combining everything, we obtain
{\footnotesize\begin{align}
    \|DG(\mathbf{x}) - DG(\overline{\mathbf{x}})\|_{\mathcal{B}((\ell^2_{\mathcal{G},\omega})^8)} &\leq  (78+24\mathscr{O}_{\mathrm{max}})\kappa^2 r^2 +  (56 + 24\mathscr{O}_{\mathrm{max}})\kappa \|\overline{u}_1\|_{1,\sqrt{\omega}}r + \kappa (28 + 12\mathscr{O}_{\mathrm{max}})\|\overline{u}_2+1\|_{1,\sqrt{\omega}}r \\
    &\hspace{-2cm}+ 40\kappa \|\overline{v}_1\|_{1,\sqrt{\omega}}r + 20\kappa \|\overline{v}_2\|_{1,\sqrt{\omega}}r + 8\kappa \|\overline{h}_1\|_{1,\sqrt{\omega}}r + 4\kappa \|\overline{h}_2\|_{1,\sqrt{\omega}}r +6\kappa \mathscr{O}_{\mathrm{max}} \|\mathcal{O}_{\mathcal{G},N}^{-1}(\overline{w}_1 - \overline{w}_2)\|_{1,\sqrt{\omega}}r \\
    &=\biggl((78+24\mathscr{O}_{\mathrm{max}})\kappa^2 r + (56 + 24\mathscr{O}_{\mathrm{max}})\kappa \|\overline{u}_1\|_{1,\sqrt{\omega}} + (28 + 12\mathscr{O}_{\mathrm{max}})\kappa \|\overline{u}_2+1\|_{1,\sqrt{\omega}} \\
    &\hspace{-2cm}+ 40\kappa \|\overline{v}_1\|_{1,\sqrt{\omega}} + 20\kappa \|\overline{v}_2\|_{1,\sqrt{\omega}} + 8\kappa \|\overline{h}_1\|_{1,\sqrt{\omega}} + 4\kappa \|\overline{h}_2\|_{1,\sqrt{\omega}} +6\kappa \mathscr{O}_{\mathrm{max}} \|\mathcal{O}_{\mathcal{G},N}^{-1}(\overline{w}_1 - \overline{w}_2)\|_{1,\sqrt{\omega}}\biggr)r \\
    &\bydef Z_{2,2}(r)r
\end{align}}
as desired.
\end{proof}
\subsection{Proof of Lemma \ref{lem : Z1 full GS}}\label{apen : Z1 GS}
As done for Swift Hohenberg, we define $M(\overline{\mathbf{x}})$ and $M^N(\overline{\mathbf{x}})$ in a similar way, but for the Gray-Scott model. We omit their exact definitions. In this Appendix, we prove Lemma \ref{lem : Z1 full GS}.
\begin{proof}
We begin in the same fashion as that for Lemma \ref{lem : Z1 full}. We use \eqref{L bar} to get
\begin{align}
    \|I_d - ADF(\overline{\mathbf{x}})\|_{\mathcal{B}(X_{\mathcal{G},\omega})} &= \|I_d - A(\overline{L}_{\mathbf{0},\overline{\blambda}} + M(\overline{\mathbf{x}}))\|_{\mathcal{B}(X_{\mathcal{G},\omega})} \\
    &\leq \|I_d - A(\overline{L}_{\mathbf{0},\overline{\blambda}} + M^N(\overline{\mathbf{x}}))\|_{\mathcal{B}(X_{\mathcal{G},\omega})} + \|A(M(\overline{\mathbf{x}}) - M^N(\overline{\mathbf{x}}))\|_{\mathcal{B}(X_{\mathcal{G},\omega})} \\
    &\hspace{-2.8cm}\leq \|I_d - A(\overline{L}_{\mathbf{0},\overline{\blambda}} + M^N(\overline{\mathbf{x}}))\|_{\mathcal{B}(X_{\mathcal{G},\omega})}  + (\|A^N\|_{\mathcal{B}(X_{\mathcal{G},\omega})} + \mathcal{L}_{\infty})\|M(\overline{\mathbf{x}}) - M^N(\overline{\mathbf{x}})\|_{\mathcal{B}(X_{\mathcal{G},\omega})}\label{Split into Z1 and Zinfty GS}
\end{align}
where the last step followed from the properties of $A$. 
Now, by using similar steps to those performed in the proof of Lemma \ref{lem : Z1 full} to handle the terms involving the adjoint multiplication, we get
the estimate $\|M(\overline{\mathbf{x}}) - M^N(\overline{\mathbf{x}})\|_{\mathcal{B}(X_{\mathcal{G},\omega})} \leq Z_{\infty}$. The remainder of the computation follows exactly the same steps as those done for Lemma \ref{lem : Z1 full}. The difference is that we must estimate
\begin{align}
    \|DG^N(\overline{\mathbf{x}})\|_{\mathcal{B}((\ell^2_{\mathcal{G},\omega})^8)} &\leq (6+2\mathscr{O}_{\mathrm{max}})\| \overline{q}_1^N\|_{1,\sqrt{\omega}} + (6+2\mathscr{O}_{\mathrm{max}})\|\overline{q}_2^N\|_{1,\sqrt{\omega}} + 6\| \overline{q}_3^N\|_{1,\sqrt{\omega}} + 6\| \overline{q}_4^N\|_{1,\sqrt{\omega}} \\
    &\hspace{+0.5cm}+ 2\| \overline{q}_5^N\|_{1,\sqrt{\omega}} + 2\| \overline{q}_6^N\|_{1,\sqrt{\omega}}+ \mathscr{O}_{\mathrm{max}}\|\overline{q}_{-1}^N\|_{1,\sqrt{\omega}} +2\mathscr{O}_{\mathrm{max}}\|\overline{q}_0^N\|_{1,\sqrt{\omega}},
\end{align}
which is required to estimate the $Z_{1,2}$ bound.
\end{proof}
\renewcommand{\theequation}{C.\arabic{equation}}
\setcounter{equation}{0}
\renewcommand{\theequation}{F.\arabic{equation}}
\setcounter{equation}{0}

\section{Convolution Estimates for Polynomial Weights}\label{apen : algebra estimates}

The purpose of this appendix is to provide explicit upper bounds for the constants $\kappa$ and $\kappa_0$ that appear in the convolution estimates of Section~\ref{sec:conv_estimates}. Specifically, Lemma~\ref{lem : conv_estimates} shows that $\ell^2_\omega$ is a Banach algebra under convolution with constant $\kappa$, and Lemma~\ref{lem : 1 bounded by 2 nu} gives the embedding $\ell^2_\omega \hookrightarrow \ell^1$ with constant $\kappa_0$, where from~\eqref{eq:kappa_def_and_assumption}
\[
\kappa \bydef \sqrt{\sup_{n \in \mathbb{Z}^m} \sum_{k \in \mathbb{Z}^m} \frac{\omega_n}{\omega_{n-k}\,\omega_k}}, \qquad
\kappa_0 \bydef \sqrt{\sum_{n \in \mathbb{Z}^m} \frac{1}{\omega_n}}.
\]
Throughout this appendix, we specialise to the product polynomial weights
\begin{equation} \label{eq:tensor_weights}    
\omega_n \bydef \prod_{j=1}^m (1+|n_j|)^s
\end{equation}
for some $s > 1$. In this appendix, we verify that the weights \eqref{eq:tensor_weights} are admissible, derive an explicit expression for $\kappa_0$ (see \eqref{eq:kappa_0_general_m}), and establish the upper bound \eqref{eq:kappa_bound} for $\kappa$. Admissibility follows from Lemma~\ref{lem:kappa_0_general_m} below, the derivation of $\kappa<\infty$ in Section~\ref{sec:kappa_general_m} and from
\[
\omega_n
=
\prod_{j=1}^{m}(1+|n_j|)^s
\le
\prod_{j=1}^{m}
(1+|n_j-k_j|)^s
(1+|k_j|)^s
=
\omega_{n-k}\omega_k.
\]

The constant $\kappa_0$ is straightforward to compute for general $m$.

\begin{lemma}[{\bf The bound \boldmath{$\kappa_0$} for general \boldmath{$m$}}] \label{lem:kappa_0_general_m}
Let $s>1$ and denote by $\zeta(s)$ the Riemann zeta function. Then $\kappa_0 = \kappa_0(m,s)$ is given by
\begin{equation} \label{eq:kappa_0_general_m}
        \kappa_0 \bydef \left( 2 \zeta(s)-1 \right)^{\frac{m}{2}} < \infty.
\end{equation}
\end{lemma}

\begin{proof}
For $m=1$, we have
\[
\sum_{n \in \mathbb{Z}} \frac{1}{\omega_n} = \sum_{n \in \mathbb{Z}} \frac{1}{(1+|n|)^s}
= 1 + 2 \sum_{n \ge 1} \frac{1}{(1+n)^s} = 1 + 2 \sum_{j \ge 2} \frac{1}{j^s} = 1 + 2 (\zeta(s)-1) =
2 \zeta(s)-1.
\]
For general $m \ge 1$, the conclusion follows by writing
\[
\sum_{n \in \mathbb{Z}^m} \frac{1}{\omega_n} = 
        \sum_{n \in \mathbb{Z}^m} \prod_{j=1}^m \frac{1}{(1+|n_j|)^s} = 
        \prod_{j=1}^m \left( \sum_{n_j \in \mathbb{Z}} \frac{1}{(1+|n_j|)^s} \right)
        =  \left( 2 \zeta(s)-1 \right)^m. \qquad \qedhere
\]
\end{proof}

The computation of $\kappa$ is more involved and is carried out in the next subsection.

\subsection{The constant \texorpdfstring{\boldmath{$\kappa$}}{kappa}} \label{sec:kappa_general_m}

To simplify the presentation, given $m \ge 1$ and $n \in \mathbb{Z}^m$, denote 
\[
S_m(n) \bydef \sum_{k \in \mathbb{Z}^m} \frac{\omega_n}{\omega_k \omega_{n-k}}.
\]
Using the tensor product weights \eqref{eq:tensor_weights}, we obtain
\[
S_m(n) = \prod_{j=1}^m \sum_{k_j \in \mathbb{Z}} \frac{(1+|n_j|)^s}{(1+|k_j|)^s(1+|n_j-k_j|)^s}
= \prod_{j=1}^m S_1(n_j).
\]

Hence, it suffices to bound $S_1(n)$ for $n \in \mathbb{Z}$; the general case $m \ge 2$ then follows by the product structure. Indeed, given any constant $C_s$ satisfying
\begin{equation} \label{eq:definition_C_s_general}
    \sup_{n \in \mathbb{Z}} S_1(n) \le C_s,
\end{equation}
one obtains, for any $m \ge 2$,
\begin{equation}\label{eq:tensor_pre_kappa_bound}
\sup_{n \in \mathbb{Z}^m} S_m(n) = \sup_{n \in \mathbb{Z}^m} \prod_{j=1}^m S_1(n_j) \le
\prod_{j=1}^m \left( \sup_{n_j \in \mathbb{Z}} S_1(n_j) \right) \le C_s^m.
\end{equation}
Combining \eqref{eq:tensor_pre_kappa_bound} with the definition of $\kappa = \kappa(m,s)$ then yields the bound
\begin{equation}\label{eq:kappa_bound}
\kappa = \sqrt{\sup_{n \in \mathbb{Z}^m} \sum_{k \in \mathbb{Z}^m} \frac{\omega_n}{\omega_{n-k}\,\omega_k}} \le  C_s^{\frac{m}{2}}.
\end{equation}

The remainder of this appendix is devoted to finding a constant $C_s$ satisfying \eqref{eq:definition_C_s_general}. We present two such bounds. The first, valid for all $s>1$ and stated in Lemma~\ref{lem:exp_bound_Cs_exp}, yields a simple bound that is exponential in $s$; while not sharp, it is well suited to small values of $s$. The second, stated in Lemma~\ref{lem:Cs_comp} and valid for $s \ge 2$, is sharper and relies on a computer-assisted approach.

\begin{lemma}[\bf Exponential bound]\label{lem:exp_bound_Cs_exp}
For $s > 1$,
\begin{equation}\label{eq:exp_bound_Cs_exp}
    \sup_{n \in \mathbb{Z}} S_1(n) \le
    C_s^{\mathrm{exp}}
    \bydef 2^s(2\zeta(s)-1).
\end{equation}
\end{lemma}

\begin{proof}
Convexity of $x \mapsto x^s$ on $[0,\infty)$ (for $s \ge 1$) gives $(a+b)^s \leq 2^{s-1}(a^s+b^s)$ for all $a,b \geq 0$. Applied to the triangle inequality $1+|n| \leq (1+|k|)+(1+|n-k|)$, this yields
\[
    (1+|n|)^s \leq 2^{s-1}\bigl((1+|k|)^s + (1+|n-k|)^s\bigr).
\]
Dividing both sides by $(1+|k|)^s(1+|n-k|)^s$ gives
\[
    \frac{(1+|n|)^s}{(1+|k|)^s(1+|n-k|)^s} \leq 2^{s-1} \left( \frac{1}{(1+|n-k|)^s} + \frac{1}{(1+|k|)^s} \right).
\]
Summing over $k \in \mathbb{Z}$ and using translation invariance, for every $n \in \mathbb{Z}$,
\[
S_1(n) = \sum_{k \in \mathbb{Z}}
      \frac{(1+|n|)^s}{(1+|k|)^s(1+|n-k|)^s} \leq 2^{s-1} \cdot 2 \sum_{k \in \mathbb{Z}} \frac{1}{(1+|k|)^s} = 2^s(2\zeta(s)-1). \qquad \qedhere
\]
\end{proof}

Since the bound of Lemma~\ref{lem:exp_bound_Cs_exp} deteriorates for large $s$, we develop a computer-assisted approach for the regime $s \ge 2$, following the strategy of \cite{MR2443030,MR3077902,MR3125637}. The method combines a finite sum evaluated via interval arithmetic with an explicit analytic tail bound. We fix the notation needed to state the lemma.

Fix integers $N \ge 3$ and $M \ge 2N$.

\noindent\textbf{Computer-assisted bound for $|n| \leq N$.}
For each $|n| \leq N$, decompose
\[
    S_1(n) = S_{\mathrm{core}}(n,M) + S_{\mathrm{tail}}(n,M),
\]
where
\[
    S_{\mathrm{core}}(n,M) \bydef \sum_{|k| \leq M} \frac{(1+|n|)^s}{(1+|k|)^s(1+|n-k|)^s},
    \qquad
    S_{\mathrm{tail}}(n,M) \bydef \sum_{|k|>M} \frac{(1+|n|)^s}{(1+|k|)^s(1+|n-k|)^s}.
\]
The finite sum $S_{\mathrm{core}}(n,M)$ is evaluated rigorously via interval arithmetic; let $U_{n,M}$ denote the resulting upper bound. For the tail, note that $|n-k| \geq |k|-|n| \ge M + 1 - |n| \geq N+1$ for $|k|>M$, giving $(1+|n-k|)^{-s} \leq (2+M-|n|)^{-s}$. Hence, defining the $n$-dependent prefactor
\[
    A(n,M) \bydef \left(\frac{1+|n|}{1+M-|n|}\right)^{\!s},
\]
we obtain $S_{\mathrm{tail}}(n,M) \leq A(n,M)\sum_{|k|>M} (1+|k|)^{-s}$. Bounding this sum by an integral gives
\[
    \sum_{|k|>M} \frac{1}{(1+|k|)^{s}}
    = 2\sum_{k>M} \frac{1}{(1+k)^{s}}
    \leq 2\int_M^\infty (1+x)^{-s}\,dx
    = \frac{2(1+M)^{1-s}}{s-1}.
\]
Combining these estimates, for each $|n| \leq N$ we define
\begin{equation}\label{eq:app:CsCAPn}
C_s^{\mathrm{CAP}}(n) \bydef U_{n,M} + A(n,M) \frac{2(1+M)^{1-s}}{s-1},
\end{equation}
and observe that $C_s^{\mathrm{CAP}}(n) \ge S_1(n)$.

\noindent\textbf{Uniform analytic bound for $|n| > N$ (valid for $s \geq 2$).}
Since $S_1(-n)=S_1(n)$, it suffices to treat $n > N$. Decompose
\[
    S_1(n)
    = \underbrace{\sum_{k \leq 0} \cdots}_{S_-(n)}
    + \underbrace{\sum_{k=1}^{n-1} \cdots}_{S_0(n)}
    + \underbrace{\sum_{k \geq n} \cdots}_{S_+(n)}.
\]

\emph{Exterior contributions.}
For $k \leq 0$: $(1+|n-k|) = 1+n-k \geq 1+n$, so $(1+n)^s/(1+n-k)^s \leq 1$ and $S_-(n) \leq \sum_{k \leq 0} (1+|k|)^{-s} = \zeta(s)$. The substitution $j=k-n$ gives $S_+(n) \leq \zeta(s)$ by the same argument.

\emph{Interior contribution.}
For $1 \leq k \leq n-1$, set $m_0 \bydef n+2$ and $j \bydef k+1$, so that $j$ ranges over $\{2,\ldots,m_0-2\}$ and $1+n-k = m_0-j$. Then,
\begin{align*}
    S_0(n) &= \sum_{k=1}^{n-1} \frac{(1+n)^s}{(1+k)^s(1+n-k)^s}
     = \sum_{j=2}^{n} \frac{(1+n)^s}{j^s(2+n-j)^s}
     = \sum_{j=1}^{n+1} \frac{(1+n)^s}{j^s(2+n-j)^s} - 2  \\
    & = \sum_{j=1}^{m_0-1} \frac{(m_0-1)^s}{j^s(m_0-j)^s} - 2
    \leq \sum_{j=1}^{m_0-1} \frac{m_0^s}{j^s(m_0-j)^s} - 2.
\end{align*}
By Lemma~24 of~\cite{MR2443030} (applied with $k = m_0 = n+2 > N+2 > 5$), letting $s^* \bydef \lfloor s \rfloor$,
\begin{equation}\label{eq:app:gamma}
    \sum_{j=1}^{m_0-1} \frac{m_0^s}{j^s(m_0-j)^s} \le \gamma_{m_0}
    \bydef 2\!\left(\frac{m_0}{m_0-1}\right)^{\!s}
    + \left(\frac{4\ln(m_0-2)}{m_0}+\frac{\pi^2-6}{3}\right)
      \!\left(\frac{2}{m_0}+\frac{1}{2}\right)^{\!s^*-2}.
\end{equation}
The map $m_0 \mapsto \gamma_{m_0}$ is non-increasing for $m_0 \geq 6$ (see Remark~26 of~\cite{MR2443030}). Taking the supremum over $n > N$ and using $m_0 = n+2 \geq N+3$ yields
\begin{equation}\label{eq:app:largen_m1}
    \sup_{|n|>N} S_1(n)
    \leq 2\zeta(s) - 2 + \gamma_{N+3}, \qquad s \geq 2.
\end{equation}

The bound can now be stated precisely.

\begin{lemma}[{\bf Computer-assisted bound for \boldmath{$s\ge 2$}}] \label{lem:Cs_comp}
Let $N \ge 3$, $s \ge 2$, and let $C_s^{\mathrm{CAP}}(n)$ and $\gamma_{N+3}$ be as defined in \eqref{eq:app:CsCAPn} and \eqref{eq:app:gamma}. Setting
\begin{equation}\label{eq:Cscomp_explicit}
    C_s^{\mathrm{comp}}
    \bydef \max \left\{ \max_{|n| \leq N} C_s^{\mathrm{CAP}}(n),\; 2\zeta(s) - 2 + \gamma_{N+3} \right\},
\end{equation}
we have $\sup_{n \in \mathbb{Z}} S_1(n) \le C_s^{\mathrm{comp}}$.
\end{lemma}

\begin{proof}
For $|n| \leq N$, the bound $S_1(n) \leq C_s^{\mathrm{CAP}}(n)$ follows from \eqref{eq:app:CsCAPn}. For $|n| > N$, the bound $S_1(n) \leq 2\zeta(s) - 2 + \gamma_{N+3}$ follows from \eqref{eq:app:largen_m1}. Taking the maximum of the two bounds gives $\sup_{n \in \mathbb{Z}} S_1(n) \leq C_s^{\mathrm{comp}}$.
\end{proof}

\nocite{julia_interval}
\nocite{julia_blanco_D3D6}
\nocite{julia_cadiot_blanco_GS}
\bibliographystyle{abbrv}
\bibliography{biblio}
\end{document}